\documentclass[a4paper,twosided]{amsbook}
\usepackage[utf8]{inputenc}
\usepackage[T1]{fontenc}

\usepackage{amsthm,amsmath,amsfonts,amssymb} 

\usepackage{amsbooka} 

\usepackage[original]{imakeidx} 

\usepackage{bbm} 
\usepackage{mathrsfs} 

\usepackage{xspace} 

\usepackage{enumitem} 
\setlist[itemize]{leftmargin=5ex,label=\textbullet}
\setlist[enumerate]{left=\parindent .. 2.5\parindent,align=left,label=\rm{(\arabic*)},ref=\rm{(\arabic*)}}
\newlist{enumerate-roman}{enumerate}{2}
\setlist[enumerate-roman]{left=\parindent .. 2.5\parindent,align=left,label=\rm{(\roman*)},ref=\rm{(\roman*)}}

\usepackage{tikz-cd} 

\usepackage{color} 
\definecolor{mybrown}{rgb}{.8,.47,.18} 
\definecolor{myblue}{rgb}{.31,.22,.82} 
\definecolor{myred}{rgb}{.76,.22,.16} 

\usepackage[colorlinks=true,citecolor=myblue,linkcolor=myblue,urlcolor=myblue,backref=page]{hyperref} 
\usepackage[alphabetic,backrefs]{amsrefs} 

\newtheoremstyle{thms}
	{}
	{}
	{\itshape} 
	{} 
	{\scshape } 
	{.} 
	{ } 
	{} 

\newtheoremstyle{defs}
	{}
	{}
	{\normalfont} 
	{} 
	{\scshape } 
	{.} 
	{ } 
	{} 

\newtheoremstyle{defs}
	{}
	{}
	{\normalfont} 
	{} 
	{\scshape } 
	{.} 
	{ } 
	{} 

\newtheoremstyle{rems}
	{}
	{}
	{\normalfont} 
	{} 
	{\scshape } 
	{.} 
	{ } 
	{} 

\theoremstyle{thms}
\newtheorem{thm}{Theorem}[subsection]
\newtheorem*{thm*}{Theorem}

\newtheorem{prop}[thm]{Proposition}
\newtheorem*{prop*}{Proposition}
\newtheorem{lem}[thm]{Lemma}
\newtheorem{coro}[thm]{Corollary} 
\newtheorem*{coro*}{Corollary}

\theoremstyle{defs}
\newtheorem{dfn}[thm]{Definition} 
\newtheorem{nota}[thm]{Notation}
\newtheorem{conv}[thm]{Convention} 
\newtheorem{num}[thm]{} 

\theoremstyle{rems}
\newtheorem{rem}[thm]{Remark}
\newtheorem{exm}[thm]{Example}

\newcommand{\introtag}[1]{\smallskip\noindent {\hypertarget{#1}{{\bf(#1)}}}}
\newcommand{\introref}[1]{\hyperlink{#1}{{\bf(#1)}}}

\newcommand{\chfinitecw}{Ch.~\ref{ch:finitecw}\xspace}
\newcommand{\eqrefdmt}[1]{Ch.~\ref{ch:dmt}, \eqref{#1}\xspace}
\newcommand{\chdmt}{Ch.~\ref{ch:dmt}\xspace}
\newcommand{\chcancellation}{Ch.~\ref{ch:cancellation}\xspace}
\newcommand{\chcomparison}{Ch.~\ref{ch:comparison}\xspace}
\newcommand{\chspectra}{Ch.~\ref{ch:spectra}\xspace}

\newcommand{\NN}{\mathbb N}
\newcommand{\ZZ}{\mathbb Z}
\newcommand{\QQ}{\mathbb Q}

\newcommand{\id}{\mathrm{id}} 

\DeclareMathOperator{\Hom}{Hom}
\DeclareMathOperator{\uHom}{\underline{Hom}} 
\newcommand{\homm}[3]{\mathrm{Hom}_{#1}(#2,#3)}
\DeclareMathOperator{\End}{End}
\DeclareMathOperator{\coker}{coker}
\DeclareMathOperator*{\colim}{colim}
\DeclareMathOperator*{\hocolim}{hocolim}

\newcommand{\Ab}{\mathrm{Ab}} 
\newcommand{\sAb}{\mathrm{sAb}} 

\newcommand{\unit}{\eta}
\newcommand{\counit}{\epsilon}
\newcommand{\tunit}{\tilde\eta}
\newcommand{\tcounit}{\tilde\epsilon}

\newcommand{\dimn}[1]{\mathrm{dim}(#1)}
\newcommand{\charac}[1]{\mathrm{char}(#1)}

\newcommand{\Spec}{\mathrm{Spec}}
\newcommand{\OO}{\mathcal O}
\newcommand{\AAA}{\mathbb A}
\newcommand{\Aone}{\mathbb{A}^1}
\newcommand{\PP}{\mathbb P} 
\newcommand{\Gm}{\mathbb{G}_m} 
\newcommand{\Gmx}[1]{\mathbb{G}_{m,#1}} 
\newcommand{\Gmk}{\Gmx{k}} 
\newcommand{\Gmpt}{\mathbb{G}_{m,1}} 

\newcommand{\cX}{\mathcal X} 

\newcommand{\Lb}{\mathcal{L}} 
\newcommand{\Mb}{\mathcal{M}} 
\newcommand{\Nb}{\mathcal{N}} 
\newcommand{\Frob}{\mathrm{Fr}} 
\newcommand{\Or}{\mathrm{Q}} 

\DeclareMathOperator{\Th}{Th} 
\newcommand{\fTh}{\mathcal{T}h} 
\DeclareMathOperator{\cK}{\mathcal K} 

\DeclareMathOperator{\Pic}{Pic} 

\newcommand{\tdiv}{\widetilde{\mathrm{div}}} 
\newcommand{\tord}{\widetilde{\mathrm{ord}}} 

\newcommand{\K}{\mathrm{K}} 
\newcommand{\KM}{\mathrm{K}^{\mathrm{M}}} 
\newcommand{\KMW}{\mathrm{K}^{\mathrm{MW}}} 
\newcommand{\I}{\mathrm{I}} 

\newcommand{\Norm}{\mathrm{N}} 

\newcommand{\sheaf}{\mathbf} 
\newcommand{\sKM}{\sheaf{K}^{\mathrm{M}}} 
\newcommand{\sKMW}{\sheaf{K}^{\mathrm{M\hspace{-.2ex}W}}} 
\newcommand{\sKW}{\sheaf{K}^{\mathrm{W}}}
\newcommand{\sW}{\sheaf{W}}
\newcommand{\sI}{\sheaf{I}} 
\newcommand{\sIbar}{\overline{\sI}} 

\newcommand{\GW}{\mathrm{GW}} 
\newcommand{\rk}{\mathrm{rk}} 

\newcommand{\W}{\mathrm{W}} 
\newcommand{\Tr}{\mathrm{Tr}} 

\newcommand{\tch}[1]{%
  \mathchoice{\widetilde{\mathrm{CH}}^{\raisebox{-.5ex}{$\scriptstyle#1$}}}
             {\widetilde{\mathrm{CH}}^{\raisebox{-.8ex}{$\scriptstyle#1$}}}
             {}
             {}
}
\newcommand{\tchi}[2]{%
  \mathchoice{\widetilde{\mathrm{CH}}_{#2}^{\raisebox{-.5ex}{$\scriptstyle#1$}}}
             {\widetilde{\mathrm{CH}}_{#2}^{\raisebox{-.7ex}{$\scriptstyle#1$}}}
             {}
             {}
}
\newcommand{\ch}[2]{\tch{#1}(#2)} 
\newcommand{\cht}[3]{\tch{#1}(#2,#3)} 
\newcommand{\chs}[3]{\tchi{#1}{#2}(#3)} 
\newcommand{\chst}[4]{\tchi{#1}{#2}(#3,#4)} 
\newcommand{\chht}[3]{\tchi{}{#1}(#2,#3)} 
\newcommand{\chhtnotw}[2]{\tchi{}{#1}(#2)} 

\DeclareMathOperator{\CH}{CH}

\newcommand{\sm}[1]{\mathrm{Sm}_{#1}}
\newcommand{\smg}{\mathrm{Sm}} 
\newcommand{\smk}{\mathrm{Sm}_k} 
\newcommand{\barsm}[1]{\overline{\mathrm{Sm}}_{#1}}

\newcommand{\cor}[1]{\widetilde{\mathrm{Cor}}_{#1}} 
\newcommand{\cork}{\cor k} 
\newcommand{\corkR}{\cor{k,R}} 
\newcommand{\Fr}{\mathrm{Fr}} 
\newcommand{\ZFr}{\ZZ\mathrm{F}} 

\newcommand{\icor}[1]{\widetilde{\mathrm{ICor}}_{#1}} 
\newcommand{\corV}[1]{\mathrm{Cor}_{#1}} 
\newcommand{\corVk}{\corV k} 
\newcommand{\corVkR}{\corV{k,R}} 
\newcommand{\Wcork}{\mathrm{WCor}_k} 
\newcommand{\grph}[1]{\tilde{\gamma}_{#1}} 
\newcommand{\Adm}{\mathcal{A}} 

\newcommand{\supp}{\mathrm{supp}} 
\newcommand{\Vanish}{\mathrm{V}} 

\newcommand{\psh}{\mathrm{PSh}} 
\newcommand{\pshkR}{\psh(k,R)} 
\newcommand{\pshMW}{\widetilde{\mathrm{PSh}}} 
\newcommand{\pshMWk}{\pshMW(k)} 
\newcommand{\pshMWkR}{\pshMW(k,R)} 
\newcommand{\pshV}{\mathrm{PSh}^{\mathrm{tr}}} 
\newcommand{\pshVkR}{\pshV(k,R)} 
\newcommand{\pshfr}{\psh^{\mathrm{Fr}}} 
\newcommand{\pshfrkR}{\pshfr(k,R)} 
\newcommand{\QpshfrkR}{\mathcal{Q}\pshfrkR} 
\newcommand{\SpshfrkR}{\mathcal{S}\pshfrkR} 

\DeclareMathOperator{\MWcorr}{\tilde{\mathrm{c}}} 
\DeclareMathOperator{\Vcorr}{\mathrm{c}} 
\DeclareMathOperator{\Icorr}{\mathrm{I}\tilde{\mathrm{c}}}

\newcommand{\MWprep}{\MWcorr} 

\newcommand{\Rt}{R_t} 

\newcommand{\MWrepZt}{\widetilde{\ZZ}_t} 
\newcommand{\MWrepZ}{\widetilde{\ZZ}} 
\newcommand{\MWrepRt}{\widetilde{R}_t} 
\newcommand{\MWrepR}{\widetilde{R}} 
\newcommand{\IrepRt}{\mathrm{I}\widetilde{R}_t}
\newcommand{\VrepRt}{R^{\mathrm{tr}}_t} 
\newcommand{\VrepR}{R^{\mathrm{tr}}} 
\newcommand{\VrepZ}{\ZZ_{\mathrm{tr}}} 

\newcommand{\tZ}{\widetilde\ZZ} 
\newcommand{\tR}{\widetilde R} 

\newcommand{\tZcbx}[1]{\tZ\{#1\}} 
\newcommand{\tRcbx}[1]{\tR\{#1\}} 
\newcommand{\tRcbtx}[1]{\tR\{#1\}} 
\newcommand{\ItRcbx}[1]{\mathrm{I}\tR\{#1\}} 
\newcommand{\ItRcbtx}[1]{\mathrm{I}\tR_t\{#1\}} 
\newcommand{\tZxcby}[2]{\tZ_{#1}\{#2\}} 
\newcommand{\RVcbx}[1]{R^{\mathrm{tr}}\{#1\}} 
\newcommand{\RVcbtx}[1]{R_t^{\mathrm{tr}}\{#1\}} 

\newcommand{\tZpx}[1]{\tZ(#1)} 
\newcommand{\tRx}[1]{\widetilde R(#1)} 
\newcommand{\tZxpy}[2]{\tZ_{#1}(#2)} 

\newcommand{\ItZ}{\mathrm{I}\tZ}

\newcommand{\tZtr}{\tZ_{\mathrm{tr}}} 

\newcommand{\nis}{\mathrm{Nis}}
\newcommand{\zar}{\mathrm{Zar}}
\newcommand{\et}{\mathrm{\acute{e}t}}
\newcommand{\cdh}{\mathrm{cdh}}

\newcommand{\sh}{\mathrm{Sh}} 
\newcommand{\shxkR}[1]{\sh_{#1}(k,R)} 
\newcommand{\shtkR}{\shxkR{t}} 
\newcommand{\shMW}{\widetilde{\sh}} 
\newcommand{\shMWtk}{\shMW_t(k)}
\newcommand{\shMWzk}{\shMW_{\zar}(k)} 
\newcommand{\shMWzL}{\shMW_{\zar}(L)} 
\newcommand{\shMWxkR}[1]{\shMW_{#1}(k,R)}
\newcommand{\shMWxk}[1]{\shMW_{#1}(k)}
\newcommand{\shMWtkR}{\shMWxkR{t}}
\newcommand{\shMWkZ}{\shMW(k,\ZZ)}
\newcommand{\shV}{\sh^{\mathrm{tr}}} 
\newcommand{\shVk}{\shV(k)} 
\newcommand{\shVxkR}[1]{\shV_{#1}(k,R)}
\newcommand{\shVtkR}{\shVxkR{t}}

\newcommand{\tfO}{\tilde{\mathcal O}} 
\newcommand{\fO}{\mathcal O} 

\DeclareMathOperator{\otr}{\tilde \otimes} 

\newcommand{\Ci}[1]{\operatorname{C}_{#1}} 
\newcommand{\Cstar}[1]{\Ci{*}(#1)} 
\newcommand{\CstarS}{\operatorname{C}^{\mathrm{0}}_*} 
\newcommand{\CBM}{\operatorname{C}_{\mathrm{BM}}} 
\DeclareMathOperator{\Comp}{C} 
\DeclareMathOperator{\KComp}{K} 
\newcommand{\KCompb}{\KComp^{\mathrm{b}}} 
\DeclareMathOperator{\Der}{D} 

\newcommand{\Cger}{\mathrm{C}} 

\newcommand{\ome}[2]{\omega_{#1/#2}} 
\newcommand{\piso}{\mathfrak p} 

\newcommand{\ext}{\mathrm{ext}}
\newcommand{\res}{\mathrm{res}}
\newcommand{\bc}{\mathrm{bc}}
\newcommand{\tr}{\mathrm{tr}}

\renewcommand{\H}{\mathrm{H}} 
\newcommand{\HM}{\H_{\mathrm{M}}} 
\newcommand{\HMet}{\H_{\mathrm{M},\et}} 
\newcommand{\HMW}{\H_{\mathrm{MW}}} 
\newcommand{\HMWs}[1]{\H_{\mathrm{MW},#1}} 
\newcommand{\HMWc}{\H_{\mathrm{MW,c}}} 
\newcommand{\HMWet}{\H_{\mathrm{MW,\et}}} 
\newcommand{\HhMW}{\H^{\mathrm{MW}}} 
\newcommand{\HhMWBM}{\H^{\mathrm{MW,BM}}} 
\newcommand{\HW}{\H_{\mathrm{W}}} 
\newcommand{\HI}{\H_{\mathrm{I}}} 

\newcommand{\hypH}{\mathbb H} 

\newcommand{\bH}{\sheaf{H}}
\newcommand{\shHMW}{\bH_{\mathrm{MW}}} 

\newcommand{\spectrum}{\mathscr}
\newcommand{\HH} {\spectrum H}
\newcommand{\EE} {\spectrum E}
\newcommand{\EEBM} {\spectrum E^{\mathrm{BM}}}
\newcommand{\FF} {\spectrum F}
\newcommand{\FFBM} {\spectrum F^{\mathrm{BM}}}
\newcommand{\un}{\mathbbm 1} 

\newcommand{\sHM}{\spectrum H_{\mathrm{M}}} 
\newcommand{\HtZ}{\H\tZ} 
\newcommand{\sHMBM}{\spectrum H^{\mathrm{M},\mathrm{BM}}} 
\newcommand{\sHMx}[1]{\spectrum H_{\mathrm{M},#1}} 
\newcommand{\sHMW}{\spectrum H_{\mathrm{MW}}} 
\newcommand{\sHMWBM}{\spectrum H^{\mathrm{MW},\mathrm{BM}}} 
\newcommand{\sHlMW}{\spectrum H^{\mathrm{MW}}} 
\newcommand{\sHMet}{\spectrum H_{\mathrm{M},\et}} %
\newcommand{\sHMWet}{\spectrum H_{\mathrm{MW},\et}} %
\newcommand{\sHMWc}{\spectrum H_{\mathrm{MW},c}} %
\newcommand{\sHAone}{\spectrum H_{\Aone}} %
\newcommand{\sHBMAone}{\spectrum H^{\mathrm{BM},\Aone}} %

\newcommand{\HstAone}{\H_{\mathrm{st}-\Aone}} 
\newcommand{\spKM}{\spectrum{K}^{\mathrm{M}}} 
\newcommand{\spKMW}{\spectrum{K}^{\mathrm{M\hspace{-.2ex}W}}} 

\newcommand{\sKO}{\spectrum{K\hspace{-.6ex}O}} 

\newcommand{\m}{\mathcal{M}} 
\newcommand{\Mk}{\m_k} 
\newcommand{\fMk}{\mathrm{f}\m_k} 
\newcommand{\tMk}{\widetilde{\m}^{\mathrm{tr}}_k} 
\newcommand{\MFk}{\mathbf{MF}_k} 
\newcommand{\MFMWZ}{\mathbf{M}\MWrepZ} 
\newcommand{\Mheart}{\mathbf{M}^{\heartsuit}} 
\newcommand{\MheartZ}{\Mheart\MWrepZ} 

\newcommand{\bpi}{\boldsymbol{\pi}} 
\newcommand{\piaone}{{\bpi}^{\Aone}} 

\newcommand{\spMW}{\widetilde{\mathrm{Sp}}} 
\newcommand{\spMWxkR}[1]{\spMW_{#1}(k,R)}
\newcommand{\spMWkR}{\spMWxkR{t}}

\newcommand{\DAone}{\Der_{\Aone}} 
\newcommand{\DA}[1]{\DAone(#1)}
\newcommand{\DAkR}{\DA{k,R}}
\newcommand{\DAk}{\DA{k}} 
\newcommand{\DAket}{\Der_{\Aone,\et}(k)} 
\newcommand{\DAxkR}[1]{\Der_{\Aone,#1}(k,R)} 
\newcommand{\DAe}[1]{\Der_{\Aone}^{\mathrm{eff}}(#1)} 
\newcommand{\DAek}{\DAe{k}} 
\newcommand{\DAekR}{\DAe{k,R}} 
\newcommand{\DAexkR}[1]{\Der_{\Aone,#1}^{\mathrm{eff}}(k,R)} 
\newcommand{\DM}{\operatorname{DM}} 
\newcommand{\DMk}{\DM(k)} 
\newcommand{\DMkR}{\DM(k,R)} 
\newcommand{\DMkZ}{\DM(k,\ZZ)} 
\newcommand{\DMxkR}[1]{\DM_{#1}(k,R)} 
\newcommand{\DMeff}{\DM^{\mathrm{eff}}} 
\newcommand{\DMe}[1]{\DMeff(#1)} 
\newcommand{\DMekR}{\DMe{k,R}} 
\newcommand{\DMexkR}[1]{\DMeff_{#1}(k,R)} 
\newcommand{\DMgm}[1]{\operatorname{DM}_{\mathrm{gm}}(#1)}
\newcommand{\DMgmk}{\DMgm{k}}

\newcommand{\DMegm}[1]{\operatorname{DM}_{\mathrm{gm}}^{\mathrm{eff}}(#1)}
\newcommand{\DMegmk}{\DMegm{k}}

\newcommand{\DMme}[1]{\operatorname{DM}_{-}^{\mathrm{eff}}(#1)} 

\newcommand{\DMcdh}{\operatorname{DM}_{\mathrm{cdh}}}

\newcommand{\Dmtilde}{\widetilde{\operatorname{DM}}} 
\newcommand{\DMt}[1]{\Dmtilde(#1)} 
\newcommand{\DMtkR}{\DMt{k,R}}
\newcommand{\DMtk}{\DMt{k}}
\newcommand{\DMtkZ}{\DMt{k,\ZZ}}
\newcommand{\DMtxkR}[1]{\Dmtilde_{#1}(k,R)} 
\newcommand{\Dmtildeeff}{\widetilde{\operatorname{DM}}^{\raisebox{-.6ex}{$\scriptstyle \mathrm{eff}$}}} 
\newcommand{\DMte}[1]{\Dmtildeeff(#1)} 
\newcommand{\DMtekR}{\DMte{k,R}} 
\newcommand{\DMtekZ}{\DMte{k,\ZZ}} 
\newcommand{\DMtekinv}{\DMte{k,\ZZ[1/e]}} 
\newcommand{\DMtek}{\DMte{k}} 
\newcommand{\DMtex}[2]{\Dmtildeeff_{#1}(#2)} 
\newcommand{\DMtexkR}[1]{\DMtex{#1}{k,R}} 
\newcommand{\DMtgm}[1]{\widetilde{\operatorname{DM}}_{\mathrm{gm}}(#1)} 
\newcommand{\DMtgmk}{\DMtgm{k}}

\newcommand{\DMtegm}[1]{\widetilde{\operatorname{DM}}_{\mathrm{gm}}^{\raisebox{-.6ex}{$\scriptstyle \mathrm{eff}$}}(#1)} 
\newcommand{\DMtegmk}{\DMtegm{k}}

\newcommand{\DMtegmkZ}{\DMtegm{k,\ZZ}}
\newcommand{\WDM}{\mathrm{WDM}} %
\newcommand{\WDMkR}{\WDM(k,R)} %
\newcommand{\WDMe}{\WDM^{\mathrm{eff}}} %
\newcommand{\WDMekR}{\WDMe(k,R)} %

\newcommand{\DMGWk}{\DM^{\mathrm{GW}}(k)} 

\newcommand{\Mot}{\mathrm{M}} 
\newcommand{\Motc}{\mathrm{M}^{\mathrm{c}}} 
\newcommand{\tMot}{\widetilde{\mathrm{M}}} 

\newcommand{\cT}{\mathcal{T}} 
\newcommand{\cP}{\mathcal{P}} 
\newcommand{\tP}{\tilde{\mathcal{P}}} 
\newcommand{\glb}{\mathcal Det_{\ZZ}} 

\DeclareMathOperator{\SH}{SH}
\newcommand{\SHk}{\SH(k)}
\newcommand{\SHS}{\SHk^{\mathrm{S}^1}}
\newcommand{\SHkeff}{\SHk^{\mathrm{eff}}}
\newcommand{\Spt}{\mathrm{SH}} 
\newcommand{\eff}{\mathrm{eff}}

\newcommand{\derL}{\mathbf{L}} 
\newcommand{\derR}{\mathbf{R}} 

\newcommand{\T} {\mathfrak T} 

\newcommand{\cupp}{{\scriptstyle \cup}} 
\newcommand{\capp}{{\scriptstyle \cap}} 

\makeindex
\makeindex[name=notation,title=Index of Notation]

\begin{document}

\frontmatter

\title{Milnor-Witt Motives}


\author{Tom Bachmann}
\address{Fakult\"at Mathematik \\ Universit\"at Duisburg-Essen, Campus Essen \\ Thea-Leymann-Strasse 9 \\ 45127 Essen}
\email{tom.bachmann@zoho.com}

\author{Baptiste Calmès}
\address{Univ.\ Artois\\
UR 2462, Laboratoire de Mathématiques de Lens (LML)\\
F-62300 Lens\\
France}
\email{baptiste.calmes@univ-artois.fr}

\author{Frédéric Déglise}
\address{Institut Mathématique de Bourgogne - UMR 5584\\
Université de Bourgogne\\
9 avenue Alain Savary\\
BP 47870\\
21078 Dijon Cedex\\
France}
\email{frederic.deglise@ens-lyon.fr}

\author{Jean Fasel}
\address{Institut Fourier-UMR 5582\\ 
Université Grenoble-Alpes\\
CS 40700\\
38058 Grenoble Cedex 9\\
France}
\email{jean.fasel@univ-grenoble-alpes.fr}

\author{Paul Arne Østvær}
\address{Department of Mathematics F. Enriques\\
University of Milan\\
Italy}
\address{
Department of Mathematics\\ 
University of Oslo\\
Norway
}
\email{paularne@math.uio.no}
\email{paul.oestvaer@unimi.it}

\date{\today}

\subjclass[2010]{Primary: 11E70, 13D15, 14F42, 19E15, 19G38; Secondary: 11E81, 14A99, 14C35, 19D45}

\keywords{Finite correspondences, Milnor-Witt $K$-theory, Chow-Witt groups, Motivic cohomology}

\dedicatory{
\newpage
To Andrei Suslin and Vladimir Voevodsky%
}

\begin{abstract}
We develop the theory of Milnor-Witt motives and motivic cohomology. Compared to Voevodsky's theory of motives and his motivic cohomology, the first difference appears in our definition of Milnor-Witt finite correspondences, where our cycles come equipped with quadratic forms. This yields a weaker notion of transfers and a derived category of motives that is closer to the stable homotopy theory of schemes. We prove a cancellation theorem when tensoring with the Tate object, we compare the diagonal part of our Milnor-Witt motivic cohomology to Minor-Witt K-theory and we provide spectra representing various versions of motivic cohomology in the $\AAA^1$-derived category or the stable homotopy category of schemes.
\end{abstract}

\maketitle

\tableofcontents

\renewcommand{\theequation}{\Alph{equation}}

\chapter{Introduction}
\label{ch:intro}

\section{Beilinson and Voevodsky's motivic complexes}
\label{sec:VoevodskyTheory}

The modern form of motivic cohomology takes its roots in the discovery of higher K-theory, notably by Quillen (1972), and the almost concomitant first conjectures of Lichtenbaum on special values of L-functions (1973, \cite{Lic1}).
The initial conjectures of Lichtenbaum were soon confirmed by Borel (1977, \cite{Bor1}) and extended by Bloch (1980, \cite{Bl80} for example).
As we know, the work of Bloch --- and probably works of Deligne (see \cite{Del1},
in particular section 4.3) --- inspired Beilinson to formulate the existence of motivic cohomology, as extension groups in a conjectural category of mixed motives, bound to satisfy a very precise formalism;
the latter constitutes the well-known Beilinson conjectures (first formulated in 1982, \cite{Be87}).
In line with the initial conjectures of Lichtenbaum, they stipulate in particular that rational motivic cohomology is given by the $\gamma$-graded parts of Quillen higher K-theory\footnote{See below, formula \eqref{eq:motivic_coh&Kth} for the precise formulation.}
(Beilinson, and Soulé \cite{Sou}).
In fact, the existence of the higher Chern character (Gillet, \cite{Gil1})
provided the hoped-for universal property of the latter.

A decisive step in the theory, going beyond Beilinson's definition of rational motivic cohomology,
was the introduction of higher Chow groups by Bloch in 1984 (see \cite{Bl80})
through an explicit complex of cycles.\footnote{There was also a definition of Landsburg of these
 complexes, which remained in a preprint form for many years; it was finally published in 
 \cite{Lan} after an important technical modification and using the paper of Bloch.}
These complexes were the first occurrence, with integral coefficients,
of Beilinson's conjectural motivic complexes.
The construction of Bloch quickly inspired several further works among which is
the introduction by Suslin of a homology theory modeled both on higher Chow groups
and on singular homology (1987, in a Luminy talk).

Suslin's definition ultimately led Voevodsky to his brilliant realization of Beilinson's program,
over a perfect field $k$ and with coefficients in an arbitrary ring $R$ \footnote{Note that Voevodsky had first introduced
the theory of (finite type) rational h-motives, over arbitrary bases,
in his PhD. This theory was finally proved to satisfy all the requirements of Beilinson,
save the existence of the motivic t-structure, in \cite{CD12}.},
which first appeared in Chapter 5 of the cornerstone book \cite{FSV}.
Let us summarize what Voevodsky did in this work (specializing to integral coefficients for simplicity), and its relation with the Beilinson's conjectures,
 (see \cite{Be87}*{5.10}):

\introtag{V1} He defines the triangulated category of \emph{(mixed) geometric motives} (resp.\ \emph{effective geometric motives}) $\DMgmk$
\index{motive!geometric}%
\index{motives, category of!(mixed) geometric}%
\index[notation]{dmgeomk@$\DMgmk$}%
(resp.\ $\DMegmk$). 
\index{motive!effective geometric}%
\index{motives, category of!(mixed) effective geometric}%
\index[notation]{dmeffgeomk@$\DMegmk$}%
This category is conjecturally the derived category of
the abelian category of mixed motives over $k$ with integral coefficients,
as demanded in part A of Beilinson's conjectures.%
\footnote{At present, the \emph{motivic} $t$-structure 
\index{motivic!$t$-structure}%
on $\DMgmk$ is missing.}

Note that a fundamental theorem of Voevodsky, not part of Beilinson's conjectures,
is the so-called cancellation theorem: under the resolution of singularities assumption, the natural functor
\[
\DMegmk \rightarrow \DMgmk
\]
is fully faithful; \cite{FSV}*{Chap.~5, 4.3.1}. This theorem was extended later to an arbitrary
perfect field $k$ in \cite{Voe10}.

\introtag{V2} He defines a triangulated category of ``big'' motives, called \emph{motivic complexes} and denoted by $\DMme k$
\index{motives, category of!big}%
\index{motivic complex}%
\index{motives, category of!effective}
\index[notation]{dmeff@$\DMme k$}%
which contains $\DMegmk$ as a full subcategory.
The category $\DMme k$ contains in particular the Tate twist motivic complexes $\ZZ(n)$ 
\index{motive!Tate}%
\index{motivic complex!Tate twist}%
\index[notation]{zn@$\ZZ(n)$}%
for $n \in \ZZ$. 
In the book, Voevodsky proved the properties asked by Beilinson in part D of his conjectures
with the exception of the Beilinson-Soulé vanishing property (\emph{op.~cit.}, 5.10, D, (ii)) and the Beilinson-Lichtenbaum conjecture (\emph{op.~cit.}, 5.10, D, (v)).\footnote{The Beilinson-Soulé vanishing property is currently unknown. The Beilinson-Lichtenbaum conjecture has been proved by Voevodsky in \cite{Voe1}.} Let us summarize and be more precise:
\begin{itemize}
\item Voevodsky's complexes are in particular complexes of sheaves for the Nisnevich topology;
 Beilinson asked for the  Zariski topology, but one can consider Voevodsky's complexes in the latter
 thanks to \cite{FSV}*{Chap.~5, 3.2.7}.
\item The Tate motivic complexes satisfy the following computations:
\[
\ZZ(n)=\begin{cases}
0 & \text{if $n<0$.\footnotemark} \\
\ZZ & \text{if $n=0$.} \\
\Gm[-1] & \text{if $n=1$.\footnotemark}
\end{cases}
\addtocounter{footnote}{-1}
\footnotetext{This follows by definition.}
\stepcounter{footnote}
\footnotetext{cf.\ \cite{Voe1}*{3.4.2, 3.4.3}}
\]
The next property, hinted at by Beilinson's conjecture, part D, (iii),
 was not proved in the book \cite{FSV}, but rather in another paper of Suslin and Voevodsky
 \cite{SV2}*{Thm.~3.4}: for any $n \geq 0$, if one views $\ZZ(n)$ as a complex of Nisnevich sheaves
 on the site of smooth $k$-schemes, one has an isomorphism:
\[
\underline{\H}^n(\ZZ(n)) \simeq \sKM_n
\]
\index{Milnor K-theory!sheaf}%
\index[notation]{kms@$\sKM_n$}%
where the right hand side denotes the $n$-th graded piece of
the unramified (pre)sheaf on the site of smooth $k$-schemes
associated with Milnor K-theory (see for example \cite{Ros}, pp. 360
for the definition and section 12 for its structure as a presheaf).
\end{itemize}
\introtag{V3} Voevodsky did not construct the six functors formalism demanded in part A of the Beilinson's
conjecture\footnote{in particular, he did not define the category relative to a base. This was done
later, following many indications and previous definitions of Voevodsky,
in \cite{CD12} and \cite{CD15}, based on constructions of \cite{FSV}*{Chap.~2} and \cite{Ayoub}.}
but anyway establishes many of its formal consequences. Here is a dictionary:
\begin{itemize}
\item Let $p:X \rightarrow \Spec(k)$ be a smooth scheme. Voevodsky defines its associated homological motive $\Mot(X)$
\index{motive!homological}%
\index[notation]{mx@$\Mot(X)$}%
which in terms of the six functors is: $\Mot(X)=p_!p^!(\ZZ)$,
where $\ZZ$ is the constant motive. Then Voevodsky proves the projective bundle and blow-up formulas
and construct the blow-up and Gysin triangles (see \cite{FSV}*{Ch.~5, \textsection 3.5}).\footnote{In
terms of Grothendieck six functors formalism,
these formulas are ``known'' consequences of the localization triangle,
together with the (oriented) relative purity isomorphism (for a smooth morphism);
see for example \cite{CD12}*{A.5.1}, respectively point (Loc) and point (4) for their formulation.}
\item Let $p:X \rightarrow \Spec(k)$ be a separated morphism of finite type.\footnote{In \cite{FSV},
one does not pay attention to the assumptions that $k$-schemes are separated or not, though
this is used for example when one assumes that the graph of a morphism $f:X \rightarrow Y$
of such $k$-schemes defines a closed subscheme of $X \times_k Y$. This is why we feel free to add
this assumption in this introduction.} Voevodsky defines the associated (homological) motive $\Mot(X)$,
and compactly supported (homological) motive $\Motc(X)$, which in terms of the six functors formalism
would be\footnote{see \cite{CD15}*{8.7, 8.10}, for this result, up to inverting the characteristic of $k$ if it is positive.}:
\begin{align*}
\Mot(X):=p_!p^!(\ZZ), \qquad \Motc(X) =p_*p^!(\ZZ).
\end{align*}
\index{motive!compactly supported}%
\index[notation]{mxc@$\Motc(X)$}%
Under a resolution of singularities assumption (true in characteristic $0$),
he also proves some formulas for these motives which would follow from the six functors formalism:
functoriality, blow-up distinguished triangle, K\"unneth formula (\cite{FSV}*{Ch.~5, Section~4}),
duality (loc.~cit.\ Thm.~4.3.5).
\end{itemize}

\introtag{V4} Voevodsky not only defines motivic cohomology but also, in collaboration
with Friedlander, four motivic theories
(\cite{FSV}*{Chap.~4, \textsection 9}) associated with
a $k$-scheme of finite type $X$ and a couple $(n,p) \in \ZZ^2$:
\begin{itemize}
\item \emph{motivic cohomology}:
\[\mathrm{H}^p(X,\ZZ(n))=\Hom_{\DMgmk}(\Mot(X),\ZZ(n)[p])\]
\item \emph{motivic cohomology with compact support}:
\[\mathrm{H}^p_c(X,\ZZ(n))=\Hom_{\DMgmk}(\Motc(X),\ZZ(n)[p])\]
\item \emph{motivic homology}:
\[\mathrm{H}_p(X,\ZZ(n))=\Hom_{\DMgmk}(\ZZ(n)[p],\Mot(X))\]
\item \emph{motivic Borel-Moore homology}:
\[\mathrm{H}_p^{\mathrm{BM}}(X,\ZZ(n))=\Hom_{\DMgmk}(\ZZ(n)[p],\Motc(X)).\]
\end{itemize}
\index{motivic!cohomology}%
\index{motivic!cohomology!with compact support}%
\index{motivic!homology}%
\index{motivic!homology!Borel-Moore}%
\index{Borel-Moore|see{motivic Borel-Moore homology}}%
Assuming resolution of singularities, Friedlander and Voevodsky prove
the whole formalism expected from these theories (functoriality, localization
and blow-up long exact sequences, duality).\footnote{The results of Friedlander
and Voevodsky have been extended when $k$ has positive characteristic $p$
in the work of Kelly \cite{Kel17}.}
This is very close to Bloch-Ogus formalism of a
\emph{Poincaré duality theory with supports}
(\cite{BO74}*{1.3}) as demanded in Part B of Beilinson's conjectures.\footnote{In
\cite{FSV}, motivic cohomology with support, required in Bloch-Ogus formalism
is not formally introduced (but see the proof of \cite{FSV}*{Chap.~5, 3.5.4}).
At present, a quick way to get the latter formalism is to use the six functors
formalism, \cite{CD15}, as in \cite{BO74}*{2.1}.}

A key computation in the book is the following relation with higher Chow
groups, proved by Suslin in \cite{FSV}*{Chap.~6, Thm.~3.2}:
assuming $k$ has characteristic $0$, for a quasi-projective $k$-scheme $X$
equidimensional of dimension $d$,
one has an isomorphism:
\begin{equation}\label{eq:BMmotivic&hChow}
\text{H}_p^{\text{BM}}(X,\ZZ(n))=\CH_n(X,p-2n).
\end{equation}
\index[notation]{chxm@$\CH_n(X,m)$}%
If one assumes in addition that $X$ is smooth,
duality theorems between motivic cohomology and motivic Borel-Moore homology
--- recalled in point \introref{V3} --- imply that one further has:
\begin{equation}\label{eq:motivic&hChow}
\H^p(X,\ZZ(n))=\CH^n(X,2n-p).
\end{equation}
\index[notation]{chxmn@$\CH^n(X,m)$}%
\index{higher Chow groups}%
\index{Bloch|see{higher Chow groups}}%
One deduces the relation between rational motivic cohomology
and K-theory asked in Part B of Beilinson's conjecture,
again for a smooth $k$-scheme:
\begin{equation}\label{eq:motivic_coh&Kth}
\H^p(X,\QQ(n)) \simeq \K_{2n-p}^{(n)}(X),
\end{equation}
where the right hand side is the $n$-th graded  part for the $\gamma$-filtration
on $\K_{2n-p}(X) \otimes_\ZZ \QQ$ (cf.\ \cite{Sou}).\footnote{Formula 
\eqref{eq:motivic&hChow}, and therefore Formula \eqref{eq:motivic_coh&Kth},
was extended to arbitrary fields by Voevodsky in \cite{Voe2}.
Formula \eqref{eq:BMmotivic&hChow} was extended to arbitrary fields $k$,
up ot inverting its characteristic exponent, in \cite{CD15}*{8.10 and 8.12}.}
The two main technical tools in Voevodsky's theory are:
\begin{itemize}
\item The category $\corVk$
\index[notation]{cork@$\corVk$}%
of \emph{finite correspondences} over $k$
\index{correspondence!finite}%
whose objects are smooth $k$-schemes (separated of finite type)
and morphisms are finite correspondences. A finite correspondence
from $X$ to $Y$ is an algebraic cycle on $X \times_k Y$ whose support
is finite equidimensional over $X$ (\emph{i.e.} each irreducible component
is finite and surjective over an irreducible component of $X$).
\item The category $\shVk$
\index[notation]{shtrk@$\shVk$}%
of \emph{sheaves with transfers} over $k$,
\index{sheaves!with transfers}%
that is the additive presheaves of abelian groups over $\corVk$
whose restriction to the category of smooth $k$-schemes is a sheaf
for the Nisnevich topology (a variant for the étale topology is
also available).
\end{itemize}
The aim of our book is to extend these tools, and the results \introref{V1} to \introref{V4}
enumerated above, by replacing algebraic cycles and their Chow groups
with a quadratic (or more accurately a symmetric bilinear) variant which corresponds to the so-called
\emph{Chow-Witt groups}.
\index{Chow-Witt!group}

\section{$\Aone$-homotopy, orientation and Chow-Witt groups}

Voevodsky expressed early on his philosophical belief that the theory of motives should be the homological part of a homotopy theory.%
\footnote{In the preface of his
1992 PhD thesis, where he introduces h-motives, Voevodsky wrote:
\begin{quote}
The ``homology theory of schemes'' we obtain this way is related to
the would-be homotopy theory of schemes in the same way as usual singular
homologies of topological spaces are related to classical homotopy theory.
\end{quote}}
This would-be homotopy theory was defined by Voevodsky in collaboration with Morel, 
based on earlier works of Joyal and Jardine \cite{J87}, and on the techniques involved in the theory of motivic complexes:
Nisnevich sheaves and $\Aone$-homotopy, but without the consideration of transfers. 
As it turns out, there is a close relation between the stable homotopy category
$\SH(k)$
\index{stable!homotopy category}%
\index[notation]{shk@$\SH(k)$}%
of schemes over $k$ and the stable category of motives over $k$ reflected by an adjunction of triangulated categories:
\[
\gamma^*:\SH(k) \leftrightarrows \DMk:\gamma_*
\]
which is the algebraic analogue of the adjunction in algebraic topology
derived from the Dold-Kan equivalence. 
Here, $\DMk$
\index[notation]{dmk@$\DMk$}%
denotes ``Voevodsky's big category of motives'' in which the tensor product with the Tate motive $\ZZ(1)$ is inverted.

Indeed the functor $\gamma_*$ induces a fully faithful functor on the rationalized category,
and identifies the objects of $\DMk_\QQ$ as the rational spectra
which are orientable (this is due to Morel, see \cite{CD12}*{5.3.35, 14.2.3, 14.2.16, 16.1.4}).
Orientable means here having a module structure over the ring spectrum
representing rational motivic cohomology.%
\footnote{This is equivalent to
having a module structure over the rational algebraic cobordism spectrum,
and should be thought as ``having an action of the Chern classes''.
For objects with a ring structure,
this corresponds to the classical notion of orientation of a ring spectrum,
borrowed from topology. See \cite{Deg12}, in particular Prop. 5.3.1.}
Roughly, this theorem means that, with rational coefficients, 
admitting transfers is a property rather than a structure, 
and corresponds to the property of being orientable.

The integral coefficient case is much more subtle. 
The existence of motivic Steenrod operations shows that $\gamma_*$ is not fully faithful.%
\footnote{As in topology this follows by looking at the map:
\[
\H^{n,i}(k,\ZZ/l)=\Hom_{\DMkZ}(\ZZ/l,\ZZ/l(i)[n])
\xrightarrow{\gamma_*} 
\Hom_{\SH(k)}(\HH \ZZ/l,\HH \ZZ/l(i)[n])
\]
where $\ZZ/l$ is the constant motive with $\ZZ/l$ coefficients in $\DMkZ$
and $\HH \ZZ/l=\gamma_*(\ZZ/l)$ is the ring spectrum representing motivic
cohomology with $\ZZ/l$-coefficients. 
The right hand side is strictly bigger (for instance in degrees $n=1$ and $i=0$)
as it contains the motivic Steenrod operations
(cf.\ \cite{Voe4} in characteristic $0$
and \cite{HKO} in positive characteristic).}
While this might seem discouraging at first, one should take stock in the fact that homological algebra is a part of stable homotopy theory.
That is, 
if $A$ is an associative ring, 
then the Grothendieck-Verdier derived category $\Der(A)$ of right $A$-modules is equivalent to the homotopy category of the Eilenberg-MacLane ring spectrum $\HH A$ \cite{R87}.
The model categorical aspects of this categorical equivalence were solidified in \cite{SS03}.
In our motivic setting there is a zig-zag of Quillen equivalences between modules over the integral motivic cohomology spectrum $\sHM \ZZ$
\index{motivic!cohomology!spectrum}%
\index{spectrum!motivic cohomology}%
and $\DMk$ \cite{RO06}. 
This was proven over fields of characteristic zero in \cite{RO08} using resolution of singularities, 
and extended to positive characteristic $p$ after inversion of $p$ in \cites{CD15,HKO}.
For a fuller review of these results we refer to \cite{L18}.

We also note the following fundamental theorem discovered by Morel.
\begin{thm*}[Morel]
Given an infinite perfect field $k$, one has a canonical isomorphism:
\[
\End_{\SH(k)}(S^0) \simeq \GW(k)
\]
where the right hand side is the Grothendieck-Witt ring of $k$
\index{Grothendieck-Witt!group}%
\index[notation]{gw@$\GW$}%
--- the Grothendieck ring of isomorphism classes of $k$-vector
spaces endowed with nondegenerate symmetric bilinear forms. Moreover, this isomorphism fits into the following commutative diagram:
\[
\begin{tikzcd}[row sep=small]
\End_{\SH(k)}(S^0) \ar[r,"\gamma^*"] \ar[d,"\simeq",dash]
 & \End_{\DMk}(\ZZ) \ar[d,"\simeq",dash] \\
\GW(k) \ar[r,"\rk"] & \ZZ
\end{tikzcd}
\]
where the map $\rk:\GW(k) \rightarrow \ZZ$ is induced by the rank of a symmetric bilinear form.
\end{thm*}
Intuitively, the stable degree of an endomorphism of the simplicial sphere
$S^n$, in the $\Aone$-homotopical sense, is the class of a symmetric bilinear form
while in the motivic sense, it is just an integer. It is interesting to note
that the degree in the sense of classical algebraic topology coincides with
the one obtained in motivic terms.\footnote{See \cite{KW} for further 
concrete computations of degrees.} 
In particular, the $\Aone$-homotopy theory
uncovers new phenomena of more arithmetic nature.

The first step to capture these phenomena was taken up by Barge and Morel in
\cite{BM00, BM99}. Motivated by Nori's ideas on Euler classes and Voevodsky's work
on the Milnor conjecture, they figured out an extension of algebraic cycle theory
where multiplicities are (classes of) symmetric bilinear forms rather than integers.
Inspired by Rost's theory of cycles with coefficients (cf.\ \cite{Ros}),
they proposed a definition of the so-called (top) \emph{Chow-Witt group}
$\ch*X$
\index[notation]{chx@@$\ch*X$}%
\index{Chow-Witt!group}%
of a smooth scheme over a field.
The idea of the construction is to replace Milnor K-theory $\KM_*$
\index[notation]{km@$\KM_*$}%
\index{Milnor K-theory}%
--- the fundamental example in Rost's theory ---
by Milnor-Witt K-theory $\KMW_*$. 
\index[notation]{kmw@$\KMW_*$}%
\index{Milnor-Witt!K-theory}%
This appears natural
once we know the following extension of Morel's theorem (stated
above): for any integer $n \in \ZZ$, one gets the following commutative
diagram:
\[
\begin{tikzcd}
\Hom_{\SH(k)}(S^0,\Gm^{\wedge,n}) \ar[r,"\gamma^*"] \ar[d,"\simeq"',dash]
 & \Hom_{\DMk}(\ZZ,\ZZ(n)[n]) \ar[d,"\simeq",dash] \\
\KMW_*(k) \ar[r] & \KM_*(k).
\end{tikzcd}
\]
It is therefore no surprise that we get the following commutative diagram between Chow-Witt and Chow rings:
\[
\begin{tikzcd}
\cht*X{\omega_{X/k}} \ar[r] \ar[d,"\widetilde \deg"'] 
 & \CH^*(X) \ar[d,"\deg"] \\
\GW(k) \ar[r,"\rk"] & \ZZ,
\end{tikzcd}
\]
where the first horizontal map is defined for any $k$-scheme $X$,
and the vertical maps are defined when $X$ is proper, if one restricts
to $0$-dimensional cycles. The theory was fully worked out in \cite{Fasel08}
(see Corollary 10.4.5 for the above commutative diagram).

Recall that the Milnor-Witt K-theory ring of a field $k$ is generated
on the one hand by units $u \in k^*$, as the Milnor K-theory ring,
and on the other hand by the (algebraic) Hopf map $\eta:\Gm \rightarrow S^0$.
This map is an obstruction for a ring spectrum $\EE$ to be orientable
(see \cite{Mor1}*{6.2.1}).
With rational coefficients, it is in fact the only obstruction to be oriented:
when its action on a rational spectrum $\EE$ (even without ring structure) vanishes,
then $\EE$ is actually a motive \emph{i.e.} an object of $\DMk$
(see \cite{CD12}*{16.2.13}).
This result can be extended to an integral spectrum $\EE$
if one assumes in addition that $\EE$ is concentrated in one degree for the homotopy
$t$-structure (see \cite{Deg10}*{1.3.4}). Then the vanishing of the action
of the Hopf map $\eta$ is equivalent to the existence of transfers on the
corresponding homotopy sheaf in the sense of Voevodsky's theory.

The previous considerations show that certain cohomology theories,
such as the Chow-Witt groups, do not admit transfers in general.
On the other hand, weak notions of transfers have been used in $\Aone$-homotopy theory
and are critical in the proof of some of its fundamental results
(see \cite{Morel12}).
 
\bigskip

The object of this book is to find a suitable notion of transfers
that allows to describe the symmetric bilinear phenomena of $\Aone$-homotopy theory.
More precisely, to make the end of the preceding section explicit,
we propose a generalization of Voevodsky's theory of finite correspondences,
sheaves with transfers and motivic complexes,
where Chow groups are replaced by Chow-Witt groups.
 
\section{Milnor-Witt motives}

Here is a synthetic description of the results proved in this book,
organized in a way corresponding to the presentation of Voevodsky's theory
in Section \ref{sec:VoevodskyTheory}. Again $k$ is a perfect field.

\introtag{MW1} 
The main technical innovation of our series of articles is the introduction of the additive symmetric monoidal category $\cork$ 
\index[notation]{corkt@$\cork$}%
of smooth $k$-schemes with morphisms given by  the so-called \emph{finite Milnor-Witt correspondences}
--- finite MW-correspondences for short ---
\index{Milnor-Witt!finite correspondence}%
a suitable generalization of Voevodsky's finite correspondences whose
coefficients are symmetric bilinear forms. One has in addition canonical
additive monoidal functors:
\[
\begin{tikzcd}
\smk \ar[r,"\tilde \gamma"] \ar[rr,"\gamma", bend right] & \cork \ar[r,"\pi"] & \corVk
\end{tikzcd}
\]
\index[notation]{cork@$\corVk$}%
where $\gamma$ is the classical graph functor and $\pi$ is induced by the function which to a symmetric bilinear form associates its rank.

Modeled on Voevodsky's definition of geometric motives, we define
the triangulated symmetric monoidal category of \emph{geometric MW-motives}
(resp.\ \emph{effective geometric MW-motives}) $\DMtgmk$
(resp.\ $\DMtegmk$) and a commutative diagram of functors:
\index[notation]{dmtgeomk@$\DMtgmk$}%
\index{Milnor-Witt!motives, category of!geometric}%
\index{motives!MW-|see{Milnor-Witt motives}}
\index[notation]{dmteffgeomk@$\DMtegmk$}%
\index{Milnor-Witt!motives, category of!effective geometric}%
\index{MW|see{Milnor-Witt}}%
\[
\begin{tikzcd}
\DMtegmk \ar[r,"\pi^*"] \ar[d,"\tilde \Sigma^\infty"']
 & \DMegmk \ar[d,"\Sigma^\infty"]  \\
\DMtgmk \ar[r,"\pi^*"]  & \DMgmk.
\end{tikzcd}
\]
An effective geometric MW-motive is simply a bounded complex
with coefficients in the additive category $\cork$; in particular,
$\pi^*$ is the obvious functor which sends such a complex to a 
geometric motive by applying $\pi$ at each stage.
Besides, when $k$ is infinite,
we are able to extend Voevodsky's cancellation theorem,
which in our context gives that the functor $\tilde \Sigma^\infty$
is fully faithful. Note also that when $(-1)$ is a sum of squares
in $k$, both functors $\pi^*[1/2]$ are equivalences of categories
(the induced functor on the categories localized at the integer $2$).

\introtag{MW2}
Using the notion of finite MW-correspondences,
we can carry on the analogue of Voevodsky's theory of motivic complexes,
with coefficients in an arbitrary commutative ring $R$.
We first introduce, for $t$ the Nisnevich or étale
topology on the category of smooth $k$-schemes,
the notion of $t$-sheaves of $R$-modules with MW-transfers
\index{Milnor-Witt!sheaf with transfers|see{Milnor-Witt sheaf}}%
\index{Milnor-Witt!sheaf}%
\index{Milnor-Witt!$t$-sheaf}%
\index{transfers!Milnor-Witt sheaf with|see{Milnor-Witt sheaf}}%
\index{sheaf!with Milnor-Witt transfers|see{Milnor-Witt sheaf}}%
over the base field $k$. We simply call them \emph{MW-$t$-sheaves}
or \emph{MW-sheaves} when $t=\nis$.
The corresponding category $\shMWtkR$ 
\index[notation]{shMWkR@$\shMWtkR$}%
enjoys
all the good properties proved by Voevodsky for sheaves with transfers,
and for their homotopy invariant version.
Beware however that the presheaf $\tilde c_R(X)$
on $\cork$ represented by a smooth $k$-scheme $X$:
\[
U \mapsto \tilde c_R(X)(U):=\cork(U,X) \otimes_\ZZ R
\]
is not a Nisnevich sheaf in general. However, the associated $t$-sheaf
with MW-transfers exists and we let $\MWrepRt(X)$
\index[notation]{ctildex@$\MWrepRt(X)$}%
be the corresponding MW-$t$-sheaf.

We can then define the category of MW-motivic complexes $\DMtekR$,
\index[notation]{dmteffkr@$\DMtekR$}%
when $t=\nis$, as the $\Aone$-localization of the derived category $\Der(\shMWxkR{\nis})$
and finally the category of MW-motivic spectra $\DMtkR$
\index[notation]{dmtkr@$\DMtkR$}%
as the $\PP^1$-stabilization
of $\DMtekR$ (or rather of its underlying model category).
These categories are triangulated symmetric monoidal, and possess an internal Hom.
As our constructions are parallel to the classical constructions of motivic homotopy categories, 
we obtain a commutative diagram of triangulated symmetric monoidal categories:
\begin{equation}\label{eq:diagr_D_motivic}
\begin{tikzcd}
\DAekR \ar[r,"\derL \tilde \gamma^*"] \ar[d]
 & \DMtekR \ar[r,"\derL \tilde \pi^*"] \ar[d,"\tilde \Sigma^\infty"]
 & \DMekR \ar[d,"\Sigma^\infty"] \\
\DAkR\ar[r,"\derL \tilde \gamma^*"]
 & \DMtkR \ar[r,"\derL \tilde \pi^*"]
 & \DMkR
\end{tikzcd}
\end{equation}
\index[notation]{daeffkr@$\DAekR$}%
\index[notation]{dakr@$\DAkR$}%
where the categories on the left column are the $\Aone$-derived and
stable $\Aone$-derived categories of Morel.
All these functors are triangulated monoidal, and admit a right adjoint.
The relation with the geometric objects comes from the following commutative diagram:
\[
\begin{tikzcd}
\DMtegmk \ar[rr,"\pi^*"] \ar[dd] \ar[rd,hook] 
 & & \DMegmk \ar[dd] \ar[rd,hook] &  \\
& \DMtekZ\ar[rr,"\derL \tilde \pi^*" near start, crossing over] 
 & & \DMe{k,\ZZ} \ar[dd,"\Sigma^\infty"] \\
\DMtgmk\ar[rr,"\pi^*" near start] \ar[rd,hook]
 & & \DMgmk \ar[rd,hook] & \\
& \DMtkZ \ar[rr,"\derL \tilde \pi^*"] \ar[from=uu,"\tilde \Sigma^\infty" near start, crossing over]
 & & \DMkZ 
\end{tikzcd}
\]
where the back square is that of \introref{MW1}, and each diagonal arrow is triangulated monoidal, fully faithful, with essential image the subcategory of compact objects in the target.
As in \introref{MW1}, we get the analogue of Voevodsky's cancellation theorem for MW-motives:
when $k$ is an infinite perfect field, the functor $\tilde \Sigma^\infty$ in the above diagram
is fully faithful.
Also, when $(-1)$ is a sum of squares in $k$ and $2$ is invertible in $R$,
the functor $\derL \pi^*$ in the above diagram is an equivalence of categories.
We will give a much more precise relation between Voevodsky's motives and MW-motives below.

Following the classical method of Voevodsky and using our results
on homotopy invariant MW-sheaves,
we can show that $\DMtekR$ is in fact the full subcategory
of $\Der(\shMWxkR{\nis})$ consisting of complexes whose Nisnevich cohomology sheaves
(or equivalently, cohomology presheaves) are homotopy invariant.
This immediately yields a t-structure on $\DMtekR$,
called the homotopy t-structure. It also gives an explicit $\Aone$-resolution
functor, obtained as the analogue of \emph{Suslin's singular chain functor}.
In particular, the effective MW-motivic complex $\tMot(X)$
\index[notation]{mtx@$\tMot(X)$}%
associated to a smooth scheme $X$ is 
the MW-$t$-sheaf associated to the complex which in degree $n \geq 0$ is
\[
U \mapsto \tilde c_R(\Delta^n \times U,X).
\]
We can then define the Tate twist for MW-motives $\MWrepRt(n)$,
\index[notation]{rtn@$\MWrepRt(n)$}%
using the same formula as the Tate twist for Voevodsky's motives.
For $n>0$, we get an effective motivic complex:
\[
\tilde R(n)[n]=\tilde R(\Gm^{n})/\oplus_{i=1}^n \tilde R(\Gm^{n-1}),
\]
where in this sum, the $i$-th map is induced by the inclusion
$\Gm^{n-1} \rightarrow \Gm^n$ setting the $i$-th coordinate to $1$.
We define $\tilde R(-n)$ in $\DMtekR$ by the formula:
\[
\tilde R(-n):=\uHom_{\DMtekR}\big(\tilde R(n),\tilde R\big).
\]
When $k$ is an infinite perfect field,
$\tilde R(-n)$ is the infinite loop space associated with the tensor inverse of 
$\tilde \Sigma^\infty \tilde R(n)$ 
computed in $\DMtkR$.
We then get the following computations of these complexes:
\[
\tilde \ZZ(n)=\begin{cases}
\sW[-n] & \text{if } n<0, \\
\sKMW_0 & \text{if } n=0,
\end{cases}
\]
where $\sW$ is the unramified Witt sheaf 
\index{Witt!unramified sheaf}%
\index[notation]{ws@$\sW$}%
and $\sKMW_0$ the $0$-th unramified Milnor-Witt K-theory sheaf.
Our Milnor-Witt motivic cohomology does therefore capture symmetric bilinear invariants, fulfilling our aim.

Besides, one can further explore the links with Voevodsky's motivic
complexes. Indeed, given any integer $n \in \ZZ$, we get that
the following computation of MW-sheaves:
\[
\underline{\H}^n(\tilde \ZZ(n))=\sKMW_n
\]
where $\sKMW_n$ is the $n$-th unramified Milnor-Witt K-theory sheaf, and the left-hand term is the sheaf associated to the relevant cohomology group.

We define the $R$-linear MW-motivic cohomology of a smooth $k$-scheme $X$
in bidegree $(p,n) \in \ZZ^2$ as:
\[
\HMW^{p,n}(X,R)=\Hom_{\DMtkR}(\tMot(X),\tilde R(n)[p]).
\]
\index[notation]{hmwpqxr@$\HMW^{p,n}(X,R)$}%
\index{Milnor-Witt!motivic!cohomology}%
When the base field $k$ is perfect and infinite, this can be identified
with the Nisnevich (or Zariski) hypercohomology:
\[
\HMW^{p,n}(X,R) \simeq \H^p_\nis(X,\tilde R(n)).
\]
Moreover, provided $p \geq 2n-1$, we get:
\[
\HMW^{p,n}(X,\ZZ)=\H^{p-n}(X,\sKMW_n)
\]
\index[notation]{hpxkmwq@$\H^p(X,\sKMW_q)$}
which finally contains as a particular case the following identification
 of MW-motivic cohomology and Chow-Witt groups:
\[
\HMW^{2n,n}(X,\ZZ) \simeq \ch nX.
\]
\index[notation]{chtx@$\ch*X$}%
\index{Chow-Witt!group}%

\introtag{MW4} As in the context of motivic cohomology,
we associate four theories in the Milnor-Witt context.
Here, our approach differs from that of \cite{FSV}*{Ch.~4}
as we can use the six functors formalism on
the motivic stable homotopy category.

In fact, as we have the adjunction of categories
\[
\begin{tikzcd}
\DA k \ar[r,"\derL \tilde \gamma^*",shift left=.4ex]
 & \DMtkR,\ar[l,"\tilde \gamma_*",shift left=.4ex]
\end{tikzcd}
\]
(which actually factorizes through $\DAkR$ via the extension-restriction of scalars adjunction)
we can define the MW-motivic ring spectrum\footnote{Note we can also view this ring spectrum
in the category $\SH(k)$ by using the canonical functor $\DA k \rightarrow \SH(k)$.}
as
\[
\sHMW R:=\tilde \gamma_*(\tilde R).
\]
\index{Milnor-Witt!motivic!ring spectrum}%
\index[notation]{hmws@$\sHMW R$}%
The advantage of this definition is that, given diagram \eqref{eq:diagr_D_motivic},
we immediately get a morphism of ring spectra:
\[
\varphi:\sHMW R \rightarrow \sHM R.
\]
\index[notation]{hms@$\sHM R$}%

Once we have a ring spectrum, we can derive the usual four theories,
based on the six functors formalism on the categories $\DA X$
for various schemes $X$.
Note however before giving these definitions that
it is important for MW-motivic cohomology to take into account a more
general twist than the usual Tate twist. Given a virtual
vector bundle $v$ over a scheme $X$ (in the sense of \cite{Del87}*{\S 4}), we let $\Th(v)$ be its Thom space seen in $\DA X$. 
Then, for any separated $k$-scheme $p:X\to \Spec(k)$ of finite
type, any integer $n\in \ZZ$
and any virtual bundle $v$ over $X$, we define
\begin{itemize}
\item \emph{MW-motivic cohomology}:
\index{Milnor-Witt!motivic!cohomology}%
\index[notation]{hmwxvr@$\HMW^n(X,v,R)$}%
\begin{align*}
& \HMW^n(X,v,R):=
 \Hom_{\DA k}\Big(\un_k,p_*\big(p^*(\sHMW R) \otimes \Th(v)\big)[n]\Big).
\intertext{%
\item \emph{MW-motivic cohomology with compact support}:
\index{Milnor-Witt!motivic!cohomology with compact support}%
\index[notation]{hmwcxvr@$\HMWc^n(X,v,R)$}%
}
& \HMWc^n(X,v,R):=
\Hom_{\DA k}\Big(\un_k,p_!\big(p^*(\sHMW R) \otimes \Th(v)\big)[n]\Big),
\intertext{%
\item \emph{MW-motivic homology}:
\index{Milnor-Witt!motivic!homology}%
\index[notation]{hmwhxvr@$\HhMW_n(X,v,R)$}%
}
& \HhMW_n(X,v,R):=
\Hom_{\DA k}\Big(\un_k,p_!\big(p^!(\sHMW R) \otimes \Th(-v)\big)[-n]\Big)
\intertext{%
\item \emph{MW-motivic Borel-Moore homology}:
\index{Milnor-Witt!motivic!Borel-Moore homology}%
\index[notation]{hmwbmhxvr@$\HhMWBM_n(X,v,R)$}%
}
& \HhMWBM_n(X,v,R):=
\Hom_{\DA k}\Big(\un_k,p_*\big(p^!(\sHMW R) \otimes \Th(-v)\big)[-n]\Big).
\end{align*}
\end{itemize}
These theories have all the expected properties:
basic functoriality, Gysin morphisms for smooth maps, products, 
descent long exact sequence with respect to Nisnevich and cdh topologies,
duality. We refer the reader to the introduction of chapter \ref{ch:spectra} of this book for the detailed list of all these properties.
Note however a difference with the classical situation: As MW-motivic cohomology 
doesn't satisfy the projective bundle theorem, the relevant formulas need 
special care. As an illustration, the duality property for a smooth scheme $X$ 
with tangent bundle $T_X$ reads as
\[
\HMW^n(X,v)\simeq \HhMWBM_{-n}(X,T_X-v)
\]
for any virtual vector bundle $v$ and any integer $n\in\ZZ$. The presence of 
$T_X$, reminiscent of Serre duality, explains why we have to consider twists by 
arbitrary (virtual) vector bundles in these theories. 

The formal properties of Borel-Moore homology allow to construct a Bloch-Ogus 
type spectral sequence of the form
\[
E^1_{p,q}(X,v):=\bigoplus_{x\in X_{(p)}} \HhMWBM_{p+q}(k(x),v_{\vert_{k(x)}})\implies \HhMWBM_{p+q}(X,v). 
\]
Besides, we perform computations allowing to deduce that 
\[
\HhMWBM_{p+n}(k(x),-n)\simeq 
\KMW_{p+n}(k(x))\otimes \omega_{k(x)/k}
\]
for $x\in X_{(p)}$, $n\in\ZZ$ and $\omega_{k(x)/k}$ is the determinant of the 
module of differentials $\Omega_{k(x)/k}$. In particular, the line $q=n$ in the 
spectral sequence for $v=-\AAA^n$ is a complex of the form
\[
\begin{tikzcd}[column sep=small]
\cdots & \displaystyle{\bigoplus_{x\in X_{(p)}} 
\hspace{-1ex}
\KMW_{p+n}(k(x))\otimes \omega_{k(x)/k}} \ar[l] & 
\displaystyle{\bigoplus_{x\in X_{(p+1)}} 
\hspace{-2ex}
\KMW_{p+1+n}(k(x))\otimes \omega_{k(x)/k}} \ar[l] & 
\cdots \ar[l]
\end{tikzcd}
\]
which we denote by $\CBM(X,\sKMW_{d+n})$ where $d=\dimn X$. 
\index[notation]{cbmxkmw@$\CBM(X,\sKMW_{n})$}%
These complexes enjoy nice functorial properties detailed in the text. 
Using these properties, we are able to compute the differentials, showing in particular that they coincide with the differentials defined by Morel in his so-called Rost-Schmid complex when $X$ is smooth. 
Furthermore, our analysis of the terms in the spectral sequence yields the following isomorphism
\[
\HhMWBM_i(X,n)\simeq 
\H_{n+i}(\CBM(X,\sKMW_{d-n}))
\]
for $i=0,1$ and $n\in \NN$. This can be seen as the analogue in degrees $0$ and 
$1$ of (\ref{eq:BMmotivic&hChow}) in Paragraph \introref{V4} above in our context. In 
particular, we can view MW-motivic Borel-Moore homology as the analogue of 
higher Chow groups in our context. Indeed, there are homomorphisms 
\[
\HhMWBM_i(X,n)\to \H^{\text{BM}}_{i+2n}(X,\ZZ(n))\to 
\CH_n(X,i)
\] 
when $\charac k=0$\footnote{This remains valid in positive characteristic 
$p$ at the cost of inverting $p$ in the coefficients (\cite{CD15}*{Cor.~8.12})}. Also, the 
complex $\CBM(X,\sKMW_{d-n})$ allows one to define 
Chow-Witt groups of singular schemes (of finite type over $k$) as 
\[
\cht nX{\ome Xk}=\H_n(\CBM(X,\sKMW_{d-n})).
\]
When $X$ is smooth, the afore mentioned duality yields the formula
\[
\HMW^0(X,v)\simeq \cht rX{\det(v)}
\]
where $r=\dimn v$ which is the analogue of \eqref{eq:motivic&hChow} in 
Paragraph \introref{V4} above and a generalization of the already stated isomorphism
\[
\HMW^{2n,n}(X,\ZZ) \simeq \ch nX.
\]

\introtag{MW5}
Let us now elaborate on the links between Hermitian K-theory
\index{Hermitian K-theory}
and the motivic cohomology we define in this book. Let us first recall that one of the motivating principle of the development of Chow-Witt groups is the idea that they should have the same relationship with Hermitian K-theory 
as the relationship between Chow groups and K-theory. In view of this, it is 
natural to expect that MW-motivic cohomology should appear in an analogue of the 
ordinary Atiyah-Hirzebruch spectral sequence relating K-theory and ordinary 
motivic cohomology. To start with, recall from \cite{Spitzweck12} that there is a spectral sequence computing Hermitian K-theory, namely the very effective 
slice spectral sequence. It takes the form 
\[
E_2^{p,q}=\H^{p+q}(X,{\tilde s}_{-q}(\sKO))\Rightarrow 
\H^{p+q}(X,\sKO)=\sKO^{p+q}(X)
\]
where $\sKO$
\index[notation]{kos@$\sKO$}%
\index{Hermitian K-theory!spectrum}%
\index{spectrum!Hermitian K-theory}%
is the spectrum representing Hermitian K-theory in the 
stable homotopy category and ${\tilde s}_{-q}(\sKO)$ are its so-called 
\emph{very effective slices}. The latter have been computed in \cite{B17} 
and take the form
\[
{\tilde s}_{-q}(\sKO)=\begin{cases} (\PP^1)^{\wedge (-q)}\wedge 
{\tilde s}_{0}(\sKO) & \text{if $-q=0$ modulo $4$.} \\
(\PP^1)^{\wedge (-q)}\wedge\sHM(\ZZ/2) & \text{if $-q=1$ modulo $4$.} \\
(\PP^1)^{\wedge (-q)}\wedge\sHM\ZZ & \text{if $-q=2$ modulo $4$.} \\
0 & \text{if $-q=3$ modulo $4$.} \\
\end{cases}
\]
where $\sHM\ZZ$ (resp.\ $\sHM(\ZZ/2)$) is the spectrum representing ordinary 
motivic cohomology (resp.\ ordinary motivic cohomology modulo $2$). Further, the 
very effective slice ${\tilde s}_{0}(\sKO)$ fits in an exact triangle 
\[
\sHM(\ZZ/2)[1]\to {\tilde s}_{0}(\sKO)\to \tilde{\H}\ZZ
\]
where $\tilde{\H}\ZZ$ (in the notation of loc.~cit.)\ is the genuinely new piece of 
the zeroth slice. In this book, we prove that $\tilde{\H}\ZZ$ is actually the 
spectrum $\sHMW \ZZ$ defined above. Apart from showing that MW-motivic cohomology 
indeed plays an important part in the computation of Hermitian K-theory, it 
has other interesting consequences. The homotopy t-structure in the stable 
homotopy category induces a commutative diagram of spectra 
\begin{equation}
\label{eq:fdl_square}
\begin{tikzcd}
\sHMW \ZZ \ar[r] \ar[d,"\varphi"']
& \tau_{\leq 0}\sHMW\ZZ \ar[d] \ar[r,"\sim"] & \spKMW \ar[d] \\
\sHM\ZZ \ar[r] &  \tau_{\leq 0}\sHM\ZZ \ar[r,"\sim"] & \spKM
\end{tikzcd}
\end{equation}
where the right-hand spectra are the ones representing Milnor-Witt and Milnor K-theory and the isomorphisms are given by the isomorphisms 
$\underline{\H}^n(\tilde \ZZ(n))=\sKMW_n$ and $\underline{\H}^n(\ZZ(n))=\sKM_n$ 
stated above. It follows from \cite{B17}*{Thm.~11} that the outer 
square is homotopy Cartesian. As a corollary, we get the following computation 
of MW-motivic cohomology groups for any finitely generated field extension $L/k$:
\[
\HMW^{p,q}(L,\ZZ)=\begin{cases} 0 & \text{if $p>q$} \\ \KMW_p(L) & 
\text{if $p=q$} \\ \HM^{p,q}(L,\ZZ) & \text{if $p<q$.}\end{cases}
\]
\index[notation]{hmpqxr@$\HM^{p,q}(X,R)$}%
where $\HM^{p,q}(L,\ZZ)$ is Voevodsky's motivic cohomology.

Another remarkable consequence of the above homotopy Cartesian square
is the computation of the étale variant of the MW-motivic ring spectrum.
Recall that there is an étale variant of the $\Aone$-derived category,
denoted by $\DAket$, 
and that the étale sheafification functor induces an adjunction of triangulated categories:
\[
\begin{tikzcd}
\DAk \ar[r,"\alpha^*",shift left=.4ex]
 & \DAket \ar[l,"\alpha_*",shift left=.4ex]
\end{tikzcd}
\]
\index[notation]{dak@$\DAk$}%
\index[notation]{daetk@$\DAket$}%
the functor $\alpha_*$ being induced by the obvious forgetful functor.
 Applying the functor $\alpha^*$ to the homotopy Cartesian square \eqref{eq:fdl_square},
 we get a homotopy cartesian square:
\[
\begin{tikzcd}
\sHMWet \ZZ \ar[r] \ar[d,"\varphi_\et"']
 & \alpha^*(\spKMW) \ar[d] \\
\sHMet\ZZ \ar[r] & \alpha^*(\spKM).
\end{tikzcd}
\]
Note finally that, according to \chspectra, Theorem~\ref{thm:compare_HMW_et&HM_et},
the map $\varphi_\et$ is an isomorphism of ring spectra.\footnote{Note also that one can prove 
 that the right vertical map is an isomorphism as in the proof of \chdmt,
 Corollary~\ref{cor:adjunctions_corr}\ref{item:adjunction-pi-equivalence},
 and especially Lemma~\ref{lemm:GW_etale_trivial}.}

\section{Beyond the analogy}

We now describe some results specific to our situation. When we invert $2$ in the coefficients of $\DMtkR$, the link with ordinary motives becomes much more precise. To state the result, note that the construction of finite MW-correspondences can be adapted to other strictly $\Aone$-invariant sheaves. Taking in particular the unramified Witt sheaf $\sW$, we obtain motivic categories $\WDMekR$ and $\WDMkR$
\index[notation]{wdmeffkr@$\WDMekR$}%
\index[notation]{wdmkr@$\WDMkR$}%
for any ring $R$ which are the respective analogues of $\DMtekR$ and $\DMtkR$. Moreover, the relationship between Milnor-Witt K-theory and the unramified Witt sheaf yields a functor
\[
\DMtkR \to \WDMkR
\]
which is also defined at the level of the effective categories. When $R$ is a $\ZZ[\frac 12]$-algebra, we obtain 
an equivalence of categories
\[
\DMtkR\to \DMkR\times \WDMkR
\]
induced by the above functor and the usual functor $\DMtkR \to \DMkR$. This decomposition identifies the ``plus part'' of $\DMtkR$ with $\DMkR$ and the ``minus part'' with $\WDMkR$.
Note that the relationship between the category $\WDMkR$ and other similarly defined 
categories related to the unramified Witt sheaf (Bachmann, Levine-Ananievskiy-Panin) still needs to be clarified.

Let us note that the canonical map
\[
\DAkR \rightarrow \DMtkR
\]
is certainly not an isomorphism in general. Indeed, we will show in further work that every ring spectrum coming from $\DMtkR$ is SL-oriented in the sense of Panin-Walter \cite{PW10}, while a result of Ananyevskyi and Sosnilo shows that the ring spectrum corresponding to the unit of $\Der_{\Aone}(k,\ZZ[1/2])$ is not SL-oriented.
However, with rational coefficient, Grigory Garkusha proved in \cite{Gar19} that the above map is in fact an equivalence of triangulated categories.

Another important notion of correspondences, called \emph{framed correspondences}, has been introduced by Voevosdky and implemented by Panin and Garkusha.
They use them to describe an explicit $\Aone$-resolution functor
for motivic spectra. There is a precise relation between our approach and their work.
Indeed, any framed correspondence gives rise to a finite MW-correspondence,
and we use this fact to obtain some of the fundamental results satisfied by
homotopy invariant MW-sheaves. Although recent work of Kolderup (\cite{Kolderup17}) allows one to avoid this trick, the relation between MW-correspondences and framed correspondences has other important applications (see \cite{Ananievskiy18}) and should be studied further. 

Finally, let us say a few words about our assumptions on the perfect base field $k$. The definition of finite Chow-Witt correspondences, which takes place in \chfinitecw makes no use of further assumptions. Starting from \chdmt, we however have to further assume that $k$ is infinite and of characteristic different from $2$ in order to obtain our main results. At present, we don't know how to eliminate these extra assumptions. 

\begin{center}
------------------------------------
\end{center}

\section*{Acknowledgements}

We gratefully acknowledge hospitality and support from Institut Mittag-Leffler and the universities of Artois, Dijon, Essen, Grenoble, Oslo and Southern California.

The exposition and results presented here owe a lot to many people. In particular, we wish to thank Alexey Ananievskiy, Aravind Asok, Peng Du, Marc Hoyois, Håkon Andreas Kolderup, Fabien Morel, Ivan Panin, Oleg Podkopaev, Antoine Touzé, Maria Yakerson and Nanjun Yang for useful conversations and remarks on preliminary versions of the book. We would also like to thank Marco Schlichting for explaining us his joint work with Simon Markett on a Grayson type spectral sequence computing higher Grothendieck-Witt groups, Jean Barge for his idea on how to compute the kernel of the map from Milnor-Witt motivic cohomology to ordinary motivic cohomology and Grigory Garkusha for conversations about framed correspondences and for explaining us how to construct the long exact sequence of motivic cohomology groups appearing in \chdmt, Section \ref{sec:AA1derived}.

\renewcommand{\theequation}{\arabic{equation}}

\mainmatter
\chapter[Finite MW-correspondences]{The category of finite Milnor-Witt correspondences\author{Baptiste Calmès and Jean Fasel}}
\label{ch:finitecw}

\section*{Abstract}
We introduce the category of finite MW-correspondences over a perfect field $k$. For any essentially smooth scheme $X$ and integers $p,q\in\ZZ$, we then define MW-motivic cohomology groups $\HMW^{p,q}(X,\ZZ)$ and begin the study of their relationship with ordinary motivic cohomology groups.

\section*{Introduction}

Let $k$ be a perfect field and let $\smk$ be the category of smooth separated schemes of finite type over $k$. One of the central ideas of V.~Voevodsky in his construction of motivic cohomology is the definition of the category of finite correspondences $\corVk$ (see for instance \cite{FSV_a} or \cite{Mazza06_a}). Roughly speaking, the category $\corVk$ is obtained from $\smk$ by taking smooth schemes as objects and formally adding transfer morphisms $\tilde f:Y\to X$ for any finite surjective morphism $f:X\to Y$ of schemes. There is an obvious functor $\smk\to \corVk$ and the presheaves (of abelian groups) on $\smk$ endowed with transfer morphisms for finite surjective morphisms become naturally presheaves on $\corVk$, also called presheaves with transfers. Classical Chow groups or Chow groups with coefficients \`a la Rost are examples of such presheaves. Having the category of finite correspondences in hand, it is then relatively easy to define motivic cohomology, which is an algebro-geometric analogue of singular cohomology in topology. 


However, there are also many examples of interesting (pre-)sheaves without transfers in the above sense. Our main examples here are the Chow-Witt groups \cite{Fasel08_a} or the cohomology of the (stable) homotopy sheaves $\piaone_i(X,x)$ of a pointed smooth scheme $(X,x)$, most notably the Milnor-Witt K-theory sheaves $\sKMW_n$ for $n\in\ZZ$ \cite{Morel12_a}. Such sheaves naturally appear in the Gersten-Grothendieck-Witt spectral sequence computing higher Grothendieck-Witt groups, aka Hermitian K-theory \cite{Fasel09_a} or in the unstable classification of vector bundles over smooth affine schemes \cite{Asok14_a,Asok12_a}, and thus are far from being exotic. 

Although these sheaves do not have transfers for general finite morphisms, they do have transfers for finite surjective morphisms with trivial relative canonical sheaf (depending on a trivialization of the latter), and one can hope to formalize this notion and then follow Voevodsky's construction of the derived category of motives from finite correspondences. In his work on the Friedlander-Milnor conjecture, Morel introduced a notion of generalized transfers in order to deal with this situation \cite{Morel11_a}. Our approach in this article is a bit different in spirit. We enlarge the category of smooth schemes using finite MW-correspondences. Roughly speaking, we replace the Chow groups (or cycles) in Voevodsky's definition by Chow-Witt groups (or cycles with extra symmetric bilinear information) and define in this way the category of finite MW-correspondences $\cor k$. The obvious functor $\smk\to \corVk$ factors through our category; namely there are functors $\smk\to \cor k$ and $\cor k\to \corVk$ whose composite is the classical functor. Given $X,Y$ smooth, the homomorphism $\cor k(X,Y)\to \corVk(X,Y)$ is in general neither injective (by far) nor surjective (yet almost). We call the presheaves on $\cor k$ \emph{presheaves with MW-transfers}. It is easy to see that presheaves with MW-transfers in our sense are also presheaves with generalized transfers in Morel's sense, and indeed the two notions are the same (for sheaves) \cite{Feld21b_a}*{Theorem 3.2.4}. A presheaf on $\corVk$ is also a presheaf on $\cor k$, but the examples above are genuine presheaves with generalized transfers, so our notion includes many more examples than the classical one. 

As hinted above, the construction of $\cor k$ is based on the theory of Chow-Witt groups and we spend a good part of the paper explaining the notions involved and recalling the relevant functorial properties. In a first version of this paper the base field $k$ was assumed to be of characteristic different from $2$, an assumption stemming from \cite{Fasel08_a} where this was required in order to prove the main properties. This assumption was recently lifted in Feld's work \cite{Feld20_a} allowing a characteristic free construction.

Having $\cor k$ at hand, we define MW-motivic cohomology groups $\HMW^{p,q}(X,\ZZ)$ for any smooth scheme $X$ and any integers $p,q\in \ZZ$ and study their basic properties. The main difference with the classical groups is that the MW-motivic cohomology groups are non trivial for $q<0$, in which range they can be identified with the cohomology of the Gersten-Witt complex defined in \cite{Balmer02_a} (in case the characteristic of the base field is different from $2$).

\subsection*{Conventions}
\label{subsection:conventions}

The schemes are separated of finite type over some perfect field $k$. If $X$ is a smooth connected scheme over $k$, we denote by $\Omega_{X/k}$ the sheaf of differentials of $X$ over $\Spec(k)$ and write $\ome{X}{k}:=\det\Omega_{X/k}$ for its canonical sheaf. In general we define $\ome{X}{k}$ connected component by connected component. We use the same notation if $X$ is the localization of a smooth scheme at any point. If $k$ is clear from the context, we omit it from the notation. If $f:X\to Y$ is a morphism of (localizations of) smooth schemes, we set $\omega_{f}=\ome{X}{k}\otimes f^*\ome{Y}{k}^\vee$. If $X$ is a scheme and $n\in\NN$, we denote by $X^{(n)}$ the set of codimension $n$ points in $X$. For a point $x\in X^{(n)}$ with $X$ smooth, we'll denote by $\Lambda(x)$ the one dimensional $k(x)$-vector space $\wedge^n(\mathfrak m_x/\mathfrak m_x^2)^\vee$ where $\wedge$ denotes the exterior power, $(-)^\vee$ denotes the dual and $\mathfrak m_x$ is the maximal ideal corresponding to $x$ in $\OO_{X,x}$.


\section{Milnor-Witt K-theory}
\label{sec:MWKtheory}

In this section, we first recall the definition of Milnor-Witt K-theory of a field and its associated sheaf following \cite{Morel12_a}*{\S 3}. We then recall the definition of Chow-Witt groups and spend some time on their functorial properties.

For any field $F$, let $\KMW_*(F)$
\index[notation]{kmw@$\KMW_*$}%
\index{Milnor-Witt!K-theory}%
be the $\ZZ$-graded associative (unital) ring freely generated by symbols $[a]$, for each $a\in F^\times$, of degree $1$ and by a symbol $\eta$ in degree $-1$ subject to the relations
\begin{enumerate-roman}
\item
\label{relation:steinberg}%
$[a][1-a]=0$ for any $a\neq 0,1$;
\item
\label{relation:linearity}%
$[ab]=[a]+[b]+\eta[a][b]$ for any $a,b\in F^\times$;
\item
\label{relation:commute}%
$\eta[a]=[a]\eta$ for any $a\in F^\times$;
\item
\label{relation:hyperbolic}%
$\eta(2+\eta[-1])=0$.
\end{enumerate-roman}

If $a_1,\ldots,a_n\in F^\times$, we denote by $[a_1,\ldots,a_n]$ the product $[a_1]\cdot\ldots\cdot[a_n]$. Let $\GW(F)$ be the Grothendieck-Witt ring
\index[notation]{gw@$\GW$}%
\index{Grothendieck-Witt!ring}%
\index{Grothendieck-Witt!group}%
of non degenerate bilinear symmetric forms on $F$. Associating to a form its rank yields a surjective ring homomorphism 
\[
\rk:\GW(F)\to \ZZ
\]
whose kernel is the \emph{fundamental ideal} $\I(F)$.
\index{fundamental ideal}%
\index[notation]{i@$\I$}%
We can consider for any $n\in\NN$ the powers $\I^n(F)$ and we set $\I^n(F)=\W(F)$ for $n\leq 0$, where the latter is the Witt ring of $F$. It follows from \cite{Morel12_a}*{Lemma 3.10} that we have a ring isomorphism
\[
\GW(F)\to \KMW_0(F)
\]
defined by $\langle a\rangle\mapsto 1+\eta[a]$. We will thus identify $\KMW_0(F)$ with $\GW(F)$ later on. In particular, we will denote by $\langle a\rangle$ the element $1+\eta[a]$ and by $\langle a_1,\ldots,a_n\rangle$ the element $\langle a_1\rangle+\ldots+\langle a_n\rangle$. Following Morel, we also introduce $\epsilon:=-\langle -1\rangle$ and $h:=2+\eta[-1]=1-\epsilon$.

If $\KM_*(F)$ denotes the Milnor K-theory ring defined in \cite{Milnor69_a}*{\S 1}, we have a graded surjective ring homomorphism 
\[
f:\KMW_*(F)\to \KM_*(F)
\]
defined by $f([a])=\{a\}$ and $f(\eta)=0$. Note that the left-hand side is $\epsilon$-graded commutative by \cite[Lemma 3.7]{Morel12_a} and that the right-hand side is $-1$-graded commutative. Moreover, $\ker f$ is the principal (two-sided) ideal generated by $\eta$ \cite{Morel04_a}*{Remarque 5.2}. We sometimes refer to $f$ as the \emph{forgetful} homomorphism. On the other hand, let 
\[
H:\KM_*(F)\to \KMW_*(F)
\]
be defined by $H(\{a_1,\ldots,a_n\})=h[a_1,\ldots,a_n]=(1-\epsilon)[a_1,\ldots,a_n]$. Using relation \ref{relation:steinberg} above, it is easy to check that $H$ is a well-defined graded homomorphism of $\KMW_*(F)$-modules (where $\KM_*(F)$ has the module structure induced by $f$), that we call the \emph{hyperbolic} homomorphism. As $f(\eta)=0$, we see that $fH:\KM_n(F)\to \KM_n(F)$ is the multiplication by $2$ homomorphism.

For any $a\in F^\times$, let $\langle\langle a\rangle\rangle:=\langle a\rangle -1\in \I(F)\subset \GW(F)$ and for any $a_1,\ldots,a_n\in F^\times$ let $\langle\langle a_1,\ldots,a_n\rangle\rangle$ denote the product $\langle\langle a_1\rangle\rangle\cdots\langle\langle a_n\rangle\rangle$ (our notation differs from \cite{Morel04_a}*{\S 2} by a sign). By definition, we have $\langle\langle a_1,\ldots a_n\rangle\rangle\in \I^m(F)$ for any $m\leq n$. In particular, we can define a map
\[
\KMW_n(F)\to \I^n(F)
\]
by $\eta^s[a_1,\ldots,a_{n+s}]\mapsto \langle\langle a_1,\ldots,a_{n+s}\rangle\rangle$ for any $s\in\NN$ and any $a_1,\ldots,a_{n+s}\in F^\times$. It follows from \cite{Morel12_a}*{Definition 3.3} and \cite{Morel04_a}*{Lemme 2.3} that this map is a well-defined homomorphism. Moreover, the diagram
\begin{equation}\label{diag:jn}
\begin{tikzcd}
\KMW_n(F)\ar[r]\ar[d,"f"'] & \I^n(F) \ar[d] \\
\KM_n(F) \ar[r,"s_n"] & \I^n(F)/\I^{n+1}(F)
\end{tikzcd}
\end{equation}
where $s_n$ is the map defined in \cite{Milnor69_a}*{Theorem 4.1} is Cartesian by \cite{Morel04_a}*{Théorème~5.3} (in characteristic $2$, the diagram can also be seen to be Cartesian using \cite{Morel12_a}*{Remark 3.12} but we will not use this fact in this book). 

\subsection{Residues}\label{sec:residues}

Suppose now that $F$ is endowed with a discrete valuation $v:F^\times \to \ZZ$ with valuation ring $\OO_v$, uniformizing parameter $\pi$ and residue field $k(v)$. The following theorem is due to Morel \cite{Morel12_a}*{Theorem~3.15}.
\begin{thm} \label{thm:res}
There exists a unique homomorphism of graded groups
\[
\partial_v^{\pi}:\KMW_*(F)\to \KMW_{*-1}(k(v))
\]
\index[notation]{dvpi@$\partial_v^\pi$}%
\index{Milnor-Witt!residue homomorphism}%
\index{residue homomorphism|see{Milnor-Witt residue homomorphism}}%
commuting with the product by $\eta$ and such that $\partial_v^{\pi}([\pi,u_2,\ldots,u_n])=[\overline u_2,\ldots,\overline u_n]$ and $\partial_v^{\pi}([u_1,\ldots,u_n])=0$ for any units $u_1,\ldots,u_n\in \OO_v^\times$.
\end{thm}

\begin{rem}
In analogy with \cite[Lemma 2.1]{Milnor69_a}), the first formula should read as $\partial_v^{\pi}([\pi])=1$ for $n=1$.
\end{rem}

As for Milnor K-theory, $v: F^\times \to \ZZ$, there also exists a specialization map
\[
s_v^\pi:\KMW_*(F)\to \KMW_*(k(v)),
\]
\index{Milnor-Witt!specialization homomorphism}%
\index{specialization homomorphism|see{Milnor-Witt specialization homomorphism}}%
\index[notation]{svpi@$s_v^\pi$}%
which is a graded ring map, and that can be deduced from $\partial_v^\pi$ by the formula 
\begin{equation} \label{eq:sandres}
s_v^\pi(\alpha)=\partial_v^\pi([\pi]\alpha)-[-1]\partial_v^\pi(\alpha)
\end{equation}
(actually, one usually constructs $\partial_v^\pi$ and $s_v^\pi$ together by a trick of Serre, see \cite{Morel12_a}*{Lemma~3.16}). 
\begin{lem}
Both the kernel of $\partial_v^\pi$ and the restriction of $s_v^\pi$ to this kernel are independent of the choice of the uniformizer $\pi$.
\end{lem}
\begin{proof}
If $\pi$ and $u\pi$ are uniformizers, for some $u \in \OO_v^\times$, then $\partial_v^{u\pi}=\langle \bar{u}\rangle\partial_v^\pi$. Indeed, by uniqueness in Theorem \ref{thm:res}, it suffices to check this equality on elements of the form $[u_1,\ldots,u_n]$ or $[\pi,u_2,\ldots,u_n]$ with the $u_i$'s units, and on these it is straightforward (alternatively, use \cite[Lemma 3.19]{Morel12_a} for $E=F$). Formula \eqref{eq:sandres} then shows that if $\alpha\in \ker(\partial_v)$, then $s_v^\pi(\alpha)=\partial_v^\pi([\pi]\alpha)$, and we compute
\begin{align*}
s_v^{u\pi}(\alpha) &= \partial_v^{u\pi}([u\pi]\alpha) = \langle\bar{u}\rangle \partial_v^\pi([u\pi]\alpha) \\
 &= \langle\bar{u}\rangle \partial_v^\pi([u]\alpha + \langle u\rangle[\pi]\alpha) \\
 &= \langle\bar{u}\rangle \epsilon [\bar{u}] \partial_v^\pi(\alpha) + \partial_v^\pi([\pi]\alpha) && \text{by \cite{Morel12_a}*{Prop.~3.17}} \\
 &= s_v^\pi(\alpha) && \text{since $\alpha \in \ker(\partial_v^\pi)$} \qedhere
\end{align*}
\end{proof}

This lemma allows one to define unramified Milnor-Witt K-theory sheaves as follows. If $X$ is a smooth integral $k$-scheme, any point $x\in X^{(1)}$ defines a discrete valuation $v_x$ for which we can choose a uniformizing parameter $\pi_x$. We then set for any $n\in \ZZ$
\[
\KMW_n(X):=\ker \Big(\KMW_n\big(k(X)\big)\to \bigoplus_{x\in X^{(1)}}\KMW_{n-1}\big(k(x)\big)\Big)
\]
where the map is induced by the residue homomorphisms $\partial_{v_x}^{\pi_x}$. This kernel is independent of the choices of uniformizers $\pi_x$ by the lemma.
 
Let $i:V\subset X$ be a codimension $1$ closed smooth subvariety defining a valuation $v$ on $k(X)$ with uniformizing parameter $\pi$. We then have $k(V)=k(v)$ and the graded ring map 
\[
s_v^{\pi}:\KMW_*\big(k(X)\big)\to \KMW_*\big(k(V)\big)
\]
restricts to a map $\KMW_*\big(X\big)\to \KMW_*\big(k(V)\big)$ independent of the choice of the uniformizer because $\KMW_n(X) \subseteq \ker(\partial_v)$ so the lemma applies again. Finally, it actually lands in $\KMW_*\big(V\big)$ by \cite{Morel12_a}*{proof of Lemma~2.12} and thus defines a morphism 
\[
i^*:\KMW_n(X)\to \KMW_n(V)
\]
satisfying $i^*(\alpha)=\partial_v^{\pi}([\pi]\alpha)$. Working inductively and locally, the same method shows that we can define pull-back maps $j^*$ for any smooth closed immersion $j:Z\to X$ \cite{Morel12_a}*{p.~21}. On the other hand, one can check that a smooth morphism $h:Y\to X$ induces a homomorphism $h^*:\KMW_n(X)\to \KMW_n(Y)$ (e.g. \cite{Feld20_a}*{Proposition 6.6}) and it follows from the standard graph factorization $X\to X\times_k Y\to Y$ that any morphism $f:X\to Y$ gives rise to a pull-back map $f^*$. Thus $X\mapsto \KMW_n(X)$ defines a presheaf $\sKMW_n$ on $\smk$ which turns out to be a Nisnevich sheaf \cite{Morel12_a}*{Proposition 2.8, Lemma~2.12} (one may also use \cite{Feld21_a}*{Theorem 4.1.7}). We call it the ($n$-th)\emph{Milnor-Witt sheaf}. 
\index[notation]{kmws@$\sKMW_*$}%
\index{Milnor-Witt!K-theory!sheaf}%

Recall that one can also define residues for Milnor K-theory \cite{Milnor69_a}*{Lemma~2.1} and therefore an unramified Nisnevich sheaf $\sKM_n$ on $\smk$. 
\index{Milnor!K-theory!sheaf}%
\index[notation]{km@$\KM_*$}%
It is easy to check that both the forgetful and the hyperbolic homomorphisms commute with residue maps. As a consequence, we get morphisms of sheaves $f:\sKMW_n\to \sKM_n$ and $H:\sKM_n\to \sKMW_n$ for any $n\in\ZZ$ and the composite $fH$ is the multiplication by $2$ map. 

The multiplication map $\KMW_n(F)\times \KMW_m(F)\to \KMW_{n+m}(F)$ induces for any $m,n\in\ZZ$ a morphism of sheaves
\[
\sKMW_n\times \sKMW_m\to \sKMW_{n+m}
\]
that is compatible with the corresponding product on Milnor K-theory sheaves via the forgetful map.
 
\subsection{Twisting by line bundles} \label{sec:twisting}

We will also need a version of Milnor-Witt K-theory twisted by graded line bundles, slightly generalizing usual twists by line bundles following  \cite{Schmid98_a}*{\S 1.2} and \cite{Morel12_a}*{\S 5}. Recall first from \cite{Del87_a}*{\S 4} the construction of the category $\glb(X)$ of ($\ZZ$-)graded line bundles over a scheme $X$. 
\index{line bundle!graded}%
\index[notation]{detz@$\glb$}%
The objects of $\glb(X)$ are pairs $(i,\Lb)$ where $i$ is a locally constant integer on $X$ and $\Lb$ is a line bundle. Morphisms of pairs are given by isomorphisms of line bundles in case the locally constant integers agree and are empty else. In particular, the automorphisms of $(i,\Lb)$ coincide with the automorphisms of $\Lb$, i.e.\ with $\OO_X^\times$. The category $\glb(X)$ is endowed with a tensor product defined by
\[
(i,\Lb)\otimes (i^\prime,\Lb^\prime)=(i+i^\prime,\Lb\otimes \Lb^\prime).
\]
The associativity constraint is given by the associativity constraint of the usual tensor product. Moreover, there is a commutativity isomorphism
\[
(i+i^\prime,\Lb\otimes \Lb^\prime)\to (i^\prime+i,\Lb^\prime\otimes \Lb).
\]
given by the isomorphism $\Lb\otimes \Lb^\prime\to \Lb^\prime\otimes \Lb$ defined on sections by $l\otimes l^\prime\mapsto (-1)^{ii^\prime}l^\prime\otimes l$. Each graded line bundle $(i,\Lb)$ has a tensor inverse, namely $(-i,\Lb^\vee)$ where the isomorphism
\[
(-i,\Lb^\vee)\otimes (i,\Lb)\to (0,\OO_X)
\]
is induced by the usual isomorphism $\Lb^\vee\otimes \Lb\simeq \OO_X$ (note that the isomorphism $(i,\Lb)\otimes (-i,\Lb^\vee)\to  (0,\OO_X)$ is then induced by $(-1)^i$-times the canonical isomorphism $\Lb\otimes \Lb^\vee\simeq \OO_X$). 
One checks that $\glb(X)$ is a commutative Picard category in the sense of \cite{Del87_a}*{\S 4.1}.

Let $F$ be a field and let $(i,L)$ be an object of $\glb(F)$. One can consider the group ring $\ZZ[F^\times]$ and the $\ZZ[F^\times]$-module $\ZZ[L^\times]$ where $L^\times=L\setminus \{0\}$. Letting $a \in F^\times$ act by multiplication by $\langle a\rangle$ defines a $\ZZ$-linear action of the group $F^\times$ on $\GW(F)=\KMW_0$, which therefore extends to a ring morphism $\ZZ[F^\times] \to \KMW_0(F)$. Thus, we get a $\ZZ[F^\times]$-module structure on $\KMW_n(F)$ for any $n$, and the action is central (since $\KMW_0(F)$ is central in $\KMW_n(F)$). We then define the $n$-th Milnor-Witt group of $F$ twisted by $(i,L)$ as 
\[
\KMW_n(F,(i,L))=\KMW_n(F) \otimes_{\ZZ[F^\times]} \ZZ[L^\times]. 
\]
\index[notation]{kmwfil@$\KMW_*(F,(i,L))$}%
\index{Milnor-Witt!twisted K-theory}
Note that the right-hand group doesn't depend on $i$. The introduction of this extra integer will nevertheless have its importance when dealing with products of twisted Milnor-Witt $K$-theory groups.
On Nisnevich sheaves, we perform a similar construction. Let $\ZZ[\Gm]$ be the Nisnevich sheaf on $\smk$ associated to the presheaf $U \mapsto \ZZ[\OO(U)^\times]$. The morphism of sheaves of groups $\Gm\to (\sKMW_0)^\times$ defined by $u\mapsto \langle u\rangle$ for any $u \in \OO(U)^\times$ extends to a morphism of sheaves of rings $\ZZ[\Gm] \to \sKMW_0$, turning $\sKMW_n$ into a $\ZZ[\Gm]$-module, the action being central. 

Let now $(i,\Lb)$ be an object of $\glb(X)$, and let $\ZZ[\Lb^\times]$ be the Nisnevich sheafification of $U \mapsto \ZZ[\Lb(U)\setminus 0]$. Following \cite{Morel12_a}*{Chapter 5}, we define the Nisnevich sheaf on $\sm{X}$, the category of smooth schemes over $X$, by 
\[
\sKMW_n(i,\Lb)= \sKMW_n\otimes_{\ZZ[\Gm]} \ZZ[\Lb^\times].
\]
\index[notation]{kmwil@$\sKMW_*(i,\Lb)$}%
\index{Milnor-Witt!K-theory!twisted sheaf}%
(this is the sheaf tensor product and again, it doesn't depend on $i$).


\section{Transfers in Milnor-Witt K-theory}\label{sec:cohomological}

\index{transfer|see{Milnor-Witt K-theory transfer}}%
\index{Milnor-Witt!K-theory!transfer}%

A very important feature of Milnor-Witt K-theory is the existence of transfers for finite field extensions. They are more subtle than the transfers for Milnor K-theory, and for this reason we explain them in some details in this section insisting on the relations with transfers of symmetric bilinear forms.  We suppose in this section that the fields are of characteristic different from $2$. In the latter case, one can also define transfers following \cite{Morel12_a}*{Section \S 5.1} but at the cost of technicalities involving canonical sheaves that would render the following discussion less transparent.

Recall first that the \emph{geometric} transfers in Milnor-Witt K-theory are defined, for a monogeneous finite field extension $K=F(x)$, using the split exact sequence \cite{Morel12_a}*{Theorem~3.24}
\[
\begin{tikzcd}
0 \ar[r] & \KMW_n(F) \ar[r] & \KMW_n(F(t)) \ar[r,"\sum\partial_x^{\pi}"] & \displaystyle{\bigoplus_{x\in (\Aone_F)^{(1)}} \KMW_{n-1}(F(x))} \ar[r] & 0
\end{tikzcd}
\] 
where $\partial_x^{\pi}$ are the residue homomorphisms of Theorem \ref{thm:res} associated to the valuations corresponding to $x$ and uniformizing parameters $\pi$ the minimal polynomial of $x$ over $F$. If $\alpha\in \KMW_{n-1}(F(x))$, its geometric transfer is defined following \cite{Morel12_a}*{Section \S 4.2} by choosing a preimage in $\KMW_n(F(t))$ and then applying the residue homomorphism $-\partial_{\infty}$ corresponding to the valuation at infinity (with uniformizing parameter $-\frac 1t$). The corresponding homomorphism $\KMW_{n-1}(K)=\KMW_{n-1}(F(x))\to \KMW_{n-1}(F)$ is denoted by $\tau_F^{K}$. It turns out that the geometric transfers do not generalize well to arbitrary finite field extensions $K/F$, and one has to modify them in a suitable way as follows.  

Let again $F\subset K$ be a field extension of degree $n$ generated by $x\in K$. Let $p$ be the minimal polynomial of $x$ over $F$. We can decompose the field extension $F\subset K$ as $F\subset K^{sep}\subset K$, where $K^{sep}$ is the separable closure of $F$ in $K$. If $\charac F=l\neq 0$, then the minimal polynomial $p$ can be written as $p(t)=p_0(t^{l^m})$ for some $m\in\NN$ and $p_0$ separable. Then $F^{sep}=F(x^{l^m})$ and $p_0$ is the minimal polynomial of $x^{l^m}$ over $L$. Following \cite{Morel12_a}*{Definition~4.26}, we set $\omega_0(x):= p_0^\prime(x^{l^m})\in K^\times$ if $l=\charac F\neq 0$ and $\omega_0(x)=p^\prime(x)\in K^\times$ if $\charac F=0$. Morel then defines \emph{cohomological} transfers as the composites
\[
\KMW_n(K)\stackrel{\langle \omega_0(x)\rangle }{\longrightarrow} \KMW_n(K)\stackrel{\tau_F^K}{\longrightarrow} \KMW_n(F)
\]
and denotes them by $\Tr_F^K(x)$. If now $K/F$ is an arbitrary finite field extension, we can write
\[
F=F_0\subset F_1\subset\ldots\subset F_m=K
\]
where $F_{i}/F_{i-1}$ is finite and generated by some $x_i\in F_i$ for any $i=1,\ldots,m$. We then set $\Tr^K_F:=\Tr_F^{F_1}(x_1)\circ\ldots \circ \Tr_{F_{m-1}}^{K}(x_m)$. It turns out that this definition is independent of the choice of the subfields $F_i$ and of the generators $x_i$ \cite{Morel12_a}*{Theorem~4.27}.

\subsection{Transfers of bilinear forms}\label{subsec:bilinear}

The definition of geometric and cohomological transfers can be recovered from the transfers in Milnor K-theory as well as the Scharlau transfers on bilinear forms as we now explain.

Recall that the Milnor-Witt K-theory group $\KMW_n(F)$ fits into a Cartesian square
\[
\begin{tikzcd}
\KMW_n(F) \ar[r] \ar[d] & \I^n(F) \ar[d] \\ 
\KM_n(F)\ar[r,"s_n"] & \overline {\I}^n(F)
\end{tikzcd}
\]
for any $n\in\ZZ$, where $\I^n(F)=\W(F)$ for $n<0$, $\KM_n(F)=0$ for $n<0$ and $\overline \I^n(F):=\I^n(F)/\I^{n+1}(F)$ for any $n\in\NN$. 

If $F\subset K$ is a finite field extension, then any non-zero $K$-linear homomorphism $f:K\to F$ induces a transfer morphism $f_*:\GW(K)\to \GW(F)$ (\cite{Lam73_a}*{VII, Proposition~1.1}). It follows from \cite{Milnor73_a}*{Lemma~1.4} that this homomorphism induces transfer homomorphisms $f_*:\I^n(K)\to \I^n(F)$ for any $n\in\ZZ$ and therefore transfer homomorphisms $\overline {f_*}:\overline {\I^n}(K)\to \overline {\I^n}(F)$ for any $n\in\NN$. Recall moreover that if $g:K\to F$ is another non zero $K$-linear map, there exists a unit $b\in K^\times$ such that the following diagram commutes (\cite{Lam73_a}*{Remark~(C), p.194})
\begin{equation}\label{equ:trace}
\begin{tikzcd}
\GW(K) \ar[r,"\langle b\rangle"] \ar[rd,"g_*"'] & \GW(K) \ar[d,"f_*"] \\
 & \GW(F).
\end{tikzcd}
\end{equation}

Using the split exact sequence of \cite{Milnor69_a}*{Theorem 2.3} and the residue at infinity, one can also define transfer morphisms $\Norm_{F/L}:\KM_n(K)\to \KM_n(F)$ (the notation reflects the fact that $\Norm_{F/L}$ coincides in degree $1$ with the usual norm homomorphism, e.g. \cite[Lemma 1.6]{Suslin80}).

\begin{lem}
For any non-zero linear homomorphism $f:K\to F$ and any $n\in\NN$, the following diagram commutes
\[
\begin{tikzcd}
\KM_n(K) \ar[r,"\Norm_{K/F}"] \ar[d,"s_n"'] & \KM_n(F) \ar[d,"s_n"] \\
\overline {\I}^n(K) \ar[r,"\overline f_*"] & \overline {\I}^n(F)
\end{tikzcd}
\]
\end{lem}

\begin{proof}
Observe first that for any $b\in K^\times$, we have $\langle -1,b\rangle\cdot \overline {\I}^n(K)=0$. It follows thus from Diagram \eqref{equ:trace} that $\overline f_*=\overline g_*$ for any non-zero linear homomorphisms $f,g:K\to F$. Now both the transfers for Milnor K-theory and for $\overline {\I}^n$ are functorial, and it follows that we can suppose (writing $F=F_0\subset F_1\subset\ldots\subset F_m=K$ as in the definition of cohomological transfers) that the extension $F\subset K$ is monogeneous, say generated by $x\in K$. Thus, $K\simeq F[x]/p$ for a unique monic irreducible polynomial $p\in F[x]$. Given any irreducible monic polynomial $q\in F[x]$ of degree $r$ we may define a homomorphism $f_q:F[x]/q\to F$ by $f_q(x^i)=0$ if $i=0,\ldots,r-2$ and $f_q(x^{r-1})=1$. Granted this, \cite{Scharlau72_a}*{Theorem~4.1} states that the following diagram commutes
\[
\begin{tikzcd}
0 \ar[r] & \W(F) \ar[r] & \W(F(t)) \ar[r,"\sum\partial_x^{\pi}"]\ar[d,"\partial_{\infty}^{-\frac 1t}"] & \displaystyle{\bigoplus_{q} \W(F[x]/q)} \ar[r] \ar[ld,"\sum (f_q)_*"]& 0 \\
 & &  \W(F) & 
\end{tikzcd}
\]
As the residue homomorphism for Witt groups respects the filtration by the powers of the fundamental ideal (use \cite{Milnor69_a}*{Lemma 5.7}), we obtain a commutative diagram
\[
\begin{tikzcd}
0 \ar[r] & \overline {\I}^{n+1}(F) \ar[r] & \overline {\I}^{n+1}(F(t)) \ar[r,"\sum\partial_x^{\pi}"]\ar[d,"\partial_{\infty}^{-\frac 1t}"] & \displaystyle{\bigoplus_{q} \overline {\I}^n(F[x]/q)} \ar[r] \ar[ld,"\sum (f_q)_*"]& 0 \\
 & &  \overline {\I}^n(F) & 
\end{tikzcd}
\]
We then conclude using the unicity of transfers (e.g. \cite{Fasel20_a}*{Lemma 1.15}).
\end{proof}

As a consequence of the lemma, we see that any non-zero linear map $f:K\to F$ induces a transfer homomorphism $f_*:\KMW_n(K)\to \KMW_n(F)$ for any $n\in\ZZ$. 

\begin{lem}\label{lem:geometrictransfer}
Let $F\subset K$ be a monogeneous field extension of degree $n$ generated by $x\in K$. Then the geometric transfer is equal to the transfer $f_*$ where $f$ is the $K$-linear map defined by $f(x^i)=0$ if $i=0,\ldots,n-2$ and $f(x^{n-1})=1$.
\end{lem}

\begin{proof}
The argument is the same as in the previous lemma, using this time the commutative diagram
\[
\begin{tikzcd}
0 \ar[r] & \I^{n+1}(F) \ar[r] & \I^{n+1}(F(t)) \ar[r,"\sum\partial_x^{\pi}"]\ar[d,"\partial_{\infty}^{-\frac 1t}"] & \displaystyle{\bigoplus_{q} \I^n(F[x]/q)} \ar[r] \ar[ld,"\sum (f_q)_*"]& 0 \\
 & &  \I^n(F) & 
\end{tikzcd}
\]
and the Cartesian square 
\[
\begin{tikzcd}
\KMW_n(F) \ar[r] \ar[d] & \I^n(F) \ar[d] \\ 
\KM_n(F)\ar[r,"s_n"] & \overline {\I}^n(F).
\end{tikzcd}
\qedhere
\]
\end{proof}

\begin{lem}\label{lem:trace}
Suppose that $F\subset K$ is a separable field extension, generated by $x\in K$. Then the cohomological transfer coincides with the transfer obtained via the trace map $K\to F$.
\end{lem}

\begin{proof}
This is an immediate consequence of \cite{Serre68a}*{III, \S 6, Lemme~2}. 
\end{proof}

If the extension $F\subset K$ is purely inseparable, then the cohomological transfer coincides by definition with the geometric transfer. It follows that the cohomological transfer can be computed using trace maps (when the extension is separable), the other homomorphisms described in Lemma \ref{lem:geometrictransfer} (when the extension is inseparable) or a combination of both via the factorization $F\subset K^{sep}\subset K$.

\subsection{Canonical orientations}
\label{sec:canonicalorientations}%

In order to properly define the category of finite MW-correspondences, we will have to use differential forms to twist Milnor-Witt K-theory groups. In this section, we collect a few useful facts about orientations of relative sheaves, starting with the general notion of an orientation of a line bundle.

Let $X$ be a scheme, and let $\Nb$ be a line bundle over $X$. Recall from \cite{Morel12_a}*{Definition~4.3} that an orientation of $\Nb$
\index{line bundle!orientation}%
\index{orientation!of a line bundle}%
is a pair $(\Lb,\psi)$, where $\Lb$ is a line bundle over $X$ and $\psi:\Lb\otimes \Lb\simeq \Nb$ is an isomorphism. Two orientations $(\Lb,\psi)$ and $(\Lb^\prime,\psi^\prime)$ are said to be equivalent if there exists an isomorphism $\alpha:\Lb\to \Lb^\prime$ such that the diagram
\[
\begin{tikzcd}[column sep=2ex]
\Lb\otimes \Lb \ar[rr,"\alpha\otimes\alpha"] \ar[rd,"\psi"'] &  & \Lb^\prime\otimes\Lb^\prime \ar[ld,"\psi^\prime"] \\
 & \Nb & 
\end{tikzcd}
\]
commutes. The set of equivalence classes of orientations of $\Nb$ is denoted by $\Or(\Nb)$.
\index[notation]{qn@$\Or(\Nb)$}%
Any invertible element $x$ in the global sections of $\Lb$ gives a trivialization $\OO_X \simeq \Lb$ sending $1$ to $x$. This trivialization can be considered as an orientation $(\OO_X,q_x)$ of $\Lb$ via the canonical identification $\OO_X \otimes \OO_X \simeq \OO_X$ given by the multiplication. In other words, on sections, $q_x(a \otimes b)=abx$. Clearly, $q_x=q_{x'}$ if and only if $x=u^2 x'$ for some invertible global section $u$.

We can extend the notion of orientation to graded line bundles in an obvious way. Namely, an orientation of $(i,\Nb)$ is an orientation of $\Nb$. More concretely, we can consider orientations as morphisms $\psi:(0,\Lb)\otimes (i,\Lb)\to (i,\Nb)$ with equivalence relation given as above. Note that a graded line bundle admits an orientation if and only if the underlying line bundle admits one.

Let $k\subset L\subset F$ be field extensions such that $F/L$ is finite and $F/k$ and $L/k$ are finitely generated and separable (possibly transcendental). Let $\omega_{F/L}:=\omega_{F/k}\otimes_L\omega_{L/k}^\vee$ be its relative $F$-vector space (according to our conventions, this vector space is the same as $\omega_f$, where $f:\Spec(F)\to \Spec(L)$ is the morphism induced by $L\subset F$). Our goal is to choose for such an extension a canonical orientation of $\omega_{F/L}$. We again suppose that $k$ is of characteristic different from $2$ (else, the situation is more difficult in view of \cite{Morel12_a}*{Section \S 4.1}).

Suppose first that the extension $L\subset F$ is purely inseparable. In that case, we have a canonical bijection of $\Or(F)$-equivariant sets $\Frob:\Or(\omega_{L/k}\otimes_L F)\to \Or(\omega_{F/k})$ induced by the Frobenius \cite{Schmid98_a}*{\S 2.2.4}. Any choice of a non-zero element $x$ of $\omega_{L/k}$ yields an $L$-linear homomorphism $x^\vee:\omega_{L/k}\to L$ defined by $x^\vee(x)=1$ and an orientation $\Frob(q_x)\in \Or(\omega_{F/k})$. We thus obtain a class in $\Or(\omega_{F/k}\otimes_L\omega_{L/k}^\vee)$ represented by $\Frob(q_x)\otimes q_{x^\vee}$.

If $x'=ux$ for some $u\in L^\times$, then $(x^\prime)^\vee=u^{-1}x^\vee$ and $\Frob(q_{x'})$ satisfies $\Frob(q_{x'})=q_u \Frob(q_x)\in \Or(\omega_{F/k})$. It follows that $\Frob(q_x)\otimes q_{x^\vee}=\Frob(q_{x'})\otimes q_{(x')^\vee}\in \Or(\omega_{F/k}\otimes_L\omega_{L/k}^\vee)$ and this class is thus independent of the choice of $x\in \omega_{L/k}$. By definition, we have $\omega_{F/L}:=\omega_{F/k}\otimes_L\omega_{L/k}^\vee$ and we therefore get a canonical orientation in $\Or(\omega_{F/L})$.

Suppose next that the extension $L\subset F$ is separable. In that case, the module of differentials $\Omega_{F/L}=0$ and we have a canonical isomorphism $\omega_{F/k}\simeq \omega_{L/k}\otimes_L F$. It follows that $\omega_{F/L}\simeq F$ canonically and we choose the orientation $q_1$ given by $1\in F$ under this isomorphism.
 
We have thus proved the following result.

\begin{lem}\label{lem:orientation}
Let $k\subset L\subset F$ be field extensions such that $F/L$ is finite, $L/k$ and $F/k$ are finitely generated and separable (possibly transcendental). Then there is a canonical orientation of $\omega_{F/L}$.
\end{lem}

\begin{proof}
It suffices to consider $L\subset F^{sep}\subset F$ and the two cases described above. 
\end{proof}

\section{Chow-Witt groups}\label{sec:chowwitt}

From now on, we assume that all schemes are smooth over a perfect field $k$ of arbitrary characteristic. 
Recall the construction of Milnor-Witt K-theory sheaves twisted by line bundles from section \ref{sec:twisting}.

\begin{dfn} \label{dfn:twistedCHWitt}
For any (smooth) scheme $X$, any graded line bundle $(i,\Lb)$ over $X$, any closed subset $Z\subset X$ and any $n\in\NN$, we define the \emph{$n$-th Chow-Witt group} (twisted by $(i,\Lb)$, supported on $Z$) by $\chst nZX{(i,\Lb)}:=\H^n_Z(X,\sKMW_n(i,\Lb))$. 
\index{Chow-Witt!group}%
\index[notation]{chtnzx@$\chst nZX{(i,\Lb)}$}%
If $(i,\Lb)=(0,\OO_X)$ (resp. $Z=X$), we omit $(i,\Lb)$ (resp. $Z$) from the notation. 
\end{dfn}
Provided the base field $k$ is perfect, the Rost-Schmid complex defined by Morel in \cite{Morel12_a}*{Chapter 5} provides a flabby resolution of $\sKMW_n(i,\Lb)$ for any $n$ and any $(i,\Lb)$, and we can use it to compute the cohomology of this sheaf (\cite{Morel12_a}*{Theorem 5.41}). Alternatively, we can use the so-called Milnor-Witt complex (with coefficients in Milnor-Witt $K$-theory) of Feld \cite{Feld20_a}*{Section \S 7} to compute the relevant Chow-Witt groups. This follows from \cite{Feld20_a}*{Theorem 8.1} (compare also with \cite{Feld20_a}*{Proposition 7.6}). In case $k$ is of characteristic different from $2$, it follows from \eqref{diag:jn} above and \cite{Fasel08_a}*{Remark 7.3.1} that this definition coincides with the one given in \cite{Fasel08_a}*{Définition~10.2.14}. 

More precisely, the degree $j$-th term of either the Rost-Schmid complex or the relevant Milnor-Witt complex is of the form
\[
\displaystyle{\bigoplus_{x\in X^{(j)}} \KMW_{n-j}(k(x),(j,\Lambda(x))\otimes (i,\Lb_x))}
\]
where $\Lb_x$ is the pull-back of $\Lb$ to $\Spec(k(x))$ and, according to our conventions $\Lambda(x)=\wedge^n(\mathfrak m_x/\mathfrak m_x^2)^\vee$. The differentials, which are described in \cite{Morel12_a}*{Chapter~5}, build on the residue maps defined in Section \ref{sec:residues}.  Note that the Rost-Schmid complexes for $(i,\Lb)$ and $(j,\Lb)$ are canonically isomorphic for each $i,j\in \ZZ$, the extra grading being only a convenient way to treat the products in Chow-Witt theory. For this reason, we sometimes simply write $\chst nZX{\Lb}$ in place of $\chst nZX{(i,\Lb)}$ when $i$ is clear from the context. 

The groups $\chst nZX{(i,\Lb)}$ are contravariant in $X$ and covariant in $(i,\Lb)$ in the following way. If $f:X\to Y$ is a morphism of smooth schemes and $(i,\Lb)$ is a graded line bundle on $Y$, then there is a morphism $f^*:\chst nZY{\Lb} \to \chst n{f^{-1}(Z)}X{f^*\Lb}$. In case $k$ is of characteristic different from $2$, $f^*$ is defined explicitly at the level of the homology of the Rost-Schmid complex in 
\cite{Fasel07_a}*{Definition 7.1}. In general, one may appeal to \cite{Feld20_a}*{Definition 5.6} and \cite{Feld20_a}*{Definition 10.1} using the graph factorization to get a concrete description. If $(i,\Lb)\to (i,\Lb^\prime)$, there is an induced homomorphism $\chst nZY{\Lb}\to \chst nZY{\Lb^\prime}$ given by the underlying isomorphism of line bundles.

 If $f:X\to Y$ is a finite morphism between smooth schemes of respective (constant) dimension $d_X$ and $d_Y$, then there is a push-forward map 
\[
f_*:\chst nZX{(d_X,\omega_{X/k})\otimes (i,f^*\Lb)}\to \chst {n+d_Y-d_X}{f(Z)}Y{(d_Y, \omega_{Y/k})\otimes(i,\Lb)}
\]
for any line bundle $\Lb$ over $Y$ \cite{Morel12_a}*{Corollary~5.30}. More generally, one can define a push-forward map for any proper morphism $f:X\to Y$ between smooth schemes of constant dimensions $d_X$ and $d_Y$ using \cite{Feld20_a}*{Definition 5.5} and \cite{Feld20_a}*{Proposition 6.6} (see also \cite{Fasel08_a}*{Corollaire~10.4.5} in characteristic different from $2$). Actually, the push-forward map can be slightly generalized if one considers supports. If $f:X\to Y$ is a morphism of smooth schemes and $Z\subset X$ is a closed subscheme which is finite over $W\subset Y$, then we can define a push-forward map (with the corresponding gradings for the line bundles)
\[
f_*:\chst nZX{\omega_{X/k}\otimes f^*\Lb}\to \chst {n+d_Y-d_X}{W}Y{\omega_{Y/k}\otimes\Lb}
\]
along the formula given in \cite{Morel12_a}*{p.~125}. Indeed, it suffices to check that the proof of \cite{Morel12_a}*{Corollary~5.30} holds in that case, which is direct. Alternatively, one can use \cite{Feld20_a}*{Section \S 7} to obtain isomorphisms
\[
\chht{d_X-n}{Z}{i^*\Lb} \to \chst nZX{\omega_{X/k}\otimes\Lb}
\]
for any closed subset $i:Z\subset X$ (here $Z$ is endowed with its reduced structure and is in general not smooth) and any line bundle $\Lb$ over $X$. Note that the left-hand side is defined for any scheme of finite type over $k$. We obtain the push-forward with supports 
\[
f_*:\chst nZX{\omega_{X/k}\otimes f^*\Lb}\to \chst {n+d_Y-d_X}{W}Y{\omega_{Y/k}\otimes\Lb}
\]
using again \cite{Feld20_a}*{Definition 5.5, Proposition 6.6}.

Observe now that the forgetful morphism of sheaves $\sKMW_n(\Lb)\to \sKM_n$ yields homomorphisms $\chst nZX{\Lb}\to \CH^n_Z(X)$ for any $n\in\NN$, while the hyperbolic morphism $\sKM_n\to \sKMW_n(\Lb)$ yields homomorphisms $\CH^n_Z(X)\to \chst nZX{\Lb}$ for any $n\in\NN$. The composite $\sKM_n\to \sKMW_n(\Lb)\to \sKM_n$ being multiplication by $2$, the composite homomorphism
\[
\CH^n_Z(X)\to \chst nZX{\Lb}\to \CH^n_Z(X)
\]
is also the multiplication by $2$. Both the hyperbolic and forgetful homomorphisms are compatible with the pull-back and the push-forward maps. Moreover, the total Chow-Witt group of a smooth scheme $X$ is endowed with a ring structure refining the intersection product on Chow groups (i.e.\ the forgetful homomorphism is a ring homomorphism). More precisely, as for usual Chow groups, there is an external product 
\[
\chst nZX{(i,\Lb)}\times \chst mWX{(j,\Nb)}\to \chst {m+n}{Z\times W}{X\times X}{(i,p_1^*\Lb)\otimes (j,p_2^*\Nb)}
\]
commuting to pull-backs and push-forwards. Strictly speaking, the external product defined either in \cite{Fasel07_a}*{Section \S 4} or \cite{Feld20_a}*{Section \S 11} ignores supports. One may however use \cite{Feld20_a}*{Section \S 11} to check that it takes the above form. 

 By pulling back along the diagonal, it yields a cup-product
\[
\chst nZX{(i,\Lb)}\times \chst mWX{(j,\Nb)}\to \chst {m+n}{Z\cap W}X{(i,\Lb)\otimes(j,\Nb)}
\]
\index{cup-product!Chow-Witt group}%
for any $m,n\in\ZZ$, any graded line bundles $(i,\Lb)$ and $(j,\Nb)$ over $X$ and any closed subsets $Z,W\subset X$. It is associative, and its unit is given by the pull-back to $X$ of $\langle 1\rangle\in \sKMW_0(k)=\GW(k)$ \cite{Fasel07_a}*{\S 6} (or \cite{Feld20_a}*{Section \S 11} once again). 
This product structure is graded commutative in the sense that the following diagram, in which the horizontal morphisms are the multiplication maps, the left vertical arrow is the permutation of the factors and the right vertical map is the commutativity isomorphism of graded line bundles,
\[
\label{diag:graded-commutative}%
\begin{tikzcd}
\chst nZX{(i,\Lb)}\times \chst mWX{(j,\Nb)} \ar[r] \ar[d] & \chst {m+n}{Z\cap W}X{(i,\Lb)\otimes(j,\Nb)} \ar[d] \\
\chst mWX{(j,\Nb)}\times \chst nZX{(i,\Lb)} \ar[r] &  \chst {m+n}{Z\cap W}X{(j,\Nb)\otimes(i,\Lb)} 
\end{tikzcd}
\]
is $\langle -1\rangle^{(n+i)(m+j)}$-commutative \cite{Fasel07_a}*{Remark~6.7} (or \cite{Fasel20_a}*{\S 3.4}).

For the sake of completeness, recall that Chow-Witt groups satisfy homotopy invariance by \cite{Fasel08_a}*{Corollaire~11.3.2} or \cite{Feld20_a}*{Theorem 9.4}: If $p:V\to X$ is a vector bundle, then the pull-back homomorphism
\[
p^*:\chst nZX{\Lb}\to \chst n{p^{-1}(Z)}V{p^*\Lb}
\]
is an isomorphism. 

\subsection{Some useful results}

The goal of this section is to state the analogues of some classical formulas for Chow-Witt groups. Most of them are ``obvious'' in the sense that their proofs are basically the same as for Chow groups. Before stating our first result, let us recall that two morphisms $f:X\to Y$ and $g:U\to Y$ are \emph{Tor-independent} if for every $x\in X$, $y\in Y$ and $u\in U$ such that $f(x)=g(u)=y$ we have $Tor_n^{\OO_{Y,y}}(\OO_{X,x},\OO_{U,u})=0$ for $n\geq 1$. In the following proposition, we denote by $d_Z$ the dimension of a given smooth scheme $Z$.

\begin{prop}[Base change formula]\label{prop:basechange}
Let
\[
\begin{tikzcd}
X^\prime\ar[r,"v"] \ar[d,"g"'] & X\ar[d,"f"] \\
Y^\prime\ar[r,"u"'] & Y
\end{tikzcd}
\]
be a Cartesian square of smooth schemes with $f$ proper. Suppose that $f$ and $u$ are Tor-independent. Let $d:=d_Y-d_X$ and $\tau_u=(\omega_{Y^\prime/k}\otimes u^*\omega_{Y/k}^\vee,d_{Y^\prime}-d_Y)$ (resp. $\tau_v=(\omega_{X^\prime/k}\otimes v^*\omega_{X/k}^\vee,d_{X^\prime}-d_X)$).

Then the following diagram commutes for any $n\in\NN$ and any graded line bundle $(i,\Lb)$
\[
\begin{tikzcd}
\chst nZX{(d_X,\omega_{X/k})\otimes (i,f^*\Lb)}\ar[r,"v^*"] \ar[d,"f_*"'] & \chst n{v^{-1}(Z)}{X^\prime}{-\tau_v\otimes(d_{X^\prime},\omega_{X^\prime/k})\otimes(i,v^*f^*\Lb)}\ar[d,"g_*"] \\
\chst {n+d}{f(Z)}Y{(d_Y,\omega_{Y/k})\otimes (i,\Lb)}\ar[r,"u^*"'] & \chst {n+d}{u^{-1}f(Z)}{Y^\prime}{-\tau_u\otimes(d_{Y^\prime},\omega_{Y^\prime/k})\otimes(i,u^*\Lb)}
\end{tikzcd}
\]
Here, we have used the canonical isomorphism
\[
g^*:-\tau_u\to -\tau_v
\]
induced by $g^*\Omega_{Y^\prime/Y}\simeq \Omega_{X^\prime/X}$, $\omega_{X^\prime/X}\simeq \omega_{X^\prime/k}\otimes v^*\omega_{X/k}^\vee$ and $\omega_{Y^\prime/Y}\simeq \omega_{Y^\prime/k}\otimes u^*\omega_{Y/k}^\vee$ (see Remark \ref{rem:choice} below for more explanations).
\index{Chow-Witt!group!base-change}%
\end{prop}

\begin{proof}
We first break the square into two Cartesian squares
\[
\begin{tikzcd}
X^\prime \ar[r,"{(g,v)}"] \ar[d,"g"'] & Y^\prime \times X \ar[r,"p_X"] \ar[d,"1\times f"'] & X \ar[d,"f"] \\
Y^\prime \ar[r,"\Gamma_u"'] & Y^\prime\times Y \ar[r,"p_Y"'] & Y
\end{tikzcd}
\] 
where $\Gamma_u$ is the graph of $u$ and the right-hand horizontal morphisms are the respective projections. By \cite{Fasel08_a}*{Théorème~12.3.6} or \cite{Feld20_a}*{Proposition 6.1}, we know that the generalized base change formula holds for the right-hand square. Moreover, the morphisms $(1\times f)$ and $\Gamma_u$ are Tor-independent since $f$ and $u$ are \cite{Fasel20_a}*{Lemma~3.17}. We are thus reduced to show that the formula holds if $u:Y^\prime\to Y$ (and therefore $v:X^\prime\to X$) is a regular embedding of smooth schemes. This follows from \cite{Asok16_a}*{Theorem~2.12} or \cite{Feld20_a}*{Proposition 10.7} (note that the assumptions of \cite{Asok16_a}*{Theorem~2.12} are indeed satisfied by \cite{Fasel20_a}*{Lemma~3.17}).
\end{proof}

\begin{rem}\label{rem:choice}
In all the formulas involving line bundles as ``orientations'', we have to specify the isomorphisms we use to identify them. In this paper, we will need the Base change formula in the situation where $u$ is smooth, in which case it reads as follows. For any line bundle $\Lb$ over $Y$, any integers $n\in \NN$, $i\in \ZZ$ and any closed subset $Z\subset X$, we consider the homomorphisms
\[
v^*:\chst nZX{(d_X,\omega_{X/k})\otimes (i,f^*\Lb)}\to \chst n{v^{-1}(Z)}{X^\prime}{(d_X,v^*\omega_{X/k})\otimes (i,v^*f^*\Lb)}
\]
and the isomorphisms 
\[
(d_X,v^*\omega_{X/k})\to (d_X,\omega_{X^\prime/k}\otimes \omega_{X^\prime/X}^\vee)\to (d_X,\omega_{X^\prime/k}\otimes g^*(\omega_{Y^\prime/Y})^\vee)
\]
induced by the canonical isomorphisms $v^*\omega_{X/k}\simeq \omega_{X^\prime/k}\otimes \omega_{X^\prime/X}^\vee$ given by the first fundamental exact sequence
\[
v^*\Omega_{X/k}\to \Omega_{X^\prime/k}\to \Omega_{X^\prime/X}\to 0
\]
and $g^*(\omega_{Y^\prime/Y})\simeq \omega_{X^\prime/X}$ given by \cite{Kunz86_a}*{Corollary~4.3}. Next we write
\[
(d_X,\omega_{X^\prime/k}\otimes g^*(\omega_{Y^\prime/Y})^\vee)=(d_{X^\prime},\omega_{X^\prime/k})\otimes (d_X-d_{X^\prime},g^*(\omega_{Y^\prime/Y})^\vee)
\]
to obtain a homomorphism $v^*$ with target 
\[
\chst n{v^{-1}(Z)}{X^\prime}{(d_{X^\prime},\omega_{X^\prime/k})\otimes (d_X-d_{X^\prime},g^*(\omega_{Y^\prime/Y})^\vee)\otimes (i,v^*f^*\Lb)}
\]
It follows that $g_*v^*$ lands into
\[
\chst {n+d_{Y^\prime}-d_{X^\prime}}{gv^{-1}(Z)}{Y^\prime}{(d_{Y^\prime},\omega_{Y^\prime/k})\otimes (d_X-d_{X^\prime},\omega_{Y^\prime/Y}^\vee)\otimes (i,u^*\Lb)}.
\]
For the other composite, we consider first 
\[
f_*:\chst nZX{(d_X,\omega_{X/k})\otimes (i,f^*\Lb)}\to \chst {n+d_Y-d_X}{f(Z)}Y{(d_Y,\omega_{Y/k})\otimes (i,\Lb)}
\]
whose composite with $u^*$ lands into
\[
\chst {n+d_Y-d_X}{u^{-1}f(Z)}{Y^\prime}{(d_Y,u^*\omega_{Y/k})\otimes (i,u^*\Lb)}
\]
We identify it with the latter group using
\[
u^*\omega_{Y/k}\simeq \omega_{Y^\prime/k}\otimes \omega_{Y^\prime/Y}^\vee.
\]
and $d_X-d_{X^\prime}=d_Y-d_{Y^\prime}$.
\end{rem}

\begin{rem}\label{rem:noncartesian}
There is no need for $f$ to be proper in the above proposition, as long as we consider supports which are proper over the base. More precisely, suppose that we have a Cartesian square
\[
\begin{tikzcd}
X^\prime \ar[r,"v"] \ar[d,"g"'] & X \ar[d,"f"] \\
Y^\prime\ar[r,"u"'] & Y
\end{tikzcd}
\]
of smooth schemes with $f$ and $u$ Tor-independent. Let $M\subset X$ be a closed subset such that the composite morphism $M\subset X\stackrel f\to Y$ is proper (here $M$ is endowed with its reduced scheme structure). Then the formula $u^*f_*=g_*v^*$ holds for any $\alpha\in \chst nMX{(d_X,\omega_{X/Y})\otimes (i,f^*\Lb)}$. The proof is the same as the proof of the proposition, taking supports into account. One may also directly use \cite{Feld20_a}*{Proposition 6.1, Proposition 10.7}.
\end{rem}

\begin{coro}[Projection formula]
\label{cor:pformula}%
Let $m,n\in\NN$ and let $\Lb$, $\Nb$ be line bundles over $Y$. Let $Z\subset X$ and $W\subset Y$ be closed subsets. If $f:X\to Y$ is a proper morphism of (constant) relative dimension $c\in \ZZ$, we have 
\[
f_*(\alpha)\cdot \beta=f_*(\alpha\cdot f^*(\beta))
\]
for any $\alpha\in \chst mZX{(d_X-d_Y,\omega_{f})\otimes (i,f^*\Lb)}$ and $\beta\in \chst nWY{(j,\Nb)}$.
\index{Chow-Witt!group!projection formula}%
\end{coro}

\begin{proof}
It suffices to use \cite{Feld20_a}*{Proposition 10.7} on the following Cartesian square in which the horizontal maps are regular embeddings:
\[
\begin{tikzcd}[column sep=7ex]
X \ar[r,"{(1\times f)\Delta_X}"] \ar[d,"f"'] & X \times Y \ar[d,"{(f\times 1)}"] \\
Y\ar[r,"\Delta_Y"'] & Y\times Y
\end{tikzcd}
\]
to get $\Delta_Y^*(f\times 1)_*(\alpha\times\beta)=f_*(\alpha\cdot f^*\beta)$. Note that the relevant condition 
\[
f^*N_{Y}(Y\times Y)\simeq N_X(X \times Y)
\] 
is satisfied as $N_{Y}(Y\times Y)$ is canonically isomorphic to $\Omega_{Y/k}^\vee$ and $N_X(X \times Y)$ is canonically isomorphic to $f^*\Omega_{Y/k}^\vee$  (use for instance \cite{Fasel20_a}*{\S 1.4}).

The result now follows from the equality $(f\times 1)_*(\alpha\times\beta)=f_*\alpha\times \beta$ which can be obtained using the explicit definition of the exterior products \cite{Feld20_a}*{Section 11} and the projection formula for fields \cite{Feld20_a}*{Axiom R2b}.
\end{proof}

\begin{rem}
\label{rem:leftmodule}%
There is also a formula for the left-module structure \cite{Fasel20_a}*{Theorem 3.19}: It suffices to use the graded commutativity of the twisted Chow-Witt groups (page \pageref{diag:graded-commutative}, Section \ref{sec:chowwitt}) and the commutation rule for graded line bundles. 
\end{rem}

\begin{lem}[Flat excision]
Let $f:X\to Y$ be a flat morphism of smooth schemes. Let $V\subset Y$ be a closed subset such that the morphism $f^{-1}(V)\to V$ induced by $f$ is an isomorphism. Then the pull-back morphism
\[
f^*:\chst iV{Y}{\Lb}\to \chst i{f^{-1}(V)}X{f^*\Lb}
\]
is an isomorphism for any $i\in\NN$ and any line bundle $\Lb$ over $Y$.
\index{Chow-Witt!group!excision}%
\end{lem}

\begin{proof}
As recalled above, Chow-Witt groups can be computed using the flabby resolution provided by the Rost-Schmid complex of \cite{Morel12_a}*{Chapter~5} or the Milnor-Witt complex (with coefficients in Milnor-Witt $K$-theory) of Feld \cite{Feld20_a}*{Section \S 7}, which both coincide in characteristic different from $2$ with the complex considered in \cite{Fasel08_a}*{Définition~10.2.7}. Now Chow-Witt groups with supports are obtained by considering the subcomplex of points supported on a certain closed subset. The lemma follows now from the fact that in our case $f^*$ induces (by definition) an isomorphism of complexes. 
\end{proof}

\begin{conv}
From now on, we omit the grading when we write graded line bundles. We only indicate when a grading is different from the ``obvious'' choices we made in this section.
\end{conv}

We now consider the problem of describing the cohomology of $X\times \Gm$ with coefficients in $\sKMW_j$ (for $j\in\ZZ$) in terms of the cohomology of $X$. First observe that the pull-back along the projection $p:X\times \Gm\to X$ endows the cohomology of $X\times \Gm$ with the structure of a module over the cohomology of $X$. Let $t$ be a parameter of $\Gm$. The class $[t]$ in $\KMW_1(k(t))$ actually lives in its subgroup $\sKMW_1(\Gm)$ since it clearly has trivial residues at all closed points of $\Gm$. 
Pulling back to $X\times \Gm$ along the projection to the second factor, we get an element in $\sKMW_1(X\times \Gm)$ that we still denote by $[t]$.

\begin{lem}\label{lem:explicitcontraction}
For any $i\in\NN$, any $j\in \ZZ$ and any smooth scheme $X$ over $k$, we have an isomorphism
\[
\H^{i}(X\times \Gm,\sKMW_j)=\H^i(X,\sKMW_j)\oplus \H^{i}(X,\sKMW_{j-1})\cdot [t].
\]
of $\H^0(X,\sKMW_0)=\sKMW_0(X)$-modules.
\end{lem}

\begin{proof}
The long exact sequence associated to the open immersion $X\times \Gm\subset X\times \Aone_k$ reads as
\[
\cdots\to \H^i(X\times\Aone_k,\sKMW_j)\to \H^i(X\times\Gm,\sKMW_j)\stackrel\partial\to \H^{i+1}_{X\times \{0\}}(X\times \Aone,\sKMW_j)\to \cdots
\] 
By homotopy invariance, the pull-back along the projection to the first factor $X\times\Aone_k\to X$ induces an isomorphism $\H^i(X,\sKMW_j)\to \H^i(X\times\Aone_k,\sKMW_j)$. Pulling-back along the morphism $X\to X\times\Gm$ defined by $x\mapsto (x,1)$ we get a retraction of the composite homomorphism
\[
\H^i(X,\sKMW_j)\to \H^i(X\times\Aone_k,\sKMW_j)\to \H^i(X\times\Gm,\sKMW_j)
\]
and it follows that the long exact sequence splits into short split exact sequences
\[
0\to \H^i(X\times\Aone_k,\sKMW_j)\to \H^i(X\times\Gm,\sKMW_j)\stackrel\partial\to \H^{i+1}_{X\times \{0\}}(X\times \Aone,\sKMW_j)\to 0
\] 
Now the push-forward homomorphism (together with the obvious trivialization of the normal bundle to $X\times\{0\}$ in $X\times\Aone_k$) yields an isomorphism \cite{Fasel08_a}*{Remarque~10.4.8}
\[
\iota:\H^i(X,\sKMW_{j-1})\to \H^{i+1}_{X\times \{0\}}(X\times \Aone,\sKMW_j)
\]
and it suffices then to check that the composite 
\[
\H^i(X,\sKMW_{j-1})\stackrel{[t]}\to \H^i(X,\sKMW_{j})\to \H^i(X\times \Gm,\sKMW_{j})\stackrel{\iota^{-1}\partial}\to \H^i(X,\sKMW_{j-1}) 
\]
is an isomorphism to conclude. This follows from \cite{Feld20_a}*{Lemma 6.5}  (see also \cite{Morel12_a}*{Proposition~3.17. (2)}).
\end{proof}

\begin{rem} \label{rem:ttozero}
By pull-back along the morphism $\Spec(k) \to \Gm$ sending the point to $1$, the element $[t] \in \sKMW_1(\Gm)$ maps to $[1]=0$ in $\sKMW_1(k)$. Therefore this pull-back gives the splitting of $\H^i(X,\sKMW_j)$ in the decomposition above. 
\end{rem}


\section{Finite Milnor-Witt correspondences}
\label{sec:fmwcorr}

\subsection{Admissible subsets}\label{sec:fGW}

Let $X$ and $Y$ be smooth schemes over $\Spec(k)$ and let $T\subset X\times Y$ be a closed subset. Any irreducible component of $T$ maps to an irreducible component of $X$ through the projection $X \times Y \to X$. 
\begin{dfn}
\label{dfn:admissible_subset}%
If, when $T$ is endowed with its reduced structure, the canonical morphism $T\to X$ is finite and maps any irreducible component of $T$ surjectively on an irreducible component of $X$, we say that $T$ is an \emph{admissible subset} of $X\times Y$. 
\index{admissible subset}%
We denote by $\Adm(X,Y)$ 
\index[notation]{axy@$\Adm(X,Y)$}%
the poset of admissible subsets of $X\times Y$, partially ordered by inclusions. When convenient, we sometimes consider $\Adm(X,Y)$ as the category associated to this poset.
\end{dfn}

\begin{rem} \label{rem:admissible}
Since the empty set has no irreducible component, it is admissible.
An irreducible component of an admissible subset is clearly admissible, and the irreducible admissible subsets are minimal (non-trivial) elements in $\Adm (X,Y)$. Furthermore, any finite union of admissible subsets is admissible.
\end{rem}

\begin{lem} \label{lem:admissiblesheaf}
If $f:X' \to X$ is a morphism between smooth schemes, then $T \mapsto (f \times \id_Y)^{-1}(T)$ defines a map $\Adm(X,Y) \to \Adm(X',Y)$. Furthermore, the presheaf $U \mapsto \Adm(U,Y)$ thus defined is a sheaf for the Zariski topology. 
\end{lem}
\begin{proof}
Finiteness and surjectivity are stable by base change by \cite{EGA1_a}*{6.1.5} and \cite{EGA2_a}*{3.5.2}, so the map is well-defined. The injectivity condition in the sheaf sequence is obvious. To prove the exactness in the middle, being closed is obviously a Zariski local property, so the union of the closed subsets in the covering defines a global closed subset. Both finiteness and surjectivity are properties that are Zariski local on the base, so this closed subset is admissible. 
\end{proof}

If $Y$ is equidimensional, $d=\dimn Y$ and $p_Y:X\times Y\to Y$ is the projection, we define a covariant functor 
\[
\Adm(X,Y)\to {\Ab}
\]
by associating to each admissible subset $T\in \Adm(X,Y)$ the group 
\[
\chst dT{X\times Y}{(d,p_Y^*\omega_{Y/k})}
\]
and to each morphism $T^\prime\subset T$ the extension of support homomorphism 
\[
\chst d{T^\prime}{X\times Y}{(d,p_Y^*\omega_{Y/k})}\to \chst dT{X\times Y}{(d,p_Y^*\omega_{Y/k})}.
\]
Using that functor, we set
\[
\cork(X,Y)=\varinjlim_{T\in \Adm(X,Y)}\chst dT{X\times Y}{(d,p_Y^*\omega_{Y/k})}.
\]
\index[notation]{cortk@$\cork$}%
If $Y$ is not equidimensional, then $Y=\coprod_jY_j$ with each $Y_j$ equidimensional and we set
\[
\cork(X,Y)=\prod_{j} \cork(X,Y_j).
\]
By additivity of Chow-Witt groups, if $X =\coprod_iX_i$ and $Y=\coprod_j Y_j$ are the respective decompositions of $X$ and $Y$ in irreducible components, we have
\[
\cork(X,Y)=\prod_{i,j} \cork(X_i,Y_j).
\]
The reader familiar with Voevodsky's finite correspondences is invited to look at Remark \ref{rem:cortocor} to compare the above construction with the classical one. 

\begin{nota}\label{nota:omega}
From now on, we simply write $\chst dT{X\times Y}{\omega_{Y}}$ in place of the heavier $\chst dT{X\times Y}{(d,p_Y^*\omega_{Y/k})}$ and only resort to the full notation when special care is needed.
\end{nota}

\begin{exm}\label{ex:basic}
Let $X$ be a smooth scheme of dimension $d$. Then 
\[
\cork({\Spec(k)},X)=\bigoplus_{x\in X^{(d)}} \chst d{\{x\}}X{\omega_X}=\bigoplus_{x\in X^{(d)}} \GW(k(x),\omega_{k(x)/k}).
\]
The first equality follows from the fact that any admissible subset in ${\Spec(k)}\times X=X$ is a finite set of closed points, and the second from the identification of the relevant group in the Rost-Schmid complex.

On the other hand, $\cork(X,\Spec(k))=\ch 0X=\sKMW_0(X)$ for any smooth scheme $X$.
\end{exm}

The group $\cork(X,Y)$ admits an alternate description which is often useful. Let $X$ and $Y$ be smooth schemes, with $Y$ equidimensional. For any closed subscheme $T\subset X\times Y$ of codimension $d=\dimn Y$, we have an inclusion 
\[
\chst dT{X\times Y}{{\omega_{Y}}}\hspace{1ex}\subset \bigoplus_{x\in (X\times Y)^{(d)}}\hspace{-2ex}\KMW_0(k(x),\Lambda(x)\otimes ({\omega_{Y}})_x)
\]
and it follows that
\[
\cork(X,Y)=\hspace{-3ex}\bigcup_{T\in \Adm(X,Y)}\hspace{-3ex}\chst dT{X\times Y}{{\omega_{Y}}}\hspace{1ex}\subset \hspace{-2ex}\bigoplus_{x\in (X\times Y)^{(d)}}\hspace{-3ex}\KMW_0(k(x),\Lambda(x)\otimes ({\omega_{Y}})_x).
\]
In general, the inclusion $\cork(X,Y)\subset \bigoplus_{x\in (X\times Y)^{(d)}}\KMW_0(k(x),\Lambda(x)\otimes ({\omega_{Y}})_x)$ is strict as shown by Example \ref{ex:basic}. As an immediate consequence of this description, we see that the map
\[
\chst dT{X\times Y}{{\omega_{Y}}}\to \cork(X,Y)
\]
is injective for any $T\in \Adm(X,Y)$. 

If $U$ is an open subset of a smooth scheme $V$, since an admissible subset $T \in \Adm(V,Y)$ intersects with $U \times Y$ as an admissible subset by Lemma \ref{lem:admissiblesheaf}, the pull-backs along $V \times Y \to U \times Y$ on Chow-Witt groups with support induce at the limit a map $\cork(V,Y) \to \cork(U,Y)$.

\begin{lem} \label{lem:restrictioninj}
This map is injective.
\end{lem}

\begin{proof}
We can assume $Y$ equidimensional, of dimension $d$. Let $Z=(V\setminus U)\times Y$. Let moreover $T\subset V\times Y$ be an admissible subset. Since $T$ is finite and surjective over $V$, the subset $Z\cap T$ is of codimension at least $d+1$ in $V\times Y$, which implies that $\chst d{Z\cap T}{V\times Y}{\omega_{Y}}=0$. The long exact sequence of localization with support then shows that the homomorphism
\[
\chst d{T}{V\times Y}{\omega_Y}\to \chst d{T\cap (U\times Y)}{U\times Y}{\omega_Y}
\]
is injective. On the other hand, we have a commutative diagram
\[
\begin{tikzcd}
\chst d{T}{V\times Y}{\omega_Y} \ar[r,hook] \ar[d,hook] & \chst d{T\cap (U\times Y)}{U\times Y}{\omega_Y} \ar[d,hook] \\
\cork(V,Y)\ar[r] & \cork(U,Y)
\end{tikzcd}
\]
with injective vertical maps. Since any $\alpha\in \cork(V,X)$ comes from the group $\chst d{T}{V\times Y}{\omega_Y}$ for some $T\in \Adm(X,Y)$, the homomorphism $\cork(V,Y)\to \cork(U,Y)$ is injective. 
\end{proof}

\begin{dfn}\label{dfn:support}
Let $\alpha\in \cork(X,Y)$, where $X$ and $Y$ are smooth. If $Y$ is equidimensional, let $d=\dimn Y$. The \emph{support} of $\alpha$ is the closure of the set of points $x\in (X\times Y)^{(d)}$ such that the component of $\alpha$ in $\KMW_0\big(k(x),\Lambda(x)\otimes ({\omega_{Y}})_x\big)$ is nonzero. 
If $Y$ is not equidimensional, then we define the support of $\alpha$ as the union of the supports of the components appearing in the equidimensional decomposition. 
\index{support!of a MW-correspondence}%
\end{dfn}

\begin{lem} \label{lem:supportadmis}
The support of an $\alpha \in \cork(X,Y)$ is an admissible subset, say $T$, and $\alpha$ is then in the image of the inclusion $\chst dT{X\times Y}{\omega_{Y}} \subset \cork(X,Y)$.
\end{lem}

\begin{proof}
By definition of $\cork(X,Y)$ as a direct limit, the support of $\alpha$ is included in some admissible subset $T \in \Adm(X,Y)$. Being finite and surjective over $X$, any irreducible component $T_i$ of $T$ is of codimension $\dimn Y$ in $X \times Y$. Therefore the support of $\alpha$ is exactly the union of all $T_i$ such that the component of $\alpha$ on the generic point of $T_i$ is non-zero. This is an admissible subset by Remark \ref{rem:admissible}. 

To obtain the last part of the statement, let $S \subset T$ be the support of $\alpha$ and let $U$ be the open subscheme $(X \times Y) \setminus S$. Consider the commutative diagram 
\[
\begin{tikzcd}
 \chst dT {X \times Y}{\omega_{Y}} \ar[r] \ar[d] & \chst d{T \setminus S} {U}{(\omega_{Y})_{|U}} \ar[d] \\
 \hspace{-5ex}{\displaystyle\bigoplus_{x\in (X\times Y)^{(d)}\cap T}}\hspace{-4.5ex} \KMW_0\big(k(x),\Lambda(x)\otimes ({\omega_{Y}})_x\big) \ar[r] & \hspace{-3ex}{\displaystyle\bigoplus_{x\in U^{(d)}\cap T}}\hspace{-2.5ex}\KMW_0\big(k(x),\Lambda(x)\otimes ({\omega_{Y}})_x\big)
\end{tikzcd}
\]
with injective vertical maps (still for dimensional reasons). By definition of the support, $\alpha$ maps to zero in the lower right group, so it maps to zero in the upper right one. Therefore, it comes from the previous group in the localization exact sequence for Chow groups with support, and this group is $\chst dS{X\times Y}{\omega_{Y}}$. 
\end{proof}

\begin{exm}
In contrast with usual correspondences, in general, an element of $\cork(X,Y)$ cannot be written as the sum of elements with irreducible support. Indeed, let $X=Y=\Aone$. Let moreover $T_1=\{x=y\}$ and $T_2=\{x=-y\}$ both in $\Aone\times \Aone$, so that $T_1\cap T_2=\{0\}\subset \Aone\times \Aone$. We can consider $\langle x\rangle\otimes \overline{x-y}$ in $\KMW_0(k(T_1),\Lambda(t_1))$ and $\langle x\rangle\otimes \overline{x+y}$ in $\KMW_0(k(T_2),\Lambda(t_2))$, where $t_i$ is the generic point of $T_i$ (see our conventions on page \pageref{subsection:conventions} for the definition of $\Lambda(-)$). The residue of the first one is 
\[
\langle 1\rangle\otimes \big( \overline{x-y}\wedge \overline x \big)=\langle 1\rangle\otimes \big(-\overline y\wedge \overline x\big)=\langle -1\rangle \otimes \big(\overline y\wedge \overline x\big) 
\]
in $\KMW_{-1}(k,\Lambda(0))$, while the residue of the second one is
\[
\langle 1\rangle\otimes \big(\overline{x+y}\wedge \overline x\big)=\langle 1\rangle \otimes \big(\overline y\wedge \overline x\big)
\]
in the same group. As $\langle -1\rangle+\langle 1\rangle=0\in \KMW_{-1}(k)$, it follows that the sum of the two elements above defines an unramified element in $ \chs 1{T_1\cup T_2}{\Aone\times \Aone}$. As the canonical sheaf $\omega_{\Aone}$ is trivial, we obtain an element of $\cork(\Aone,\Aone)$ which is not the sum of elements with irreducible supports (each component is ramified).
\end{exm}

Let $\alpha \in \cork(X,Y)$ with support $T$ be restricted to an element denoted by $\alpha_{|U} \in \cork(U,Y)$.  
\begin{lem} \label{lem:supporttoopen}
The support of $\alpha_{|U}$ is $T \cap U$, in other words the image of $T$ by the map $\Adm(X,Y) \to \Adm(U,Y)$. 
\end{lem}
\begin{proof}
It is straightforward from the definition of the support.
\end{proof}

\subsection{Composition of finite MW-correspondences}
\label{sec:compositionMWcorr}%

Let $X$, $Y$ and $Z$ be smooth schemes of respective dimensions $d_X,d_Y$ and $d_Z$, with $X$ and $Y$ connected. Let $V\in \Adm(X,Y)$ and $T\in \Adm(Y,Z)$ be admissible subsets. Consider the following commutative diagram where all maps are canonical projections:
\begin{equation}
\label{eq:composition}%
\begin{tikzcd}
X\times Z \ar[rrrd,bend left=1ex,"r_Z"] \ar[rddd,bend right=1ex,"p_X"'] & & & \\
 & X \times Y \times Z \ar[r,"q_{YZ}"'] \ar[d,"p_{XY}"] \ar[lu,"q_{XZ}"'] & Y \times Z \ar[r,"q_Z"'] \ar[d,"p_Y"] & Z \\
 & X \times Y \ar[r,"q_Y"'] \ar[d,"p"] & Y & \\
 & X & & 
\end{tikzcd}
\end{equation}
We have homomorphisms
\[
(p_{XY})^*:\chst {d_Y}V{X\times Y}{{\omega_{Y}}}\to \chst {d_Y}{(p_{XY})^{-1}V}{X\times Y\times Z}{(p_{XY})^*{\omega_{Y}}}
\]
and
\[
(q_{YZ})^*:\chst {d_Z}T{Y\times Z}{{\omega_{Z}}}\to \chst {d_Z}{(q_{YZ})^{-1}T}{X\times Y\times Z}{(q_{YZ})^*{\omega_{Z}}}.
\]
Let $M=(p_{XY})^{-1}V\cap (q_{YZ})^{-1}T$, endowed with its reduced structure. It follows from \cite{Mazza06_a}*{Lemmas 1.4 and 1.6} that every irreducible component of $M$ is finite and surjective over $X$. As a consequence, the map $M\to q_{XZ}(M)$ is finite and the push-forward
\[
(q_{XZ})_*:\chst {d_Y+d_Z}{M}{X\times Y\times Z}{\omega_{X\times Y\times Z} \otimes q_{XZ}^*\Lb}\to \chst {d_Z}{q_{XZ}(M)}{X\times Z}{\omega_{X\times Z}\otimes \Lb}
\]
is well-defined for any line bundle $\Lb$ over $X\times Z$. In particular for $\Lb=p_X^*\omega_{X/k}^\vee$, we get a push-forward map
\begin{multline}
\chst {d_Y+d_Z}{M}{X\times Y\times Z}{\omega_{X\times Y\times Z} \otimes (p_{XY})^*p^*\omega_{X/k}^\vee} \stackrel{(q_{XZ})_*} \longrightarrow \\
  \chst {d_Z}{q_{XZ}(M)}{X\times Z}{\omega_{X\times Z}\otimes p_X^*\omega_{X/k}^\vee}.
\end{multline}
\begin{lem}
We have canonical isomorphisms 
\begin{multline}
\label{eq:side1}%
(d_X+d_Y+d_Z,\omega_{X\times Y\times Z}) \otimes (-d_X,(p_{XY})^*p^*\omega_{X/k}^\vee)\simeq \\
(d_Y,(p_{XY})^*{\omega_{Y}})\otimes (d_Z,(q_{YZ})^*{\omega_{Z}}).
\end{multline}
and
\begin{equation}\label{eq:side2}
(d_X+d_Z,\omega_{X\times Z})\otimes (-d_X,p_X^*\omega_{X/k}^\vee)\simeq (d_Z,\omega_Z).
\end{equation}
\end{lem}

\begin{proof}
We have a canonical isomorphism 
\begin{multline}
(d_X+d_Y+d_Z,\omega_{X\times Y\times Z})\simeq \\
(d_X,(p_{XY})^*p^*\omega_{X/k})\otimes (d_Y,(p_{XY})^*q_{Y}^*\omega_{Y/k})\otimes (d_Z,(q_{YZ})^*q_Z^*\omega_{Z/k}). 
\end{multline}
Using the (graded) commutativity of the tensor product of graded line bundles, we obtain the isomorphism (\ref{eq:side1}). For the second isomorphism, we use the same argument with the canonical isomorphism
\[
(d_X+d_Z,\omega_{X\times Z})\simeq (d_X,p_X^*\omega_{X/k})\otimes (d_Z,r_Z^*\omega_{Z/k})
\]
obtained from exact sequences (1.2) and (1.3) in \cite{Fasel20_a}*{\S 1.4}.
\end{proof}

As a consequence, if we have cycles  $\beta\in  \chst {d_Y}V{X\times Y}{\omega_Y}$ and $\alpha\in \chst {d_Z}T{Y\times Z}{\omega_Z}$ the expression
\[
\alpha\circ \beta:=(q_{XZ})_*[(q_{YZ})^*\beta\cdot (p_{XY})^*\alpha]
\] 
is well-defined. Moreover, it follows from \cite{Mazza06_a}*{Lemma~1.7} that $q_{XZ}(M)$ is an admissible subset of $X\times Z$.
All the above homomorphisms commute with extension of supports, and therefore we get a well-defined composition 
\[
\circ:\cork(X,Y)\times \cork(Y,Z)\to \cork(X,Z).
\]

\begin{rem}\label{rem:everythingcommutes}
Note that in the definition of the composition we could have considered the product $(p_{XY})^*\alpha\cdot (q_{YZ})^*\beta$ in place of the product $(q_{YZ})^*\beta\cdot (p_{XY})^*\alpha$. This is actually the same, remembering the fact that $\omega_Y$ sits in degree $d_Y$, $\omega_Z$ in degree $d_Z$ and using graded commutativity of the product on Chow-Witt groups.
\end{rem}

The proof that the composition we defined is associative follows essentially from the fact that the intersection product is associative, but there are some subtleties involved when dealing with the required line bundles so we write it for the sake of completeness. 

\begin{lem}\label{lem:associative}
The composition of finite MW-correspondences is associative.
\end{lem}

\begin{proof}
We will use the following Cartesian square where all morphisms are projections onto the respective factors:
\begin{equation}\label{eq:assoc1}
\begin{tikzcd} 
 & & X \times Y \ar[r] & Y & \\
 & X \times T & X \times Y \times T \ar[r,"q_{YT}"] \ar[l,"q_{XT}"'] \ar[u,"s_{XY}"] & Y \times T \ar[u] &  \\
Z \times T \ar[d] & X \times Z \times T \ar[d,"p_{XZ}"'] \ar[u,"p_{XT}"] \ar[l,"r_{ZT}"'] & X \times Y \times Z \times T \ar[r,"q_{YZT}"] \ar[d,"p_{XYZ}"'] \ar[l,"q_{XZT}"'] \ar[u,"p_{XYT}"] & Y \times Z \times T \ar[r,"q_{ZT}"] \ar[d,"p_{YZ}"'] \ar[u,"p_{YT}"] & Z \times T \ar[d] \\
Z & X \times Z \ar[d] \ar[l] & X \times Y \times Z \ar[r,"q_{YZ}"] \ar[d,"p_{XY}"'] \ar[l,"q_{XZ}"'] & Y \times Z \ar[r] \ar[d] & Z  \\
 & X & X \times Y \ar[r] \ar[l] & Y. & 
\end{tikzcd}
\end{equation}

Now, let $\alpha\in \cork(X,Y)$, $\beta\in \cork(Y,Z)$ and $\gamma\in \cork(Z,T)$. In our computation, we treat them as elements of some Chow-Witt group, and consider their push-forwards and pull-backs as usual. We omit the relevant line bundles, all our choices being canonical and already mentioned in the previous paragraphs. The composite $\beta\circ\alpha$ is represented by the cycle $(q_{XZ})_*(p_{XY}^*\alpha\cdot q_{YZ}^*\beta)$, while the composite $\gamma\circ (\beta\circ\alpha)$ is given by
\[
(p_{XT})_*\left( p_{XZ}^*(q_{XZ})_*(p_{XY}^*\alpha\cdot q_{YZ}^*\beta)\cdot r_{ZT}^*\gamma\right).
\]
Using the base change formula (Proposition \ref{prop:basechange}, with the canonical isomorphisms of Remark \ref{rem:choice}), we obtain
\[
p_{XZ}^*(q_{XZ})_*(p_{XY}^*\alpha\cdot q_{YZ}^*\beta)=(q_{XZT})_*p_{XYZ}^*(p_{XY}^*\alpha\cdot q_{YZ}^*\beta).
\]
Next, we can use the projection formula (Corollary \ref{cor:pformula}) to get
\[
(q_{XZT})_*p_{XYZ}^*(p_{XY}^*\alpha\cdot q_{YZ}^*\beta)\cdot r_{ZT}^*\gamma=(q_{XZT})_*\left(p_{XYZ}^*(p_{XY}^*\alpha\cdot q_{YZ}^*\beta)\cdot q_{XZT}^*r_{ZT}^*\gamma\right).
\]
Since the product on Chow-Witt groups is associative, it follows that the composite $\gamma\circ (\beta\circ\alpha)$ is the push-forward along the projection $X\times Y\times Z\times T\to X\times T$ of the product of the pull-backs of $\alpha,\beta,\gamma$ along the respective projections. 

We now turn to the computation of $(\gamma\circ \beta)\circ\alpha$. The composite $\gamma\circ\beta$ is given by $(p_{YT})_*(p_{YZ}^*\beta\cdot q_{ZT}^*\gamma)$, while $(\gamma\circ \beta)\circ\alpha$ is of the form
\[
(q_{XT})_*\left(s_{XY}^*\alpha\cdot q_{YT}^*(p_{YT})_*(p_{YZ}^*\beta\cdot q_{ZT}^*\gamma)\right).
\]
Using the base change formula once again, we obtain
\[
q_{YT}^*(p_{YT})_*(p_{YZ}^*\beta\cdot q_{ZT}^*\gamma)=(p_{XYT})_*q_{YZT}^*(p_{YZ}^*\beta\cdot q_{ZT}^*\gamma).
\]
Here, $(p_{XYT})_*q_{YZT}^*(p_{YZ}^*\beta\cdot q_{ZT}^*\gamma)$ is a cycle of codimension $d_T$ (the dimension of $T$) with coefficients in the line bundle $\omega_T$.
Using Remark \ref{rem:everythingcommutes}, we see that 
\[
s_{XY}^*\alpha\cdot (p_{XYT})_*q_{YZT}^*(p_{YZ}^*\beta\cdot q_{ZT}^*\gamma)= (p_{XYT})_*q_{YZT}^*(p_{YZ}^*\beta\cdot q_{ZT}^*\gamma)\cdot s_{XY}^*\alpha
\]
and we can use the projection formula (Corollary \ref{cor:pformula}) to obtain 
\[
(p_{XYT})_*q_{YZT}^*(p_{YZ}^*\beta\cdot q_{ZT}^*\gamma)\cdot s_{XY}^*\alpha=(p_{XYT})_*\left(q_{YZT}^*(p_{YZ}^*\beta\cdot q_{ZT}^*\gamma)\cdot p_{XYT}^*s_{XY}^*\alpha\right).
\]
Now, $q_{YZT}^*(p_{YZ}^*\beta\cdot q_{ZT}^*\gamma)$ is a cycle of codimension $d_Z+d_T$ with coefficients in $\omega_Z\otimes \omega_T$ and $p_{XYT}^*s_{XY}^*\alpha$ is a cycle of codimension $d_Y$ with coefficients in $\omega_Y$. Applying Remark \ref{rem:everythingcommutes} once again, we obtain
\[
q_{YZT}^*(p_{YZ}^*\beta\cdot q_{ZT}^*\gamma)\cdot p_{XYT}^*s_{XY}^*\alpha=p_{XYT}^*s_{XY}^*\alpha\cdot q_{YZT}^*(p_{YZ}^*\beta\cdot q_{ZT}^*\gamma)
\]
showing that $(\gamma\circ \beta)\circ\alpha$ is also equal to the push-forward along the $X\times Y\times Z\times T\to X\times T$ of the product of the pull-backs of $\alpha,\beta,\gamma$ along the respective projections.
\end{proof}

\subsection{Morphisms of schemes and finite MW-correspondences}
\label{subsec:embedding}%

Let $X,Y$ be smooth schemes of respective dimensions $d_X$ and $d_Y$. Let $f:X\to Y$ be a morphism and let $\Gamma_f:X\to X\times Y$ be its graph. Then $\Gamma_f(X)$ is of codimension $d_Y$ in $X\times Y$, finite and surjective over $X$. If $p_X:X\times Y\to X$ is the projection map, then $p_X\Gamma_f=\id$ and it follows that we have isomorphisms $\OO_X\simeq \omega_{X/k}\otimes \Gamma_f^*p_X^*\omega_{X/k}^\vee$ and $\omega_{X\times Y/k}\otimes p_X^*\omega_{X/k}^\vee\simeq p_Y^*\omega_{Y/k}$, 
where the latter is obtained via 
\[
(d_X+d_Y,\omega_{X\times Y/k})\simeq (d_X,p_X^*\omega_{X/k})\otimes (d_Y,p_Y^*\omega_{Y/k})
\]
and the usual commutation rule.

We therefore obtain a finite push-forward
\[
(\Gamma_f)_*: \sKMW_0(X)\to \chst {d_Y}{\Gamma_f}{X\times Y}{\omega_Y}
\]
We denote by $\grph f$
\index[notation]{gaf@$\grph f$}%
\index{graph|see{Milnor-Witt finite correspondences}}%
the class of $(\Gamma_f)_*(\langle 1\rangle)$ in $\chst {d_Y}{\Gamma_f}{X\times Y}{\omega_Y}$. In particular, when $X=Y$ and $f=\id$, we set $1_X:=\grph {\id}$. Using \cite{Fasel07_a}*{Proposition 6.8} we can check that $1_X$ is the identity for the composition defined in the previous section.

\begin{exm}\label{ex:action}
Let $X$ be a smooth scheme over $k$. The diagonal morphism induces a push-forward homomorphism
\[
\Delta_*:\sKMW_0(X)\to \chst {d_X}{X}{X\times X}{\omega_X}
\] 
which satisfies $\Delta_*(\langle 1\rangle)=1_X$. For any $\alpha\in \sKMW_0(X)$, any admissible subset $T\in \Adm(X,X)$ and any $\beta\in \chst {d_X}{T}{X\times X}{\omega_X}$, one checks that the composition $\beta\circ\alpha$ coincides with the intersection product $\beta\cdot\alpha$. It follows that $\Delta_*$ induces a ring homomorphism $\sKMW_0(X)\to \cork(X,X)$. For any smooth scheme $Y$, composition of morphisms endows the group $\cork(Y,X)$ with the structure of a left $\sKMW_0(X)$-module and a right $\sKMW_0(Y)$-module.
\end{exm}

\begin{dfn} \label{def:cortilde}
Let $\cork$ be the category whose objects are smooth schemes and whose morphisms are the abelian groups $\cork(X,Y)$ defined in Section \ref{sec:fGW}. We call it the \emph{category of finite MW-correspondences over $k$}.
\index{Milnor-Witt!finite correspondence}%
\index{finite!MW-correspondence|see{Milnor-Witt finite correspondence}}%
\index[notation]{cortk@$\cork$}%
\end{dfn}

We see that $\cork$ is an additive category, with disjoint union as direct sum.
We now check that associating $\grph f$ to any morphism of smooth schemes $f:X\to Y$ gives a functor $\tilde\gamma:\smk\to \cork$. We have a commutative diagram
\[
\begin{tikzcd} 
 X\ar[r,"\Gamma_f"]\ar[d,"\Gamma_{g\circ f}"'] & X\times Y\ar[r]\ar[d,"1\times \Gamma_g"] & Y\ar[d,"\Gamma_g"']\ar[rd,"g"] & \\
X\times Z\ar[d]\ar[r,"\Gamma_f\times 1"'] & X\times Y\times Z\ar[r,"q_{YZ}"]\ar[d,"p_{XY}"] & Y\times Z\ar[d]\ar[r] & Z \\
 X\ar[r,"\Gamma_f"'] & X\times Y\ar[r]\ar[d] & Y & \\
 & X & & 
\end{tikzcd} 
\]
in which all squares are Cartesian, and an equality $\pi_{XZ}\circ (\Gamma_f\times 1)=\id_{X\times Z}$, where $\pi_{XZ}$ is the projection. Using the base change formula (twice), we see that the composite of $\grph g$ with $\grph f$ is given by
\[
(\pi_{XZ})_*\left((\Gamma_f\times 1)_*(\langle 1\rangle)\cdot (1\times \Gamma_g)_*(\langle 1\rangle)\right).
\]
Using the projection formula and $(\pi_{XZ})_*(\Gamma_f\times 1)_*=\id$, the expression becomes 
\[
(\Gamma_f\times 1)^*(1\times \Gamma_g)_*(\langle 1\rangle),
\]
which we identify with $(\Gamma_{g\circ f})_*(\langle 1\rangle)$ using the base change formula once again.
\begin{rem}\label{rem:cortocor} 
The category of finite correspondences as defined by Voevodsky can be recovered by replacing Chow-Witt groups by Chow groups in our definition. Indeed, when $Y$ is equidimensional of dimension $d=\dimn Y$ and $T \in \Adm(X,Y)$,
\[
\CH^d_T(X \times Y) = \bigoplus_{x \in (X \times Y)^{(d)}\cap T} \hspace{-2ex}\ZZ
\]
since the previous group in the Gersten complex is zero because $T$ is of codimension $d$, and the following group is also zero because there are no negative K-groups. Writing $T=T_1\cup\ldots\cup T_n$ for the irreducible components of $T$, it follows that $\CH^d_T(X \times Y)=\ZZ\cdot\{T_1\}\oplus \ldots\oplus \ZZ\cdot\{T_n\}$. Further, If $T\subset T^\prime$ with $T^\prime$ admissible, we may decompose $T^\prime$ in irreducible components $T^\prime=T_1\cup\ldots\cup T_n\cup T^\prime_{n+1}\cup\ldots\cup T^\prime_{n+m}$ and the extension of support homomorphism $\CH^d_T(X \times Y)\to \CH^d_{T^\prime}(X \times Y)$ is then just the inclusion
\[
\ZZ\cdot\{T_1\}\oplus \ldots\oplus \ZZ\cdot\{T_n\}\to \ZZ\cdot\{T_1\}\oplus \ldots\oplus \ZZ\cdot\{T_n\}\oplus \ZZ\{T^\prime_{n+1}\}\oplus\ldots\oplus \ZZ\{T^\prime_{n+m}\}.
\]
Using this, we see that 
\[
\varinjlim_{T\in \Adm(X,Y)}\CH^d_T(X \times Y)=\corVk(X,Y),
\]
the group of Voevodsky's finite correspondences from $X$ to $Y$ \cite{Mazza06_a}*{Definition 1.1}. 

By the same procedure, it is of course possible to define finite correspondences using other cohomology theories with support, provided that they satisfy the classical axioms used in the definition of the composition (base change, etc.).
\end{rem}
The forgetful homomorphisms 
\[
\chst dT{X\times Y}{\omega_Y}\to \CH^d_T(X\times Y)
\] 
and the previous remark yield a functor $\pi:\cork\to \corVk$ (use \cite{Fasel07_a}*{Prop.~6.12} and the fact that the relevant constructions for Chow-Witt groups such as proper push-forwards and pull-backs refine the classical constructions for Chow groups) which is additive, and the ordinary functor $\gamma:\smk\to \corVk$ is the composite functor $\smk\stackrel{\tilde\gamma}\to \cork\stackrel{\pi}\to \corVk$.

On the other hand, the hyperbolic homomorphisms
\[
\CH^d_T(X\times Y)\to \chst dT{X\times Y}{\omega_Y}
\]
yield a homomorphism $H_{X,Y}:\corVk(X,Y)\to \cork(X,Y)$ for any smooth schemes $X,Y$ but not a functor $\corVk\to \cork$ since $H_{X,X}$ doesn't preserve neither the identity nor the compositions. The composite $\pi_{X,Y}H_{X,Y}$ is just the multiplication by $2$, as explained in Section \ref{sec:chowwitt}.
\medskip

We now give two examples showing how to compose a finite MW-correspondence with a morphism of schemes.

\begin{exm}[Pull-back]\label{ex:flat_example}
Let $X,Y,U \in \smk$ and let $f:X \to Y$ be a morphism. Let $(f\times 1):(X\times U)\to (Y\times U)$ be induced by $f$ and let $T\in \Adm(Y,U)$ be an admissible subset. Then $F:=(f\times 1)^{-1}(T)$ is an admissible subset of $X\times U$ by \cite{Mazza06_a}*{Lemma 1.6}. It follows that the pull-back of cycles $(f\times 1)^*$ induces a homomorphism $\cork(Y,U)\to \cork(X,U)$. The interested reader is invited to check that it coincides with the composition with $\grph f$.
\end{exm}

\begin{exm}[Push-forwards]\label{ex:push-forwards}

Let $X$ and $Y$ be smooth schemes of dimension $d$ and let $f:X\to Y$ be a finite morphism such that any irreducible component of $X$ surjects to the irreducible component of $Y$ it maps to. Contrary to the classical situation, we don't have a finite MW-correspondence $Y\to X$ associated to $f$ in general. However, we can define one if $\omega_f$ admits an orientation.  

Let then $(\Lb,\psi)$ be an orientation of $\omega_f$. We define a finite MW-correspondence $\alpha(f,\Lb,\psi)\in \cork(Y,X)$ as follows. Let $\Gamma^t_f:X\to Y\times X$ be the transpose of the graph of $f$. Then $X$ is an admissible subset and we have a transfer morphism 
\[
(\Gamma^t_f)_*:\sKMW_0(X,\omega_f)\to \chst {d}{\Gamma^t_f(X)}{Y\times X}{\omega_X}.
\]
Composing with the homomorphism 
\[
\chst {d}{\Gamma^t_f(X)}{Y\times X}{\omega_X}\to \cork(Y,X),
\]
we get a map $\sKMW_0(X,\omega_f)\to \cork(Y,X)$. Now the isomorphism $\psi$ together with the canonical isomorphism $\sKMW_0(X)\simeq \sKMW_0(X,\Lb\otimes\Lb)$ yield an isomorphism $\sKMW_0(X)\to \sKMW_0(X,\omega_f)$. We define the finite MW-correspondence $\alpha(f,\Lb,\psi)=\alpha(f,\psi)$ as the image of $\langle 1\rangle$ under the composite
\[
\sKMW_0(X)\to \sKMW_0(X,\omega_f)\to \cork(Y,X).
\]
If $(\Lb^\prime,\psi^\prime)$ is equivalent to $(\Lb,\psi)$, then it is easy to check that the correspondences $\alpha(f,\Lb,\psi)$ and $\alpha(f,\Lb^\prime,\psi^\prime)$ are equal. Thus any element of $\Or(\omega_f)$ yields a finite MW-correspondence. In general, different choices of elements in $\Or(\omega_f)$ yield different correspondences.

When $g:Y \to Z$ is another such morphism with an orientation $(\Mb,\phi)$ of $\omega_g$, then $(\Lb \otimes f^*\Mb, \psi \otimes f^*\phi)$ is an orientation of $\omega_{g \circ f}=\omega_f \otimes f^*\omega_g$, and we have $\alpha(f,\Lb,\psi)\circ \alpha(g,\Mb,\phi) = \alpha(g \circ f,\Lb \otimes f^*\Mb, \psi \otimes f^*\phi)$. 

Let now $U$ be a smooth scheme of dimension $n$ and let $T\in \Adm(X,U)$. The commutative diagram
\[
\begin{tikzcd}
X\times U \ar[r,"f\times 1"] \ar[rd] & Y \times U \ar[d,"p_Y"] \\ 
 & Y
\end{tikzcd}
\]
where $p_Y:Y\times U$ is the projection on the first factor and \cite{Mazza06_a}*{Lemma~1.4} show that $(f\times 1)(T)\in A(Y,U)$ in our situation. Moreover, we have a push-forward morphism
\[
(f\times 1)_*:\chst nT{X\times U}{\omega_U\otimes \omega_f}\to \chst n{(f\times 1)(T)}{Y\times U}{\omega_U}
\] 
Using the trivialization $\psi$, we get a push-forward morphism
\[
(f\times 1)_*:\chst nT{X\times U}{\omega_U}\to \chst n{(f\times 1)(T)}{Y\times U}{\omega_U}
\] 
Now this map commutes with the extension of support homomorphisms, and it follows that we get a homomorphism 
\[
(f\times 1)_*:\cork(X,U)\to \cork(Y,U)
\]
depending on $\psi$. It is an interesting exercise to check that $(f\times 1)_*(\beta)=\beta\circ \alpha(f,\psi)$ for any $\beta\in \cork(X,U)$, using the base change formula as well as the projection formula.

In particular, when $U=Y$ and $\beta=\grph f$, using $\psi$ we can push-forward along $f$ as $\sKMW_0(X) \simeq \sKMW_0(X,\omega_f) \to \sKMW_0(Y)$ to obtain an element $f_*\langle 1 \rangle$ in $\sKMW_0(Y)$, and we have $\grph f \circ \alpha(f,\psi)=f_*\langle 1 \rangle \cdot \id_Y$, using the action of Example \ref{ex:action}. 
\end{exm}

\begin{rem}\label{rem:notsurjective}
Suppose that $X$ is connected, $f:X\to Y$ is a finite surjective morphism with relative bundle $\omega_f$ and that $\omega_f\neq 0$ in $\Pic(X)/2$ (take for instance the map $\PP^2\to \PP^2$ defined by $[x_0:x_1:x_2]\mapsto [x_0^2:x_1^2:x_2^2]$). Consider the finite correspondence $Y\to X$ corresponding to the transpose of the graph of $f$. As in the previous example, we have a transfer homomorphism 
\[
(\Gamma_f^t)_*:\sKMW_0(X,\omega_f)\to \chst {d}{\Gamma_f^t(X)}{Y\times X}{\omega_X}.
\]
making the diagram
\[
\begin{tikzcd}
\sKMW_0(X,\omega_f) \ar[r,"\simeq"] \ar[d] & \chst d{\Gamma_f^t(X)}{Y\times X}{\omega_X} \ar[d] \\
\ZZ \ar[r,"\simeq"] & \CH^d_{\Gamma_f^t(X)}(Y\times X)
\end{tikzcd}
\]
commutative, where the vertical homomorphisms are the forgetful maps and the horizontal ones are isomorphisms, as one clearly sees using the Rost-Schmid complex (or alternatively the Milnor-Witt complex). Since $\omega_f$ is not a square in $\Pic(X)$, the left-hand vertical map is not surjective: it is equal to the rank map. It follows that the map $\cork(Y,X)\to \corVk(Y,X)$ is not surjective. Indeed, if we enlarge the support to a larger admissible set $X'$, one of its irreducible components will be $X$, and the (usual) Chow groups with support in $X'$ will be isomorphic to several copies of $\ZZ$, one for each irreducible component, and the forgetful map cannot surject to the copy of $\ZZ$ corresponding to $X$.
\end{rem}

\subsection{Tensor products}
\label{sec:tensorproducts}

Let $X_1,X_2,Y_1,Y_2$ be smooth schemes over $\Spec(k)$. Let $d_1=\dimn{Y_1}$ and $d_2=\dimn{Y_2}$.

Let $\alpha_1\in \chst {d_1}{T_1}{X_1\times Y_1}{\omega_{Y_1}}$ and $\alpha_2\in \chst {d_2}{T_2}{X_2\times Y_2}{\omega_{Y_2}}$ for some admissible subsets $T_i\subset X_i\times Y_i$. The exterior product defined in \cite{Fasel07_a}*{\S 4} gives a cycle 
\[
(\alpha_1\times \alpha_2)\in \chst {d_1+d_2}{T_1\times T_2}{X_1\times Y_1\times X_2\times Y_2}{p_{Y_1}^*\omega_{Y_1/k}\otimes p_{Y_2}^*\omega_{Y_2/k}}
\]
where $p_{Y_i}:X_1\times Y_1\times X_2\times Y_2\to Y_i$ is the projection to the corresponding factor. 
Let $\sigma:X_1\times Y_1\times X_2\times Y_2\to X_1\times X_2\times Y_1\times Y_2$ be the transpose isomorphism. Applying $\sigma_*$, we get a cycle
\[
\sigma_*(\alpha_1\times \alpha_2)\in \chst {d_1+d_2}{\sigma(T_1\times T_2)}{X_1\times X_2\times Y_1\times Y_2}{p_{Y_1}^*\omega_{Y_1/k}\otimes p_{Y_2}^*\omega_{Y_2/k}}.
\]
On the other hand, $p_{Y_1}^*\omega_{Y_1/k}\otimes p_{Y_2}^*\omega_{Y_2/k}=\omega_{Y_1\times Y_2}$ and it is straightforward to check that (the underlying reduced scheme) $\sigma(T_1\times T_2)$ is finite and surjective over $X_1\times X_2$. Thus $\sigma_*(\alpha_1\times \alpha_2)$ defines a finite MW-correspondence between $X_1\times X_2$ and $Y_1\times Y_2$. 

\begin{dfn}\label{def:tensor}
If $X_1$ and $X_2$ are smooth schemes over $k$, we define their tensor product as $X_1\otimes X_2:=X_1\times X_2$. If $\alpha_1\in \cork(X_1,Y_1)$ and $\alpha_2\in \cork(X_2,Y_2)$, then we define their tensor product as $\alpha_1\otimes \alpha_2:=\sigma_*(\alpha_1\times \alpha_2)$. 
\index{Milnor-Witt!finite correspondence!tensor product}%
\index{tensor product!of finite MW-correspondences}%
\end{dfn}

\begin{lem}
The tensor product $\otimes$ together with the obvious associativity and symmetry isomorphisms endows $\cork$ with the structure of a symmetric monoidal category. 
\end{lem}

\begin{proof}
Straightforward.
\end{proof}


\section{Presheaves on \texorpdfstring{$\cork$}{finte MW-correspondences}}

\begin{dfn}
A \emph{presheaf with MW-transfers} is a contravariant additive functor $\cork\to {\Ab}$. 
\index{presheaf!with Milnor-Witt transfers}%
\index[notation]{pshtk@$\pshMWk$}%
We will denote by $\pshMWk$ the category of presheaves with MW-transfers. Let $t$ be $\zar$, $\nis$ and $\et$, respectively the Zariksi, Nisnevich or étale topology on $\smk$. We say that a presheaf with MW-transfers is a $t$-sheaf with MW-transfers if when restricted to $\smk$, it is a sheaf in the $t$-topology. We denote by $\shMWtk$ the category of $t$-sheaves with MW-transfers.
\index[notation]{shtk@$\shMWtk$}%
\end{dfn}

\begin{rem}
The sheaves with MW-transfers are closely related to sheaves with generalized transfers as defined in \cite{Morel11_a}*{Definition~5.7}. Indeed, let $M$ be a Nisnevich sheaf with MW-transfers. Then it is endowed with an action of $\sKMW_0$ by Example \ref{ex:action}. Following the procedure described in Section \ref{sec:limits}, one can define $M(F)$ for any finitely generated field extension $F/k$. If $F\subset L$ is a finite field extension and $k$ is of characteristic different from $2$, then the canonical orientation of Lemma \ref{lem:orientation} together with the push-forwards defined in Example \ref{ex:push-forwards} show that we have a homomorphism $\Tr_F^L:M(L)\to M(F)$. In characteristic $2$, one has to integrate the canonical modules $\omega_{L/k}$ and $\omega_{K/k}$. One can then check that the axioms listed in \cite{Morel11_a}*{Definition 5.7} are satisfied. Conversely, a Nisnevich sheaf with generalized transfers in the sense of Morel yields a Nisnevich sheaf with MW-transfers by \cite{Feld21b_a}*{Theorem 3.2.4}. 
\end{rem}

Recall first that there is a forgetful additive functor $\pi:\cork\to \corVk$. It follows that any (additive) presheaf on $\corVk$ defines a presheaf on $\cork$ by composition. We now give a more genuine example. 

\begin{lem}
For any $j\in\ZZ$, the contravariant functor $X\mapsto \sKMW_j(X)$ is a presheaf on $\cork$.
\end{lem}

\begin{proof}
Let then $X,Y$ be smooth schemes and $T\subset X\times Y$ be an admissible subset. Let $\beta \in \sKMW_j(Y)$ and $\alpha\in \chst {d_Y}T{X\times Y}{\omega_Y}$. We set 
\[
\alpha^*(\beta):= p_*(p_Y^*(\beta)\cdot \alpha)
\]
where $p$ and $p_Y$ are the respective projections, and $p_*$ is defined using the canonical isomorphism $(d_Y,\omega_Y)\simeq (d_X+d_Y,\omega_{X\times Y})\otimes (-d_X,\omega_X^\vee)$. Note further that $p_Y^*(\beta)\cdot \alpha$ is supported on $T$ by definition of the intersection product. One may check that $\alpha^*$ is additive. Next, we observe that $p_Y^*(\beta)$ commutes with any element in the total Chow-Witt group of $X\times Y$ (e.g. \cite{Fasel20_a}*{\S 3.4} with $a=i=0$ and $a^\prime=i^\prime=j^\prime=d_X$ or \cite{Feld20_a}*{Theorem 11.6}).

If $T\subset T^\prime$ are both admissible, we have a commutative diagram
\[
\begin{tikzcd}
\chst {d_Y}T{X\times Y}{\omega_Y} \ar[r] \ar[rd,"p_*"'] & \chst {d_Y}{T^\prime}{X\times Y}{\omega_Y}\ar[d,"p_*"] \\
 & \ch iX
\end{tikzcd}
\]
where the top horizontal morphism is the extension of support. It follows that $\alpha\mapsto \alpha^*$ defines a map $\cork(X,Y)\to \homm {\Ab}{\H^i(Y,\sKMW_j)}{\H^i(X,\sKMW_j)}$ (we had $i=0$ in the above discussion, but the argument is the same for any $i\in\NN$ using $a^\prime=i^\prime=j^\prime=d_X$ for the commutativity statement). We now check that this map preserves the respective compositions. With this in mind, consider again the diagram (\ref{eq:composition})
\[
\begin{tikzcd}
X\times Z \ar[rrrd,bend left=1ex,"r_Z"] \ar[rddd,bend right=1ex,"p_X"'] & & & \\
 & X \times Y \times Z \ar[r,"q_{YZ}"'] \ar[d,"p_{XY}"] \ar[lu,"q_{XZ}"'] & Y \times Z \ar[r,"q_Z"'] \ar[d,"p_Y"] & Z \\
 & X \times Y \ar[r,"q_Y"'] \ar[d,"p"] & Y & \\
 & X & & 
\end{tikzcd}
\]
Let $\alpha_1\in \chst {d_Y}{T_1}{X\times Y}{\omega_Y}$ and $\alpha_2\in \chst {d_Z}{T_2}{Y\times Z}{\omega_Z}$ be correspondences, with $T_1\subset X\times Y$ and $T_2\subset Y\times Z$ admissible. Let moreover $\beta\in \sKMW_j(Z)$. By definition, we have 
\[
(\alpha_2\circ\alpha_1)^*(\beta)=(p_X)_*[r_Z^*\beta\cdot (q_{XZ})_*(p_{XY}^*\alpha_1\cdot q_{YZ}^*\alpha_2)].
\]
Using the projection formula and the fact that $r_Z^*\beta$ commutes with the total Chow-Witt group, we have 
\begin{eqnarray*}
 (p_X)_*[r_Z^*\beta\cdot (q_{XZ})_*(p_{XY}^*\alpha_1\cdot q_{YZ}^*\alpha_2)]& = & (p_X)_*[(q_{XZ})_*(p_{XY}^*\alpha_1\cdot q_{YZ}^*\alpha_2)\cdot r_Z^*\beta]   \\
 & = & (p_X)_*[(q_{XZ})_*(p_{XY}^*\alpha_1\cdot q_{YZ}^*\alpha_2\cdot q_{XZ}^*r_Z^*\beta)] \\
 & = & p_*(p_{XY})_*(p_{XY}^*\alpha_1\cdot q_{YZ}^*\alpha_2\cdot q_{XZ}^*r_Z^*\beta).
\end{eqnarray*} 
On the other hand,
\begin{eqnarray*}
\alpha_1^*\circ \alpha_2^*(\beta) & = & \alpha_1^*((p_Y)_*(q_Z^*\beta\cdot \alpha_2)) \\
 & = & p_*(q_Y^*(p_Y)_*(q_Z^*\beta\cdot \alpha_2)\cdot \alpha_1).
\end{eqnarray*}
By base change, $q_Y^*(p_Y)_*=(p_{XY})_*q_{YZ}^*$ and it follows that 
\begin{eqnarray*}
\alpha_1^*\circ \alpha_2^*(\beta) & = & p_*((p_{XY})_*(q_{YZ}^*q_Z^*\beta\cdot q_{YZ}^*\alpha_2)\cdot \alpha_1) \\
\end{eqnarray*}
Using the projection formula once again, the latter is equal to 
\[
p_*(p_{XY})_*(q_{YZ}^*q_Z^*\beta\cdot q_{YZ}^*\alpha_2\cdot p_{XY}^*\alpha_1)
\]
We conclude using Remark \ref{rem:everythingcommutes} and the fact that $\beta$ commutes with the total Chow-Witt group.
\end{proof}

\begin{rem}
More generally, the contravariant functor $X\mapsto \H^i(X,\sKMW_j)$ is a presheaf on $\cork$ for any $j\in\ZZ$ and $i\in\NN$. The above proof applies with minor modifications, or we can alternatively use the facts that $\sKMW_j$ is a homotopy invariant Nisnevich sheaf with MW-transfers, that the category of Nisnevich sheaves with MW transfers has enough injectives and that Zariski and Nisnevich cohomology coincide for invariant Nisnevich sheaves with MW-transfers. We refer the reader to \chdmt for details.

In contrast, the sheaf $\sKMW_j$ for $j\in\ZZ$ doesn't have transfers in the sense of Voevodsky. If this were the case, then it would define an object in the category of (effective) motives and it would follow from \cite{Mazza06_a}*{Corollary 15.5} that the projective bundle formula would hold for the cohomology of this sheaf. This is not the case by \cite{Fasel13_a}*{Theorem 11.7}.
\end{rem}

\subsection{Extending presheaves to limits}\label{sec:limits}

We consider the category $\cP$ of filtered projective systems $((X_\lambda)_{\lambda \in I},f_{\lambda\mu})$ of smooth quasi-compact schemes over $k$, with affine étale 
transition morphisms $f_{\lambda\mu}:X_\lambda \to X_\mu$. Morphisms in $\cP$ are defined by 
\[
\Hom((X_\lambda)_{\lambda \in I}, (X_\mu)_{\mu \in J})=\varprojlim_{\mu\in J} \big(\varinjlim_{\lambda\in I} \Hom(X_\lambda,X_{\mu})\big)
\]
as in \cite{EGA4-3_a}*{8.13.3}. The limit of such a system exists in the category of schemes by loc.\ cit.\ 8.2.5, and by 8.13.5, the functor sending a projective system to its limit defines an equivalence of categories from $\cP$ to the full subcategory $\barsm{k}$ of schemes over $k$ that are limits of such systems.
\begin{rem}\label{rem:constantprojective}
It follows from this equivalence of categories that such a projective system converging to a scheme that is already finitely generated (e.g.\ smooth) over $k$ has to be constant above a large enough index.
\end{rem}

Let now $F$ be a presheaf of abelian groups on $\smk$. We extend $F$ to a presheaf $\bar F$ on $\cP$ by setting on objects $\bar F((X_\lambda)_{\lambda \in I})=\varinjlim_{\lambda\in I} F(X_\lambda)$. An element of the set $\varprojlim_{\mu\in J} \big(\varinjlim_{\lambda\in I} \Hom(X_\lambda,X_{\mu})\big)$ yields a morphism $\varinjlim_{\mu\in J} F(X_\mu)\to \varinjlim_{\lambda\in I} F(X_\lambda)$, this respects composition, and thus $\bar F$ is well-defined. Using the above equivalence of categories, it follows immediately that $\bar F$ defines a presheaf on $\barsm{k}$ which extends $F$ in the sense that $\bar F$ and $F$ coincide on $\smk$ since a smooth scheme can be considered as a constant projective system.

Since any finitely generated field extension $L/k$ can be written as a limit of smooth schemes in the above sense, we can consider in particular $F(\Spec(L))$ and to shorten the notation, we often write $F(L)$ instead of $\bar F(\Spec(L))$ in what follows.

We will mainly apply this limit construction when $F$ is $\cork(- \times X,Y)$, $\sKMW_*$ or the MW-motivic cohomology groups $\HMW^{p,q}(-,\ZZ)$, to be defined in \ref{def:genmotiviccohom}. 
\medskip

We now slightly extend this equivalence of categories to a framework useful for Chow-Witt groups with support. We consider the category $\cT$ of triples $(X,Z,\Lb)$ where $X$ is a scheme over $k$, with a closed subset $Z$ and a line bundle $\Lb$ over $X$. A morphism from $(X_1,Z_1,\Lb_1)$ to $(X_2,Z_2,\Lb_2)$ in this category is a pair $(f,i)$ where $f:X_1 \to X_2$ is a morphism of $k$-schemes such that $f^{-1}(Z_2)\subseteq Z_1$ and $i:f^*\Lb_2 \to \Lb_1$ is an isomorphism of line bundles over $X_1$. The composition of two such morphisms $(f,i)$ and $(g,j)$ is defined as $(f\circ g, j \circ g^*(i))$. Let $\tP$ be the category of projective systems in that category such that:
\begin{itemize}
\item The objects are smooth quasi-compact and the transition maps are affine étale;
\item For any index $\alpha$, there exists $\mu\geq \alpha$ such that for any $\lambda\geq \mu$ we have $f_{\lambda\mu}^{-1}(Z_\mu)=Z_\lambda$,
\end{itemize}
with morphisms defined by a double limit of $\Hom$ groups in $\cT$, as for $\cP$.
Let $X$ be the inverse limit of the $X_\lambda$. All the pull-backs of the various $\Lb_\lambda$ to $X$ are canonically isomorphic by pulling back the isomorphisms $i_{\lambda\mu}$, and by the last condition, the inverse images of the closed subsets $Z_\lambda$ stabilize to a closed subset $Z$ of $X$. In other words, the inverse limit of the system exists in $\cT$. 

Given a limit scheme $X=\varprojlim X_\lambda$, any line bundle $\Lb$ on $X$ is obtained as a pull-back from a line bundle on some $X_\lambda$, by \cite[8.5.2 (ii) and 8.5.5]{EGA4-3_a}. However, the closed subsets of $X$ that are obtained as limits of closed subsets $Z_\lambda$ of $X_\lambda$ are the constructible ones, as follows from \cite[8.3.11]{EGA4-3_a} by noting that all the $Z_\lambda$ are constructible because $X_\lambda$ is noetherian. If $X$ itself happens to be noetherian, then all closed $Z$ in $X$ are constructible and are thus limits. Summarizing: 
\begin{prop} \label{prop:equivtriple}
The functor sending a system to its inverse limit is fully faithful. In other words, it defines an equivalence of categories from $\tP$ to the full subcategory of $\cT$ of objets $(X,Z,\Lb)$ such that $X$ is in $\barsm{k}$, and $Z$ is constructible (which is automatic when $X$ is noetherian).
\end{prop}

As previously, any presheaf of abelian groups or sets $F$ defined on the full subcategory of $\cT$ with smooth underlying schemes can be extended uniquely to the subcategory of $\cT$ with underlying schemes that are limits. 

If the limit scheme $X$ happens to be smooth over a field extension $L$ of $k$, for any nonnegative integer $i$, the pull-back induces a map
\[
\varinjlim_{\lambda \in I} \chst i{Z_\lambda}{X_\lambda} {\Lb_\lambda} \to \chst i Z X \Lb.
\]
\begin{lem} \label{lem:limitChowiso}
This map is an isomorphism.
\end{lem}
\begin{proof}
When $X$ is smooth over $k$ itself, $(X,Z,\Lb)$ can be considered as a constant projective system, and the conclusion follows from Remark \ref{rem:constantprojective} and the equivalence of categories of Proposition \ref{prop:equivtriple}.

In the general case, considering an $X_\lambda$ from which $Z$ is pulled-back and its open complement in $X_\lambda$, one can use the associated long exact sequence of localization, its compatibility to flat pull-backs and the exactness of filtered colimits to reduce to the case where $Z=X$. 

The line bundle $\Lb$ also comes from some $X_\lambda$, and the Chow-Witt group being defined as sheaf cohomology of the Milnor-Witt sheaf with coefficients in $\Lb$, the statement thus follows from the fact that (Zariski and Nisnevich) sheaf cohomology commutes to such inverse limits of schemes, by \cite[Cor.~8.7.7]{SGA4-2_a}.
\end{proof}

\subsection{Representable presheaves}
\label{sec:representable}%

\begin{dfn}
\label{dfn:reprpresheaf}%
Let $X \in \smk$. We denote by $\MWprep(X)$ the presheaf $\cork(-,X)$. 
\index[notation]{ctx@$\MWprep(X)$}%
\index{presheaf!representable|see{Milnor-Witt representable presheaf}}%
\index{Milnor-Witt!presheaf!representable}%
If $x:\Spec(k) \to X$ is a rational point, we denote by $\MWprep(X,x)$ the cokernel of the morphism $\MWprep(\Spec(k)) \to \MWprep(X)$ (which is split injective since $\tilde\gamma$ is a functor). More generally, let $X_1,\ldots,X_n$ be smooth schemes pointed respectively by the rational points $x_1,\ldots,x_n$. We define $\MWprep\big((X_1,x_1)\wedge\ldots\wedge (X_n,x_n)\big)$
\index[notation]{ctx1@$\MWprep\big((X_1,x_1)\wedge\ldots\wedge (X_n,x_n)\big)$}%
as the cokernel of the split injective map
\[
\begin{tikzcd}
[column sep=20ex]
\displaystyle \bigoplus_i \MWprep(X_1\times \ldots \times X_{i-1}\times X_{i+1}\times \ldots \times X_n) 
\ar[r,"{\id_{X_1}\times \cdots \times x_i \times \cdots \times \id_{X_n}}"] 
&  \MWprep(X_1\times\ldots\times X_n).
\end{tikzcd}
\]
\end{dfn}

\begin{exm}
It follows from Example \ref{ex:basic} that $\MWprep(\Spec(k))=\sKMW_0$.
\end{exm}

Our next goal is to check that for any smooth scheme $X$, the presheaf $\MWprep(X)$ is actually a sheaf in the Zariski topology. We start with an easy lemma. For a smooth connected scheme $Y$, we consider $k(Y)$ as a projective system as in the previous section, obtaining a well-defined expression $\cork(k(Y),X)$. 

\begin{lem}\label{lem:unramified}
Let $Y \in \smk$ be connected, with function field $k(Y)$. Then for any $X \in \smk$, the homomorphism $\cork(Y,X)\to \cork(k(Y),X)$ is injective.
\end{lem}

\begin{proof}
It follows from Lemma \ref{lem:restrictioninj} that all transition maps in the system converging to $\cork(k(Y),X)$ are injective. 
\end{proof}

\begin{prop} \label{prop:zarsheaf}
For any $X \in \smk$, the presheaf $\MWprep(X)$ is a sheaf in the Zariski topology.
\end{prop}

\begin{proof}
It suffices to prove that for any cover of a scheme $Y$ by two Zariski open sets $U$ and $V$, the equalizer sequence
\[
\begin{tikzcd}
\cork(Y,X) \ar[r] & \cork(U,X) \coprod \cork(V,X) \ar[r,shift right=.5ex] \ar[r,shift left=.5ex] & \cork(U \cap V,X)
\end{tikzcd}
\]
is exact in the well-known sense. Injectivity of the map on the left follows from Lemma \ref{lem:unramified}. To prove exactness in the middle, let $\alpha$ and $\beta$ be elements in the middle groups, equalizing the right arrows and with respective supports $E \in \Adm(U,X)$ and $F \in \Adm(V,X)$ by Lemma \ref{lem:supportadmis}. Since $\alpha$ restricted to $V$ is $\beta$ restricted to $U$, we must have $E \cap V=F \cap U$ by Lemma \ref{lem:supporttoopen}. By Lemma \ref{lem:admissiblesheaf}, the closed set $T=E \cup F$ is admissible, and the conclusion now follows from the long exact sequence of localization for Chow-Witt groups with support in $T$. 
\end{proof}

\begin{exm} \label{exm:notnis}
In general, the Zariski sheaf $\MWprep(X)=\cork(-,X)$ is not a Nisnevich sheaf (and therefore not an étale sheaf either). Suppose that $k$ is of characteristic different from $2$ and set $\Aone_{a_1,\ldots,a_n}:=\Aone\setminus \{a_1,\ldots,a_n\}$ for any rational points $a_1,\ldots,a_n\in \Aone(k)$ and consider the elementary Nisnevich square
\[
\begin{tikzcd}
\Aone_{0,1,-1} \ar[r] \ar[d,"g"'] & \Aone_{0,-1} \ar[d,"f"] \\
\Aone_{0,1} \ar[r] & \Aone_0
\end{tikzcd}
\]
where the morphism $f:\Aone_{0,-1}\to \Aone_0$ is given by $x\mapsto x^2$, the horizontal maps are inclusions and $g$ is the base change of $f$. 

We now show that $\MWprep(\Aone_{0,1})$ is not a Nisnevich sheaf. In order to see this, we prove that the sequence
\[
\MWprep(\Aone_{0,1})(\Aone_{0})\to \MWprep(\Aone_{0,1})(\Aone_{0,1})\oplus \MWprep(\Aone_{0,1})(\Aone_{0,-1})\to \MWprep(\Aone_{0,1})(\Aone_{0,1,-1})
\]
is not exact in the middle. Let $\Delta:\Aone_{0,1}\to \Aone_{0,1} \times \Aone_{0,1}$ be the diagonal embedding. It induces an isomorphism
\[
\sKMW_0(\Aone_{0,1})\to \chst 1{\Delta(\Aone_{0,1})}{\Aone_{0,1}\times \Aone_{0,1}}{\omega_{\Aone_{0,1}}}
\]
and thus a monomorphism $\sKMW_0(\Aone_{0,1})\to \MWprep(\Aone_{0,1})(\Aone_{0,1})$. Since the restriction homomorphism $\sKMW_0(\Aone_0)\to \sKMW_0(\Aone_{0,1})$ is injective, it follows that the class $\eta\cdot [t]=-1+\langle t\rangle $ is non trivial in $\sKMW_0(\Aone_{0,1})$ and its image $\alpha_t$ in $\MWprep(\Aone_{0,1})(\Aone_{0,1})$ is also non trivial. We claim that its restriction in $\MWprep(\Aone_{0,1})(\Aone_{0,1,-1})$ is trivial. Indeed, consider the Cartesian square
\[
\begin{tikzcd}
\Aone_{0,1,-1} \ar[r,"\Gamma_g"] \ar[d,"g"'] & \Aone_{0,1,-1} \times \Aone_{0,1}\ar[d,"{(g\times 1)} "] \\
\Aone_{0,1} \ar[r,"\Delta"'] & \Aone_{0,1} \times \Aone_{0,1}
\end{tikzcd}
\]
From Example \ref{ex:flat_example}, we know that the restriction of $\alpha_t$ to $\MWprep(\Aone_{0,1})(\Aone_{0,1,-1})$ is represented by $(g\times 1)\Delta_*(-1+\langle t\rangle)$. By base change, the latter is $(\Gamma_g)_*g^*(\alpha_t)$. Now we have $g^*(-1+\langle t\rangle)=-1+\langle t^2\rangle=0$ in $\sKMW_0(\Aone_{0,1,-1})$. In conclusion, we see that the correspondence $(\alpha_t,0)$ is in the kernel of the homomorphism
\[
\MWprep(\Aone_{0,1})(\Aone_{0,1})\oplus \MWprep(\Aone_{0,1})(\Aone_{0,-1})\to \MWprep(\Aone_{0,1})(\Aone_{0,1,-1}).
\] 
To conclude that $\MWprep(\Aone_{0,1})$ is not a sheaf, it then suffices to show that $\alpha_t$ cannot be the restriction of a correspondence in $\MWprep(\Aone_{0,1}) (\Aone_0)$. To see this, observe first that the restriction map $\Adm(\Aone_0,\Aone_{0,1})\to \Adm(\Aone_{0,1},\Aone_{0,1})$ is injective as well as the homomorphism $\MWprep(\Aone_{0,1})(\Aone_0)\to \MWprep(\Aone_{0,1})(\Aone_{0,1})$. Suppose then that $\beta\in \MWprep(\Aone_{0,1})(\Aone_0)$ is a correspondence whose restriction is $\alpha_t$, and let $T=\supp(\beta)$. From the above observations, we see that $\Delta(\Aone_{0,1})=\supp(\alpha_t)=T\cap (\Aone_{0,1}\times \Aone_{0,1})$. In particular, we see that $T$ is irreducible and its generic point is the generic point of the (transpose of the) graph of the inclusion $\Aone_{0,1}\to \Aone_0$. But the graph is closed, but not finite over $\Aone_0$. It follows that $\beta$ doesn't exist and thus that $\MWprep(\Aone_{0,1})$ is not a Nisnevich sheaf.
\end{exm}

Next, recall that the presheaf $\VrepZ(X):=\corVk(-,X)$
\index[notation]{ztrx@$\VrepZ(X)$}%
\index{presheaf!with transfers!representable}%
is also a presheaf with MW-transfers (using the functor $\cork\to \corVk$). The morphism $\MWprep(X)\to \VrepZ(X)$ is easily seen to be a morphism of MW-presheaves. 

\begin{lem}
The morphism of MW-presheaves $\MWprep(X)\to \VrepZ(X)$ induces an epimorphism of Zariski sheaves.
\end{lem}

\begin{proof}
Let $Y$ be the localization of a smooth scheme at a point. We have to show that the map $\MWprep(X)(Y)\to \VrepZ(X)(Y)$ is surjective. It is sufficient to prove that for any elementary correspondence $T\subset Y\times X$, the map
\[
\chst {d_X}T{Y\times X}{\omega_X}\to \CH^{d_X}_T(Y\times X)=\ZZ
\]
is surjective. Note that the map $T\to Y\times X\to Y$ is finite and surjective, and therefore that $T$ is the spectrum of a semi-local domain. Let $\tilde T$ be the normalization of $T$ into its field of fractions. Note that the singular locus of $\tilde T$ is of codimension at least $2$ and therefore that its image in $Y\times X$ is of codimension at least $d_X+2$. As both $\chst {d_X}T{Y\times X}{\omega_X}$ and $\CH^{d_X}_T(Y\times X)$ are invariant if we remove points of codimension $d_X+2$, we may suppose that $\tilde T$ is smooth and therefore that we have a well-defined push-forward
\[
i_*:\cht 0{\tilde T}{\omega_{{\tilde T}/k}\otimes i^*\Lb}\to \chst {d_X}T{Y\times X}{\omega_{Y\times X/k}\otimes \Lb} 
\] 
for any line bundle $\Lb$ on $Y\times X$, where $i$ is the composite $\tilde T\to T\subset Y\times X$. As $\tilde T$ is semi-local, we can find a trivialization of $\omega_{{\tilde T}/k}\otimes i^*\Lb$ and get a push-forward homomorphism
\[
i_*:\ch 0{\tilde T}\to \chst {d_X}T{Y\times X}{\omega_{Y\times X/k}\otimes \Lb} 
\] 
Now, $\langle 1\rangle\in \ch 0{\tilde T}$ and we can use the commutative diagram
\[
\begin{tikzcd}
\ch 0{\tilde T} \ar[r,"i_*"] \ar[d] & \chst {d_X}T{Y\times X}{\omega_{Y\times X/k} \otimes \Lb} \ar[d] \\
\CH^0(\tilde T) \ar[r,"i_*"'] & \CH^{d_X}_T(Y\times X)
\end{tikzcd}
\]
to conclude that the composite 
\[
\ch 0{\tilde T}\to \chst {d_X}T{Y\times X}{\omega_{Y\times X/k}\otimes \Lb}\to \CH^{d_X}_T(Y\times X)
\]
is surjective for any line bundle $\Lb$. The claim follows.
\end{proof}

Now, we turn to the task of determining the kernel of the morphism $\MWprep(X)\to \VrepZ(X)$. With this in mind, recall that we have for any $n\in \NN$ an exact sequence of sheaves
\[
0\to \I^{n+1}\to \KMW_n \to \KM_n \to 0.
\] 
If $T\subset Y\times X$ is of pure codimension $d_X$, the long exact sequence associated to the previous exact sequence reads as
\[
\H^{d_X-1}_T(Y\times X,\sKM_{d_X})\to \H^{d_X}_T(Y\times X,\I^{d_X+1},\omega_X)\to \chst {d_X}T{Y\times X}{\omega_X} \to \CH^{d_X}_T(Y\times X)
\]
and the left-hand group is trivial since there are no points of codimension $d_X-1$ supported on $T$. Consequently, the middle map is injective. Moreover, if $\alpha\in \cork(Y^\prime,Y)$ and $\pi(\alpha)\in \corVk(Y^\prime,Y)$ is its image in $\corVk$ under the usual functor, the diagram
\[
\begin{tikzcd}
\chst {d_X}T{Y\times X}{\omega_X} \ar[r] \ar[d,"\alpha^*"'] & \CH^{d_X}_T(Y\times X) \ar[d,"{\pi(\alpha)^*}"] \\
\chst {d_X}T{Y^\prime\times X}{\omega_X} \ar[r] & \CH^{d_X}_T(Y^\prime\times X)
\end{tikzcd}
\]
commutes, where the vertical arrows are the respective composition with $\alpha$ and $\pi(\alpha)$ (in other words, $\pi$ is a morphism of presheaves with MW-transfers). These observations motivate the following definition.

\begin{dfn} \label{dfn:icorXY}
For smooth connected schemes $X,Y$, we set 
\[
\icor k(Y,X)=\varinjlim_{T\in \Adm(Y,X)} \H^{d_X}_T(Y\times X,\I^{d_X+1},\omega_X).
\]
As before, we extend this definition to all smooth schemes $X,Y$ by additivity. The above diagram shows that $\icor k(-,X)$ is an element of $\pshMWk$ that we denote by $\Icorr(X)$.
\index[notation]{icork@$\icor k$}%
\index[notation]{icorx@$\Icorr(X)$}%
\end{dfn}

The above arguments prove that we have an exact sequence of Zariski sheaves (use for instance the analogue of Lemma \ref{lem:supportadmis} for $\icor k(Y,X)$)
\[
0\to \Icorr(X) \to \MWprep(X)\to \VrepZ(X)\to 0.
\]

\begin{rem} \label{rem:wittcor}
Instead of $\icor k(Y,X)$, we could consider the abelian group
\[
\varinjlim_{T\in \Adm(Y,X)} \H^{d_X}_T(Y\times X,\I^{d_X},\omega_X)=\varinjlim_{T\in \Adm(Y,X)} \H^{d_X}_T(Y\times X,\W,\omega_X)
\]
In this fashion, we would obtain a category of sheaves with Witt-transfers. We will come back to this construction in \chdmt, \S\ref{sec:relordmot}.
\end{rem}

The category $\pshMWk$ of presheaves with MW-transfers admits a unique symmetric monoidal structure such that the embedding $\cork\to  \pshMWk$ given by $X\mapsto \MWprep(X)$ is symmetric monoidal (\chdmt, \S\ref{sec:MWtransfers}). We denote by $F\otimes G$ the tensor product of two presheaves, and we observe that $\MWprep(X)\otimes \MWprep(Y)=\MWprep(X\times Y)$ by definition. It turns out that $\pshMWk$ is a closed monoidal category (e.g. \cite{Mazza06_a}*{Lecture 8}), with internal Hom functor $\uHom$ determined by $\uHom(\MWprep(X),F)=F(X\times -)$. Using this concrete description and the fact that a covering $U\to Y$ pulls-back to a covering $U\times X\to Y\times X$ for any Grothendieck topology, the proof of the next lemma is direct.

\begin{lem}
Suppose that $F\in \pshMWk$ is a $\tau$-sheaf. Then $\uHom(\MWprep(X),F)$ is also a $\tau$-sheaf for any $X \in \smk$.
\end{lem}

\begin{dfn}
Let $F\in\pshMWk$ and $n \in \NN$. Then the $n$-th contraction of $F$, denoted by $F_{-n}$, is the presheaf $\uHom\big(\MWprep\big((\Gm,1)^{\wedge n}\big),F\big)$. Thus, $F_{-n}=(F_{-n+1})_{-1}$.
\end{dfn}

\begin{exm}\label{ex:MilnorWittcontraction}
If $F=\sKMW_j$ for some $j\in\ZZ$, then it follows from Lemma \ref{lem:explicitcontraction} that $F_{-1}=\sKMW_{j-1}$.
\end{exm}

\begin{dfn}
A presheaf $F$ in $\pshMWk$ is said to be homotopy invariant if the map
\[
F(X)\to F(X\times\Aone)
\]
induced by the projection $X\times \Aone\to X$ is an isomorphism for any $X\in \smk$.
\end{dfn}

\begin{exm}
We already know that $\MWprep(k)$ coincides with the Nisnevich sheaf $\sKMW_0$. It follows from \cite{Fasel08_a}*{Corollaire~11.3.3} that this sheaf is homotopy invariant.
\end{exm}

\subsection{The module structure}
\label{sec:modulestructure}%

Recall from Example \ref{ex:action} that the Zariski sheaf $\MWprep(X)$ is endowed with an action of a left $\sKMW_0(X)$-module for any smooth scheme $X$. Pulling back along the structural morphism $X\to \Spec(k)$, we obtain a ring homomorphism $\KMW_0(k)\to \sKMW_0(X)$ and it follows that $\MWprep(X)$ is a sheaf of $\KMW_0(k)$-modules. If $X\to Y$ is a morphism of smooth schemes, it is readily verified that the morphism $\MWprep(X)\to \MWprep(Y)$ is a morphism of sheaves of $\KMW_0(k)$-modules. 

Let now $F$ be an object of $\pshMWk$. If $X$ is a smooth scheme, the presheaf $\uHom (\MWprep(X),F)$ is naturally endowed with the structure of a right $\KMW_0(k)$-module by precomposing with the morphism $\MWprep(X)\to \MWprep(X)$ induced by some $\alpha\in \KMW_0(k)$. As $\KMW_0(k)$ is commutative, it follows that $\uHom (\MWprep(X),F)$ is also a Zariski sheaf of (left) $\KMW_0(k)$-modules. If $X\to Y$ is a morphism of smooth schemes, a direct computation shows that the induced morphism $\uHom (\MWprep(Y),F)\to \uHom (\MWprep(X),F)$ is a morphism of $\KMW_0(k)$-modules.

\section{Milnor-Witt motivic cohomology}
\index{Milnor-Witt!motivic!cohomology}%

We define the pointed scheme $\Gmpt =(\Gm,1)$.

\begin{dfn}
For any $q\in\ZZ$, we define the Zariski sheaf $\tZcbx q$ by
\[
\tZcbx q:=
\begin{cases} \MWprep(\Gmpt^{\wedge q}) & \text{if $q>0$.} \\
\MWprep(k)& \text{if $q=0$}. \\
\uHom \big(\Gmpt^{\wedge q}, \MWprep(k)\big) & \text{if $q<0$.}
\end{cases}
\]
\index[notation]{ztq@$\tZcbx q$}%
\end{dfn}
Note that the sheaves considered are all sheaves of $\KMW_0(k)$-modules and that for $q<0$,the sheaf $\tZcbx q$ is the $(-q)$-th contraction of $\MWprep(k)=\sKMW_0$. It follows thus from Example \ref{ex:MilnorWittcontraction} that $(\sKMW_0)_{q}=\sKMW_q=\sW$ for $q<0$, where the latter is the Zariski sheaf associated to the presheaf $X\mapsto \W(X)$ (Witt group).

\begin{rem}
Of course, we could have defined $\tZcbx q$ to be $\W$ when $q<0$ from the beginning. The advantage of our definition is that it is more uniform. It will also make the product to be defined more natural.
\end{rem}

\subsection{Motivic sheaves}\label{sec:motivic}

Let $\Delta^{\bullet}$ be the cosimplicial object on $\smk$ defined by 
\[
\Delta^n:=\Spec(k[x_0,\ldots,x_n]/(\sum_{i=0}^nx_i-1)).
\]
and the usual faces and degeneracy maps. Given any $F\in \pshMWk$, we get a simplicial object $\uHom\big(\MWprep(\Delta^{\bullet}),F\big)$ in $\pshMWk$ which is a simplicial sheaf in case $F$ is one. 

\begin{dfn}
The \emph{Suslin-Voevodsky singular construction on $F$} is the complex associated to the simplicial object $\uHom\big(\MWprep(\Delta^{\bullet}),F\big)$. Following the conventions, we denote it by $\Cstar F$.
\index[notation]{cf@$\Cstar F$}%
\index{singular complex|see{Suslin(-Voevodsky) singular complex}}%
\index{Suslin(-Voevodsky) singular complex}%
\end{dfn}

The following lemma is well-known.

\begin{lem}
Suppose that $F$ is a homotopy invariant presheaf on $\cork$. Then the natural homomorphism $F\to \Cstar F$ is a quasi-isomorphism of complexes (here $F$ is considered as a complex concentrated in degree $0$).
\end{lem}

\begin{dfn}
For any integer $q\in\ZZ$, we define $\tZpx q$ as the complex of Zariski sheaves of $\KMW_0(k)$-modules $\Cstar{\tZcbx q}[-q]$.
\index[notation]{ztq@$\tZpx q$}%
\end{dfn}

Following \cite{Mazza06_a}, we use the cohomological notation and we then have by convention $\tZpx q^i=\Ci{q-i}\tZcbx q$. It follows that $\tZpx q$ is bounded above for any $q\in\ZZ$.

\begin{dfn} \label{def:genmotiviccohom}
The \emph{MW-motivic cohomology groups} $\HMW^{p,q}(X,\ZZ)$ are defined to be the hypercohomology groups 
\[
\HMW^{p,q}(X,\ZZ):=\hypH^p_{\zar}(X,\tZpx q).
\] 
These groups have a natural structure of $\KMW_0(k)$-modules.
\index[notation]{hmwpqxz@$\HMW^{p,q}(X,\ZZ)$}%
\index{Milnor-Witt!motivic!cohomology}%
\end{dfn}

The groups $\HMW^{p,q}(X,\ZZ)$ are thus contravariant in $X \in \smk$. 

\begin{exm} \label{exm:negativecohomology}
By Example \ref{ex:basic}, we have $\HMW^{p,0}(X,\ZZ)=\H^{p}(X,\KMW_0)$ while the identification $(\sKMW_0)_{q}=\sKMW_q=\sW$ for $q<0$ yields $\HMW^{p,q}(X,\ZZ)=\H^{p-q}(X,\sW)$ for $q<0$.
\end{exm}

\begin{rem}
The presheaves $X\mapsto \HMW^{p,q}(X,\ZZ)$ are in fact presheaves with MW-transfers for any $p,q\in\ZZ$. This follows from \chdmt, Corollary~\ref{cor:W-motivic&generalized}.
\end{rem}

Next, consider the presheaf $\Icorr(X)$ defined in Section \ref{sec:representable}. 
\begin{dfn} \label{dfn:Imotiviccohomology}
For any $q\geq 0$, we set $\ItZ(q)=\Cstar{\Icorr(\Gmpt^{\wedge q}})[-q]$. 
\index[notation]{iztq@$\ItZ(q)$}%
We denote by 
\[
\HI^{p,q}(X,\ZZ)=\hypH^p_{\zar}(X,\ItZ(q))
\]
\index[notation]{hipqxz@$\HI^{p,q}(X,\ZZ)$}%
the hypercohomology groups of the complex $\ItZ(q)$.
\end{dfn}

We will see in \chdmt, \S \ref{sec:AA1derived} that for any $q\geq 0$, the exact sequence of Zariski sheaves
\[
0\to \Icorr(\Gmpt^{\wedge q}) \to \MWprep(\Gmpt^{\wedge q})\to \VrepZ(\Gmpt^{\wedge q})\to 0
\]
yields a long exact sequence of motivic cohomology groups (in fact, a long exact sequence of $\KMW_0(k)$-modules)
\[
\ldots \to \HI^{p,q}(X,\ZZ)\to \HMW^{p,q}(X,\ZZ)\to \H^{p,q}(X,\ZZ)\to \HI^{p+1,q}(X,\ZZ)\to \ldots
\]

\subsection{Base change of the ground field}

The various constructions considered until now are functorial with respect to the ground field $k$, and we now give details about some of the aspects of this functoriality. Let $L$ be a field extension of $k$, both $L$ and $k$ assumed to be perfect. 
For any scheme $X$ over $k$, let $X_L=\Spec(L)\times_{\Spec(k)} X$ be its extension to $L$. When $X$ is smooth, the canonical bundle $\omega_X$ pulls-back over $X_L$ to the canonical bundle $\omega_{X_L}$ of $X_L$. 

By conservation of surjectivity and finiteness by base change, the pull-back induces a map $\Adm(X,Y) \to \Adm(X_L,Y_L)$. From the contravariant functoriality of Chow-Witt groups with support, passing to the limit, one then immediately obtains an \emph{extension of scalars} functor
\[
\begin{tikzcd}
\cork \ar[r,"\ext_{L/k}"] & \cor L
\end{tikzcd}
\]
\index[notation]{extlk@$\ext_{L/k}$}%
sending an object $X$ to $X_L$. This functor is monoidal, since $(X \times_k Y)_L \simeq X_L \times_L Y_L$ canonically.

When $L/k$ is finite, $L/k$ is automatically separable since $k$ is perfect (and $L$ is automatically perfect). Viewing an $L$-scheme $Y$ as a $k$-scheme $Y_{|k}$ by composing its structural morphism with the smooth morphism $\Spec(L) \to \Spec(k)$ defines a \emph{restriction of scalars} functor $\res_{L/K}: \sm L \to \sm k$.
\index[notation]{reslk@$\res_{L/k}$}%
For any $L$-scheme $Y$ and $k$-scheme $X$, there are morphisms 
\[
\unit_Y:  Y \to L \times_k Y = (Y_{|k})_L \quad \text{and} \quad  \counit_X:(X_L)_{|k}= L \times_k X \to X
\]
induced by the structural morphism of $Y$ and $\id_Y$ for $\unit$, and the projection to $X$ for $\counit$. 

\begin{lem} \label{lem:finiterestriction}
Let $X \in \smk$ and $Y,Y' \in \sm L$.
\begin{enumerate}
\item \label{item:adjsm} The morphisms of functors $\unit$ and $\counit$ given on objects by $\unit_Y$ and  $\counit_X$ define an adjunction $(\res_{L/k},\ext_{L/k}):\sm L\leftrightarrows \smk$. 
\item \label{item:proj} The natural morphism $(Y \times_L X_L)_{|k} \to Y_{|k} \times_k X$ is an isomorphism (adjoint to the morphism $\unit_Y \times_L \id_{X_L}$).
\item \label{item:unit} The morphism $\unit_Y$ is a closed embedding of codimension $0$, identifying $Y$ with the connected components of $(Y_{k})_L$ living over $\Spec(L)$, diagonally inside $\Spec(L) \times_k \Spec(L)$.
\item \label{item:closedsm} The natural morphism $(Y \times_L Y')_{|k} \to Y_{|k} \times_k Y'_{|k}$ is a closed embedding of codimension $0$, thus identifying the source as a union of connected components of the target.
\end{enumerate}
\end{lem}

\begin{proof}
Part \ref{item:adjsm} is straightforward. Part \ref{item:proj} can be checked locally, when $X=\Spec(A)$ with $A$ a $k$-algebra and $Y=\Spec(B)$ with $B$ an $L$-algebra, and the morphism considered is the canonical isomorphism $B \otimes_L L \otimes_k A \simeq B \otimes_k A$. To prove \ref{item:unit}, note that the general case follows from base-change to $Y$ of the case $Y=\Spec(L)$, which corresponds to $L\otimes_k L \to L$ via multiplication. Since the source is a product of separable extensions of $L$, one of which is isomorphic to $L$ by the multiplication map, the claim holds.
Up to the isomorphism of part \ref{item:proj}, the morphism $(Y \times_L Y')_{|k} \to Y_{|k} \times_k Y'_{|k}$ corresponds to the morphism $\unit_Y \times \id_{Y'}$ so \ref{item:closedsm} follows from \ref{item:unit} by base change to $Y'$. All schemes involved being reduced, the claims about identifications of irreducible components are clear.
\end{proof}

The functor $\res_{L/k}$ extends to a functor
$\begin{tikzcd}
\cor L \ar[r,"\res_{L/k}"] & \cork
\end{tikzcd}$
as we now explain. For any $X,Y \in \sm L$, we have $\Adm (X,Y)\subseteq \Adm (X_{|k},Y_{|k})$ by part \ref{item:closedsm} of Lemma \ref{lem:finiterestriction}, so for any $T$ in it the push-forward induces a map 
\[
\chst dT{X \times Y}{\omega_{Y}} \to \chst dT{X_{|k} \times_k Y_{|k}}{\omega_{Y_{|k}}}
\]
because $\omega_{L/k}$ is canonically trivial (the extension is separable). Passing to the limit, it gives a map $\cor L (X,Y) \to \cork (X_{|k},Y_{|k})$, and it is compatible with the composition of correspondences; this is an exercise using base change, where the only nontrivial input is that when $X,Y,Z \in \sm L$, the push-forward from $X\times_L Y \times_L Z$ to $X_{|k} \times_k Y_{|k} \times_k Z_{|k}$ respects products, which follows from part \ref{item:closedsm} of Lemma \ref{lem:finiterestriction}. 

The adjunction $(\res_{L/k},\ext_{L/k})$ between $\sm L$ and $\smk$ from part \ref{item:adjsm} of Lemma \ref{lem:finiterestriction} extends to an adjunction between $\cor L$ and $\cork$, using the same unit and counit, to which we apply the graph functors to view them as correspondences (see after Definition \ref{def:cortilde}). 

We are now going to define another adjunction $(\ext_{L/k}, \res_{L/k})$ (note the reversed order) that only exists at the level of correspondences.
The unit and counit
\[
\tunit_X:  X \to (X_L)_{|k} \quad \text{and} \quad  \tcounit_Y:(Y_{|k})_L \to Y 
\]
are defined respectively as the transpose of $\counit$ and $\unit$, using Example \ref{ex:push-forwards}. By part \ref{item:adjsm} of Lemma \ref{lem:finiterestriction} and by the composition properties of the transpose construction it is clear that they do define an adjunction $(\ext_{L/k},\res_{L/k})$. Since the extension $L/k$ is finite separable, the trace $L\to k$ induces a transfer homomorphism $\KMW_0(L)\to\KMW_0(k)$ (Section \ref{subsec:bilinear}) which coincides with the cohomological transfer by Lemma \ref{lem:trace}.

\begin{lem} \label{lem:compoextres}
The composition $\counit \circ \tunit$ is the multiplication (via the $\KMW_0(k)$-module structure) by the trace form of $L/k$ and the composition $\tcounit \circ \unit$ on $Y$ is the projection to the component of $(Y_{|k})_L$ corresponding to $Y$ and mentioned in \ref{item:unit} of Lemma \ref{lem:finiterestriction}. 
\end{lem}
\begin{proof}
By Example \ref{ex:push-forwards}, we obtain that $\counit \circ \tunit$ is the multiplication by $\counit_* \langle 1\rangle$. Since $\counit_X$ is obtained by base change to $X$ of the structural map $\sigma:\Spec(L) \to \Spec(k)$, the element $\counit_* \langle 1 \rangle$ is actually the pull-back to $X$ of $\sigma_*\langle 1 \rangle$, which is the trace form of $L/k$ by the definition of finite push-forwards for Chow-Witt groups (or the Milnor-Witt sheaf $\sKMW_0$). Similarly, but this time by base change of the diagonal $\delta: \Spec(L) \to \Spec(L) \times \Spec(L)$, we obtain that $\unit \circ \tcounit$ is the multiplication by $\delta_* \langle 1 \rangle$, which is the projector to the component corresponding to $\Spec(L)$, and thus to $Y$ by base change. 
\end{proof}

The functors $\ext_{L/k}$ and $\res_{L/k}$ between $\smk$ and $\sm L$ are trivially continuous for the Zariski topology: they send a covering to a covering and they preserve fiber products: $(X \times_Z X')_L \simeq X_L \times_{Z_L} X'_L$ and $(Y \times_T Y')_{|k}\simeq Y_{|k} \times_{T_{|k}} Y'_{|k}$. Therefore, they induce functors between categories of Zariski sheaves with transfers 
\[
\res_{L/k}^*: \shMWzk  \to \shMWzL \quad \text{and} \quad \ext_{L/k}^*: \shMWzL \to \shMWzk.
\] 
\index[notation]{shtk@$\shMWzk$}%
In order to avoid confusion, we set $\bc_{L/k} =\res_{L/k}^*$ and $\tr_{L/k}=\ext_{L/k}^*$ in order to suggest the words ``base change'' and ``transfer'', but we still use the convenient notation $F_L$ for $\bc_{L/k}(F)$ and $G_{|k}$ for $\tr_{L/k}(G)$. We thus have, by definition
\[
F_L(U)=F(U_{|k})\qquad \text{and} \qquad G_{|k}(V)=G(V_L)
\]
for any $F \in \shMWzk$, $G \in \shMWzL$, $U \in \sm L$ and $V \in \smk$.
It is formal that the adjunction $(\res_{L/k},\ext_{L/k})$ induces an adjunction $(\bc_{L/k},\tr_{L/k})$ with unit $\counit^*$ and counit $\unit^*$, while the adjunction $(\ext_{L/k},\res_{L/k})$ induces an adjunction $(\tr_{L/k},\bc_{L/k})$ with unit $\tcounit^*$ and counit $\tunit^*$. 

\begin{lem}
For any $X \in \smk$ and $Y \in \sm L$, we have natural isomorphisms 
\[
\big(\MWprep(X)\big)_L\simeq \MWprep(X_L) \in \shMWzL \quad\text{and}\quad \big(\MWprep(Y)\big)_{|k} \simeq \MWprep(Y_{|k}) \in \shMWzk.
\]
In the same spirit, $\big(\uHom(\MWprep(X), \MWprep(k))\big)_L\simeq \uHom\big(\MWprep(X_L),\MWprep(L)\big) \in \shMWzL$, while on the other hand $\uHom\big(\MWprep(Y), \MWprep(L)\big)_{|k} \simeq \uHom\big(\MWprep(Y_{|k}),\MWprep(k)\big) \in \shMWzk$.
\end{lem}

\begin{proof}
We have $\big(\MWprep(X)\big)_L(U) \simeq \cork (U_{|k},X) \simeq \cor L (U,X_L) \simeq \MWprep(X_L)(U)$, by the adjunction $(\res_{L/k},\ext_{L/k})$ and $\big(\MWprep(Y)\big)_{|k}(V)\simeq \cor L (V_L,Y) \simeq \cork (V,Y_{|k}) \simeq \MWprep(Y_{|k})(V)$ by the transposed adjunction $(\ext_{L/k},\res_{L/k})$.

Similarly, $\uHom \big(\MWprep(X),\MWprep(k)\big)_L(U)\simeq \cork (U_{|k}\times X,k) \simeq \cork ((U\times_L X_L)_{|k},k) \simeq \cor L (U\times_L X_L,L) \simeq \uHom \big(\MWprep(X_L),\MWprep(L)\big)(U)$. Finally, $\uHom \big(\MWprep(Y),\MWprep(L)\big)_{|k}(V)\simeq \cor L (V_L \times_L Y,L) \simeq \cork ((V_L \times_L Y)_{|k},k) \simeq \cork (V\times Y_{|k},k) \simeq \uHom \big(\MWprep(Y_{|k}),\MWprep(k)\big)(V)$.
\end{proof}

\begin{coro} \label{coro:extressmash}
For any sequence of pointed schemes $(X_1,x_1),\ldots,(X_n,x_n)$, we have
\[
\big(\MWprep\big((X_1,x_1)\wedge \ldots \wedge (X_n,x_n)\big)\big)_L \simeq \MWprep\big(((X_1)_L,(x_1)_L)\wedge\ldots\wedge ((X_n)_L,(x_n)_L)\big) \in \shMWzL
\]
and
\[
(\uHom \big(\MWprep\big((X_1,x_1)\wedge \ldots \wedge (X_n,x_n)\big),\MWprep(k)\big)_L \simeq \uHom \big(\MWprep(\wedge_{i=1}^n((X_i)_L,(x_i)_L)),\MWprep(k)\big)
\]
in $\shMWzL$.
\end{coro}
\begin{proof}
It follows immediately from the lemma applied to the split exact sequences defining the smash-products. 
\end{proof}

The same type of result for smash products would hold for restrictions, but the restriction of an $L$-pointed scheme is an $L_{|k}$-pointed scheme, not an $k$-pointed scheme. Nevertheless, the counit $\tunit^*$ of the adjunction $(\tr_{L/k},\bc_{L/k})$ induces maps
\begin{gather}
\big((\MWprep((X_1,x_1)\wedge \ldots \wedge (X_n,x_n)))_L\big)_{|k} \to \MWprep\big((X_1,x_1)\wedge\ldots\wedge (X_n,x_n)\big) \label{eq:mapsmashup} \\
\big(\uHom \big(\MWprep(\wedge_{i=1}^n(X_i,x_i)),\MWprep(k)\big)_L\big)_{|k} \to \uHom\big(\MWprep(\wedge_{i=1}^n(X_i,x_i)),\MWprep(k)\big). \label{eq:mapsmashdown}
\end{gather}
while the unit $\counit^*$ of the transposed adjunction $(\bc_{L/k},\tr_{L/k})$ induces maps
\begin{gather}
\MWprep\big((X_1,x_1)\wedge\ldots\wedge (X_n,x_n)\big) \to \big((\MWprep\big((X_1,x_1)\wedge \ldots \wedge (X_n,x_n)\big))_L\big)_{|k}  \label{eq:maptildesmashup} \\
\uHom (\MWprep(\wedge_{i=1}^n(X_i,x_i)),\MWprep(k)) \to \big((\uHom (\MWprep(\wedge_{i=1}^n(X_i,x_i)),\MWprep(k)))_L\big)_{|k}. \label{eq:maptildesmashdown}
\end{gather}

\begin{lem} \label{lem:extresCstar}
For any $F \in \shMWzk$ and any $G \in \shMWzL$, we have canonical isomorphisms $(\Cstar F)_L \simeq \Cstar{F_L}$ and $(\Cstar G)_{|k} \simeq \Cstar{G_{|k}}$.
\end{lem}
\begin{proof}
It is straightforward, using $Y_{|k} \times \Delta^n \simeq (Y \times_L \Delta^n_L)_{|k}$ for $F$ and $X_L \times_L \Delta_L^n \simeq (X \times_k \Delta^n)_L$ for $G$.
\end{proof}

To avoid confusion, let us write $\tZxcby kq$, $\tZxpy kq$, etc. for the (complexes of) sheaves over $k$, and $\tZxcby Lq$, $\tZxpy Lq$ etc. for the same objects over $L$.

For any $q \in \ZZ$, using Corollary \ref{coro:extressmash}, we obtain, an isomorphism ${\tZxcby kq}_L \simeq \tZxcby Lq$. Using the maps \eqref{eq:mapsmashup} for $q\geq 0$ and \eqref{eq:mapsmashdown} for $q<0$, applied to copies of $\Gmpt$, we obtain a morphism $(\tZxcby Lq)_{|k} \to \tZxcby kq$. Symmetrically, using the maps \eqref{eq:maptildesmashup} and \eqref{eq:maptildesmashdown} we obtain a morphism $\tZxcby kq \to (\tZxcby Lq)_{|k}$. Using Lemma \ref{lem:extresCstar}, they induce an isomorphism and morphisms 
\begin{equation} \label{eq:transferZq}
{\tZxcby kq}_L \simeq \tZxpy Lq, 
\qquad
\begin{tikzcd}[column sep=3ex]
(\tZxpy Lq)_{|k} \ar[r,"\tunit^*"] & \tZxpy kq
\end{tikzcd}
\qquad \text{and} \qquad 
\begin{tikzcd}[column sep=3ex]
\tZxpy kq \ar[r,"\counit^*"] & (\tZxpy Lq)_{|k}.
\end{tikzcd}
\end{equation} 

\begin{lem} \label{lem:crossediso}
For any Zariski sheaf $F$ on $\sm L$ and any $X \in \sm L$, we have $\H^*(X,F_{|k})=\H^*(X_L,F)$. For any Zariski sheaf $G$ on $\smk$ and any $Y \in \smk$, we have $\H^*(Y,G_L)=\H^*(Y_{|k},G)$. More generally, if $F$ and $G$ are complexes of Zariski sheaves, we have the same result for hypercohomology.
\end{lem}
\begin{proof}
Let a Zariski sheaf $F$ be flabby if restricted to the small site of any scheme, it gives a flabby sheaf in the usual sense: restrictions are surjective. The Zariski cohomology can then be computed using resolutions by flabby sheaves. Both functors $\bc_{L/k}$ and $\tr_{L/k}$ preserve flabby resolutions. So, given a flabby resolution $I^\bullet$ of $F$, we have 
\[
\H^i(X_L,F)=\H^i(I^\bullet(X_L))= \H^i(\ext_{L/k}^*I^\bullet (X))=\H^i(X,F_{|k}).
\]
A similar proof holds for $G$ and $Y$. The claim about hypercohomology is proved similarly using flabby Cartan-Eilenberg resolutions.
\end{proof}

Using the morphisms \eqref{eq:transferZq} and Lemma \ref{lem:crossediso}, we obtain for any $X \in \smk$ and any $q$ two morphisms 
\[
\hypH^p(X, \tZxpy kq) \to \hypH^p(X,(\tZxpy Lq)_k) \simeq \hypH^p(X_L,\tZxpy Lq)
\]
and
\[
\hypH^p(X_L, \tZxpy Lq) \simeq \hypH^p(X,(\tZxpy Lq)_k) \to \hypH^p(X,\tZxpy kq).
\]
In other words:
\begin{equation}
\begin{tikzcd}
{\HMW^{p,q}}_k(X,\ZZ) \ar[r,"\bc_{L/k}"] & {\HMW^{p,q}}_L(X_L,\ZZ) 
\end{tikzcd} 
\quad \text{and} \quad 
\begin{tikzcd}
{\HMW^{p,q}}_L(X_L,\ZZ) \ar[r,"\tr_{L/k}"] & {\HMW^{p,q}}_k(X,\ZZ).
\end{tikzcd}
\end{equation}
Using Lemma \ref{lem:compoextres}, we obtain: 
\begin{lem}
On MW-motivic cohomology, the composition $\tr_{L/k} \circ \bc_{L/k}$ is the multiplication (via the $\KMW_0(k)$-module structure) by the trace form of the extension $L/k$.
\end{lem}

We now compare the MW-cohomology groups when computed over $k$ for a limit scheme that is of the form $X_L$ and when computed over some extension $L$ of $k$ (where $L$ is still asumed to be perfect).

If $L$ is a finitely generated extension of $k$, that we view as the inverse limit $L=\varprojlim U$ of schemes $U \in \smk$, we obtain a natural map
\[
\varinjlim_{T \in \Adm(U\times X,Y)} \hspace{-3ex}\chst dT {U \times X \times Y}{\omega_{Y}} \to \hspace{-1ex} \varinjlim_{T' \in \Adm(X_L,Y_L)}\hspace{-3ex}\chst d{T'} { X_L \times Y_L}{\omega_{Y_L}} 
\]
induced by pull-backs between the groups where $T'=T_L$.
When $d=\dimn Y$, this is a map 
\[
\cork (U \times X,Y) \to \cor L (X_L,Y_L)
\]
functorial in $U$. Taking the limit over $U$, we therefore obtain a map
\[
\begin{tikzcd}
\cork (X_L,Y) = \varinjlim \cork(U \times X,Y) \ar[r,"\Psi_{L/k}"] & \cor L (X_L,Y_L)
\end{tikzcd}
\]
using the construction of section \ref{sec:limits} to define the left hand side.
\begin{prop}
The map $\Psi_{L/k}$ is an isomorphism.
\end{prop}
\begin{proof}
To shorten the notation, we drop the canonical bundles in the whole proof, since they behave as they should by pull-back. 
The source of $\Psi_{L/k}$ can be defined using a single limit, as 
\[
\varinjlim_{(U,T)} \chs dT{U\times X \times Y}
\]
where the limit runs over the pairs $(U,T)$ with $(U,T)\leq (U',T')$ if there is a map $U' \subseteq U$ and $T\cap U' \subseteq T'$. The corresponding transition map on Chow-Witt groups is the restriction to $U'$ composed with the extension of support from $T \cap U'$ to $T'$. Note that both of these maps are injective, as explained in the proof of Lemma \ref{lem:restrictioninj} for the first one and right before that same Lemma for the second one. 
The maps $f_{U,T}: \chs dT{U \times X \times Y} \to \varinjlim_{V \subseteq U} \chs d{T \cap V}{V \times X \times Y}$ sending the initial group in the direct system to its limit are again injective, and their target can be identified with $\chs d{T_L}{X_L \times Y_L}$ by Lemma \ref{lem:limitChowiso} (independently of $U$, except that $T$ lives on $U$). The $f_{U,T}$ therefore induce an injective map
\begin{equation} \label{eq:limitmap}
\varinjlim_{(U,T)} \chs dT{U\times X \times Y} \to \varinjlim_{(U,T)} \chs d{T_L}{X_L \times Y_L}
\end{equation}
This map is also surjective because any $\chs d{T\cap V}{V \times X \times Y}$ in the target of $f_{U,T}$ is in the image of the same group on the source, after passing from $(U,T)$ to $(V, T \cap V)$. Finally, the isomorphism \eqref{eq:limitmap} actually maps to $\cor L(X_L,Y_L)$ because any $T' \in \Adm(X_L,Y_L)$ actually comes by pull-back from some open $U$, by \cite{EGA4-3_a}*{8.3.11}. It is easy to see that in defining that isomorphism we have just expanded the definition of $\Psi_{L/k}$
\end{proof}



\numberwithin{equation}{thm}
\renewcommand{\theequation}{\thethm.\alph{equation}}
\renewcommand{\theparagraph}{\thesubsection.\alph{paragraph}}
\chapter[MW-motivic complexes]{Milnor-Witt motivic complexes\author{Frédéric Déglise and Jean Fasel}}
\label{ch:dmt}

\section*{Abstract}
The aim of this work is to develop a theory parallel
to that of motivic complexes based on cycles
and correspondences with coefficients in symmetric bilinear forms.
This framework is closer to the point of view of $\Aone$-homotopy than the original one envisioned by Beilinson and set up by Voevodsky.

\section*{Introduction}

The aim of this paper is to define the various categories of MW-motives built out of the category of finite Chow-Witt correspondences constructed in \chfinitecw, and to study the motivic cohomology groups intrinsic to these categories. 
In Section \ref{sec:MWtransfers}, we start with a quick reminder of the basic properties of the category $\cork$. 
We then proceed with our first important result, namely that the sheaf (in either the Nisnevich or the étale topologies) associated to a MW-presheaf, i.e.\ an additive functor $\cork\to \Ab$,  is a MW-sheaf. 
The method follows closely Voevodsky's method and relies on Lemma \ref{lm:corr_main}. 
We also discuss the monoidal structure on the category of MW-sheaves. 
In the second part of the paper, we prove the analogous for MW-sheaves of a famous theorem of Voevodsky saying that the sheaf (with transfers) associated to a homotopy invariant presheaf with transfers is strictly $\Aone$-invariant. 
Our method here is quite lazy. We heavily rely on the fact that an analogue theorem holds for quasi-stable sheaves with framed transfers by \cite{Garkusha15_b}*{Theorem~1.1}, provided the base field is infinite perfect of characteristic different from $2$. 
Having this theorem at hand, it suffices to construct a functor from the category of linear framed presheaves to $\cork$ to prove the theorem. This functor is of independent interest and this is the reason why we take this shortcut. 
However, a direct proof of this theorem due to H.~A.~Kolderup (still relying on ideas of Panin-Garkusha) can be found in \cite{Kolderup17_b}. 
In Section \ref{sec:MWmotives}, we finally pass to the construction of the categories of MW-motives starting with a study of different model structures on the category of possibly unbounded complexes of MW-sheaves. The ideas here are closely related to \cite{CD12_b}. The category of effective motives $\DMtekR$ (with coefficients in a ring $R$) 
\index[notation]{dmteffkr@$\DMtekR$}%
is then defined as the category of $\Aone$-local objects in this category of complexes. 
Using the analogue of Voevodsky's theorem proved in Section \ref{sec:framed}, these objects are easily characterized by the fact that their homology sheaves are strictly $\Aone$-invariant. This allows as usual to give an explicit $\Aone$-localization functor, defined in terms of the Suslin (total) complex. The category of geometric objects is as in the classical case the subcategory of compact objects of $\DMtekR$. 
Our next step is the formal inversion of the Tate motive in $\DMtekR$ to obtain the stable category of MW-motives $\DMtkR$ (with coefficients in $R$). We can then consider motivic cohomology as groups of extensions in this category, a point of view which allows proving in Section \ref{sec:MWcohom} many basic property of this version of motivic cohomology, including a commutativity statement and a comparison theorem between motivic cohomology and Chow-Witt groups.

\subsection*{Conventions} 

In all this work, we fix a base field $k$, assumed to be perfect.
All schemes considered are assumed to be separated of finite type over $k$, unless explicitly stated.

We fix a ring of coefficients $R$.
We also consider a Grothendieck topology $t$ on the site of
smooth $k$-schemes, which in practice will be either the Nisnevich
or the étale topology. In section 3 and 4 we will restrict
to these two latter cases.

\section{Milnor-Witt transfers on sheaves}\label{sec:MWtransfersonsheaves}

\subsection{Reminder on Milnor-Witt correspondences}

\begin{num}
We use the definitions and constructions of \chfinitecw.

In particular, for any smooth schemes $X$ and $Y$ (with $Y$ connected of dimension $d$),
we consider the following group of \emph{finite MW-correspondences} 
from $X$ to $Y$:
\begin{equation}\label{eq:def_corr}
\MWcorr(X,Y):=\varinjlim_T \chst dT{X \times Y}{\omega_Y}
\end{equation}
\index[notation]{cortxy@$\MWcorr(X,Y)$}%
where $T$ runs over the ordered set of reduced (but not necessarily irreducible) closed subschemes in
 $X \times Y$ whose projection to $X$ is finite equidimensional and $\omega_Y$ is the pull back of the canonical sheaf of $Y$ along the projection to the second factor.
This definition is extended to the case where $Y$ is not connected
by additivity. When considering the coefficients ring $R$, we put:
\begin{equation}\label{eq:R-lin_MWcorr}
\MWcorr(X,Y)_R:=\MWcorr(X,Y) \otimes_\ZZ R.
\end{equation}
 Recall that the same formula is used when considering $R$-linear finite correspondences.
 In the sequel, we drop the index $R$ from the notation when there is no possible confusion. 
 
Because there is a natural morphism from Chow-Witt groups
 (twisted by any line bundle) to Chow groups, we get a canonical map:
\begin{equation}\label{eq:from_corr2corV}
\pi_{XY}:\MWcorr(X,Y) \rightarrow \Vcorr(X,Y)
\end{equation}
for any smooth schemes $X$ and $Y$, where the right hand side
is the group of Voevodsky's finite correspondences which
is compatible with the composition --- see \emph{loc.\ cit.}\ 
Remark 4.12. Let us recall the following result. 
\end{num}

\begin{lem}\label{lm:corr&corrV_upto2tor}
If $2\in R^\times$, the induced map
\[
\pi_{XY}:\MWcorr(X,Y) \rightarrow \Vcorr(X,Y)
\]
is a split epimorphism.
\end{lem}
The lemma comes from the basic fact that the following composite map
\[
\KM_n(F) \xrightarrow{(1)} \KMW_n(F,\Lb)
 \xrightarrow{(2)} \KM_n(F)
\]
is multiplication by $2$, where (1) is the map from Milnor K-theory
of a field $F$ to Milnor-Witt K-theory of $F$ twisted by the $1$-dimensional
$F$-vector space $\Lb$ described in \chfinitecw, \S\ref{sec:MWKtheory} and (2) is the map killing $\eta$
(see the discussion after \chfinitecw, Definition~\ref{dfn:twistedCHWitt}).

\begin{rem}\ 
\label{rem:finite_corr&plim}%
\begin{enumerate}
\item In fact, a finite MW-correspondence amounts to a finite correspondence
$\alpha$ together with a symmetric bilinear form over the function field
of each irreducible component of the support of $\alpha$
satisfying some condition related with residues;
see \chfinitecw, Definition \ref{dfn:support}.
\item Every finite MW-correspondence between smooth schemes $X$ and $Y$ has a well defined support (\chfinitecw, \ref{dfn:support}). Roughly speaking, it is the minimal closed subset of $X\times Y$ on which the correspondence is defined.
\item Recall that the Chow-Witt group in degree $n$ 
of a smooth $k$-scheme $X$ can be defined as the $n$-th Nisnevich cohomology
group of the $n$-th unramified Milnor-Witt sheaf $\sKMW_n$ (this cohomology being computed using an explicit flabby resolution of $\sKMW_n$).
This implies that the definition can be uniquely extended to the case
where $X$ is an essentially smooth $k$-scheme.
Accordingly, one can extend the definition of finite 
MW-correspondences to the case of essentially smooth
$k$-schemes using formula \eqref{eq:def_corr}, see \chfinitecw, \S\ref{sec:limits}. The definition
of composition, being essentially an intersection product, extends to that generalized setting.
We will use that fact in the proof of Lemma \ref{lm:corr_main}.
\item Consider the notations of the previous point.
Assume that the essentially smooth $k$-scheme
$X$ is the projective limit of a projective system of essentially
smooth $k$-schemes $(X_i)_{i \in I}$. Then the canonical map:
\[
\Big({\varinjlim}_{i \in I^{op}} \MWcorr(X_i,Y)\Big) \longrightarrow \MWcorr(X,Y)
\]
is an isomorphism. This readily follows from formula \eqref{eq:def_corr}
and the fact that Chow-Witt groups, as Nisnevich cohomology, commute
with projective limits of schemes. See also \chfinitecw, \S\ref{sec:limits} for an extended discussion of these facts. 
 \item For any smooth schemes $X$ and $Y$, the group $\MWcorr(X,Y)$ is endowed with a structure of a left $\sKMW_0(X)$-module and a right $\sKMW_0(Y)$-module (\chfinitecw, Example~\ref{ex:action}). Pulling back along $X\to \Spec k$, it follows that $\MWcorr(X,Y)$ is a left $\sKMW_0(k)$-module and it is readily verified that the category $\cork$, defined just below, is in fact $\sKMW_0(k)$-linear. Consequently, we can also consider $\sKMW_0(k)$-algebras as coefficient rings. 
\end{enumerate}
\end{rem}

\begin{num}
Recall from \emph{loc.\ cit.}\ that there is a composition product
of MW-correspondences which is compatible with the projection
map $\pi_{XY}$.
\end{num}
\begin{dfn} \label{dfn:corkcorVk}
We denote by $\cork$ (resp.\ $\corVk$) the additive category whose objects
are smooth schemes and morphisms are finite MW-correspondences
(resp.\ correspondences). If $R$ is a ring,
we let $\corkR$ (resp.\ $\corVkR$) be
the category $\cork \otimes_\ZZ R$ (resp.\ $\corVk \otimes_\ZZ R$).
\index[notation]{cortk@$\cork$}%
\index[notation]{cork@$\corVk$}%
\index[notation]{cortkR@$\corkR$}%
\index[notation]{corkR@$\corVkR$}%
\index{Milnor-Witt!finite correspondence}%
\index{finite!correspondence}%

We denote by
\begin{equation}\label{eq:from_smc2smcV}
\pi:\cork \rightarrow \corVk
\end{equation}
the additive functor which is the identity on objects and the map
 $\pi_{XY}$ on morphisms.
\end{dfn}

As a corollary of the above lemma, the induced functor
\[
\pi:\corkR \rightarrow \corVkR,
\]
is full when $2\in R^\times$. Note that the corresponding result without inverting $2$ is wrong by \chfinitecw, \ref{rem:notsurjective}.

\begin{num}
\label{num:symmonstruc}%
The external product of finite MW-correspondences induces
a symmetric monoidal structure on $\cork$ which on objects is given by
the cartesian product of $k$-schemes. One can check that the
functor $\pi$ is symmetric monoidal, for the usual
symmetric monoidal structure on the category $\corVk$.
 
Finally, the graph of any morphism $f:X \rightarrow Y$ can be
seen not only as a finite correspondence $\gamma(f)$ from $X$ to $Y$
but also as a finite MW-correspondence $\tilde \gamma(f)$
such that $\pi \tilde \gamma(f)=\gamma(f)$. One obtains in this
way a canonical functor:
\begin{equation}\label{eq:from_sm2smc}
\tilde \gamma:\smk \rightarrow \cork
\end{equation}
which is faithful, symmetric monoidal, and such that
 $\pi \circ \tilde \gamma=\gamma$, where $\gamma:\smk \rightarrow \corVk$ is
 the classical graph functor with values in finite correspondences. 
\end{num}

\subsection{Milnor-Witt transfers}
\label{sec:MWtransfers}

\begin{dfn}
We let $\pshMWkR$ (resp.\ $\pshVkR$, resp.\ $\pshkR$) be the category of additive presheaves of $R$-modules on $\cork$ (resp.\ $\corVk$, resp.\ $\smk$). 
Objects of $\pshMWkR$ will be simply called \emph{MW-presheaves}.
\index[notation]{pshtkR@$\pshMWkR$}%
\index[notation]{pshtrkR@$\pshVkR$}%
\index[notation]{pshkR@$\pshkR$}%
\index{Milnor-Witt!presheaf}%
\end{dfn}

\begin{dfn}\label{def:representablepsht}
We denote by $\MWprep_R(X)$
\index[notation]{ctrx@$\MWprep_R(X)$}%
\index[notation]{ctx@$\MWprep(X)$}%
the representable presheaf $Y\mapsto \MWcorr (Y,X)\otimes_\ZZ R$. As usual, we also write $\MWprep(X)$ in place of $\MWprep_R(X)$ in case the context is clear.
\end{dfn}

The category of MW-presheaves is an abelian Grothendieck category.\footnote{\label{fn:Grothendieck}Recall
that an abelian category
is called Grothendieck abelian if it admits a family of generators,
admits small sums and filtered colimits are exact.
The category of presheaves over any essentially small category $\mathscr S$
with values in the category of $R$-modules is a basic example of Grothendieck
abelian category. In fact, it is generated by representable presheaves of
$R$-modules. The existence of small sums is obvious
and the fact filtered colimits are exact can be reduced to the similar fact
for the category of $R$-modules by taking global sections over objects
of $\mathscr S$.} 
It admits a unique symmetric monoidal
structure such that the Yoneda embedding
\[
\cork \rightarrow \pshMWkR,\phantom{i} X \mapsto\MWprep(X)
\]
is symmetric monoidal (see e.g. \cite{Mazza06_b}*{Lecture~8}). From the functors \eqref{eq:from_smc2smcV}
 and \eqref{eq:from_sm2smc}, we derive as usual adjunctions of categories:
\[
\begin{tikzcd}[row sep=20pt, column sep=20pt]
\pshkR \ar[r,shift left=3pt,"{\tilde \gamma^*}"]
 & \pshMWkR \ar[r,shift left=3pt,"{\pi^*}"] \ar[l,shift left=3pt,"{\tilde \gamma_*}"] 
 & \pshVkR \ar[l,shift left=3pt,"\pi_*"]
\end{tikzcd}
\]
such that $\tilde \gamma_*(F)=F \circ \tilde \gamma$,
 $\pi_*(F)=F \circ \pi$. 
The left adjoints $\tilde \gamma^*$ and $\pi^*$ are easily described as follows. For a smooth scheme $X$, let $R(X)$ be the presheaf (of abelian groups) such that $R(X)(Y)$ is the free $R$-module generated by $\Hom(Y,X)$ for any smooth scheme $Y$. 
\index[notation]{rx@$R(X)$}%
The Yoneda embedding yields $\Hom_{\pshkR}(R(X),F)=F(X)$ for any presheaf $F$, and in particular 
\[
\Hom_{\pshkR}(R(X),\tilde\gamma_*(F))=F(X)=\Hom_{\pshMWkR}(\MWprep_R(X),F)
\] 
for any $F\in\pshMWkR$. We can thus set $\tilde\gamma^*(R(X))= \MWprep_R(X)$. On the other hand, suppose that 
\[
F_1\to F_2\to F_3\to 0
\]
is an exact sequence in $\pshkR$. The functor $\Hom_{\pshkR}(-,F)$ being left exact for any $F\in \pshkR$, we find an exact sequence of presheaves
\[
0\to \Hom_{\pshkR}(F_3,\tilde\gamma_*(G))\to \Hom_{\pshkR}(F_2,\tilde\gamma_*(G))\to \Hom_{\pshkR}(F_1,\tilde\gamma_*(G))
\]
for any $G\in\pshMWkR$ and by adjunction an exact sequence
\[
0\to \Hom_{\pshkR}(\tilde\gamma^*(F_3),G)\to \Hom_{\pshkR}(\tilde\gamma^*(F_2),G)\to \Hom_{\pshkR}(\tilde\gamma^*(F_1),G)
\]
showing that $\tilde\gamma^*(F_3)$ is determined by $\tilde\gamma^*(F_2)$ and $\tilde\gamma^*(F_1)$, i.e.\ that the sequence
\[
\tilde\gamma^*(F_1)\to \tilde\gamma^*(F_2)\to \tilde\gamma^*(F_3)\to 0
\]
is exact. This gives the following formula. If $F$ is a presheaf, we can choose a resolution by (infinite) direct sums of representable presheaves (e.g. \cite{Mazza06_b}*{proof of Lemma~8.1})
\[
F_1\to F_2\to F\to 0
\]
and compute $\tilde\gamma^*F$ as the cokernel of $\tilde\gamma^*F_1\to \tilde\gamma^*F_2$. We let the reader define $\tilde \gamma^*$ on morphisms and check that it is independent (up to unique isomorphisms) of choices. A similar construction works for $\pi^*$. Note that the left adjoints $\tilde \gamma^*$
and $\pi^*$ are symmetric monoidal and right-exact.

\begin{lem}\label{lm:compare_transfers}
The functors $\tilde \gamma_*$ and $\pi_*$ are faithful. If $2$ is invertible in the ring $R$ then the functor $\pi_*$ is also full.
\end{lem}

\begin{proof}
The faithfulness of both $\tilde \gamma_*$ and $\pi_*$ are obvious.
To prove the second assertion, we use the
fact that the map \eqref{eq:from_corr2corV} from
finite MW-correspondences to correspondences is surjective
after inverting $2$ (Lemma \ref{lm:corr&corrV_upto2tor}).
In particular, given a MW-presheaf $F$,
the property $F=\pi_*(F_0)$ is equivalent to the property on $F$
that for any $\alpha, \alpha' \in \MWcorr(X,Y)$ with $\pi(\alpha)=\pi(\alpha')$ then $F(\alpha)=F(\alpha')$.
It is then clear that a natural transformation
between two presheaves with transfers $F_0$ and $G_0$
is the same thing as a natural transformation
between $F_0 \circ \pi$ and $G_0 \circ \pi$.
\end{proof}

\begin{dfn}
We define a \emph{MW-$t$-sheaf} (resp.\ $t$-sheaf with transfers)
\index{Milnor-Witt!$t$-sheaf}%
\index{sheaf!with transfers}
to be a presheaf with
MW-transfers (resp.\ with transfers) $F$ such that
$\tilde \gamma_*(F)=F \circ \tilde \gamma$ (resp.\ $F \circ \gamma$)
is a sheaf for the given topology $t$.
When $t$ is the Nisnevich topology, we will simply
say \emph{MW-sheaf} and when $t$ is the étale topology
we will say \emph{étale MW-sheaf}.
\index{Milnor-Witt!sheaf}%

We denote by $\shMWtkR$ the category of MW-$t$-sheaves,
seen as a full subcategory of the $R$-linear
category $\pshMWkR$. When $t$ is the Nisnevich topology,
we drop the index in this notation.
\index[notation]{shtkr@$\shMWtkR$}%
\end{dfn}
Note that there is an obvious forgetful functor
\[
\tfO_t:\shMWtkR \rightarrow \pshMWkR
 \text{ (resp.\ } \mathcal O^{\mathrm{tr}}_t:\shVtkR \rightarrow \pshVkR)
\]
\index[notation]{ot@$\tfO_t$}%
\index[notation]{shtrkr@$\shVtkR$}%
which is fully faithful.
In what follows, we will drop
the indication of the topology $t$ in the above functors,
as well as their adjoints.

\begin{exm}
Given a smooth scheme $X$, the presheaf $\MWprep(X)$ is in general not a MW-sheaf (see \chfinitecw, Ex.~\ref{exm:notnis}). Note however that $\MWprep(\Spec k)$ is the unramified $0$-th Milnor-Witt sheaf $\sKMW_0$ (defined in \cite{Morel12_b}*{\S 3}) by \emph{loc.\ cit.}\ Ex. 4.4.
\end{exm}

As in the case of the theory developed by Voevodsky, the theory of 
MW-sheaves rely on the following fundamental lemma, whose proof
is adapted from Voevodsky's original argument.
 
\begin{lem}\label{lm:corr_main}
Let $X$ be a smooth scheme and $p:U \rightarrow X$ be a $t$-cover
where $t$ is the Nisnevich or the étale topology.

Then the following complex
\[
\cdots \xrightarrow{\ d_n\ } \MWprep(U^n_X) \longrightarrow
\cdots \longrightarrow \MWprep(U \times_X U) \xrightarrow{\ d_1\ } \MWprep(U)
\xrightarrow{\ d_0\ }\MWprep(X) \rightarrow 0
\]
where $d_n$ is the differential associated with the \v Cech simplicial
scheme of $U/X$, is exact on the associated $t$-sheaves.
\end{lem}

\begin{proof}
We have to prove that the fiber of the above complex at a $t$-point
is an acyclic complex of $R$-modules.  Taking into
account Remark \ref{rem:finite_corr&plim}(4), we are reduced to 
prove, given an essentially smooth local henselian scheme $S$, that
the following complex
\[
C_*:=\cdots \xrightarrow{\ d_n\ } \MWcorr(S,U^n_X) \longrightarrow
\cdots \longrightarrow \MWcorr(S,U \times_X U) \xrightarrow{\ d_1\ } \MWcorr(S,U)
\xrightarrow{\ d_0\ } \MWcorr(S,X) \rightarrow 0
\]
is acyclic.

Let $\Adm=\Adm(S,X)$ be the set of admissible subsets in $S\times X$ (\chfinitecw, Definition~\ref{dfn:admissible_subset}). 
Given any $T \in \Adm$, and an integer $n \geq 0$,
we let $C_n^{(T)}$ be the subgroup of $\MWcorr(S,U^n_X)$ consisting of
MW-correspondences whose support is in the closed subset $U^n_T:=T \times_X U^n_X$ of $S\times U^n_X$.
The differentials are given by direct images along projections,
so they respect the support condition on MW-correspondence
associated with $T \in \mathcal F$ and make $C_*^{(T)}$
into a subcomplex of $C_*$.

It is clear that $C_*$ is the filtering union of the subcomplexes
$C_*^{(T)}$ for $T \in \mathcal F$ so it suffices to prove that,
for a given $T \in \mathcal F$, the complex $C_*^{(T)}$ is split. We prove the result when $R=\ZZ$, the general statement follows after tensoring with $R$.
Because $S$ is henselian and $T$ is finite over $S$, the scheme $T$
is a finite sum of local henselian schemes. Consequently,
the $t$-cover $p_T:U_T \rightarrow T$, which is in particular
étale and surjective, admits a splitting $s$. It follows from \cite{Milne12_b}*{Proposition~2.15} that $s$ is an isomorphism onto a connected component of $U_T$.
We therefore obtain maps $s \times 1_{U_T^n}:U_T^n\to U^{n+1}_T$ such that $U^{n+1}_T=U_T^n\sqcup D^{n+1}_{T}$ for any $n\geq 0$ and a commutative diagram
\[
\begin{tikzcd}
U^n_T \ar[d] \ar[r] & (S\times U_X^{n+1})\setminus D^{n+1}_{T} \ar[d] \\
U^{n+1}_T \ar[r] \ar[d] & S\times U_X^{n+1} \ar[d] \\
U_T^n \ar[r] & S\times U_X^{n}
\end{tikzcd}
\]
in which the squares are Cartesian and the right-hand vertical maps are étale. By étale excision, we get isomorphisms
\[
\chst *{U_T^n}{S\times U_X^{n}}{\omega_{U_X^{n}}}\to \chst *{U_T^n}{(S\times U_X^{n+1})\setminus D^{n+1}_T}{\omega_{U_X^{n+1}}}
\]
and
\[
\chst *{U_T^n}{S\times U_X^{n+1}}{\omega_{U_X^{n+1}}}\to \chst *{U_T^n}{(S\times U_X^{n+1})\setminus D^{n+1}_T}{\omega_{U_X^{n+1}}}.
\]
Putting these isomorphisms together, we obtain an isomorphism
\[
\chst *{U_T^n}{S\times U_X^{n}}{\omega_{U_X^{n}}}\to \chst *{U_T^n}{S\times U_X^{n+1}}{\omega_{U_X^{n+1}}}
\]
that we can compose with the extension of support to finally obtain a homomorphism
\[
(s \times 1_{U_T^n})_*:\chst *{U_T^n}{S\times U_X^{n}}{\omega_{U_X^{n}}}\to \chst *{U_T^{n+1}}{S\times U_X^{n+1}}{\omega_{U_X^{n+1}}}
\]
yielding a contracting homotopy
\[
(s \times 1_{U_T^n})_*:C_n^{(T)} \rightarrow C_{n+1}^{(T)}.\qedhere
\]
\end{proof}

\begin{num}\label{num:rightadjoint}
As in the classical case, one can derive from this lemma
the existence of a left adjoint $\tilde a$ to the functor $\tfO$.
The proof is exactly the same as in the case of sheaves
with transfers (cf.\ \cite{CD12_b}*{10.3.9} for example) but we include it here for the convenience of the reader.

Let us introduce some notation. If $P$ is a presheaf on $\smk$,
 we define a presheaf with MW-transfers
\begin{equation}\label{eq:df_gamma^!}
\tilde\gamma^!(P):Y \mapsto \Hom_{\pshkR}(\tilde\gamma_*(\MWprep(Y)),P).
\end{equation}
\index[notation]{gashr@$\tilde\gamma^"!$}%
and we observe that $\tilde\gamma^!$ is right adjoint to the functor $\tilde\gamma_*$. The latter, having both a left and a right adjoint, is then exact.
Given a natural transformation
\[
\phi:P \rightarrow \tilde\gamma_*\tilde\gamma^!(P)
\]
and smooth schemes $X$ and $Y$, 
we define a pairing
\[
P(X) \times \MWcorr(Y,X) \rightarrow P(Y),
 (\rho,\alpha) \mapsto \langle \rho,\alpha \rangle_\phi
 :=[\phi_X(\rho)]_Y(\alpha)
\]
where $\phi_X(\rho)$ is seen as a natural transformation
$\MWprep(X) \rightarrow P$.
The following lemma is tautological.
\end{num}

\begin{lem}\label{lm:W-transfers_basic_existence}
Let $P$ be a presheaf on $\smk$.
Then there is a bijection between the following data:
\begin{itemize}
\item Presheaves with MW-transfers $\tilde P$ such that $\tilde\gamma_*(\tilde P)=P$;
\item Natural transformations $\phi:P \rightarrow \tilde\gamma_*\tilde\gamma^!(P)$
such that:
\begin{enumerate}
\item[\rm{(W1)}] $\forall \rho \in P(X), \langle \rho, \id_X \rangle_\phi=\rho$.
\item[\rm{(W2)}] $\forall (\rho,\beta,\alpha)
 \in P(X) \times \MWcorr(Y,X) \times \MWcorr(Z,Y),
  \langle\langle \rho,\beta \rangle_\phi,\alpha\rangle_\phi
	 =\langle \rho,\beta \circ \alpha\rangle_\phi$;
\end{enumerate}
\end{itemize}
according to the following rules:
\[
\begin{tikzcd}[row sep=.5ex, column sep=0ex]
\tilde P \ar[r,mapsto] & \big(P=\tilde\gamma_*(\tilde P) \xrightarrow{\mathrm{ad}'} \tilde\gamma_*\tilde\gamma^!\tilde\gamma_*(\tilde P)=\tilde\gamma_*\tilde\gamma^!(P)\big) \\
\big(P,\langle.,\alpha\rangle_\phi\big)  \ar[r,mapsfrom] & \phi,
\end{tikzcd}
\]
where $\mathrm{ad}'$ is the unit map for the adjunction $(\tilde\gamma_*,\tilde\gamma^!)$. 
\end{lem}
Before going further, we note the following corollary of the previous result.
 
\begin{coro}\ 
\label{cor:corr_main}%
\begin{enumerate}
\item
\label{item:tilde-gamma-mw-sheaf}%
For any $t$-sheaf $F$ on $\smk$,
 $\tilde\gamma^!(F)$ is a MW-$t$-sheaf.
\item
\label{item:cover-corr}%
Let $\alpha \in \MWcorr(X,Y)$ be a finite MW-correspondence
and $p:W \rightarrow Y$ a $t$-cover.
Then there exists a $t$-cover $q:W' \rightarrow X$
and a finite MW-correspondence $\hat \alpha:W' \rightarrow W$
such that the following diagram in $\cork$ commutes:
\begin{equation}\label{eq:lift_corr}
\begin{tikzcd}
[row sep=18pt, column sep=18pt]
W'\ar[r,"{\hat \alpha}"] \ar[d,"q"'] & W \ar[d,"p"] \\
X \ar[r,"\alpha"'] & Y.
\end{tikzcd}
\end{equation}
\end{enumerate}
\end{coro}

The first property is a direct consequence of Lemma
\ref{lm:corr_main} given Formula \eqref{eq:df_gamma^!}.
The second property is an application of the
fact that $\MWprep(W) \rightarrow \MWprep(X)$ is an epimorphism
of sheaves, obtained from the same proposition.

We are ready to state and prove the main lemma
which proves the existence of the right adjoint $\tilde a$ to $\tfO$.
 
\begin{lem}
\label{lm:exists-Wtr}%
Let $\tilde P$ be a MW-presheaf and $P:=\tilde\gamma_*(\tilde P)$. Let $F$
be the $t$-sheaf associated with $P$ and let 
$\tau:P \rightarrow F$ be the canonical natural transformation.

Then there exists a unique pair $(\tilde F,\tilde \tau)$
 such that:
\begin{enumerate}
\item
\label{item:MW-sheaf-image-gamma}%
$\tilde F$ is a MW-$t$-sheaf such that $\tilde\gamma_*(\tilde F)=F$.
\item
\label{item:coincides-with-tau}%
$\tilde \tau:\tilde P \rightarrow \tilde F$ is a natural transformation
 of MW-presheaves such that the induced transformation
\[
P=\tilde\gamma_*(\tilde P) \xrightarrow{\tilde\gamma_*(\tilde \tau)} \tilde\gamma_*(\tilde F)=F
\]
coincides with $\tau$.
\end{enumerate}
\end{lem}
\begin{proof}
Let us construct
$\tilde F$. Applying Lemma \ref{lm:W-transfers_basic_existence} to $\tilde P$ and $P$, we get a 
canonical natural transformation: $\psi:P \rightarrow \tilde\gamma_*\tilde\gamma^!(P)$.
Applying part \ref{item:tilde-gamma-mw-sheaf} of Corollary~\ref{cor:corr_main} and the fact that $F$ is the $t$-sheaf associated with $P$,
there exists a unique natural transformation $\phi$ which
fits into the following commutative diagram:
\[
\begin{tikzcd}
[row sep=20pt,column sep=20pt]
P \ar[r,"{\psi}"] \ar[d,"\tau"'] & \tilde\gamma_*\tilde\gamma^!(P) \ar[d,"{\tilde\gamma_*\tilde\gamma^!(\tau)}"] \\
F \ar[r,"\phi"'] & \tilde\gamma_*\tilde\gamma^!(F).
\end{tikzcd}
\]
To obtain the MW-sheaf $\tilde F$ satisfying \ref{item:MW-sheaf-image-gamma},
it is sufficient according to
Lemma \ref{lm:W-transfers_basic_existence} to prove
that conditions (W1) and (W2) are satisfied for the
pairing $\langle.,.\rangle_\phi$. Before proving this,
we note that the existence of $\tilde \tau$ satisfying property \ref{item:coincides-with-tau}
is equivalent to the commutativity of the above diagram.
In particular, the unicity of $(\tilde F,\tilde \tau)$ comes
from the unicity of the map $\phi$.

Therefore, we only need to prove (W1) and (W2) for $\phi$.
Consider a couple $(\rho,\alpha) \in F(X) \times \MWcorr(Y,X)$.
Because $F$ is the $t$-sheaf associated with $P$,
there exists a $t$-cover $p:W \rightarrow X$ 
and a section $\hat \rho \in P(W)$
such that $p^*(\rho)=\tau_W(\hat \rho)$.
According to part~\ref{item:cover-corr} of Corollary~\ref{cor:corr_main},
we get a $t$-cover $q:W' \rightarrow Y$
and a correspondence $\hat \alpha \in \MWcorr(W',W)$
making the diagram \eqref{eq:lift_corr} commutative.
As $\phi$ is a natural transformation, we get
\[
q^*\langle \rho,\alpha \rangle_\phi
=\langle \rho,\alpha \circ q\rangle_\phi
=\langle \rho,p \circ \hat \alpha \rangle_\phi
=\langle p^*\rho,\hat \alpha \rangle_\phi
=\langle \tau_W(\hat \rho),\hat \alpha \rangle_\phi
=\langle \hat \rho,\hat \alpha \rangle_\psi.
\]
Because $q^*:F(X) \rightarrow F(W)$ is injective,
we deduce easily from this principle the properties (W1) and (W2)
for $\phi$ from their analog properties for $\psi$. 
\end{proof}

\begin{prop}\ 
\label{prop:exist_associated-W-t-sheaf}%
\begin{enumerate}
\item The obvious forgetful functor $\tfO:\shMWtkR \rightarrow \pshMWkR$ admits
 a left adjoint $\tilde a$ such that the following diagram commutes:
\[
\begin{tikzcd}
[row sep=12pt,column sep=16pt]
\pshkR \ar[d,"a"'] & \pshMWkR \ar[d,"{\tilde a}"] \ar[l,"{\tilde \gamma_*}"'] \\
\shtkR & \shMWtkR \ar[l,"{\tilde \gamma_*}"']
\end{tikzcd}
\]
\index[notation]{at@$\tilde a$}%
where $a$ is the usual $t$-sheafification functor with
 respect to the smooth site.
\item 
\label{item:shgrothab}%
The category $\shMWtkR$ is a Grothendieck abelian category
 and the functor $\tilde a$ is exact.
\item
\label{item:gamma-has-left-adjoint}
The functor $\tilde \gamma_*$, appearing in the lower line
 of the preceding diagram, admits a left adjoint $\tilde \gamma^*$,
 and commutes with every limit and colimit.
\index[notation]{gasta@$\tilde \gamma_*$,$\tilde \gamma^*$}%
\end{enumerate}
\end{prop}

\begin{proof}
The first point follows directly from the previous lemma:
indeed,  with the notation of this lemma, we can put:
$\tilde a(P)=\tilde F$.

For point \ref{item:shgrothab}, we first remark that the functor $\tilde a$,
being a left adjoint, commutes with every colimit. Moreover,
the functor $a$ is exact and $\tilde\gamma_*:\pshMWkR\to \pshkR$ is also exact (Paragraph \ref{num:rightadjoint}).
Therefore, $\tilde a$ is exact because of the previous commutative square
and the fact that $\tilde \gamma_*$ is faithful. Then, we easily deduce that $\shMWtkR$ is a Grothendieck abelian category from the fact that $\pshMWkR$ is such a category.

The existence of the left adjoint $\tilde \gamma^*$ follows
formally. Thus $\tilde \gamma_*$ commutes with every limit.
Because $\tilde \gamma_*$ is exact and commutes with
arbitrary coproducts, we deduce that it commutes with arbitrary
colimits, therefore proving part \ref{item:gamma-has-left-adjoint}.
\end{proof}

\begin{rem}
The left adjoint $\tilde \gamma^*$ of $\tilde \gamma_*:\shMWtkR\to \shtkR$ can be computed as the composite
\[
\shtkR\stackrel{\fO}\to \pshkR\stackrel{\tilde \gamma^*}\to \pshMWkR\stackrel{\tilde a}\to \shMWtkR.
\] 
One can also observe that, according to part~\ref{item:shgrothab}, a family of generators of
 the Grothendieck abelian category $\shMWtkR$ is obtained by applying
 the functor $\tilde a$ to a family of generators of $\pshMWkR$. 
\end{rem}
 
\begin{dfn}\label{df:gen_sht}
Given any smooth scheme $X$, we put $\MWrepRt(X)=\tilde a\left(\MWprep(X)\right)$. 
\index[notation]{ztx@$\MWrepRt(X)$}%
\end{dfn}

In particular, for a smooth scheme $X$, $\MWrepRt(X)$ is the $t$-sheaf associated with the presheaf $\MWprep(X)$,
equipped with its canonical action of MW-correspondences (Lemma \ref{lm:exists-Wtr}).
The corresponding family, for all smooth schemes $X$, generates the abelian category $\shMWtkR$.

\begin{num}\label{num:sht_monoidal}
One deduces from the monoidal structure on $\cork$ a monoidal
 structure on $\shMWtkR$ whose tensor product $\otr$ is uniquely characterized
 by the property that for any smooth schemes $X$ and $Y$:
\begin{equation}\label{eq:sht_monoidal}
\MWrepRt(X) \otr \MWrepRt(Y)=\MWrepRt(X\times Y).
\end{equation}
Explicitly, the tensor product of any two sheaves $F,G\in \shMWtkR$ is obtained by applying $\tilde a$ to the presheaf tensor product $F\otimes G$ mentioned after Definition \ref{def:representablepsht}.
In particular, the bifunctor $\otr$ commutes with colimits
 and therefore, as the abelian category $\shMWtkR$ is a Grothendieck abelian
 category, the monoidal category $\shMWtkR$ is closed. The internal Hom
 functor is characterized by the property that for any MW-$t$-sheaf $F$
 and any smooth scheme $X$,
\[
\uHom\big(\MWrepRt(X),F\big)=F(X \times -).
\]
\end{num}

As a corollary of Proposition \ref{prop:exist_associated-W-t-sheaf},
 we obtain functors between the category of sheaves we have
 considered so far.
\begin{coro}\ 
\label{cor:adjunctions_corr}%
\begin{enumerate}
\item
\label{item:exists-diagram-monoidal}%
There exists a commutative diagram of symmetric monoidal functors
\[
\begin{tikzcd}
[row sep=26pt,column sep=26pt]
\pshkR \ar[d,"{\tilde \gamma^*}"'] \ar[r,"{a_\nis}"]
 & \shxkR{\nis} \ar[d,"{\tilde \gamma^*_\nis}"] \ar[r,"{a_\et}"]
  & \shxkR{\et} \ar[d,"{\tilde \gamma^*_\et}"] \\
\pshMWkR \ar[d,"{\pi^*}"'] \ar[r,"{\tilde a_\nis}"]
 & \shMWxkR{\nis} \ar[d,"{\pi^*_\nis}"] \ar[r,"{\tilde a_\et}"]
 & \shMWxkR{\et} \ar[d,"{\pi^*_\et}"] \\
\pshVkR \ar[r,"{a^{tr}_\nis}"] & \shVxkR{\nis} \ar[r,"{a^{tr}_\et}"]
 & \shVxkR{\et}
\end{tikzcd}
\]
which are all left adjoints of an obvious forgetful functor. Each of these functors
 respects the canonical family of abelian generators.
\item 
\label{item:right-adjoint-pi-faithful}%
Let $t=\nis, \et$. 
Then the right adjoint functor $\pi_*^t:\shMWtkR \rightarrow \shtkR$ is faithful.
 When $2$ is invertible in $R$, 
 the right adjoint functor $\pi_*^t:\shVtkR \rightarrow \shMWtkR$ is fully faithful.
\item
\label{item:adjunction-pi-equivalence}%
If $k$ is of characteristic different from $2$, the adjunction 
\[
\pi^*_\et:\shMWxkR{\et} \leftrightarrows \shVxkR{\et}:\pi_*^\et
\]
is an equivalence of categories.
\end{enumerate}
\end{coro}
\begin{proof}
Property \ref{item:exists-diagram-monoidal} is a formal consequence of Proposition 
 \ref{prop:exist_associated-W-t-sheaf} and its analog for sheaves with transfers.

Property \ref{item:right-adjoint-pi-faithful} follows from the commutativity of the diagram in \ref{item:exists-diagram-monoidal},
 which induces an obvious commutative diagram for the right adjoint functors,
 the fact that the forgetful functor from sheaves to presheaves is always
 fully faithful and Lemma \ref{lm:compare_transfers}.

For \ref{item:adjunction-pi-equivalence}, it is sufficient to prove that the left adjoint $\pi^*_\et$
 is fully faithful. Indeed, if this is the case we also obtain that it is essentially surjective
 as it commutes with colimits and is surjective on generators of $\shVxkR{\et}$
 as we have: $\pi^*_\et(\MWrepR_\et(X))=\VrepR_\et(X)$.
 Therefore we prove that for any étale MW-sheaf $F$, the counit map:
\[
F \rightarrow \pi_*^\et\pi^*_\et(F)
\]
is an isomorphism.
Recall that for any étale sheaf with transfers $G$, $\pi_*^\et(G)=G \circ \pi$.
In particular, one immediately deduces that $\pi_*^\et$ commutes with coproducts and with taking fibers for the étale topology. This implies that it commutes with colimits.
As $F$ is a colimit of representable MW-sheaves, one can assume that $F=\MWrepR_\et(X)$ for a smooth $k$-scheme $X$.
Therefore, we are reduced to prove that the canonical map of étale sheaves (where we conventionally forget MW-transfers and transfers)
\[
\MWrepR_\et(X) \rightarrow \VrepR_\et(X)
\]
is an isomorphism. This can be checked on stalks and by definition of $R$-linear finite MW-correspondences (resp.\ correspondences) \eqref{eq:R-lin_MWcorr}, it is sufficient to treat the case $R=\ZZ$.
So, let $Y$ be a smooth $k$-scheme, and $y:\Spec(F) \rightarrow Y$ be a geometric point, such that $F$ is algebraically closed.
%
We let $A$ be the strict henselisation of $Y$ at the point $y$.
We need to show that the map induced by \eqref{eq:from_smc2smcV}
\[
\pi_{A,X}:\cork(\Spec(A),X) \rightarrow \corVk(\Spec(A),X)
\]
obtained after taking the obvious filtered inductive limit over étale neighborhoods of $(Y,y)$ on both sides, is an isomorphism.
Let $Z \subset X \otimes_k A$ be an admissible subset (\chfinitecw, Definition~\ref{dfn:admissible_subset}),
i.e.\ a closed integral subset whose irreducible 
components $Z_i$ are finite and surjective over $\Spec(A)$.
 As $A$ is henselian, $Z$ decomposes as a direct sum $\sqcup_{i} Z_i$ where $Z=\Spec(B_i)$ with $B_i$ a local henselian ring.
 Because $Z/\Spec(A)$ is finite, the residue field of $B_i$ is isomorphic to $F$. 
Any finite MW-correspondence $\alpha \in \cork(\Spec(A),X)$
 with support in $Z$ can then be written as
\[
\alpha=\oplus_{i\in I}\chst{\mathrm{dim}(X)}{Z_i}{X \otimes_k A}{\omega_{X/k}}
\]
Let $\tilde Z_i$ be the normalization of $Z_i$ in $k(Z_i)$. For the same reason as above, $\tilde Z_i$ is the spectrum of a strictly henselian local ring with residue field $F$. The $\tilde Z_i$ being normal, we may define $\chht 0{\tilde Z_i}L$ as usual for any line bundle $L/{\tilde Z_i}$. The scheme being local, we may even choose a trivialization of $L$ and the push-forward homomorphism induced by the finite morphism $\tilde Z_i\to X \otimes_k A$ reads as 
\[
\chhtnotw 0{\tilde Z_i}\to \chst{\mathrm{dim}(X)}{Z_i}{X \otimes_k A}{\omega_{X/k}}
\]  
The left-hand side is a subgroup of $\GW(k(Z_i))$ which is in fact isomorphic to $\ZZ$ via the rank homomorphism by the next lemma. It follows that the unique homomorphism $\ZZ\to \GW(k(Z_i))$ is an isomorphism, and, since every element in the image is automatically unramified, that $\chhtnotw 0{\tilde Z_i}=\ZZ$. As $\ZZ\to \GW(k(Z_i))$, we also deduce that the map
\[
\chhtnotw 0{\tilde Z_i}\to \chst{\mathrm{dim}(X)}{Z_i}{X \otimes_k A}{\omega_{X/k}}
\]  
is an isomorphism, and consequently that 
\[
\alpha=\oplus_{i\in I}\chst{\mathrm{dim}(X)}{Z_i}{X \otimes_k A}{\omega_{X/k}}
\]
is also an isomorphism.

\end{proof}

\begin{lem}
\label{lemm:GW_etale_trivial}%
Let $A$ be an integral henselian local ring whose residue field $F$ is algebraically closed of characteristic different from $2$,
and $K$ be its fraction field.
Then, the rank map $\GW(K) \rightarrow \ZZ$ is an isomorphism.%
\footnote{The same holds for $A$ itself.} 
\end{lem}

\begin{proof}
According to \cite[chap.~IV, Lem.~1.1]{MilHus_73}, the ring $\GW(K)$ is additively generated
 by symbols $\langle u \rangle$ for $u \in K^\times$.
 We write $u=\frac a b$ for $a,b \in A-\{0\}$.
 The monic polynomial $(t^2-a)$ in $A[t]$ admits a simple root in the (algebraically closed) residue field $F$
 --- which is of characteristic different from $2$ ---
 so it admits a solution in $A$ as $A$ is henselian. In particular, $a$ is a square and similarly, $b$ is a square.
 This implies $u=v^2$ so that $\langle u \rangle=1$. So $\GW(K)$ is additively generated by $1$, which forces
 the rank map to be an isomorphism as required.
\end{proof}
 
\section{Framed correspondences}\label{sec:framed}

\subsection{Definitions and basic properties}

The aim of this section is to make a link between the category of linear framed correspondences (after Garkusha-Panin-Voevodsky) and the category of MW-presheaves. We start with a quick reminder on framed correspondences following \cite{Garkusha14_b}.
\index{framed correspondence}%

\begin{dfn}
Let $U$ be a smooth $k$-scheme and $Z\subset U$ be a closed subset of codimension $n$. A set of regular functions $\phi_1,\ldots,\phi_n\in k[U]$ is called a framing of $Z$ in $U$ if $Z$ coincides with the closed subset $\phi_1=\ldots=\phi_n=0$. 
\end{dfn} 

\begin{dfn}
Let $X$ and $Y$ be smooth $k$-schemes, and let $n\in\NN$ be an integer. An explicit framed correspondence $c=(U,\phi,f)$ of level $n$ from $X$ to $Y$ consists of the following data:
\begin{enumerate}
\item A closed subset $Z\subset \AAA^n_X$ which is finite over $X$ (here, $Z$ is endowed with its reduced structure). 
\item An étale neighborhood $\alpha:U\to \AAA^n_X$ of $Z$.
\item A framing $\phi=(\phi_1,\ldots,\phi_n)$ of $Z$ in $U$.
\item A morphism $f:U\to Y$.
\end{enumerate} 
The closed subset $Z$ is called the \emph{support} of the explicit framed correspondence $c=(U,\phi,f)$.
\end{dfn}

\begin{rem}
One could give an alternative approach to the above definition. A framed correspondence $(U,\phi,f)$ corresponds to a pair of morphisms $\phi:U\to \AAA^n_k$ and $f:U\to Y$ yielding a unique morphism $\varphi:U\to \AAA^n_Y$. The closed subset $Z\subset U$ corresponds to the preimage of $Y\times \{0\}\subset Y\times \AAA^n_k=\AAA^n_Y$. This correspondence is unique.
\end{rem}

\begin{rem}
Note that $Z$ is not supposed to map surjectively onto a component of $X$. For instance $Z=\emptyset$ is an explicit framed correspondence of level $n$, denoted by $0_n$. If $Z$ is non-empty, then an easy dimension count shows that $Z\subset \AAA^n_X\to X$ is indeed surjective onto a component of $X$.
\end{rem}

\begin{rem}
Suppose that $X$ is a smooth connected scheme. By definition, an explicit framed correspondence of level $n=0$ is either a morphism of schemes $f:X\to Y$ or $0_0$.
\end{rem}

\begin{dfn}
Let $c=(U,\phi,f)$ and $c^\prime=(U^\prime,\phi^\prime,f^\prime)$ be two explicit framed correspondences of level $n\geq 0$. Then, $c$ and $c^\prime$ are said to be \emph{equivalent} if they have the same support and  there exists an open neighborhood $V$ of $Z$ in $U\times_{\AAA^n_X} U^\prime$ such that the diagrams
\[
\begin{tikzcd}
U \times_{\AAA^n_X} U^\prime \ar[r] \ar[d] & U^\prime \ar[d,"{f^\prime}"] \\
U \ar[r,"f"'] & Y
\end{tikzcd}
\]
and
\[
\begin{tikzcd}
U \times_{\AAA^n_X} U^\prime \ar[r] \ar[d] & U^\prime \ar[d,"{\phi^\prime}"] \\
U \ar[r,"{\phi}"'] & \AAA^n_k
\end{tikzcd}
\]
are both commutative when restricted to $V$. A \emph{framed} correspondence of level $n$ is an equivalence class of explicit framed correspondences of level $n$.
\end{dfn}

\begin{dfn}
Let $X$ and $Y$ be smooth schemes and let $n\in \NN$. We denote by $\Fr_n(X,Y)$
\index[notation]{frnxy@$\Fr_n(X,Y)$}%
the set of framed correspondences of level $n$ from $X$ to $Y$ and by $\Fr_*(X,Y)$ the set $\sqcup_n \Fr_n(X,Y)$. Together with the composition of framed correspondences described in \cite{Garkusha14_b}*{\S 2}, this defines a category whose objects are smooth schemes and morphisms are $\Fr_*(-,-)$. We denote this category by $\Fr_*(k)$ and refer to it as the \emph{category of framed correspondences}.
\index{framed correspondences!category of}%
\index[notation]{frk@$\Fr_*(k)$}%
\end{dfn}

We now pass to the linear version of the above category following \cite{Garkusha14_b}*{\S 7}, starting with the following observation. 
Let $X$ and $Y$ be smooth schemes, and let $c_{Z}=(U,\phi,f)$ be an explicit framed correspondence of level $n$ from $X$ to $Y$ with support $Z$ of the form $Z=Z_1\sqcup Z_2$. 
Let $U_1=U\setminus Z_2$ and $U_2=U\setminus Z_1$. For $i=1,2$, we get étale morphisms $\alpha_i:U_i\to X$ and morphisms $\phi_i:U_i\to \AAA^n_k$, $f_i:U_i\to Y$ by precomposing the morphisms $\alpha,\phi$ and $f$ with the open immersion $U_i\to U$. 
Note that $U_i$ is an étale neighborhood of $Z_i$ for $i=1,2$ and that $c_{Z_i}=(U_i,\phi_i,f_i)$ are explicit framed correspondences of level $n$ from $X$ to $Y$ with support $Z_i$.  

\begin{dfn}\label{def:linear}
Let $X$ and $Y$ be smooth schemes and let $n\in\NN$. Let 
\[
\ZFr_n(X,Y)=\ZZ \Fr_n(X,Y)/H
\]
\index[notation]{zfrnxy@$\ZFr_n(X,Y)$}%
\index{framed correspondences!category of linear}%
where $H$ is the subgroup generated by elements of the form $c_Z-c_{Z_1}-c_{Z_2}$ where $Z=Z_1\sqcup Z_2$ is as above and $\ZZ \Fr_n(X,Y)$ is the free abelian group on $\Fr_n(X,Y)$. The category $\ZFr_*(k)$ of \emph{linear framed correspondences} is the category whose objects are smooth schemes and whose morphisms are 
\[
\Hom_{\ZFr_*(k)}(X,Y)=\bigoplus_{n\in \NN} \ZFr_n(X,Y).
\]
\end{dfn}

\begin{rem}
Note that there is an obvious functor $\iota:\Fr_*(k)\to \ZFr_*(k)$ with $\iota(0_n)=0$ for any $n\in\NN$.
\end{rem}

The stage being set, we now compare the category of finite MW-correspondences with the above categories.

Let $U$ be a smooth $k$-scheme and let $\phi:U\to \AAA^n_k$ be a morphism corresponding to (nonzero) global sections $\phi_i\in \OO(U)$. Each section $\phi_i$ can be seen as an element of $k(U)^\times$ and defines then an element of $\KMW_1(k(U))$. Let $\vert \phi_i\vert$ be the support of $f_i$, i.e.\ its vanishing locus, and let $Z= \vert \phi_1\vert \cap \ldots \cap \vert \phi_n\vert $. Consider the residue map
\[
d:\KMW_1(k(U))\to \bigoplus_{x\in U^{(1)}} \KMW_0(k(x),\omega_x).
\]
Then $d(\phi_i)$ defines an element supported on $\vert \phi_i\vert$. As it is a boundary, it defines a cycle $Z(\phi_i)\in H^1_{\vert \phi_i\vert}(U,\sKMW_1)$. Now, we can consider the intersection product
\[
\H^1_{\vert \phi_1\vert}(U,\sKMW_1)\times \ldots \times \H^1_{\vert \phi_n\vert}(U,\sKMW_1)\to \H^n_{Z}(U,\sKMW_n)
\]
to get an element $Z(\phi_1)\cdot \ldots \cdot Z(\phi_n)$ that we denote by $Z(\phi)$.

\begin{lem}\label{lem:functor}
Any explicit framed correspondence $c=(U,\phi,f)$ induces a finite MW-correspondence $\alpha(c)$ from $X$ to $Y$. Moreover, two equivalent explicit framed correspondences $c$ and $c^\prime$ induce the same finite MW-correspondence.
\end{lem}

\begin{proof}
Let us start with the first assertion. If $Z$ is empty, its image is defined to be zero. If $c$ is of level $0$, then it corresponds to a morphism of schemes, and we use the functor $\smk\to \cork $ to define the image of $c$. We thus suppose that $Z$ is non-empty (thus finite and surjective on some components of $X$) of level $n\geq 1$.
Consider the following diagram
\[
\begin{tikzcd}
 & U\ar[d,"\alpha"'] \ar[r,"{(\phi,f)}"] & \AAA^n_Y \\
Z \ar[r] \ar[ru] & \AAA^n_X \ar[d,"{p_X}"'] & \\
 & X & 
\end{tikzcd}
\]
defining an explicit framed correspondence $(U,\phi,f)$ of level $n$. The framing $\phi$ defines an element $Z(\phi)\in \H^n_{Z}(U,\sKMW_n)$ as explained above. Now, $\alpha$ is étale and therefore induces an isomorphism  $\alpha^*\omega_{\AAA^n_X}\simeq \omega_U$. Choosing the usual orientation for $\AAA^n_k$, we get an isomorphism $\OO_{\AAA^n_X}\simeq \omega_{\AAA^n_X}\otimes (p_X)^*\omega_X^\vee$ and therefore an isomorphism
\[
\OO_U\simeq \alpha^*(\OO_{\AAA^n_X})\simeq \alpha^*(\omega_{\AAA^n_X}\otimes p_X^*\omega_X^\vee)\simeq \omega_U\otimes (p_X\alpha)^* \omega_X^\vee.
\]
We can then see $Z(\phi)$ as an element of the group $\H^n_{Z}(U,\sKMW_n,\omega_U\otimes (p_X\alpha)^* \omega_X^\vee)$. Consider next the map $(p_X\alpha,f):U\to X\times Y$ and the image $T$ of $Z$ under the map of underlying topological spaces. It follows from \cite{Mazza06_b}*{Lemma 1.4} that $T$ is closed, finite and surjective over (some components of) $X$. Moreover, the morphism $Z\to T$ is finite and it follows that we have a push-forward homomorphism
\[
(p_X\alpha,f)_*:\H^n_{Z}(U,\sKMW_n,\omega_U\otimes (p_X\alpha)^* \omega_X^\vee)\to \H^n_T(X\times Y,\sKMW_n,\omega_{X\times Y/X})
\] 
yielding, together with the canonical isomorphism $\omega_{X\times Y/X}\simeq \omega_Y$, a finite Chow-Witt correspondence $\alpha(c):=(p_X\alpha,f)_*(Z(\phi))$ between $X$ and $Y$.

Suppose next that $c=(U,\phi,f)$ and $c^\prime=(U^\prime,\phi^\prime,f^\prime)$ are two equivalent explicit framed correspondences of level $n$. Following the above construction, we obtain two cocycles $\tilde\alpha(c)\in \H^n_{Z}(U,\sKMW_n,\omega_U\otimes (p_X\alpha)^* \omega_X^\vee)$ and $\tilde\alpha(c^\prime)\in \H^n_{Z}(U^\prime,\sKMW_n,\omega_{U^\prime}\otimes (p_X\alpha^\prime)^* \omega_X^\vee)$. Now, the pull-backs along the projections 
\[
\begin{tikzcd}
U\times_{\AAA^n_X} U^\prime \ar[r,"p_2"] \ar[d,"p_1"'] & U^\prime \\
U & 
\end{tikzcd}
\]
yield homomorphisms
\[
p_1^*:\H^n_{Z}(U,\sKMW_n,\omega_U\otimes (p_X\alpha)^* \omega_X^\vee)\simeq \H^n_{p_1^{-1}(Z)}(U\times_{\AAA^n_X} U^\prime,\sKMW_n,\omega_{U\times_{\AAA^n_X} U^\prime}\otimes (p_X\alpha p_1)^* \omega_X^\vee)
\] 
and
\[
p_2^*:\H^n_{Z}(U^\prime,\sKMW_n,\omega_{U^\prime}\otimes (p_X\alpha^\prime)^* \omega_X^\vee)\simeq \H^n_{p_2^{-1}(Z)}(U\times_{\AAA^n_X} U^\prime,\sKMW_n,\omega_{U\times_{\AAA^n_X} U^\prime}\otimes (p_X\alpha p_2)^* \omega_X^\vee),
\] 
while the pull-back along the open immersion $i:V\to U\times_{\AAA^n_X} U^\prime$ induces homomorphisms 
\[
i^*:\H^n_{p_1^{-1}(Z)}(U\times_{\AAA^n_X} U^\prime,\sKMW_n,\omega_{U\times_{\AAA^n_X} U^\prime}\otimes (p_X\alpha p_1)^* \omega_X^\vee)\simeq \H^n_Z(V,\sKMW_n, \omega_V\otimes (p_X\alpha p_1i)^* \omega_X^\vee)
\]
and
\[
i^*:\H^n_{p_2^{-1}(Z)}(U\times_{\AAA^n_X} U^\prime,\sKMW_n,\omega_{U\times_{\AAA^n_X} U^\prime}\otimes (p_X\alpha p_2)^* \omega_X^\vee)\simeq \H^n_Z(V,\sKMW_n, \omega_V\otimes (p_X\alpha p_2i)^* \omega_X^\vee).
\]
Note that $p_X\alpha p_2=p_X\alpha p_1$ and that $i^*p_1^*(\tilde\alpha(c))=i^*p_2^*(\tilde\alpha(c^\prime))$ by construction. Pushing forward along $V\to U\times_{\AAA^n_X} U^\prime\to U\to X\times Y$, we get the result.
\end{proof}

\begin{exm}\label{ex:stable}
Let $X$ be a smooth $k$-scheme. Consider the explicit framed correspondence $\sigma_X$ of level $1$ from $X$ to $X$ given by $(\Aone_X,q,p_X)$ where $q:\Aone_X=\Aone\times X\to \Aone$ is the projection to the first factor and $p_X:\Aone_X\to X$ is the projection to the second factor. We claim that $\alpha(\sigma_X)=\id\in \cork (X,X)$. To see this, observe that we have a commutative diagram
\[
\begin{tikzcd}
\Aone_X \ar[r,"{(p_X,p_X)}"] \ar[d,"p_X"'] & X\times X \\
X \ar[ru,"{\triangle}"'] & 
\end{tikzcd}
\]
where $\triangle$ is the diagonal map. Following the process of the above lemma, we start by observing that $Z(q)\in \H^1_X(\Aone_X,\sKMW_1)$ is the class of $\langle 1\rangle\otimes \overline t\in \sKMW_0(k(X),(\mathfrak m/\mathfrak m^2)^*)$ where $\mathfrak m$ is the maximal ideal corresponding to $X$ in the appropriate local ring and $t$ is a coordinate of $\Aone$. Now, we choose the canonical orientation of $\Aone$ and the class of $Z(q)$ corresponds then to the class of $\langle 1\rangle \in \sKMW_0(k(X))$ in $\H^1_X(\Aone_X,\sKMW_1,\omega _{(\Aone_X/X)})$. Its push-forward under 
\[
(p_X)_*:\H^1_X(\Aone_X,\sKMW_1,\omega _{(\Aone_X/X)})\to \H^0(X,\sKMW_0)
\]
is the class of $\langle 1\rangle$ and the claim follows from the above commutative diagram yielding $(p_X,p_X)_*=\triangle_*(p_X)_*$ and the definition of the identity in $\cork (X,X)$.
\end{exm}

\begin{prop}\label{prop:bunchoffunctors}
The assignment $c=(U,\phi,f)\mapsto \alpha(c)$ made explicit in Lemma \ref{lem:functor} defines functors $\alpha:\Fr_*(k)\to \cork $ and $\alpha^\prime:\ZFr_*(k)\to \cork$ such that we have a commutative diagram of functors
\[
\begin{tikzcd}
 &  \Fr_*(k) \ar[dd,"{\alpha}"'] \ar[rd,"{\iota}"] & \\ 
\smk \ar[ru] \ar[rd,"{\tilde\gamma}"'] &  & \ZFr_*(k) \ar[ld,"{\alpha^\prime}"] \\
 & \cork\ . & 
\end{tikzcd}
\]
\end{prop}

\begin{proof}
For any smooth schemes $X,Y$ and any integer $n\geq 0$, we have a well-defined map $\alpha:\Fr_n(X,Y)\to \cork(X,Y)$ and therefore a well-defined map $\ZFr_n(X,Y)\to \cork(X,Y)$. Let $c=(U,\phi,f)$ be an explicit framed correspondence of level $n$ with support $Z$ of the form $Z=Z_1\sqcup Z_2$. Let $c_i=(U_i,\phi_i,f_i)$ be the explicit framed correspondences with support $Z_i$ obtained as in Definition \ref{def:linear}. By construction, we get $\alpha(c)=\alpha(c_1)+\alpha(c_2)$ and it follows that $\alpha:\Fr_n(X,Y)\to \cork(X,Y)$ induces a homomorphism $\alpha^\prime:\ZFr_n(X,Y)\to \cork(X,Y)$.

 It remains then to show that the functors $\alpha:\Fr_k\to \cork$ and $\alpha^\prime:\ZFr_*(k)\to \cork$ are well-defined, which amounts to prove that the respective compositions are preserved. Suppose then that $(U,\phi,f)$ is an explicit framed correspondence of level $n$ between $X$ and $Y$, and that $(V,\psi,g)$ is an explicit framed correspondence of level $m$ between $Y$ and $Z$. We use the diagram
\begin{equation}
\label{eq:diag1}%
\begin{tikzcd}
W \ar[r,"{pr_V}"] \ar[d] \ar[dd,bend right=12ex,"{pr_U}"'] & V \ar[d,"\beta"'] \ar[r,"{\psi}"] \ar[rrd,bend right=2ex,"g"'] & \AAA^m & \\
U \times \AAA^m \ar[r,"{f\times \id}"'] \ar[d] & Y \times \AAA^m \ar[d,"{p_Y}"'] & & Z \\
U \ar[r,"f"'] \ar[d,"{p_X\alpha}"'] \ar[rrd,bend right=2ex,"\phi"'] & Y & &  \\
X & & \AAA^n & 
\end{tikzcd}
\end{equation} 
in which the squares are all cartesian. The composition of $(U,\phi,f)$ with $(V,\psi,g)$ is given by $(W,(\phi\circ pr_U,\psi\circ pr_V),g\circ pr_V)$. 

On the other hand, the morphisms $(p_X\alpha,f)\circ pr_U:W\to X\times Y$ and $(p_Y\beta,g)\circ pr_V:W\to Y\times Z$ yield a morphism $\rho:W\to X\times Y\times Z$ and then a diagram
\begin{equation}
\label{eq:diag3}%
\begin{tikzcd}
 & W \ar[r,"{pr_V}"] \ar[d,"\rho"'] & V \ar[d,"{(p_Y\beta,g)}"] \\ 
W \ar[r,"\rho"] \ar[d,"{pr_U}"'] & X \times Y \times Z \ar[d,"{p_{X\times Y}}"'] \ar[r,"{p_{Y\times Z}}"] & Y\times Z \ar[d] \\
U \ar[r,"{(p_X\alpha,f)}"'] & X \times Y \ar[r] & Y
\end{tikzcd}
\end{equation}
in which all squares are cartesian. By base change (\chfinitecw, Proposition~\ref{prop:basechange}, Remark~\ref{rem:choice}), we have $(p_{X\times Y})^*(p_X\alpha,f)_*=\rho_*(pr_U)^*$ and $(p_{Y\times Z})^*(p_Y\beta,g)_*=\rho_*(pr_V)^*$.  By definition of the pull-back and the product, we have $(pr_U)^*(Z(\phi))=Z(\phi\circ pr_U)$ and $(pr_V)^*(Z(\psi))=Z(\psi\circ pr_V)$. It follows that 
\[
Z(\phi\circ pr_U,\psi\circ pr_V)=(pr_U)^*(Z(\phi))\cdot (pr_V)^*(Z(\psi)).
\]
Finally, observe that there is a commutative diagram
\[
\begin{tikzcd}
W \ar[d,equal] \ar[rr,"{g\circ pr_V}"] & & Z \\
W\ar[r,"\rho"] \ar[d,equal] & X \times Y \times Z \ar[r,"{p_{X\times Z}}"] & X \times Z \ar[d] \ar[u] \\
W \ar[rr,"{p_X\circ \alpha\circ pr_U}"'] & & X.
\end{tikzcd}
\]
Using these ingredients, we see that the composition is preserved.
\end{proof}

\begin{rem}
Note that the functor $\alpha^\prime:\ZFr_*(k)\to \cork$ is additive. It follows from Example \ref{ex:stable} that it is not faithful.
\end{rem}

\subsection{Presheaves}

Let $X$ be a smooth scheme. Recall from Example \ref{ex:stable} that we have for any smooth scheme $X$ an explicit framed correspondence $\sigma_X$ of level $1$ given by the triple $(\Aone_X,q,p_X)$ where $q$ and $p_X$ are respectively the projections onto $\Aone_k$ and $X$. The following definition can be found in \cite{Garkusha15_b}*{\S 1}.

\begin{dfn}
Let $R$ be a ring. A presheaf of $R$-modules $F$ on $\ZFr_*(k)$ is \emph{quasi-stable} if for any smooth scheme $X$, the pull-back map $F(\sigma_X):F(X)\to F(X)$ is an isomorphism. A quasi-stable presheaf is \emph{stable} if $F(\sigma_X):F(X)\to F(X)$ is the identity map for any $X$. We denote by $\pshfrkR$ the category of presheaves of $R$-modules on $\ZFr_*(k)$, by $\QpshfrkR$ the category of quasi-stable presheaves of $R$-modules on $\ZFr_*(k)$ and by $\SpshfrkR$ the category of stable presheaves of $R$-modules.
\end{dfn}

Now, the functor $\alpha^\prime:\ZFr_*(k)\to \cork$ induces a functor $\pshMWkR\to \pshfrkR$. By Example \ref{ex:stable}, this functor induces a functor
\[
(\alpha^\prime)^*:\pshMWkR\to \SpshfrkR.
\]
Recall next that a presheaf $F$ on $\smk$ is $\Aone$-invariant if the map $F(X)\to F(X\times \Aone)$ induced by the projection $X\times \Aone\to X$ is an isomorphism for any smooth scheme $X$. A Nisnevich sheaf of abelian groups $F$ is strictly $\Aone$-invariant if the homomorphisms $\H^i_{\nis}(X,F)\to \H^i_{\nis}(X\times \Aone,F)$ induced by the projection are isomorphisms for $i\geq 0$.

We can now state the main theorem of \cite{Garkusha15_b}.

\begin{thm}
\label{thm:A1local_framedPSh}%
Let $F$ be an $\Aone$-invariant quasi-stable $\ZFr_*(k)$-presheaf of $R$-modules, where $k$ is of characteristic different from $2$.
\begin{enumerate}
\item 
\label{item:A-one-invariant-quasi-stable}%
If the base field $k$ is infinite, 
the associated Nisnevich sheaf $F_{\nis}$ of $R$-modules is $\Aone$-invariant
and quasi-stable.
\item
\label{item:in-addition-Nisnevich-sheaf}%
Assume the base field $k$ is infinite (and perfect)
and the presheaf of $R$-modules $F$ is in addition a Nisnevich sheaf.
Then $F$ is strictly $\Aone$-invariant, as a Nisnevich sheaf of $R$-modules.
\end{enumerate}
\end{thm}
\begin{proof}
These results are proved in \cite{Garkusha15_b} in the case where $R=\ZZ$ but it implies the case of an arbitrary ring of coefficients, simply by forgetting the scalars. Let us give the details below.

Let $\varphi:\ZZ \rightarrow R$ be the unique morphism of rings attached
with the ring $R$. We will consider the restriction of scalars functor
\[
\varphi_*:R\text{-mod} \rightarrow \ZZ\text{-mod}.
\]
As this functor admits both a left and a right adjoint (extension
of scalars and coinduced module), it commutes with all limits and colimits.
Besides, it is conservative.

For any category $\mathscr S$, one extends the functor $\varphi_*$
to presheaves on $\mathscr S$ by applying it term-wise:
\[
\hat \varphi_*:\psh(\mathscr S,R) \rightarrow \psh(\mathscr S,\ZZ).
\]
By construction, for any object $X$ of $\mathscr S$ and any presheaf $F$ of
 $R$-modules, we have the relation:
\begin{equation}\label{eq:Garkusha1}
\Gamma(X,\hat \varphi_*F)=\varphi_*(\Gamma(X,F)).
\end{equation}
Suppose now that $\mathscr S$ is endowed with a Grothendieck topology.
Then, as $\varphi_*$ is exact and commutes with products,
the functor $\hat \varphi_*$ respects sheaves so that we get an induced functor:
\[
\tilde \varphi_*:\sh(\mathscr S,R) \rightarrow \sh(\mathscr S,\ZZ).
\]
Then, using again the fact $\varphi_*$ commutes with colimits,
and the classical formula defining the associated sheaf functor
$a:\psh(\mathscr S,?) \rightarrow \sh(\mathscr S,?)$,
we get, for any presheaf $F$ of $R$-modules, a canonical isomorphism:
\begin{equation}\label{eq:Garkusha2}
a\big(\hat \varphi_*(F)\big) \simeq \tilde \varphi_* a(F).
\end{equation}
The fact $\varphi_*$ is conservative together with
relations \eqref{eq:Garkusha1} and \eqref{eq:Garkusha2} are sufficient to prove
assertion \ref{item:A-one-invariant-quasi-stable}. Indeed, these facts 
imply that it is sufficient to check the $\Aone$-invariance and
quasi-stability of the presheaf $\hat \varphi_*(F)$ to conclude.

Let us come back to the abstract situation to prove the remaining relation.
Recall that one can compute cohomology of an object $X$ in $\mathscr S$ with
coefficients in a sheaf $F$ by considering the colimit of the \v Cech cohomology
of the various hypercovers of $X$. Thus, relation \eqref{eq:Garkusha1} and the
fact that $\varphi_*$ commutes with colimits and products implies that,
for any integer $n$, we get the following isomorphism of abelian groups,
natural in $X$:
\begin{equation}\label{eq:Garkusha3}
\H^n\big(X,\tilde \varphi_*(F)\big) \simeq \varphi_*\H^n\big(X,F).
\end{equation}
Therefore to prove assertion \ref{item:in-addition-Nisnevich-sheaf},
 using the latter relation and once again the fact $\varphi_*$
 is conservative, we are reduced to consider the
 sheaf $\tilde \varphi_*(F)$ of abelian groups, which as a presheaf
 is just $\hat \varphi_*(F)$. Using relation \eqref{eq:Garkusha1},
 the latter is $\Aone$-invariant and quasi-stable so that we are indeed
 reduced to the case of abelian groups as expected.
\end{proof}

\begin{rem}
It is worth mentioning that the considerations of the preceding proof
 are part of a standard machinery of changing coefficients for sheaves
 that can be applied in particular in our context.

We leave the exact formulation to the reader, but
 describe it in the general case of a morphism of
 rings $\varphi:R \rightarrow R'$.

The restriction of scalars functor $\varphi_*$ together
 with its left adjoint $\varphi^*$ (extension of scalars)
 and its right adjoint $\varphi_!$ (associated coinduced module),
 can be extended (using the arguments of the preceding proof or similar arguments)
 to the category of MW-presheaves or MW-sheaves as two pairs
 of adjoint functors, written for simplicity here by $(\varphi^*,\varphi_*)$
 and $(\varphi_*,\varphi^!)$. Note that the extended functor $\varphi_*$
 will still be conservative and that the
 functor $\varphi^*$ will be monoidal.

Moreover, using the definitions of the following section, the pair
 of adjoint functors $(\varphi^*,\varphi_*)$ will induce adjoint functors on
 the associated effective and stable $\Aone$-derived categories, such that
 in particular the induced functor $\varphi_*$ is still conservative
 on $\DMekR$ and $\DMkR$.
 Similarly the pair of adjoint functors $(\varphi_*,\varphi^!)$ can also be
 derived. Such considerations have been used for example in
 \cite{CD16_b}*{\textsection 5.4}.
\end{rem}

\section{Milnor-Witt motivic complexes}\label{sec:MWmotives}

\subsection{Derived category}
\label{sec:dercat}

For any abelian category $\mathcal A$, we denote by $\Comp(\mathcal A)$ the category of (possibly unbounded) complexes of objects of $\mathcal A$ and by $\KComp(\mathcal A)$ the category of complexes with morphisms up to homotopy. Finally, we denote by $\Der(\mathcal A)$ the derived category of $\Comp(\mathcal A)$. We refer to \cite{W94_b}*{\S 10} for all these notions.

\begin{num}
Recall from our notations that $t$ is now either the Nisnevich
or the étale topology.

As usual in motivic homotopy theory, our first task is to equip the
category of complexes of MW-$t$-sheaves with a good model structure.
This is done using the method of \cite{CD09_b}, thanks to Lemma \ref{lm:corr_main}
and the fact that $\shtkR$ is a Grothendieck abelian category
(Proposition~\ref{prop:exist_associated-W-t-sheaf}\ref{item:shgrothab}).

Except for one subtlety in the case of the étale topology,
 our construction is analogous to that of sheaves with transfers.
 In particular, the proof of the main point is essentially
 an adaptation of \cite{CD12_b}*{5.1.26}. In order to make a short and streamlined proof, we first recall a few
 facts from model category theory.
\end{num}

\begin{num}\label{num:model}
We will be using the adjunction of Grothendieck abelian categories:
\[
\tilde \gamma^*:\shtkR \leftrightarrows \shMWtkR:\tilde \gamma_*
\]
of Corollary \ref{cor:adjunctions_corr}. Recall from Lemma \ref{lm:compare_transfers} that the functor
 $\tilde \gamma_*$ is conservative and exact.

First, there exists the so-called injective model structure
\index{model structure!injective}%
on $\Comp(\shtkR)$ and $\Comp(\shMWtkR)$ which is defined such that
the cofibrations are monomorphisms (thus every object is cofibrant)
and weak equivalences are quasi-isomorphisms (this is classical;
see e.g. \cite{CD09_b}*{2.1}). The fibrant objects for this model structure
are called \emph{injectively fibrant}.

Second, there exists the $t$-descent (or projective) model structure
\index{model structure!$t$-descent}%
on the category
$\Comp(\shtkR)$ (see \cite{CD09_b}*{Ex.~2.3}) characterized by the following properties:
\begin{itemize}
\item the class of \emph{cofibrations} is given by the smallest class
 of morphisms of complexes closed under suspensions, pushouts,
 transfinite compositions and retracts generated by the inclusions
\begin{equation} \label{eq:model1}
\Rt(X) \rightarrow C\big(\Rt(X) \xrightarrow{\id} \Rt(X)\big)[-1]
\end{equation}
for a smooth scheme $X$, where $\Rt(X)$ 
\index[notation]{rtx@$\Rt(X)$}%
is the free sheaf of $R$-modules on $X$.
\item weak equivalences are quasi-isomorphisms.
\end{itemize}
Our aim is to obtain the same kind of model structure
on the category $\Comp(\shMWtkR)$ of complexes of MW-$t$-sheaves.
Let us recall from \cite{CD09_b} that one can describe nicely the fibrant
objects for the $t$-descent model structure.
This relies on the following definition 
for a complex $K$ of $t$-sheaves:
\begin{itemize}
\item the complex $K$ is \emph{local} if for any smooth scheme $X$
and any integer $n \in \ZZ$, the canonical map:
\begin{equation} \label{eq:model2}
\H^n\big(K(X)\big)=\Hom_{\KComp(\shtkR)}(\Rt(X),K[n])
 \rightarrow \Hom_{\Der(\shtkR)}(\Rt(X),K[n])
\end{equation}
is an isomorphism;
\item the complex $K$ is \emph{$t$-flasque} if for any smooth
scheme $X$ and any $t$-hypercover $p:\cX \rightarrow X$,
the induced map:
\begin{equation} \label{eq:model3}
\H^n\big(K(X)\big)=\Hom_{\KComp(\shtkR)}(\Rt(X),K[n])
 \xrightarrow{p^*} \Hom_{\KComp(\shtkR)}(\Rt(\cX),K[n])=\H^n\big(K(\cX)\big)
\end{equation}
is an isomorphism.

Our reference for $t$-hypercovers is \cite{DHI04_b}. Recall in particular that
$\cX$
is a simplicial scheme whose terms are arbitrary direct sums of smooth schemes.
Then the notation $\Rt(\cX)$ stands for the complex associated
with the simplicial $t$-sheaves obtained by applying
the obvious extension of the functor $\Rt$ to the category of
direct sums of smooth schemes.
Similarly, $K(\cX)$ is the total complex (with respect to products)
of the obvious double complex.
\end{itemize}
Then, let us state for further reference the following theorem (\cite{CD09_b}*{Theorem~2.5}).
\end{num}
\begin{thm}\label{thm:CD}
Let $K$ be a complex of $t$-sheaves on the smooth site. Then the following
three properties on $K$ are equivalent:
\begin{enumerate-roman}
\item $K$ is fibrant for the $t$-descent model structure,
\item $K$ is local,
\item $K$ is $t$-flasque.
\end{enumerate-roman}
\end{thm}
Under these equivalent conditions,
we will say that $K$ is $t$-fibrant.\footnote{Moreover, fibrations
for the $t$-descent model
structure are epimorphisms of complexes whose kernel is $t$-fibrant.}

\begin{num}\label{num:Wmodel}
Consider now the case of MW-$t$-sheaves.
We will define \emph{cofibrations} in $\Comp(\shMWtkR)$ as in the previous
paragraph by replacing $\Rt$ with $\MWrepRt$ in \eqref{eq:model1}, i.e.\ 
the cofibrations are the morphisms in the smallest class
of morphisms of complexes of MW-$t$-sheaves closed under suspensions, pushouts,
transfinite compositions and retracts generated by the inclusions
\begin{equation} 
\MWrepRt(X) \rightarrow C\big(\MWrepRt(X) \xrightarrow{\id} \MWrepRt(X)\big)[-1]
\end{equation}
for a smooth scheme $X$. In particular, note that bounded above complexes of MW-$t$-sheaves whose components are direct sums of sheaves of the form $\MWrepRt(X)$ are cofibrant. This is easily seen by taking the push-out of (\ref{eq:model1}) along the morphism $\MWrepRt(X)\to 0$.

Similarly, a complex $K$ in $\Comp(\shMWtkR)$ will be called
\emph{local} (resp.\ \emph{$t$-flasque}) if it satisfies
the definition in the preceding paragraph
after replacing respectively $\shtkR$ and $\Rt$
by $\shMWtkR$ and $\MWrepRt$ in \eqref{eq:model2} (resp.\ \eqref{eq:model3}).

In order to show that cofibrations and quasi-isomorphisms define
a model structure on $\Comp(\shMWtkR)$, we will have to prove
the following result in analogy with the previous theorem.
\end{num}

\begin{thm}\label{thm:decent_Wsheaves}
Let $K$ be a complex of MW-$t$-sheaves.
 Then the following conditions are equivalent:
\begin{enumerate-roman}
\item $K$ is local;
\item $K$ is $t$-flasque.
\end{enumerate-roman}
\end{thm}

The proof is essentially an adaptation of the proof of
 \cite{CD12_b}*{5.1.26, 10.3.17},
 except that the case of the étale topology needs a new argument. It will be completed as a corollary of two lemmas, the first of which is a reinforcement of Lemma \ref{lm:corr_main}.

\begin{lem}\label{lm:corr_main_strong}
Let $p:\cX \rightarrow X$ be a $t$-hypercover of a smooth scheme
 $X$. Then the induced map:
\[
p_*:\MWrepRt(\cX) \rightarrow \MWrepRt(X)
\]
is a quasi-isomorphism of complexes of MW-$t$-sheaves.
\end{lem}
\begin{proof}
In fact, we have to prove that the complex $\MWrepRt(\cX)$
 is acyclic in positive degree and that $p_*$
 induces an isomorphism $\H_0\big(\MWrepRt(\cX)\big)\simeq \MWrepRt(X)$.\footnote{Note that
 the second fact follows from Lemma \ref{lm:corr_main}
 and the definition of $t$-hypercovers, but our proof works more directly.}
 In particular, as these assertions only concerns the $n$-th homology
 sheaf of $\MWrepRt(\cX)$, we can always assume that
 $\cX \simeq \mathrm{cosk}_n(\cX)$ for a
 large enough integer $n$ (because these two simplicial objects
 have the same $(n-1)$-skeleton). In other words, we can assume
 that $\cX$ is a bounded $t$-hypercover in the terminology
 of \cite{DHI04_b}*{Def.~4.10}.

As a consequence of the existence of the injective model
 structure, the category $\Der(\shMWtkR)$ is naturally enriched over the derived
 category of $R$-modules. Let us denote by $\derR \Hom^\bullet$ the
 corresponding Hom-object.
 We have only to prove that for any complex $K$
 of MW-$t$-sheaves, the natural map:
\[
p^*:\derR \Hom^\bullet(\MWrepRt(X),K)
 \rightarrow \derR \Hom^\bullet(\MWrepRt(cX),K)
\]
is an isomorphism in the derived category of $R$-modules.
 Because there exists an injectively fibrant resolution of any complex $K$,
 and $\derR \Hom^\bullet$ preserves quasi-isomorphisms,
 it is enough to consider the case of an injectively fibrant
 complex $K$ of MW-$t$-sheaves.

In this case, $\derR \Hom^\bullet(-,K)=\Hom^\bullet(-,K)$
(as any complex is cofibrant for the injective model structure)
and we are reduced to prove that the following complex of presheaves on
the smooth site:
\[
X \mapsto \Hom^\bullet(\MWrepRt(X),K)
\]
satisfies $t$-descent with respect to bounded $t$-hypercovers
 i.e.\ sends bounded $t$-hypercovers $\cX/X$ to quasi-isomorphisms of complexes
 of $R$-modules. But Lemma \ref{lm:corr_main} (and the fact that $K$ is injectively fibrant)
 tells us that this is the case when $\cX$ is the $t$-hypercover associated
 with a $t$-cover. So we conclude using \cite{DHI04_b}*{A.6}.
\end{proof}

The second lemma for the proof of Theorem \ref{thm:decent_Wsheaves}
 is based on the previous one.
 
\begin{lem}\label{lm:decent_Wsheaves2}
Let us denote by $\Comp$, $\KComp$, $\Der$
 (respectively by  $\tilde \Comp$, $\tilde \KComp$, $\tilde \Der$)
 the category of complexes, complexes up to homotopy and derived
 category of the category $\shtkR$ (respectively $\shMWtkR$).

Given a simplicial scheme $\cX$ whose components are (possibly infinite) coproducts of smooth $k$-schemes
 and a complex $K$ of MW-$t$-sheaves, 
 we consider the isomorphism of $R$-modules obtained
 from the adjunction $(\tilde \gamma^*,\tilde \gamma_*)$:
\[
\epsilon_{\cX,K}:\Hom_{\tilde \Comp}(\MWrepRt(\cX),K)
\rightarrow \Hom_{\Comp}(\Rt(\cX),\tilde \gamma_*(K)).
\]
Then there exist unique isomorphisms $\epsilon'_{\cX,K}$
 and $\epsilon''_{\cX,K}$ of $R$-modules making the following diagram
 commutative:
\[
\begin{tikzcd}
[row sep=18pt,column sep=28pt]
\Hom_{\tilde \Comp}(\MWrepRt(\cX),K) \ar[r,"{\epsilon_{\cX,K}}"] \ar[d]
 & \Hom_{\Comp}(\Rt(\cX),\tilde \gamma_*(K)) \ar[d] \\
\Hom_{\tilde \KComp}(\MWrepRt(\cX),K) \ar[r,"{\epsilon'_{\cX,K}}"] \ar[d,"{\pi_{\cX,K}}"']
 & \Hom_{\KComp}(\Rt(\cX),\tilde \gamma_*(K)) \ar[d,"{\pi'_{\cX,K}}"] \\
\Hom_{\tilde \Der}(\MWrepRt(\cX),K) \ar[r,"{\epsilon''_{\cX,K}}"]
 & \Hom_{\Der}(\Rt(\cX),\tilde \gamma_*(K))
\end{tikzcd}
\]
where the vertical morphisms are the natural localization maps.
\end{lem}

\begin{proof}
The existence and unicity of $\epsilon'_{\cX,K}$ simply follows
 from the fact $\tilde \gamma^*$ and $\tilde \gamma_*$ are additive functors,
 so in particular $\epsilon_{\cX,K}$ is compatible with chain homotopy
 equivalences.
 
For the case of $\epsilon_{\cX,K}^{''}$, we assume that 
 the complex $K$ is injectively fibrant. In this case,
 the map $\pi_{\cX,K}$ is an isomorphism.
 This already implies the existence and unicity of the map
 $\epsilon''_{\cX,K}$. Besides, according to the previous lemma and the fact that the map
 $\pi_{\cX,K}$ is an isomorphism natural in $\cX$,
 we obtain that $K$ is $t$-flasque
  (in the sense of Paragraph \ref{num:Wmodel}).
 Because $\epsilon'_{\cX,K}$ is an isomorphism natural in $\cX$,
 we deduce that $\tilde \gamma_*(K)$ is $t$-flasque.
 In view of Theorem \ref{thm:CD}, it is $t$-fibrant. 
 As $\Rt(\cX)$ is cofibrant for the $t$-descent model
 structure on $\Comp$, we deduce that $\pi'_{\cX,K}$ is an isomorphism.
 Therefore, $\epsilon''_{\cX,K}$ is an isomorphism.

The case of a general complex $K$ now follows from the existence of an injectively
 fibrant resolution $K \rightarrow K'$ of any complex of MW-$t$-sheaves $K$.
\end{proof}

\begin{proof}[proof of Theorem \ref{thm:decent_Wsheaves}]
 The previous lemma shows that the following conditions on
 a complex $K$ of MW-$t$-sheaves are equivalent:
\begin {itemize}
\item $K$ is local (resp.\ $t$-flasque) in $\Comp(\shMWtkR)$;
\item $\tilde \gamma_*(K)$ is local (resp.\ $t$-flasque) in $\Comp(\shtkR)$.
\end {itemize}
Then Theorem \ref{thm:decent_Wsheaves} follows from Theorem \ref{thm:CD}.
\end{proof}

Here is an important corollary (analogous to \cite{FSV00_b}*{chap.~5, 3.1.8}) which is simply a restatement of Lemma \ref{lm:decent_Wsheaves2}. 

\begin{coro}\label{cor:compare_Hom&cohomology}
Let $K$ be a complex of MW-$t$-sheaves and $X$ be a smooth scheme.
 Then for any integer $n \in \ZZ$, there exists a canonical isomorphism,
 functorial in $X$ and $K$:
\[
\Hom_{\Der(\shMWtkR)}(\MWrepRt(X),K[n])=\hypH^n_t(X,K)
\]
where the right hand side stands for the $t$-hypercohomology
of $X$ with coefficients in the complex $\tilde \gamma_*(K)$
 (obtained after forgetting MW-transfers).
\end{coro}

Recall that the category $\Comp(\shMWtkR)$ is symmetric monoidal,
 with tensor product induced as usual from the tensor product
 on $\shMWtkR$ (see Paragraph \ref{num:sht_monoidal}).

\begin{coro}\label{cor:model_Der}
The category $\Comp(\shMWtkR)$
has a proper cellular model structure
(see \cite{Hirschhorn03_b}*{12.1.1 and 13.1.1})
with quasi-isomorphisms as weak equivalences and
cofibrations as defined in Paragraph \ref{num:Wmodel}.
Moreover, the fibrations for this model structure
are epimorphisms of complexes whose kernel
are $t$-flasque (or equivalently local) complexes of MW-$t$-sheaves.
Finally, this is a symmetric monoidal model structure;
in other words, the tensor product (resp.\ internal Hom functor)
admits a total left (resp.\ right) derived functor.
\end{coro}

\begin{proof}
Each claim is a consequence of \cite{CD09_b}*{2.5, 5.5 and 3.2},
applied to the Grothendieck abelian category $\shMWtkR$ with respect
to the descent structure $(\mathcal G,\mathcal H)$ (see \cite{CD09_b}*{Def. 2.2} for the notion of descent structure) defined as follows:
\begin{itemize}
\item $\mathcal G$ is the class of MW-$t$-sheaves of the form $\MWrepRt(X)$
 for smooth scheme $X$;
\item $\mathcal H$ is the (small) family of complexes which are
 cones of morphisms $p_*:\MWrepRt(\cX) \rightarrow \MWrepRt(X)$
 for a $t$-hypercover $p$.
\end{itemize}
Indeed, $\mathcal G$ generates the category $\shMWtkR$
 (see after Definition \ref{df:gen_sht}) and the condition
 to be a descent structure is given by Theorem \ref{thm:decent_Wsheaves}.

In the end, we can apply \cite{CD09_b}*{3.2} to derive the
 tensor product as the tensor structure is weakly flat (in the sense of \cite{CD09_b}*{\S 3.1})
 due to the preceding definition and formula \eqref{eq:sht_monoidal}.
\end{proof}

\begin{rem}
We can follow the procedure of \cite{Mazza06_b}*{\S 8} to compute the tensor product of two bounded above complexes of MW-$t$-sheaves. This follows from \cite{CD09_b}*{Proposition~3.2} and the fact that bounded above complexes of MW-$t$-sheaves whose components are direct sums of representable sheaves are cofibrant.
\end{rem}

\begin{dfn}
The model structure on $\Comp(\shMWtkR)$ of the above corollary is called the \emph{$t$-descent model structure}.
\index{model structure!$t$-descent}%
\end{dfn}

In particular, the category $\Der(\shMWtkR)$
is a triangulated symmetric closed monoidal category.

\begin{num}
We also deduce from the $t$-descent model structure
that the vertical adjunctions of Corollary \ref{cor:adjunctions_corr}
induce Quillen adjunctions with respect to the $t$-descent model
structure on each category involved and so admit derived functors
as follows:
\begin{equation}\label{eq:chg_top&tr_Der}
\begin{tikzcd}
[column sep=30pt,row sep=24pt]
\Der(\shxkR{\nis}) \ar[r,shift left=2pt,"{\derL \tilde \gamma^*}"] \ar[d,shift left=2pt,"a"]
 & \Der(\shMWxkR{\nis}) \ar[r,shift left=2pt,"{\derL \pi^*}"] \ar[d,shift left=2pt,"{\tilde a}"]
     \ar[l,shift left=2pt,"{\tilde \gamma_{*}}"]
 & \Der(\shVxkR{\nis})\ar[d,shift left=2pt,"{a^{\mathrm{tr}}}"]
     \ar[l,shift left=2pt,"{\pi_{*}}"] \\
\Der(\shxkR{\et}) \ar[r,shift left=2pt,"{\derL \tilde \gamma^*_\et}"]
    \ar[u,shift left=2pt,"{\derR \fO}"]
 & \Der(\shMWxkR{\et}) \ar[r,shift left=2pt,"{\derL \pi^*_\et}"]
	  \ar[u,shift left=2pt,"{\derR \tfO}"]
		\ar[l,shift left=2pt,"{\tilde \gamma_{\et*}}"]
 & \Der(\shVxkR{\et})
    \ar[u,shift left=2pt,"{\derR \fO}"]
		\ar[l,shift left=2pt,"{\pi_{\et*}}"]
\end{tikzcd}
\end{equation}
where the adjoint pair associated étale sheaf and forgetful functor
are denoted by $(a,\fO)$
 and similarly for MW-transfers and transfers.
When the functors are exact, they are trivially derived and so we
 have used the same notation as for their counterpart for
 sheaves.

Note that by definition, the left adjoints in this diagram
 are all monoidal functors and send the object represented
 by a smooth scheme $X$ (say in degree $0$) to the analogous
 object
\begin{equation}\label{eq:derived&rep}
\derL\tilde \gamma^*\big(\Rt(X)\big)=\MWrepRt(X),
\qquad
 \derL\pi^*\big(\MWrepRt(X)\big)=\VrepRt(X).
\end{equation}
\end{num}

\subsection{The $\Aone$-derived category}\label{sec:AA1derived}

We now adapt the usual $\Aone$-homotopy machinery to our context.

\begin{dfn}
We define the category $\DMtexkR t$ of MW-motivic complexes
for the topology $t$
\index[notation]{dmteffkr@$\DMtexkR t$}%
\index{Milnor-Witt!motives, category of!effective}%
as the localization of the triangulated category
$\Der(\shMWtkR)$ with respect to the localizing triangulated
subcategory\footnote{Recall that according to Neeman \cite{Nee01_b}*{3.2.6},
localizing means stable by coproducts.}
$\mathcal T_{\Aone}$
generated by complexes of the form:
\[
\cdots 0 \rightarrow \MWrepRt(\Aone_X) \xrightarrow{p_*} \MWrepRt(X) \rightarrow 0 \cdots
\]
where $p$ is the projection of the affine line relative to
an arbitrary smooth $k$-scheme $X$.
As usual, we define the \emph{MW-motive $\tMot(X)$}
associated to a smooth scheme $X$ as the complex concentrated in degree
$0$ and equal to the representable MW-$t$-sheaf
$\MWrepRt(X)$. Following our previous conventions,
we mean the Nisnevich topology when the topology is not indicated.
\end{dfn}

According to this definition, it is formal that the localizing
 triangulated subcategory $\mathcal T_{\Aone}$ is
 stable under the derived tensor product of $\Der(\shMWtkR)$
 (cf.\ Corollary \ref{cor:model_Der}). In particular, it induces
 a triangulated monoidal structure on $\DMtexkR t$.

\begin{num}\label{num:A1-local}
As usual, we can apply the classical techniques of localization
to our triangulated categories
and also to our underlying model structure.
So a complex of MW-$t$-sheaf $E$ is called \emph{$\Aone$-local}
if for any smooth scheme $X$ and any integer $i \in \ZZ$,
the induced map
\[
\Hom_{\Der(\shMWtkR)}(\MWrepRt(X),E[i])
 \rightarrow \Hom_{\Der(\shMWtkR)}(\MWrepRt(\Aone_X),E[i])
\]
is an isomorphism. In view of Corollary \ref{cor:compare_Hom&cohomology},
 it amounts to ask that the $t$-cohomology of $\tilde \gamma_*(E)$ is
 $\Aone$-invariant, or in equivalent words, that $E$ is strictly
 $\Aone$-local.

Applying Neeman's localization theorem (see
\cite{Nee01_b}*{9.1.19}),\footnote{Indeed recall that the derived
category of the Grothendieck abelian category $\shMWtkR$ is a well generated triangulated category
according to \cite{Nee01-2_b}*{Th. 0.2}.} the category
$\DMtexkR t$ can be viewed as the full subcategory of $\Der(\shMWtkR)$
whose objects are the $\Aone$-local complexes. Equivalently,
the canonical functor $\Der(\shMWtkR) \rightarrow \DMtexkR t$
admits a fully faithful right adjoint whose essential image
consists in  $\Aone$-local complexes. In particular,
one deduces formally the existence of an $\Aone$-localization
functor $L_{\Aone}:\Der(\shMWtkR) \rightarrow \Der(\shMWtkR)$.

Besides, we get the following proposition by applying the 
general left Bousfield localization procedure for proper cellular
model categories (see \cite{Hirschhorn03_b}*{4.1.1}). 
We say that a morphism $\phi$ of $\Der(\shMWtkR)$
is an \emph{weak $\Aone$-equivalence} if for any $\Aone$-local object $E$,
the induced map $\Hom(\phi,E)$ is an isomorphism.
\end{num}
\begin{prop}\label{prop:model_DMt}
The category $\Comp(\shMWtkR)$
has a symmetric monoidal model structure
with weak $\Aone$-equivalences as weak equivalences and
cofibrations as defined in Paragraph \ref{num:Wmodel}.
This model structure is proper and cellular. Moreover, the fibrations for this model structure
are epimorphisms of complexes whose kernel
are $t$-flasque and
$\Aone$-local complexes.
\end{prop}

The resulting model structure on $\Comp(\shMWtkR)$
will be called the \emph{$\Aone$-local model structure}.
\index{model structure!$\Aone$-local}%
The proof of the proposition follows formally from
Corollary \ref{cor:model_Der} by the usual localization
procedure of model categories, see \cite{CD09_b}*{\textsection 3}
for details. Note that the tensor product of two bounded above complexes can be computed as in the derived category. 

\begin{num}
As a consequence of the above discussion, the category $\DMtexkR t$ is a triangulated
 symmetric monoidal closed category. Besides, it is clear
 that the functors of Corollary~\ref{cor:adjunctions_corr}
 induce Quillen adjunctions for the $\Aone$-local model
 structures. Equivalently, Diagram \eqref{eq:chg_top&tr_Der}
 is compatible with $\Aone$-localization and induces
 adjunctions of triangulated categories:
\begin{equation}\label{eq:chg_top&tr_DMeff}
\begin{tikzcd}
[column sep=30pt,row sep=24pt]
\DAekR \ar[r,shift left=2pt,"{\derL \tilde \gamma^*}"] \ar[d,shift left=2pt,"a"]
 & \DMtekR \ar[r,shift left=2pt,"{\derL \pi^*}"] \ar[d,shift left=2pt,"{\tilde a}"] \ar[l,shift left=2pt,"{\tilde \gamma_{*}}"]
 & \DMekR \ar[d,shift left=2pt,"{a^{tr}}"]
     \ar[l,shift left=2pt,"{\pi_{*}}"] \\
\DAexkR{\et} \ar[r,shift left=2pt,"{\derL \tilde \gamma^*_\et}"]
    \ar[u,shift left=2pt,"{\derR \fO}"]
 & \DMtexkR{\et} \ar[r,shift left=2pt,"{\derL \pi^*_\et}"]
	  \ar[u,shift left=2pt,"{\derR \tfO}"]
		\ar[l,shift left=2pt,"{\tilde \gamma_{\et*}}"]
 & \DMexkR{\et}.
    \ar[u,shift left=2pt,"{\derR \fO}"]
		\ar[l,shift left=2pt,"{\pi_{\et*}}"]
\end{tikzcd}
\end{equation}
In this diagram, the left adjoints are all monoidal and send
the different variant of motives represented by a smooth scheme $X$
to the analogous motive. In particular,
\[\derL \pi^*\tMot(X)=\Mot(X).\]
\index[notation]{mtx@$\tMot(X)$}%
\index[notation]{mx@$\Mot(X)$}%
Also, the functors $\tilde \gamma_{t*}$ and $\pi_{t*}$
for $t=\nis, \et$ (or following our conventions, $t=\varnothing, \et$)
are conservative. It is worth mentioning the following corollary of Proposition~\ref{cor:adjunctions_corr}\ref{item:adjunction-pi-equivalence}:
\begin{prop}
($k$ is an arbitrary perfect field.%
\footnote{Moreover, inseparable extensions being isomorphisms for the étale topology,
the definition of $\DMtexkR{\et}$ and this comparison equivalence can be extended to arbitrary fields.})
The adjunction defined above 
\[
\derL \pi^*_\et=\pi^*_\et:\DMtexkR{\et} \leftrightarrows \DMexkR{\et}:\pi_{\et*}
\]
is an equivalence of triangulated monoidal categories.
\end{prop}
The functors $\tilde \gamma_{t*}$ and $\pi_{t*}$ for $t=\nis, \et$ in diagram
\eqref{eq:chg_top&tr_Der} preserve $\Aone$-local objects
and so commute with the $\Aone$-localization functor.
Therefore one deduces from Morel's $\Aone$-connectivity
theorem \cite{Morel05_b} the following result.
\end{num}
\begin{thm}
Let $E$ be a complex of MW-sheaves concentrated in positive degrees.
 Then the complex $L_{\Aone} E$ is concentrated in positive degrees.
\end{thm}
Indeed, to check this, one needs only to apply the functor 
 $\tilde \gamma_*$ as it is conservative, and so we are reduced to Morel's
 theorem \cite{Morel05_b}*{6.1.8}.

\begin{coro}\label{cor:pre-df:htp_tstruct_eff}
Under the assumption of the previous theorem, 
 the triangulated category $\DMtekR$ admits a unique $t$-structure
 such that the functor 
 $\tilde \gamma_*:\DMtekR \rightarrow \Der\big(\shMWxkR{\nis}\big)$
 is $t$-exact.
\end{coro}
Note that the truncation functor on $\DMtekR$ is obtained by applying
 to an $\Aone$-local complex the usual truncation functor of
 $\Der(\shMWxkR{\nis})$ and then the $\Aone$-localization functor.

\begin{dfn}
\label{df:htp_tstruct_eff}%
The $t$-structure on $\DMtekR$ obtained in the previous corollary
 will be called the \emph{homotopy $t$-structure}.
\end{dfn}

\begin{rem}
Of course, the triangulated categories $\DMekR$ and $\DAekR$
 are also equipped with $t$-structures, called in each 
 cases homotopy $t$-structures
 --- in the first case, it is due to Voevodsky and in the second
 to Morel.

It is clear from the definitions that the functors
 $\tilde \gamma_*$ and $\pi_*$ in Diagram \eqref{eq:chg_top&tr_DMeff}
 are $t$-exact.
\end{rem}

As in the case of Nisnevich sheaves with transfers,
 we can describe nicely $\Aone$-local objects and
 the $\Aone$-localization functor due to the following theorem.
 
\begin{thm}\label{thm:Wsh&A1}
Assume $k$ is an infinite perfect field of characteristic not $2$. Let $F$ be an $\Aone$-invariant MW-presheaf. Then, the associated MW-sheaf $\tilde a(F)$ is
 strictly $\Aone$-invariant. Moreover, the Zariski sheaf associated with $F$
 coincides with $\tilde a(F)$ and the natural map
\[
\H^i_\zar(X,\tilde a(F)) \rightarrow
\H^i_\nis(X,\tilde a(F))
\]
is an isomorphism for any smooth scheme $X$.
\end{thm}

\begin{proof}
Recall that we have a functor $(\alpha^\prime)^*:\pshMWkR\to \SpshfrkR$. Since the presheaf $F$ is $\Aone$-invariant, it follows that $(\alpha^\prime)^*(F)$ has the same property and we may apply Theorem \ref{thm:A1local_framedPSh} to see that the Nisnevich sheaf $a(F)$ associated to the presheaf $F$ is strictly $\Aone$-invariant and quasi-stable. It follows from Proposition \ref{prop:exist_associated-W-t-sheaf} that $\tilde a(F)$ is strictly $\Aone$-invariant. 
Now, a strictly $\Aone$-invariant sheaf admits a Rost-Schmid complex in the sense of \cite{Morel12_b}*{\S 5} and it follows from \cite{Morel12_b}*{Corollary~5.43} that 
\[
\H^i_\zar(X,\tilde a(F)) \rightarrow
\H^i_\nis(X,\tilde a(F))
\]
is an isomorphism for any $i\in\NN$ and any smooth scheme $X$. It remains to show that the Zariski sheaf associated with $F$ coincides with $\tilde a(F)$. We follow \cite{Pa10_b}*{Lemma~4.1} indicating a suitable reference for the properties needed there. First, if $X$ is a smooth integral scheme and $x\in X$, it follows from \cite{Garkusha15_b}*{Theorem~2.15~(3)} that $F(\OO_{X,x})\to F(k(X))$ is injective, and we may therefore form the unramified Zariski sheaf $F_{\mathrm{un}}$ (on $X$) defined by $U\mapsto \cap_{x\in U}F(\OO_{X,x})\subset F(k(X))$. If $F_{\zar}$ is the Zariski sheaf associated to $F$, we obtain a map $F_{\zar}\to F_{\mathrm{un}}$ which is locally an isomorphism and then $F_{\zar}$ satisfies the conclusions of \cite{Pa10_b}*{Lemma~3.1, Corollary~3.2}. We can now conclude the proof as in \cite{Pa10_b}*{Lemma~4.1} using this time \cite{Garkusha15_b}*{Theorem~2.15~(5)}.
\end{proof}

\begin{rem}
See also \cite{Kolderup17_b} for a proof intrinsic to MW-motives in characteristic $0$.
\end{rem}

As in the case of Voeovdsky's theory of motivic complexes,
 this theorem has several important corollaries.
 
\begin{coro}\label{cor:A1-local_complexes}
Assume $k$ is an infinite perfect field of characteristic not $2$. Then, a complex $E$ of MW-sheaves is $\Aone$-local
if and only if its homology (Nisnevich) sheaves 
are $\Aone$-invariant.
The heart of the homotopy $t$-structure (Def. \ref{df:htp_tstruct_eff})
on $\DMtekR$ consists of $\Aone$-invariant MW-sheaves.
\end{coro}
The proof is classical. We recall it for the convenience of the reader.

\begin{proof}
Let $K$ be an $\Aone$-local complex of MW-sheaves over $k$.
 Let us show that its homology sheaves are $\Aone$-invariant.
 According to the existence of the $\Aone$-local model structure
 on $\Comp(\shtkR)$ (Proposition \ref{prop:model_DMt}),
 there exist a Nisnevich fibrant and $\Aone$-local complex
 $K'$ and a weak $\Aone$-equivalence:
\[
K \xrightarrow \phi K'.
\]
 As $K$ and $K'$ are $\Aone$-local, the map $\phi$ is a
 quasi-isomorphism. In particular, we can replace $K$ by $K'$.
 In other words, we can assume $K$ is Nisnevich fibrant thus local
 (Theorem \ref{thm:CD}). 
 Then we get:
\[
\H^n\big(K(X)\big) \simeq
 \Hom_{\Der(\shMWtkR)}\big(\MWrepRt(X),K[n]\big)
\]
according to the definition of local in Paragraph \ref{num:Wmodel}.
 This implies in particular that the cohomology presheaves of $K$ are
 $\Aone$-invariant. We conclude using Theorem \ref{thm:Wsh&A1}.

Assume conversely that $K$ is a complex of MW-sheaves
 whose homology sheaves are $\Aone$-invariant. Let us show that
 $K$ is $\Aone$-local. According to Corollary
 \ref{cor:compare_Hom&cohomology}, we need only to show its Nisnevich
 hypercohomology is $\Aone$-invariant. Then we apply the Nisnevich
 hypercohomology spectral sequence for any smooth scheme $X$:
\[
E_2^{p,q}=\H^p_\nis(X,\H^q_\nis(K)) \Rightarrow \hypH^{p+q}_\nis(X,K).
\]
As the cohomological dimension of the Nisnevich topology
 is bounded by the dimension of $X$, the $E_2$-term is concentrated
 in degree $p \in [0,\dim(X)]$ and the spectral sequence converges (\cite{SV00_b}*{Theorem~0.3}).
 It is moreover functorial in $X$. Therefore it is enough
 to show the map induced by the projection
\[
\H^p_\nis(X,\H^q_\nis(K)) \rightarrow \H^p_\nis(\Aone_X,\H^q_\nis(K))
\]
is an isomorphism to conclude. By assumption the sheaf
 $\H^q_\nis(K)$ is $\Aone$-invariant so Theorem \ref{thm:Wsh&A1} applies
 again.

As the functor $\tilde \gamma_*$ is $t$-exact by
 Corollary \ref{cor:pre-df:htp_tstruct_eff}, the conclusion 
 about the heart of $\DMtekR$ follows.
\end{proof}

\begin{coro}\label{cor:cohomcomplex}
Assume $k$ is an infinite perfect field of characteristic not $2$.
Let $K$ be an $\Aone$-local complex of MW-sheaves. Then we have 
\[
\hypH^i_{\zar}(X,K)= \hypH^i_{\nis}(X,K)
\]
for any smooth scheme $X$ and any $i\in \ZZ$.
\end{coro}
\begin{proof}
The proof uses the same principle as in the previous result.
 Let us first consider the change of topology adjunction:
\[
\alpha^*:\shxkR \zar \leftrightarrows \shxkR \nis:\alpha_*
\]
where $\alpha^*$ is obtained using the functor ``associated
 Nisnevich sheaf functor"
 and $\alpha_*$ is just the forgetful functor.
 This adjunction can be derived (using for example the injective
 model structures) and induces:
\[
\alpha^*:\Der(\shxkR \zar) \leftrightarrows \Der(\shxkR \nis):\derR \alpha_*
\]
--- note indeed that $\alpha^*$ is exact.
 Coming back to the statement of the corollary, we have to show
 that the adjunction map:
\begin{equation}\label{eq:compar_zar/nis}
 \tilde \gamma_*(K) \rightarrow \derR \alpha_* \alpha^*( \tilde\gamma_*(K))
\end{equation}
is a quasi-isomorphism. Let us denote abusively by $K$ the
 sheaf $\tilde \gamma_*(K)$. Note that this will be harmless as
 $\tilde \gamma_*(\H^q_\nis(K))=\H^q_\nis(\tilde \gamma_*K)$.
 With this abuse of notation, one has canonical identifications:
\begin{align*}
\Hom_{\Der(\shxkR \nis)}(\MWrepZt(X),\derR \alpha_* \alpha^*(K))
 &=\H^p_\zar(X,K) \\
\Hom_{\Der(\shxkR \nis)}(\MWrepZt(X),\derR \alpha_* \alpha^*(\H^q_\nis(K)))
 &=\H^p_\zar(X,\H^q_\nis(K))
\end{align*}
 Using now the tower of truncation
 of $K$ for the standard $t$-structure on $\Der(\shMWtkR)$
 --- or equivalently $\Der(\shtkR)$ ---
 and the preceding identifications, one gets a spectral sequence:
\[
^{\zar}E_2^{p,q}=\H^p_\zar(X,\H^q_\nis(K))
 \Rightarrow \hypH^{p+q}_\zar(X,K)
\]
and the morphism \eqref{eq:compar_zar/nis} induces a morphism
 of spectral sequence:
\begin{align*}
E_2^{p,q}=\H^p_\nis(X,\H^q_\nis(K))
 \longrightarrow &^{\zar}E_2^{p,q}=\H^p_\zar(X,\H^q_\nis(K)) \\
 & \Rightarrow \hypH^{p+q}_\nis(X,K)
 \longrightarrow \hypH^{p+q}_\zar(X,K).
\end{align*}
The two spectral sequences converge (as the Zariski and Nisnevich cohomological
 dimensions of $X$ are bounded). According to Theorem \ref{thm:Wsh&A1},
 the map on the $E_2$-term is an isomorphism so the map
 on the abutment must be an isomorphism and this concludes the proof.
\end{proof}

\begin{num}
\label{num:suslincomplex}%
Following Voevodsky, given a complex $E$ of MW-sheaves, we define its
associated Suslin complex as the following complex of
sheaves:\footnote{Explicitly, this complex associates
to a smooth scheme $X$ the total complex (for coproducts)
of the bicomplex $E(\Delta^* \times X)$.}
\[
\Cstar E:=\uHom(\MWrepRt(\Delta^*),E)
\]
\index{Suslin(-Voevodsky) singular complex}%
\index[notation]{cf@$\Cstar F$}%
where $\Delta^*$ is the standard cosimplicial scheme.
\end{num}

\begin{coro}
\label{cor:LA1}%
Assume $k$ is an infinite perfect field of characteristic not $2$.

Then for any complex $E$ of MW-sheaves, there
 exists a canonical quasi-isomorphism:
\[
L_{\Aone}(E) \simeq \Cstar E.
\]
\end{coro}

\begin{proof}
Indeed, according to Corollary \ref{cor:A1-local_complexes}, it is clear
 that $\Cstar E$ is $\Aone$-local. Thus the canonical map
\[
c:E \rightarrow \Cstar E
\]
induces a morphism of complexes:
\[
L_{\Aone}(c):L_{\Aone}(E) \rightarrow L_{\Aone}(\Cstar E) = \Cstar E.
\]
As $\Delta^n \simeq \AAA^n_k$, one checks easily that
the map $c$ is an $\Aone$-weak equivalence. Therefore,
the map $L_{\Aone}(c)$ is an $\Aone$-weak equivalence
of $\Aone$-local complexes, thus a quasi-isomorphism.
\end{proof}

\begin{num} \label{num:Tate}
As usual, one defines the Tate object in $\DMtekR$ by
the formula
\[
\tRx 1:=\tMot(\PP^1_k)/\tMot(\{\infty\})[-2]
\simeq \tMot(\Aone_k)/\tMot(\Aone_k-\{0\})[-2]
\simeq \tMot(\Gm)/\tMot(\{1\})[-1].
\]
\index{Tate motive|see{Milnor-Witt motive, Tate}}%
\index{Milnor-Witt!motive!Tate}%
Then one defines the effective MW-motivic cohomology of a smooth
scheme $X$ in degree $(n,i) \in \ZZ \times \NN$:\footnote{Negative
twists will be introduced in the next section.} as
\[
\HMW^{n,i}(X,R)=\Hom_{\DMtekR}(\MWrepR X,\tRx i[n])
\]
where $\tRx i={\tRx 1}^{\otimes i}$.
\index[notation]{rtn@$\tRx n$}%
\end{num}

\begin{coro}
\label{cor:W-motivic&generalized}%
Assume $k$ is an infinite perfect field of characteristic not $2$.
The effective MW-motivic cohomology defined above coincides with 
the generalized motivic cohomology groups
defined (for $R=\ZZ$) in \chfinitecw, Definition~\ref{def:genmotiviccohom}.
\end{coro}

\begin{proof}
By Corollary \ref{cor:LA1}, the Suslin complex of $R(i)$ is $\Aone$-local. It follows then from Corollary \ref{cor:cohomcomplex} that its Nisnevich hypercohomology and its Zariski hypercohomology coincide.  We conclude using \chcancellation, Corollary~\ref{cor:local}.
\end{proof}

We now spend a few lines in order to compare ordinary motivic cohomology with MW-motivic cohomology, following \chfinitecw, Definition~\ref{dfn:Imotiviccohomology}. In this part, we suppose that $R$ is flat over $\ZZ$. If $X$ is a smooth scheme, recall from \chfinitecw, Definition~\ref{dfn:icorXY} that the presheaf with MW-transfers $\Icorr(X)$
\index[notation]{icorx@$\Icorr(X)$}%
defined by 
\[
\Icorr(X)(Y)=\varinjlim_T\H^d_T\!\big(X \times Y,\I^{d+1},\omega_{Y}\big)
\]
fits in an exact sequence
\[
0\to \Icorr(X)\to \MWcorr(X)\to \VrepZ(X)\to 0.
\]
As $\MWcorr(X)$ is a sheaf in the Zariski topology, it follows that $\Icorr(X)$ is also such a sheaf. We can also consider the Zariski sheaf $\Icorr_R(X)$ defined by 
\[
\Icorr_R(X)(Y)=\Icorr(X)(Y)\otimes R.
\]

\begin{dfn}
We denote by $\IrepRt(X)$ the $t$-sheaf associated to the presheaf $\Icorr_R(X)$.
\index[notation]{izx@$\IrepRt(X)$}%
\end{dfn}

In view of Proposition \ref{prop:exist_associated-W-t-sheaf}, $\Icorr_R(X)$ has MW-transfers. Moreover, sheafification being exact and $R$ being flat, we have an exact sequence
\[
0\to \IrepRt(X)\to \MWrepRt(X)\to \VrepRt(X)\to 0
\]
of MW-$t$-sheaves (note the slight abuse of notation when we write $\VrepRt(X)$ in place of $\pi_*^t \VrepRt(X)$. We deduce from \cite{Mazza06_b}*{Lemma~2.13} an exact sequence 
\[
0\to \ItRcbtx q\to \tRcbtx q\to \RVcbtx q\to 0
\]
for any $q\in\NN$, where $\ItRcbtx q=\IrepRt(\Gm^{\wedge q})$.
\index[notation]{irtq@$\ItRcbtx q$}%

\begin{dfn}
For any $q\in \NN$, we set $\I\tRx q=\I\tRcbx q[-q]$ in $\DMtekR$
\index[notation]{irtq@$\I\tRx q$}%
and 
\[
\HI^{p,q}(X,R)=\Hom_{\DMtekR}(\tMot(X),\I\tRx q[p])
\]
\index[notation]{hipqxr@$\HI^{p,q}(X,R)$}%
for any smooth scheme $X$. Note that following our convention, we are using the Nisnevich topology in this definition.
\end{dfn}

\begin{num}
Assume that $k$ is an infinite perfect field of characteristic not $2$.
As Zariski cohomology and Nisnevich cohomology coincide by Corollary \ref{cor:cohomcomplex}, these groups coincide with the ones defined in \chfinitecw, Definition~\ref{dfn:Imotiviccohomology} (when $R=\ZZ$). We now construct a long exact sequence for any smooth scheme $X$ and any $q\in \NN$
\[
\cdots\to \HI^{p,q}(X,R)\to \HMW^{p,q}(X,R)\to \HM^{p,q}(X,R) \to \HI^{p+1,q}(X,R)\to \cdots
\]
where $\HM^{p,q}(X,R)$ is Voevodsky's motivic cohomology.
\index[notation]{hmpqxr@$\HM^{p,q}(X,R)$}%
The method we use was explained to us by Grigory Garkusha, whom we warmly thank.

For a smooth scheme $X$, consider the quotient $\tRcbx q(X)/\ItRcbx q(X)$. This association defines a presheaf with MW-transfers, whose associated sheaf is $\RVcbx q$.
\index[notation]{rq@$\RVcbx q$}%
Next, observe that Suslin's construction is exact on presheaves, and therefore we obtain an exact sequence of complexes of MW-presheaves
\[
0\to \Cstar{\I\tRcbx q}\to \Cstar{\tRcbx q}\to \Cstar{\tRcbx q/\ItRcbx q} \to 0.
\]
In particular, we obtain an exact sequence of complexes of $R$-modules
\[
0\to \Cstar {\I\tRcbx q}(Z)\to \Cstar {\tRcbx q}(Z)\to \Cstar {\tRcbx q/\ItRcbx q}(Z)\to 0.
\]
for any local scheme $Z$. Now, \chcancellation, Corollary~\ref{cor:local} shows that the morphism 
\[
\Cstar {\tRcbx q/\ItRcbx q}(Z)\to \Cstar {\RVcbx q}(Z)
\]
is a quasi-isomorphism. Consequently, we obtain an exact triangle in the derived category of MW-sheaves
\[
\Cstar {\I\tRcbx q}\to \Cstar {\tRcbx q}\to \Cstar {\RVcbx q}\to \Cstar {\I\tRcbx q}[1]
\]
and the existence of the long exact sequence 
\[
\cdots\to \HI^{p,q}(X,R)\to \HMW^{p,q}(X,R)\to \HM^{p,q}(X,R) \to \HI^{p+1,q}(X,R)\to \cdots
\]
follows.
\end{num}

\begin{num}
To end this section,
we now discuss the effective geometric MW-motives,
which are built as in the classical case.
 
\begin{dfn}
($k$ is an arbitrary perfect field.)
One defines the category $\DMtegmkZ$ of \emph{geometric effective motives}
\index[notation]{dmteffgeomkZ@$\DMtegmkZ$}%
\index{Milnor-Witt!motives, category of!effective geometric}%
over the field $k$ as the pseudo-abelianization
of the Verdier localization of
the homotopy category $\KCompb(\cork)$ associated
to the additive category $\cork$ with respect to the thick
triangulated subcategory containing complexes of the form:
 
\begin{enumerate}
\item
\label{item:elementary-nis-sequence}
$\cdots \rightarrow [W] \xrightarrow{k_*-g_*} [V] \oplus [U]
 \xrightarrow{f_*+j_*} [X]
 \rightarrow \cdots$
for an elementary Nisnevich distinguished square of smooth schemes:
\[
\begin{tikzcd}
[column sep=10pt]
W \ar[d,"g"'] \ar[r,"k"] & V\ar[d,"f"] \\
U\ar[r,"j"] & X;
\end{tikzcd}
\]
\item $\cdots  \rightarrow [\Aone_X] \xrightarrow{p_*} [X]  \rightarrow \cdots$
where $p$ is the canonical projection and $X$ is a smooth scheme.
\end{enumerate}
\end{dfn}

It is clear that the natural map $\KCompb(\cork) \rightarrow \Der(\shMWkZ)$
induces a canonical functor:
\[
\iota:\DMtegmkZ \rightarrow \DMtek.
\]
As a consequence of \cite{CD09_b}*{Theorem~6.2} (see also Example 6.3 of \emph{op. cit.})
 and Theorem \ref{thm:Wsh&A1}, we get the following result.
\begin{prop}\label{prop:gm_objects}
Assume $k$ is an infinite perfect field of characteristic not $2$.
The functor $\iota$ is fully faithful and its essential
 image coincides with each one of the following subcategories:
\begin{itemize}
\item the subcategory of compact objects of $\DMtek$;
\item the smallest thick triangulated subcategory of $\DMtek$
which contains the motives $\tMot(X)$ of smooth schemes $X$.
\end{itemize}
Moreover, when $k$ is an infinite perfect field,
one can reduce in part \ref{item:elementary-nis-sequence} of the definition of $\DMtegmkZ$ to consider
those complexes associated to a Zariski open cover $U \cup V$ of 
a smooth scheme $X$.
\end{prop}

\begin{rem}
\label{rem:compact_gen_DMte}%
Note that the previous proposition states in particular that the objects
of the form $\tMot(X)$ for a smooth scheme $X$ are compact in $\DMtekR$.
As this is a direct consequence of \cite{CD09_b}*{Theorem~6.2},
and does not use Theorem \ref{thm:Wsh&A1}, this is valid
for an arbitrary perfect field $k$.
Therefore, they form a family of compact generators of $\DMtekR$ in the
sense that $\DMtekR$ is equal to its smallest triangulated category
containing $\tMot(X)$ for a smooth scheme $X$ and stable under
coproducts.
\end{rem}

\end{num}

\subsection{The stable $\Aone$-derived category}
\label{sec:stablederivedcat}

\begin{num}
As usual in motivic homotopy theory, we now describe 
the $\PP^1$-stabilization of the category of 
MW-motivic complexes for the topology $t$ (again, $t=\nis, \et$).

Recall that the Tate twist in $\DMtexkR t$ 
\index{Tate twist}%
\index{Milnor-Witt!motive!Tate}%
is defined by one of the following
 equivalent formulas:
\[
\tRx 1:=\tMot(\PP^1_k)/\tMot(\{\infty\})[-2]
 \simeq \tMot(\Aone_k)/\tMot(\Aone_k-\{0\})[-2]
 \simeq \tMot(\Gm)/\tMot(\{1\})[-1].
\]
\index[notation]{rtn@$\tRx n$}%
In the construction of the $\PP^1$-stable category as well as
in the study of the homotopy $t$-structure, it is useful to
introduce a redundant notation of $\Gm$-twist:
\[
\tRcbx 1:=\tMot(\Gm)/\tMot(\{1\})
\]
so that $\tRcbx 1=\tRx 1[1]$. The advantage of this
definition is that we can consider $\tRcbx 1$
as a MW-$t$-sheaf instead of a complex. For $m\geq 1$, we set $\tRcbx m={\tRcbx 1}^{\otimes m}$ and we observe that $\tRx m=\tRcbx m[-m]$. 
\index[notation]{rtn@$\tRcbx n$}%

Let us recall the general process of $\otimes$-inversion of the
Tate twist in the context of model categories,
as described in our particular case in \cite{CD09_b}*{\textsection~7} (or \cite{Ho00}).
We define the category $\spMWkR$ of (abelian) Tate MW-$t$-spectra 
\index[notation]{sptkr@$\spMWkR$}%
\index{Milnor-Witt!spectrum!Tate}%
\index{spectrum!Milnor-Witt Tate|see{Milnor-Witt spectrum Tate}}%
as the additive category whose object are couples $(\mathcal F_*,\epsilon_*)$ where $\mathcal F_*$
is a sequence of MW-$t$-sheaves such that $\mathcal F_n$ is equipped with
an action of the symmetric group $\mathfrak S_n$
and, for any integer $n \geq 0$,
\[
\epsilon_n:(\mathcal F_n\{1\}:=\mathcal F_n \otimes \tRcbx 1)
\rightarrow \mathcal F_{n+1}
\]
is a morphism of MW-$t$-sheaves, called the \emph{suspension map},
such that the iterated morphism
\[
\mathcal F_n\{m\} \rightarrow \mathcal F_{n+m}
\]
is $\mathfrak S_n \times \mathfrak S_m$-equivariant for any $n\geq 0$ and $m\geq 1$.
 (see \emph{loc.\ cit.}\ for more details). The morphisms in $\spMWkR$ between couples $(\mathcal F_*,\epsilon_*)$ and $(\mathcal G_*,\tau_*)$ are sequences of $\mathfrak S_n$-equivariant morphisms $f_n:\mathcal F_n\to \mathcal G_n$ such that the following diagram of $\mathfrak S_n \times \mathfrak S_m$-equivariant maps
\[
\begin{tikzcd}
\mathcal F_n\{m\} \ar[r] \ar[d,"{f_n\{m\}}"'] & \mathcal F_{n+m} \ar[d,"{f_{n+m}}"] \\
\mathcal G_n\{m\} \ar[r] & \mathcal G_{n+m}
\end{tikzcd}
\]
 is commutative for any $n\geq 0$ and $m\geq 1$.

This is a Grothendieck abelian, closed symmetric monoidal category with tensor product described in \cite{CD09_b}*{\S 7.3, \S 7.4} (together with \cite{MacLane98_b}*{Chapter~VII, \S 4, Exercise~6}).
Furthermore, we have a canonical adjunction of abelian categories:
\begin{equation}\label{eq:suspension}
\Sigma^\infty:\shMWtkR \leftrightarrows \spMWkR:\Omega^\infty
\end{equation}
such that $\Sigma^\infty(\mathcal F)=(\mathcal F\{n\})_{n \geq 0}$
with the obvious suspension maps
and $\Omega^\infty(\mathcal F_*,\epsilon_*)=\mathcal F_0$.
\index[notation]{sigmainf@$\Sigma^\infty$,$\Omega^\infty$}%
Recall the Tate MW-$t$-spectrum $\Sigma^\infty(\mathcal F)$
is called the \emph{infinite spectrum} associated with $\mathcal F$.
The functor $\Sigma^\infty$ is monoidal (cf.\ \cite{CD09_b}*{\S 7.8}).

One can define the $\Aone$-stable cohomology of a complex
$\EE=(\EE_*,\sigma_*)$ 
of Tate MW-$t$-spectra, for any smooth scheme $X$ and any
couple $(n,m) \in \ZZ^2$:
\begin{equation}\label{eq:stable_A1_cohomology}
\HstAone^{n,m}(X,\EE)
:=\varinjlim_{r \geq \max(0,-m)}
\Big(\Hom_{\DMtekR}(\tMot(X)\{r\},\EE_{m+r}[n])\Big)
\end{equation}
\index{stable!$\Aone$-cohomology}%
\index[notation]{hsta1nmxe@$\HstAone^{n,m}(X,\EE)$}%
where the transition maps are induced by the suspension maps $\sigma_*$ and $\tMot(X)\{r\}=\tMot(X)\otimes \tRcbx r$. 
\end{num}

\begin{dfn}
We say that a morphism $\varphi:\EE \rightarrow \FF$ of complexes of Tate
MW-$t$-spectra is a \emph{stable $\Aone$-equivalence}
\index{stable!$\Aone$-equivalence}%
if for any smooth scheme $X$
and any couple $(n,m) \in \ZZ^2$, the induced map
\[
\varphi_*:\HstAone^{n,m}(X,\EE)
\rightarrow \HstAone^{n,m}(X,\FF)
\]
is an isomorphism.

One defines the category $\DMtxkR t$
\index[notation]{dmtkr@$\DMtxkR t$}%
\index{Milnor-Witt!motivic!spectra}%
of \emph{MW-motivic spectra} for the topology $t$
as the localization of the triangulated category
$\Der(\spMWkR)$ with respect to stable $\Aone$-derived equivalences.
\end{dfn}

\begin{num}\label{num:twist-1}
As usual, we can describe the above category as the homotopy
 category of a suitable model category.

First, recall that we can define the negative twist of an abelian Tate
 MW-$t$-spectrum $\mathcal F_*$ by the formula, for $n>0$:
\begin{equation}\label{eq:twist-1}
\begin{split}
\lbrack(\mathcal F_*)\{-n\}\rbrack_m=
\begin{cases}
\ZZ[\mathfrak S_m] \otimes_{\ZZ[\mathfrak S_{m-n}]} \mathcal F_{m-n}
 & \text{if $m \geq n$,} \\
0 & \text{otherwise.}
\end{cases}
\end{split}
\end{equation}
\index[notation]{fn@$(\mathcal F_*)\{-n\}$}%
with the obvious suspension maps. Note for future references that
 one has according to this definition and that of the tensor product:
\begin{equation}\label{eq:twist-2}
\mathcal F_*\{-n\} =\mathcal F_* \otimes (\Sigma^{\infty}\tR)\{-n\}.
\end{equation}

We then define the class of cofibrations in $\Comp(\spMWkR)$
as the smallest class of morphisms of complexes closed under
suspensions, negative twists, pushouts,
transfinite compositions and retracts generated by
the infinite suspensions of cofibrations in $\Comp(\shMWtkR)$.

Applying \cite{CD09_b}*{Prop.~7.13}, we get:
\end{num}
\begin{prop}
The category $\Comp(\spMWkR)$ of complexes of Tate MW-$t$-spectra
has a symmetric monoidal model structure
\index{model structure!stable $\Aone$}%
with stable $\Aone$-equivalences as weak equivalences and
cofibrations as defined above.
This model structure is proper and cellular.

Moreover, the fibrations for this model structure
are epimorphisms of complexes whose kernel
is a complex $\EE$ such that:
\begin{itemize}
\item for any $n \geq 0$, $\EE_n$ is a $t$-flasque and $\Aone$-local
complex;
\item for any $n \geq 0$, the map obtained by adjunction from the
obvious suspension map:
\[
\EE_{n+1} \rightarrow \derR \uHom(\tRcbx 1,\EE_n)
\]
is an isomorphism.
\end{itemize}
\end{prop}
Therefore, the homotopy category $\DMtxkR t$ is a triangulated symmetric
monoidal category with internal Hom. The adjoint pair \eqref{eq:suspension}
can be derived and induces an adjunction of triangulated categories:
\[
\Sigma^\infty:\DMtexkR t \leftrightarrows \DMtxkR t:\Omega^\infty.
\]
\index[notation]{sigmainf@$\Sigma^\infty$,$\Omega^\infty$}%
As a left derived functor of a monoidal functor, the functor
 $\Sigma^\infty$ is monoidal.
Slightly abusing notation, we still denote by $\tMot(X)$ the MW-motivic spectrum $\Sigma^{\infty}(\tMot(X))$.

\begin{num}
By construction, the MW-motivic spectrum $\tRcbx 1$,
 and thus $\tRx 1$ is $\otimes$-invertible in $\DMtxkR t$
 (see \cite{CD09_b}*{Prop.~7.14}). Moreover,
 using formulas \eqref{eq:twist-1} and \eqref{eq:twist-2},
 one obtains a canonical map  in $\spMWkR$:
\[
\phi:\Sigma^\infty \tRcbx 1 \otimes \big((\Sigma^\infty \tRcbx{-1}\big)
 \rightarrow \big(\Sigma^\infty \tRcbx 1\big)\{-1\}
 \rightarrow \tR.
\]
The following proposition justifies the definition of Paragraph
 \ref{num:twist-1} of negative twists.
\end{num}
\begin{prop} \label{prop:negstableA1}
The map $\phi$ is a stable $\Aone$-equivalence.
 The MW-motive $(\Sigma^\infty \tR)\{-1\}$
 is the tensor inverse of $\tRcbx 1$.
 For any MW-$t$-spectra $\EE$,
 the map obtained by adjunction from $\EE \otimes \phi$:
\[
\EE\{-1\} \rightarrow \derR \uHom (\tRcbx 1,\EE)
\]
is a stable $\Aone$-equivalence.
\end{prop}
\begin{proof}
The first assertion follows from a direct computation using
 the definition of stable $\Aone$-equivalences via the 
 $\Aone$-stable cohomology \eqref{eq:stable_A1_cohomology}.
 The other assertions are formal consequences of the first one.
\end{proof}

As in the effective case, we derive from the functors of
Corollary \ref{cor:adjunctions_corr}
Quillen adjunctions for the stable $\Aone$-local model
structures and consequently adjunctions of triangulated categories:
\begin{equation}\label{eq:chg_top&tr_DM}
\begin{tikzcd}
[column sep=30pt,row sep=24pt]
\DAkR \ar[r,shift left=2pt,"{\derL \tilde \gamma^*}"] \ar[d,shift left=2pt,"a"]
 & \DMtkR \ar[r,shift left=2pt,"{\derL \pi^*}"] \ar[d,shift left=2pt,"{\tilde a}"]
     \ar[l,shift left=2pt,"{\tilde \gamma_{*}}"]
 & \DMkR \ar[d,shift left=2pt,"{a^{tr}}"]
     \ar[l,shift left=2pt,"{\pi_{*}}"] \\
\DAxkR{\et} \ar[r,shift left=2pt,"{\derL \tilde \gamma^*_\et}"]
    \ar[u,shift left=2pt,"{\derR \fO}"]
 & \DMtxkR{\et} \ar[r,shift left=2pt,"{\derL \pi^*_\et}"]
	  \ar[u,shift left=2pt,"{\derR \tfO}"]
		\ar[l,shift left=2pt,"{\tilde \gamma_{\et*}}"]
 & \DMxkR{\et}
    \ar[u,shift left=2pt,"{\derR \fO}"]
		\ar[l,shift left=2pt,"{\pi_{\et*}}"]
\end{tikzcd}
\end{equation}
where each left adjoints is monoidal and sends a motive of a smooth
 scheme to the corresponding motive.

As a consequence of Corollary~\ref{cor:adjunctions_corr}\ref{item:adjunction-pi-equivalence},
we immediately get the following result:
\begin{prop}
\label{prop:compare_DMt_DM}%
($k$ is an arbitrary perfect field.)
The adjunction defined above 
\[
\derL \pi^*_\et=\pi^*_\et:\DMtxkR{\et} \leftrightarrows \DMxkR{\et}:\pi_{\et*}
\]
is an equivalence of triangulated monoidal categories.
\end{prop}

Formally, one can compute morphisms  of MW-motivic spectra
 as follows.
\begin{prop} \label{prop:Homlimiso}
For any smooth scheme $X$, any pair of integers $(n,m) \in \ZZ^2$
and any MW-motivic spectrum $\EE$, one has a canonical
 functorial isomorphism:
\begin{align*}
\Hom_{\DMtkR}(\tMot(X),\EE(n)[m])
 &\simeq \HstAone^{n,m}(X,\EE) \\
 &=\varinjlim_{r \geq \max(0,-m)}
 \Big(\Hom_{\DMtekR}(\tMot(X)\{r\},\EE_{m+r}[n])\Big).
\end{align*}
\end{prop}
\begin{proof}
This follows from \cite{Ayoub07_b}*{4.3.61 and 4.3.79}
 which can be applied because of Remark \ref{rem:compact_gen_DMte}
 and the fact the cyclic permutation of order 3 acts on
 $\tRcbx 3$ as the identity in $\DMtekR$ (the proof of this fact is postponed until Corollary \ref{cor:permut_Tate}).
\end{proof}

In fact, as in the case of motivic complexes, one gets a better result
 under more restrictive assumptions on the base field $k$.
 This is due to the following
 analogue of Voevodsky's cancellation theorem \cite{Voe10_b}, 
 which is proved in \chcancellation.
\begin{thm}\label{thm:cancellation}
Let $k$ be a perfect infinite field of characteristic not $2$. Then for any complexes
 $K$ and $L$ of MW-sheaves, the morphism 
\[
\Hom_{\DMtekR}(K,L) \rightarrow \Hom_{\DMtekR}(K(1),L(1)),
\]
obtained by tensoring with the Tate twist, is an isomorphism.
\end{thm}

We then formally deduce the following corollary from this result.

\begin{coro}\label{cor:cancellation}
If $k$ is an infinite perfect field of characteristic not $2$, the functor
\[
\Sigma^\infty:\DMtekR \rightarrow \DMtkR
\]
is fully faithful.
\end{coro}

\begin{rem}
According to Proposition~\ref{prop:compare_DMt_DM} and Voevodsky's cancellation theorem,
this theorem holds for an arbitrary perfect field if we use the étale topology instead
of the Nisnevich one. 
\end{rem}

\section{Milnor-Witt motivic cohomology}\label{sec:MWcohom}

\subsection{Milnor-Witt motivic cohomology as Ext-groups}

Given our construction of the triangulated category $\DMtkR$,
we can now define, in the style of Beilinson,
a generalization of motivic cohomology as follows.
 
\begin{dfn}\label{def:generalizedMW}
We define the \emph{MW-motivic cohomology} of a smooth scheme $X$
in degree $(n,i) \in \ZZ^2$ and coefficients in $R$ as:
\[
\HMW^{n,i}(X,R)=\Hom_{\DMtkR}(\tMot(X),\tRx i[n]).
\]
\index{Milnor-Witt!motivic!cohomology}%
\index[notation]{hmwnmxr@$\HMW^{n,m}(X,R)$}%
\end{dfn}
As usual, we deduce a cup-product on MW-motivic cohomology.
\index{Milnor-Witt!motivic!cohomology cup-product}%
We define its étale variant by taking morphisms in $\DMtxkR{\et}$.
Then we derive from the preceding (essentially) commutative diagram the
following morphisms of cohomology theories, all compatible with products
and pullbacks:
\begin{equation}\label{eq:MW-regulators}
\begin{tikzcd}
[row sep=14pt,column sep=14pt]
\H_{\Aone}^{n,i}(X,R) \ar[r] \ar[d] & \HMW^{n,i}(X,R) \ar[r] \ar[d]
 & \HM^{n,i}(X,R) \ar[d] \\
\H_{\Aone,\et}^{n,i}(X,R) \ar[r] & \HMWet^{n,i}(X,R) \ar[r]
 & \HMet^{n,i}(X,R).
\end{tikzcd}
\end{equation}
where $\H_{\Aone}(X,R)$ and $\H_{\Aone,\et}(X,R)$
\index[notation]{ha1nmxr@$\H_{\Aone}^{n,m}(X,R)$}%
\index[notation]{ha1etnmxr@$\H_{\Aone,\et}^{n,m}(X,R)$}%
are respectively Morel's stable $\Aone$-derived cohomology and its étale version while $\HM^{n,i}(X,R)$ and $\HMet^{n,i}(X,R)$ are respectively the motivic cohomology and the étale motivic cohomology (also called Lichtenbaum motivic cohomology).

Gathering all the information we have obtained in the previous
section on MW-motivic complexes, we get the following computation.
 
\begin{prop}\label{prop:explicit}
Assume that $k$ is an infinite perfect field of characteristic not $2$.
For any smooth scheme $X$ and any couple of integers $(n,m) \in \ZZ^2$,
the MW-motivic cohomology $\HMW^{n,m}(X,\ZZ)$ defined previously coincides
with the generalized motivic cohomology defined in \chfinitecw.

More explicitly,
\[
\HMW^{n,m}(X,\ZZ)=
\begin{cases}
\H^n_\zar(X,\tZpx m) & \text{if $m>0$,} \\
\H^n_\zar(X,\sKMW_0) & \text{if $m=0$,} \\
\H^{n-m}_\zar(X,\sW) & \text{if $m<0$.}
\end{cases}
\]
where $\sKMW_0$ (resp.\ $\sW$) is the unramified 
sheaf associated with Milnor-Witt K-theory in degree $0$ 
(resp.\ unramified Witt sheaf) --- see \cite{Morel12_b}*{\textsection 3}.
\end{prop}

\begin{proof}
The cases $m>0$ and $m=0$ are clear from the previous corollary
and Corollary \ref{cor:W-motivic&generalized}.

Consider the case $m<0$. Then we can use the following computation:
\begin{align*}
\HMW^{n,m}(X,\ZZ)&=\Hom_{\DMt{k,\ZZ}}\big(\tMot(X),\tZcbx m[n-m]\big) \\
&=\Hom_{\DMt{k,\ZZ}}\big(\tMot(X)\otimes \tZcbx{-m},\tZ[n-m]\big) \\
&=\Hom_{\DMtekZ}\big(\tMot(X),\derR\uHom(\tZcbx{-m},\tZ)[n-m]\big)
\end{align*}
where the last identification follows from the Corollary~\ref{cor:cancellation} and the usual adjunction.

As the MW-motivic complex $\tZcbx{-m}$ is cofibrant
and the motivic complex $\tZ=\sKMW_0$ is Nisnevich-local
and $\Aone$-invariant (cf.\ \chfinitecw, Ex.~\ref{ex:basic} and \cite{Fasel08_b}*{Cor.~11.3.3}),
we get:
\[
\derR\uHom(\tZcbx{-m},\tZ)=\uHom(\tZcbx{-m},\tZ)
\]
and this last sheaf is isomorphic to $\sW$ according to
\chfinitecw, \ref{exm:negativecohomology}. So the assertion now follows from
Corollaries \ref{cor:A1-local_complexes} and \ref{cor:compare_Hom&cohomology}.
\end{proof}

\begin{num}
We next prove a commutativity result for MW-motivic cohomology. 
First, note that the sheaf 
$\tRcbx 1=\tR(\Gmk)/\tRcbx 1$ is a direct factor of 
$\tR(\Gmk)$ and that the permutation map
\[
\sigma:\tR(\Gmk)\otimes \tR(\Gmk)\to \tR(\Gmk)\otimes 
\tR(\Gmk)
\]
given by the morphism of schemes $\Gmk\times \Gmk\to \Gmk\times \Gmk$ 
defined by $(x,y)\mapsto (y,x)$ induces a map
\[
\sigma:\tRcbx 1 \otimes \tRcbx 1\to \tRcbx 1\otimes \tRcbx 1.
\] 
On the other hand, recall from Remark \ref{rem:finite_corr&plim} (5), that 
$\cork$ is $\sKMW_0(k)$-linear. In particular, the class of $\epsilon=-\langle 
-1\rangle\in \sKMW_0(k)$ (and its corresponding element in 
$\sKMW_0(k)\otimes_\ZZ R$ that we still denote by $\epsilon$) yields a 
MW-correspondence 
\[
\epsilon=\epsilon\cdot \id:\MWrepR(\Gmk)\otimes \MWrepR(\Gmk)\to 
\MWrepR(\Gmk)\otimes \MWrepR(\Gmk)
\]
that also induces a MW-correspondence
\[
\epsilon:\tRcbx 1 \otimes \tRcbx 1 \to \tRcbx 1 \otimes \tRcbx 1.
\]
We can now state the following lemma (\chcancellation, Lemma~\ref{lem:hopf}).
\end{num}
\begin{lem}
The MW-correspondences $\sigma$ and $\epsilon$ are $\Aone$-homotopic.
\end{lem} 

As an obvious corollary, we obtain the following result.

\begin{coro}\label{cor:permut_Tate}
For any $i,j \in \ZZ$, the switch 
$\tRx i\otimes \tRx j\to \tRx j \otimes \tRx i$ 
is $\Aone$-homotopic to $\langle (-1)^{ij}\rangle$. 
\end{coro}

\begin{proof}
By definition, we have $\tRx i:=\tRcbx i[-i]$ and 
$\tRx j:=\tRcbx j[-j]$. We know from the previous lemma that the 
switch $\tRcbx i \otimes \tRcbx j \to \tRcbx j \otimes \tRcbx i$ is homotopic to $\epsilon^{ij}$. 
The result now follows from the compatibility isomorphisms for tensor triangulated categories (see e.g. \cite{Mazza06_b}*{Exercise~8A.2}) and the fact that 
$(-1)^{ij}\epsilon^{ij}=\langle (-1)^{ij}\rangle$.
\end{proof}

\begin{thm}\label{thm:commutative}
Let $i,j\in \ZZ$ be integers. For any smooth scheme $X$, the pairing
\[
\HMW^{p,i}(X,R)\otimes 
\HMW^{q,j}(X,R)\to \HMW^{p+q,i+j}(X,R)
\]
is $(-1)^{pq}\langle (-1)^{ij}\rangle$-commutative.
\end{thm}

\begin{proof}
The proof of \cite{Mazza06_b}*{Theorem 15.9} applies mutatis mutandis.
\end{proof}

\subsection{Comparison with Chow-Witt groups}

\subsubsection{The naive Milnor-Witt presheaf}
\label{sec:naiveMilnorWittpresheaf}
Let $S$ be a ring and let $S^\times$ be the group of units in $S$. We define the naive Milnor-Witt presheaf of $S$ as in the case of fields by considering the free $\ZZ$-graded algebra $A(S)$ generated by the symbols $[a]$ with $a\in S^\times$ in degree $1$ and a symbol $\eta$ in degree $-1$ subject to the usual relations:

\begin{enumerate}
\item 
\label{item:Milnor-Witt-multiply}%
$[ab]=[a]+[b]+\eta [a][b]$ for any $a,b\in S^\times$.
\item $[a][1-a]=0$ for any $a\in S^\times$ such that $1-a\in S^\times$.
\item $\eta [a]=[a]\eta$ for any $a\in S^\times$.
\item $\eta (\eta [-1]+2)=0$.
\end{enumerate}

This definition is functorial in $S$ and it follows that we obtain a presheaf of $\ZZ$-graded algebras on the category of smooth schemes via
\[
X\mapsto \KMW_*(\OO(X)).
\]
We denote by $\sKMW_{*,\text{naive}}$
\index[notation]{kmwnaive@$\sKMW_{*,\text{naive}}$}%
the associated Nisnevich sheaf of graded $\ZZ$-algebras and observe that this definition naturally extends to essentially smooth $k$-schemes. Our next aim is to show that this naive definition in fact coincides with the definition of the unramified Milnor-Witt $K$-theory sheaf given in \cite{Morel12_b}*{\S 3} (see also \chfinitecw, \S\ref{sec:MWKtheory}). Indeed, let $X$ be a smooth integral scheme. The ring homomorphism $\OO(X)\to k(X)$ induces a ring homomorphism $\KMW_*(\OO(X))\to \KMW_*(k(X))$ and it is straightforward to check that elements in the image are unramified, i.e.\ that the previous homomorphism induces a ring homomorphism $\KMW_*(\OO(X))\to \sKMW_*(X)$. By the universal property of the associated sheaf, we obtain a morphism of sheaves
\[
\sKMW_{*,\mathrm{naive}}\to \sKMW_*.
\]
If $X$ is an essentially smooth local $k$-scheme, it follows from \cite{Gille15_b}*{Theorem~6.3} that the map $\sKMW_{*,\mathrm{naive}}(X)\to \sKMW_*(X)$ is an isomorphism, showing that the above morphism is indeed an isomorphism.

\begin{num}\textbf{A comparison map}. \label{num:comparisonmap}
Let now $X$ be a smooth connected scheme and let $a\in \OO(X)^\times$ be an invertible global section. It corresponds to a morphism $X\to \Gmk$ and in turn to an element in $\cork(X,\Gmk)$ yielding a map 
\[
s:\OO(X)^\times\to \Hom_{\DMtek}(\tMot(X),\tZcbx 1)=\HMW^{1,1}(X,\ZZ).
\]
Consider next the correspondence $\eta[t]\in \ch 0{X\times \Gmk}=\cork (X\times \Gmk,\Spec k)$ and observe that it restricts trivially when composed with the map $X\to X\times \Gmk$ given by $x\mapsto (x,1)$. It follows that we obtain an element 
\[
s(\eta)\in \Hom_{\DMtek}(\tMot(X)\otimes \tZcbx 1,\tZ)=\Hom_{\DMt{k,\ZZ}}(\tMot(X),\tZcbx {-1})=\HMW^{-1,-1}(X,\ZZ).
\]
Using the product structure of the cohomology ring, we finally obtain a (graded, functorial in $X$) ring homomorphism 
\[
s:A(\OO(X))\to \bigoplus_{n\in \ZZ}\HMW^{n,n}(X,\ZZ),
\]
where $A(\OO(X))$ is the free $\ZZ$-graded (unital, associative) algebra generated in degree $1$ by the elements $s(a)$ and in degree $-1$ by $s(\eta)$. 
The following theorem does not need the assumption that $k$ is infinite and of characteristic not $2$.
\end{num}

\begin{thm}
\label{thm:gradedringhomo}%
Let $X$ be a smooth scheme. Then, the graded ring homomorphism
\[
s:A(\OO(X))\to \bigoplus_{n\in \ZZ}\HMW^{n,n}(X,\ZZ)
\]
induces a graded ring homomorphism
\[
s:\KMW_*(\OO(X))\to \bigoplus_{n\in \ZZ}\HMW^{n,n}(X,\ZZ)
\] 
which is functorial in $X$.
\end{thm}

\begin{proof}
We have to check that the four relations defining Milnor-Witt $K$-theory hold in the graded ring on the right-hand side. First, note that Theorem \ref{thm:commutative} yields $\epsilon s(\eta)s(a)=s(a)s(\eta)$ and the third relation follows from the fact that $\epsilon s(\eta)=s(\eta)$ by construction. Observe next that $s(\eta)s(-1)+1=\langle -1\rangle=-\epsilon$ by \chcancellation, Lemma~\ref{lem:example1} and it follows immediately that 
\[
s(\eta)(s(\eta)s(-1)+2)=s(\eta)(1-\epsilon)=0.
\] 
Next, consider the multiplication map
$
m:\Gmk\times \Gmk\to \Gmk
$
and the respective projections on the $j$-th factor
\[
p_j:\Gmk\times \Gmk\to \Gmk
\]
for $j=1,2$.
They all define correspondences that we still denote by the same symbols and it is straightforward to check that $m-p_1-p_2$ defines a morphism $\tZcbx{1}\otimes \tZcbx{1}\to \tZcbx{1}$ in $\cork$. It follows from \chcancellation, Lemma~\ref{lem:example2} that this correspondence corresponds to $s(\eta)$ through the isomorphism
\[
\Hom_{\DMtek}(\tZcbx{1},\tZ)\to \Hom_{\DMtek}(\tZcbx{1}\otimes \tZcbx{1},\tZcbx{1})
\]
given by the cancellation theorem. As a corollary, we see that the defining relation \ref{item:Milnor-Witt-multiply} of Milnor-Witt K-theory is satisfied in $\bigoplus_{n\in \ZZ}\HMW^{n,n}(X,\ZZ)$. Indeed, if $a,b\in \OO(X)^\times$, then $s(a)s(b)$ is represented by the morphism $X\to \Gmk\times \Gmk$ corresponding to $(a,b)$. Applying $m-p_1-p_2$ to this correspondence, we get $s(ab)-s(a)-s(b)$ which is $s(\eta)s(a)s(b)$ by the above discussion.

To check that the Steinberg relation holds in the right-hand side, we first consider the morphism
\[
\Aone\setminus \{0,1\}\to \Aone\setminus \{0\}\times \Aone\setminus \{0\}
\]
defined by $a\mapsto (a,1-a)$. Composing with the correspondence $\tMot(\Aone\setminus \{0\}\times \Aone\setminus \{0\})\to \tMot(\Gmk^{\wedge 2})$, we obtain a morphism
\[
\tMot(\Aone\setminus \{0,1\})\to \tMot(\Gmk^{\wedge 2}).
\]
We can perform the same computation in $\DAek$ where this morphism is trivial by \cite{Hu01_b} and we conclude that it is also trivial in $\DMtek$ by applying the functor $\derL\tilde\gamma^*$.
\end{proof}

For any $p,q\in \ZZ$, we denote by $\shHMW^{p,q}$ the (Nisnevich) sheaf associated to the presheaf $X\mapsto \HMW^{p,q}(X,\ZZ)$. 
\index[notation]{hmwpwxz@$\shHMW^{p,q}$}%
\index{Milnor-Witt!motivic!cohomology sheaf}%
The homomorphism of the previous theorem induces a morphism on induced sheaves and we have the following result.

\begin{thm}\label{thm:KMWmotivic}
Assume $k$ is an infinite perfect field of characteristic not $2$.
Then the homomorphism of sheaves of graded rings
\[
s:\sKMW_*\to \bigoplus_{n\in \ZZ}\shHMW^{n,n}
\]
is an isomorphism.
\index{comparison theorem}%
\end{thm}

\begin{proof}
Let $L/k$ be a finitely generated field extension. Then, it follows from \chcomparison, Theorem~\ref{thm:comparison} 
and from the cancellation theorem \ref{thm:cancellation}
that the homomorphism $s_L$ is an isomorphism. Now, the presheaf on $\cork$ given by $X\mapsto \bigoplus_{n\in \ZZ}\HW^{n,n}(X,\ZZ)$ is homotopy invariant by definition. It follows from Theorem \ref{thm:A1local_framedPSh} that the associated sheaf is strictly $\Aone$-invariant. Now, $\sKMW_*$ is also strictly $\Aone$-invariant and it follows from \cite{Morel12_b}*{Definition~2.1, Remark~2.3, Theorem~2.11} that $s$ is an isomorphism if and only if $s_L$ is an isomorphism for any finitely generated field extension $L/k$.
\end{proof}

\begin{thm}\label{thm:hyper}
Assume $k$ is an infinite perfect field of characteristic not $2$.
For any smooth scheme $X$ and any integers $p,n\in \ZZ$, the hypercohomology spectral sequence induces isomorphisms
\[
\HMW^{p,n}(X,\ZZ)\to \H^{p-n}(X,\sKMW_n)
\]
provided $p\geq 2n-1$. 
\end{thm}

\begin{proof}
In view of Proposition \ref{prop:explicit}, we may suppose that $n>0$. For any $q\in\ZZ$, observe that the Nisnevich sheaf $\shHMW^{q,n}$ associated to the presheaf $X\mapsto \HMW^{q,n}(X,\ZZ)$ coincides with the corresponding cohomology sheaf of the complexe $\tZpx{n}$. Now, the latter is concentrated in cohomological levels $\leq n$ and it follows that $\shHMW^{q,n}=0$ if $q>n$. On the other hand, the sheaves $\shHMW^{q,n}$ are strictly $\Aone$-invariant, and as such admit a Gersten complex whose components in degree $m$ are of the form 
\[
\bigoplus_{x\in X^{(p)}}(\shHMW^{q,n})_{-p}(k(x),\wedge^p (\mathfrak m_x/\mathfrak m_x^2)^*)
\] 
by \cite{Morel12_b}*{\S 5}. By the cancellation theorem \ref{thm:cancellation}, we have a canonical isomorphism of sheaves $\shHMW^{q-p,n-p}\simeq (\shHMW^{q,n})_{-p}$ and it follows that the terms in the Gersten resolution are of the form 
\[
\bigoplus_{x\in X^{(p)}}(\shHMW^{q-p,n-p})(k(x),\wedge^p (\mathfrak m_x/\mathfrak m_x^2)^*).
\] 
If $p\geq n$, then $\tZpx{n-p}\simeq \sKMW_{n-p}[p-n]$ and it follows that $\shHMW^{q-p,n-p}$ is the sheaf associated to the presheaf $X\mapsto \H^{q-n}(X,\sKMW_{n-p})$, which is trivial if $q\neq n$. Altogether, we see that 
\[
\H^p(X,\shHMW^{q,n})=
\begin{cases}
0 & \text{if $q>n$,} \\
0 & \text{if $p\geq n$ and $q\neq n$.}
\end{cases} 
\]
We now consider the hypercohomology spectral sequence for the complex $\tZpx{n}$ (\cite{SV00_b}*{Theorem~0.3}) $E_2^{p,q}:=\H^p(X,\shHMW^{q,n})\implies \HMW^{p+q,n}(X,\ZZ)$ which is strongly convergent since Nisnevich cohomology groups vanish above the dimension of $X$. Our computations of the sheaves $\shHMW^{q,n}$ immediately imply that 
\[
\H^{p-n}(X,\shHMW^{n,n}) = \HMW^{p,n}(X,\ZZ) \text{ if $p\geq 2n-1$}. 
\]
We conclude using Theorem \ref{thm:KMWmotivic}.
\end{proof}

\begin{rem}
The isomorphisms $\HMW^{p,n}(X,\ZZ)\to \H^{p-n}(X,\sKMW_n)$ are functorial in $X$. Indeed, the result comes from the analysis of the hypercohomology spectral sequence for the complexes $\tZpx{n}$, which are functorial in $X$.
\end{rem}

Setting $p=2n$ in the previous theorem, and using the fact that $\H^{n}(X,\sKMW_n)=\ch nX$ by definition (for $n\in\NN$), we get the following corollary.

\begin{coro}\label{cor:ChowWitt}
Assume $k$ is an infinite perfect field of characteristic not $2$.
For any smooth scheme $X$ and any $n\in \NN$, the hypercohomology spectral sequence induces isomorphisms 
\[
\HMW^{2n,n}(X,\ZZ)\to \ch nX.
\]
\end{coro}

\begin{rem}\label{rem:preThom}
Both Theorems \ref{thm:hyper} and Corollary \ref{cor:ChowWitt} are still valid if one considers cohomology with support on a closed subset $Y\subset X$, i.e.\ the hypercohomology spectral sequence (taken with support) yields an isomorphism
\[
\HMWs{Y}^{p,n}(X,\ZZ)\to \H^{p-n}_Y(X,\sKMW_n)
\]
provided $p\geq 2n-1$.

Let now $E$ be a rank $r$ vector bundle over $X$, $s:X\to E$ be the zero section and $E^0=E\setminus s(X)$. The Thom space of $E$ is the object of $\DMtekZ$ defined by 
\[
\Th(E)=\Sigma^\infty\tMot(\tZ(E)/\tZ(E^0)).
\] 
\index[notation]{the@$\Th(E)$}%
\index{Thom space}%
It follows from Corollary \ref{cor:compare_Hom&cohomology} and \cite{Asok16_b}*{Proposition~3.13} that (for $n\in\NN$)
\begin{align*}
\Hom_{\DMtekZ}(\Th(E),\tZpx{n}[2n]) & \simeq \Hom_{\DMtekR}(\tMot(\tZ(E)/\tZ(E^0)),\tZpx{n}[2n]) \\
& \simeq \HMWs{X}^{2n,n}(E,\ZZ).
\end{align*}
Using the above result, we get $\Hom_{\DMtekZ}(\Th(E),\tZpx{n}[2n])\simeq  \chs nXE$. Using finally the Thom isomorphism (\cite{Morel12_b}*{Corollary~5.30} or \cite{Fasel08_b}*{Remarque~10.4.8})
\[
\cht {n-r}X{\det(E)}\simeq \chs nXL,
\]
we obtain an isomorphism
\[
\Hom_{\DMtekZ}(\Th(E),\tZpx{n}[2n])\simeq \cht{n-r}X{\det(E)}
\]
which is functorial (for schemes over $X$).
\end{rem}

\begin{rem}
The isomorphisms of Corollary \ref{cor:ChowWitt} induce a ring homomorphism
\[
\bigoplus_{n\in\NN}\HMW^{2n,n}(X,\ZZ)\to \bigoplus_{n\in\NN} \ch nX.
\]
This follows readily from the fact that the isomorphism of Theorem \ref{thm:KMWmotivic} is an isomorphism of graded rings.
\end{rem}

\section{Relations with ordinary motives}
\label{sec:relordmot}

Our aim in this section is to show that both the categories $\DMtekR$ and $\DMtkR$ split into two factors when $2\in R^\times$, one of the factors being the corresponding category of ordinary motives. We assume that $R=\ZZ[1/2]$, the general case being obtained from this one. To start with, let $X$ be a smooth $k$-scheme and let $\Lb$ be a line bundle over $X$. On the small Nisnevich site of $X$, we have a Cartesian square of sheaves of graded abelian groups (\chfinitecw, \S\ref{sec:MWKtheory})
\[
\begin{tikzcd}
\sKMW_*(\Lb) \ar[r] \ar[d] & \sI^*(\Lb) \ar[d] \\
\sKM_* \ar[r] & \sIbar^*(\Lb)
\end{tikzcd}
\]
where $\sIbar^*(\Lb)$ is the sheaf associated to the presheaf $\sI^*(\Lb)/\sI^{*+1}(\Lb)$. Observe that $\sIbar^*(\Lb)$ is in fact independent of $\Lb$ (\cite{Fasel08_b}*{Lemme~E.1.2}), and we will routinely denote it by $\sIbar^*$ below. Next, observe that $\langle 1,1\rangle \sI^*(Y)\subset \sI^{*+1}(Y)$ for any smooth scheme $Y$, and it follows that $\sIbar^*$ is a $2$-torsion sheaf. Inverting $2$, we then obtain an isomorphism 
\begin{equation}\label{eq:splitting}
\sKMW_*(\Lb)[1/2]\simeq \sI^*(\Lb)[1/2]\times \sKM_*[1/2].
\end{equation}
This decomposition can be more concretely seen as follows. In $\sKMW_0(X)[1/2]$ (no line bundle here), we may write 
\[
1=(1+\langle -1\rangle)/2+(1-\langle -1\rangle)/2.
\]
We observe that both $e:=(1+\langle -1\rangle)/2$ and $1-e=(1-\langle -1\rangle)/2$ are idempotent, and thus decompose $\sKMW_0(X)[1/2]$ as 
\[
\sKMW_0(X)[1/2]\simeq \sKMW_0(X)[1/2]/e\oplus \sKMW_0(X)[1/2]/(1-e).
\]
Now, $\sKMW_0(X)[1/2]/e=\sW(X)[1/2]$ (as $2e=2+\eta[-1]:=h$), while the relation $\eta h=0$ (together with $h=2$ modulo $(1-e)$) imply $\sKMW_0(X)[1/2]/(1-e)=\sKM_0(X)=\ZZ$. As $\sKMW_*(\Lb)[1/2]$ is a sheaf of $\sKMW_*$-algebras, we find
\[
\sKMW_*(\Lb)[1/2]\simeq \sKMW_*(\Lb)[1/2]/e\oplus \sKMW_*(\Lb)[1/2]/(1-e)
\]
and 
\[
\sKMW_*(\Lb)[1/2]/e=\sKMW_*(\Lb)[1/2]/h=\sKW_*(\Lb)[1/2]
\]
where $\sKW(\Lb)$ is the Witt $K$-theory sheaf discussed in \cite{Morel04_b} and 
\[
\sKMW_*(\Lb)[1/2]/(1-e)=\sKM_*[1/2].
\]
This splitting has very concrete consequences on the relevant categories of motives. To explain them in their proper context, recall first from \chfinitecw, \ref{rem:wittcor} that one can introduce the category of finite W-correspondences as follows.

For smooth (connected) schemes $X$ and $Y$, let $\Wcork (X,Y)$ be the abelian group
\[
\Wcork(X,Y):=\varinjlim_{T\in \Adm(X,Y)} \H^{d_Y}_T\!\big(X \times Y,\sW,\omega_{Y}\big)
\] 
\index[notation]{wcork@$\Wcork$}%
\index{Witt!finite correspondence}%
where $\Adm(X,Y)$ is the poset of admissible subsets of $X\times Y$ (\chfinitecw, Definition~\ref{dfn:admissible_subset}), $d_Y$ is the dimension of $Y$ and $\omega_{Y}$ is the pull-back along the second projection of the canonical module on $Y$. We let $\Wcork$ be the category whose objects are smooth schemes and whose morphisms are given by the above formula. The results of \chfinitecw apply mutatis mutandis, showing in particular that $\Wcork$ is an additive category endowed with a tensor product. Moreover the results of the present paper also apply, allowing to use the main theorems in this new framework. In particular, we can build the category of effective W-motives along the same lines as those used above, and its stable version as well. 

\begin{dfn}
For any ring $R$, we denote by $\WDMekR$ the category of effective W-motives, i.e.\ the full subcategory of $\Aone$-local objects of the derived category of W-sheaves. We denote by $\WDMkR$ the category obtained from the previous one by inverting the Tate object. 
\index[notation]{wdmeffkr@$\WDMekR$}%
\index{Witt!motives, category of!effective}%
\index[notation]{wdmrk@$\WDMkR$}%
\index{Witt!motives, category of}%
\index{motives!Witt-|see{Witt motives}}%
\end{dfn}

The relations with the categories previously built can be described as follows. Observe that by definition we have 
\[
\Wcork(X,Y):=\varinjlim_{T\in \Adm(X,Y)} \H^{d_Y}_T\!\big(X \times Y,\sW,\omega_{Y}\big)=\varinjlim_{T\in \Adm(X,Y)} \H^{d_Y}_T\!\big(X \times Y,\sI^{d_Y},\omega_{Y}\big)
\] 
for any smooth schemes $X$ and $Y$. The morphism of sheaves $\sKMW_{d_Y}(\omega_Y)\to \sI^{d_Y}(\omega_Y)$ thus yields a well defined functor 
\[
\beta:\cork\to \Wcork.
\]
This functor in turn induces a functor $\beta_*$ between the categories of presheaves and sheaves (in either the Nisnevich or the étale topologies), yielding finally (exact) functors 
\[
\WDMekR\xrightarrow{\beta_*}\DMtekR
\]
and 
\[
\WDMkR\xrightarrow{\beta_*}\DMtkR.
\]
Note that both functors $\beta_*$ admit monoidal left adjoints $\derL \beta^*$ preserving representable objects.

\begin{thm}
Suppose that $2\in R^\times$. We then have equivalences of categories
\[
(\beta_*,\pi_*):\WDMekR\times \DMekR \to  \DMtekR
\]
and
\[
(\beta_*,\pi_*):\WDMkR\times \DMkR \to \DMtkR.
\]
\end{thm}

\begin{proof}
The functors $\beta$ and $\pi$ are constructed using the isomorphism of sheaves (\eqref{eq:splitting}) 
\[
\sKMW_*(\Lb)[1/2]\simeq \sI^*(\Lb)[1/2]\times \sKM_*[1/2]
\]
together with the first and second projection. We can construct functors in the other direction by using the inclusion into the relevant factor and then the inverse of the above isomorphism. The result is then clear.
\end{proof}

\begin{rem}
Still under the assumption that $2\in R^\times$, the above splitting induces isomorphisms
\[
\HMW^{p,q}(X,R)\simeq \HW^{p,q}(X,R)\times \HM^{p,q}(X,R)
\]
functorial in $X$. Here, $\HW^{p,q}(X,R)$ is the W-motivic cohomology defined in the category $\WDMekR$. Moreover, this isomorphism is compatible with pull-backs, push-forwards and products.
\end{rem}

\renewcommand{\theequation}{\arabic{equation}}

\chapter[Cancellation theorem]{A cancellation theorem for Milnor-Witt correspondences\author{Jean Fasel and Paul-Arne {\O}stv{\ae}r}}
\label{ch:cancellation}

\section*{Abstract}
We show that finite Milnor-Witt correspondences satisfy a cancellation theorem with respect to the pointed multiplicative group scheme.
This has several notable applications in the theory of Milnor-Witt motives and Milnor-Witt motivic cohomology.

\section*{Introduction}

The notion of finite Milnor-Witt correspondences was introduced in \chfinitecw and \chdmt as a natural generalization of finite correspondences in the sense \cite{FSV00_c}.
Intuitively one may view a Milnor-Witt correspondence as an ordinary finite correspondence together with a well behaved quadratic form over the function field of each irreducible component of the support of the correspondence.
This point of view allows to translate many results in the classical setting into this new framework.
In this paper we show the following result, Theorem \ref{thm:main} below, which is a cancellation theorem for Milnor-Witt correspondences. 

\begin{thm*}
Let $k$ be an infinite perfect field of characteristic not $2$. 
Then for all complexes $\mathscr C$ and $\mathscr D$ of Milnor-Witt sheaves, 
tensoring with the Tate object $\tZpx 1$
\index[notation]{zt1@$\tZpx 1$}%
\index{Milnor-Witt!motive!Tate}%
\index{cancellation theorem}%
yields an isomorphism 
\begin{equation}
\label{cancellationisomorphism}
\Hom_{\DMtekZ}(\mathscr C,\mathscr D)
\overset{\simeq}{\to} \Hom_{\DMtekZ}(\mathscr C(1),\mathscr D(1)).
\end{equation}
\end{thm*}

Here, $\DMtekZ$ is the analogue in this new framework of Voevodsky's category of motives $\DMe{k}$ (see \chdmt). 
This result extends Voevodsky's pioneering work on the cancellation theorem for finite correspondences in \cite{Voe10_c}.
It basically asserts a suspension isomorphism in Milnor-Witt motivic cohomology with respect to the multiplicative group scheme $\Gm$ viewed as an algebro-geometric sphere;
the Tate object $\tZpx 1$ is defined as $\MWrepZ(\PP^1)/\MWrepZ(\infty)[-2]$ in \chdmt, \S \ref{num:Tate}.

\subsection*{Notation}
Throughout the paper we work over a perfect field $k$ of $\charac k\neq 2$, except in some specific places where it will be explicitly stated.
We denote by $\smk$ the category of smooth separated schemes over $\Spec(k)$.
For a natural number $i\in \NN$, an integer $j\in \ZZ$, and a line bundle $\Lb$ on $X\in\smk$, 
we consider exterior powers of tangent bundles and set  
\[
\Cger^i(X,\sKMW_j,\Lb):=\bigoplus_{x\in X^{(i)}} \sKMW_{j-i}(k(x),\wedge^i(\mathfrak m_x/\mathfrak m_x^2)^\vee\otimes_{k(x)} \Lb_x).
\]
\index{Gersten complex}%
\index[notation]{cixkmwjl@$\Cger_i(X,\sKMW_j,\Lb)$}%
Here, 
$\sKMW_{\ast}$ is the Nisnevich sheaf of unramified Milnor-Witt $K$-theory as defined in \cite{Morel12_c}*{\S 3}, 
and $\Lb_x$ is the stalk of $\Lb$ at a point $x\in X^{(i)}$ of codimension $i$. To deal with sign issues, it is more convenient to work with graded line bundles as in \chfinitecw, \S\ref{sec:twisting}.
We note that $\wedge^i(\mathfrak m_x/\mathfrak m_x^2)^\vee$ is a $1$-dimensional $k(x)$-vector space for every such $x$, which is seen as a graded line bundle concentrated in degree $i$. 
The grading of $\Lb$ depends on the context, and will be indicated when necessary.
By means of the $\ZZ[\Gm]$-module structures on $\sKMW_{\ast}$ and $\ZZ[\Lb^{\times}]$, 
the Nisnevich sheaf of Milnor-Witt K-theory twisted with respect to $\Lb$ is defined in \chfinitecw, \S\ref{sec:twisting} via the sheaf tensor product 
\[
\sKMW_{\ast}(\Lb):=\sKMW_{\ast}\otimes_{\ZZ[\Gm]}\ZZ[\Lb^{\times}].
\]
\index[notation]{kmwsl@$\sKMW_{\ast}(\Lb)$}%
The choice of a twisting line bundle on $X$ can be viewed as an ``$\Aone$-local system'' on $X$.
The terms $\Cger^i(X,\sKMW_j,\Lb)$ together with the differentials defined componentwise in \cite{Morel12_c}*{\S 5} form the Rost-Schmid cochain complex;
we let $\H^i(X,\sKMW_j,\Lb)$ denote the associated cohomology groups.
\index[notation]{hixkmwjl@$\H^i(X,\sKMW_j,\Lb)$}%
We define the support of $\alpha\in \Cger^i(X,\sKMW_j,\Lb)$ to be the closure of the set of points $x\in X^{(i)}$ of codimension $i$ for which the $x$-component of $\alpha$ is nonzero.
\index{support}%

To establish notation, we recall that the Milnor-Witt $K$-theory $\sKMW_{\ast}(F)$ of any field $F$ is the quotient of the free associative algebra on generators $[F^{\times}]\cup\{\eta\}$ subject to the relations
\begin{enumerate}
\item
$[a][b]=0$, \text{for} $a+b=1$ (Steinberg relation),
\item
$[a]\eta=\eta[a]$ ($\eta$-commutativity relation),
\item
$[ab]=[a]+[b]+\eta[a][b]$ (twisted $\eta$-logarithm relation), and 
\item
$(2+[-1]\eta)\eta=0$ (hyperbolic relation).
\end{enumerate}
The grading on $\sKMW_{\ast}(F)$ is defined by declaring $|\eta|=-1$ and $| [a]| =1$ for all $a\in F^{\times}$. 
For the Grothendieck-Witt ring of symmetric bilinear forms over $F$ there is a ring isomorphism 
\[
\GW(F)\xrightarrow{\simeq}\sKMW_0(F);
\langle a\rangle\mapsto 1+\eta[a].
\]
\index[notation]{gw@$\GW$}%
We shall repeatedly use the zeroth motivic Hopf element $\epsilon:=-\langle -1\rangle=-1-\eta[-1]\in \sKMW_0(F)$ in the context of $\Aone$-homotopies between Milnor-Witt sheaves.
In fact, 
$\sKMW_{\ast}(F)$ is isomorphic to the graded ring of endomorphisms of the motivic sphere spectrum over $F$ \cite{Morel04-2_c}*{Theorem~6.2.1}.

As in \chdmt, Definition \ref{def:representablepsht}, we denote by $\MWprep(X)$ the Milnor-Witt presheaf on $\smk$ defined by 
\[
Y\mapsto \MWprep(X)(Y):=\cork(Y,X),
\] 
\index[notation]{ctx@$\MWprep(X)$}%
and by $\MWrepZ(X)$ its associated Milnor-Witt (Nisnevich) sheaf. 
\index[notation]{ztx@$\MWrepZ(X)$}%
One notable difference from the setting of finite correspondences is that the Zariski sheaf $\MWprep(X)$ is not in general a sheaf in the Nisnevich topology,
see \chfinitecw, Example~\ref{exm:notnis}.
We write $\MWprep\{1\}$ for the cokernel of the morphism of presheaves 
\[
\MWprep(\Spec(k))\to \MWprep(\Gm)
\] 
induced by the $k$-rational point $1\colon\Spec(k)\to\Gm$.
This is a direct factor of $\MWprep(\Gm)$. 
As it turns out the associated Milnor-Witt sheaf $\tZcbx 1$ is the cokernel of the morphism $\MWrepZ(\Spec(k))\to \MWrepZ(\Gm)$; see \chdmt, Proposition~\ref{prop:exist_associated-W-t-sheaf}, \ref{item:shgrothab}. 

The additive category $\cork$ has the same objects as $\smk$ and with morphisms given by finite Milnor-Witt correspondences \chdmt, Definition~\ref{dfn:corkcorVk}.
The cartesian product in $\smk$ and the external product of finite Milnor-Witt correspondences
furnish a symmetric monoidal structure on $\cork$.
Taking graphs of maps in $\smk$ yields a faithful symmetric monoidal functor 
\[
\tilde\gamma:\smk\to\cork.
\]
The tensor product $\otimes$ on Milnor-Witt presheaves defined in \chdmt, \S \ref{sec:MWtransfers} forms part of a closed symmetric monoidal structure;
the internal Hom object $\uHom$ is characterized by the property that $\uHom(\MWprep(X),{\mathcal F})(Y)={\mathcal F}(X\times Y)$ for any Milnor-Witt presheaf $\mathcal F$. 
The same holds for the category of Milnor-Witt sheaves;
here, 
$\uHom(\MWrepZ(X),{\mathcal F})(Y)={\mathcal F}(X\times Y)$ for any Milnor-Witt sheaf ${\mathcal F}$.

\subsection*{Outline of the paper}

In \S\ref{section:cdMKK} we associate to a Cartier divisor $D$ on $X$ a class $\tdiv(D)$ 
\index[notation]{divtd@$\tdiv(D)$}%
in the cohomology group with support $\H^1_{| D|}(X,\sKMW_1,\OO_X(D))$.
One of our main calculations relates the push-forward of a principal Cartier divisor on $\Gm$ defined by an explicit rational function to the zeroth motivic Hopf element $\epsilon\in\sKMW_0$.
The cancellation theorem for Milnor-Witt correspondences is shown in \S\ref{cancellationW}.
In outline the proof follows the strategy in \cite{Voe10_c}, 
but it takes into account the extra structure furnished by Milnor-Witt $K$-theory \cite{Morel12_c}.
A key result states that the twist map on the Milnor-Witt presheaf $\MWprep\{1\}\otimes\MWprep\{1\}$ is $\Aone$-homotopic to $\epsilon$. 
An appealing aspect of the proof is that all the calculations with rational functions, Cartier divisors, and the residue homomorphisms in the Rost-Schmid cochain complex corresponding to points 
of codimension one or valuations can be carried out explicitly.

The cancellation theorem has several consequences.
In \S\ref{sec:etfMWmotives} we show that tensoring complexes of Milnor-Witt sheaves $\mathscr C$ and $\mathscr D$ with $\tZcbx 1$ induces the isomorphism (\ref{cancellationisomorphism}) for the 
category of effective Milnor-Witt motives.
This is the symmetric bilinear form analogue of Voevodsky's embedding theorem for motives.
In order to prove this isomorphism we first make a comparison of  Zariski and Nisnevich hypercohomology groups in \S\ref{ZvsN}.
Several other applications for Milnor-Witt motivic cohomology groups and Milnor-Witt motives follow in \chdmt.

In \S\ref{sec:examples} we apply the cancellation theorem to concrete examples which are used in \chfinitecw and \chdmt for verifying the hyperbolic and $\eta$-twisted logarithmic relations 
appearing in the proof of the isomorphism between Milnor-Witt $K$-theory and the diagonal part of the Milnor-Witt motivic cohomology ring.

Finally, 
in \S\ref{MWmcs} we study Milnor-Witt motivic cohomology as a highly structured ring spectrum.  
The cancellation theorem shows this is an $\Omega_{T}$-spectrum.
This part is less detailed since it ventures into open problems on effective slices in the sense of \cite{Sp12_c}*{\S5}.


\section{Cartier divisors and Milnor-Witt $K$-theory}
\label{section:cdMKK}

\subsection{Cartier divisors}
\label{sec:CartierDivisors}
We refer to \cite{Ful98_c}*{Appendix~B.4} for background on Cartier divisors.
Suppose $X$ is a smooth integral scheme and let $D=\{(U_i,f_i)\}$ be a Cartier divisor on $X$. 
We denote the support of $D$ by $| D|$.
The line bundle $\OO_X(D)$ associated to $D$ is the $\OO_X$-subsheaf of $k(X)$ generated by $f^{-1}_i$ on $U_i$ \cite{Ful98_c}*{Appendix~B.4.4}. 
Multiplication by $f_i^{-1}$ yields an isomorphism between $\OO_{U_i}$ and the restriction of the line bundle $\OO_X(D)$ to $U_i$.

In the following we shall associate to the Cartier divisor $D$ a certain cohomology class 
\[
\tdiv(D)\in \H^1_{| D|}(X,\sKMW_1,\OO_X(D)),
\] 
where $\OO_X(D)$ is seen as a graded line bundle of degree $1$.
Let $x\in X^{(1)}$ and choose $i$ such that $x\in U_i$. 
Then $f_i^{-1}$ is a generator of the stalk of $\OO_X(D)$ at $x$ and a fortiori a generator of $\OO_X(D)\otimes k(X)$. 
We may therefore consider the element 
\[
[f_i]\otimes (f_i)^{-1}\in \sKMW_1(k(X),\OO_X(D)\otimes k(X)),
\] 
and its boundary under the residue homomorphism
\[
\partial_x
\colon 
\sKMW_1(k(X),\OO_X(D)\otimes k(X))
\to 
\sKMW_0(k(x),(\mathfrak m_x/\mathfrak m_x^2)^\vee\otimes_{k(x)}\OO_X(D)_x).
\]
defined in \cite{Morel12_c}*{Theorem~3.15} by keeping track of local orientations, see \cite{Morel12_c}*{Remark~3.21}.

\begin{dfn}
\label{def:Cartier}
With the notation above we set 
\[
\tord_x(D):=\partial_x([f_i]\otimes (f_i)^{-1})
\in \sKMW_0(k(x),(\mathfrak m_x/\mathfrak m_x^2)^\vee\otimes_{k(x)}\OO_X(D)_x),
\] 
\index[notation]{ordtd@$\tord(D)$}%
and 
\[
\tord(D):=\sum_{x\in X^{(1)}\cap | D|} \tord_x(D)\in \Cger^1(X,\sKMW_1,\OO_X(D)).
\]
\end{dfn}

A priori it is not clear whether $\tord_x(D)$ is well-defined for any $x\in X^{(1)}\cap | D|$ since the definition seems to depend on the choice of $U_i$. 
We address this issue in the following result.
\begin{lem}
\label{lem:Cartier}
Let $x\in X^{(1)}$ be such that $x\in U_i\cap U_j$. 
Then we have 
\[
\partial_x([f_i]\otimes (f_i)^{-1})=\partial_x([f_j]\otimes (f_j)^{-1}).
\]
\end{lem}
\begin{proof}
Note that $f_i/f_j\in \OO_{X}(U_i\cap U_j)^\times$ and a fortiori $f_i/f_j\in \OO_{X,x}^\times$. 
This implies the equalities
\[
[f_i]\otimes (f_i)^{-1}=[f_i]\otimes (f_j/f_i) (f_j)^{-1}=\langle f_j/f_i\rangle [f_i]\otimes (f_j)^{-1}.
\]
Using that $[f_j]=[(f_j/f_i)f_i]=[f_j/f_i]+\langle f_j/f_i\rangle [f_i]$ we deduce 
\[
[f_i]\otimes (f_i)^{-1}=[f_j]\otimes (f_j)^{-1}-[f_j/f_i]\otimes (f_j)^{-1}.
\]
Since $f_j/f_i\in \OO_{X,x}^\times$ we have $\partial_x([f_j/f_i])=0$, and the result follows.
\end{proof}

\begin{lem}
The boundary homomorphism
\[
d^1:\Cger^1(X,\sKMW_1,\OO_X(D))\to \Cger^2(X,\sKMW_1,\OO_X(D))
\]
vanishes on the class $\tord(D)$ in Definition \ref{def:Cartier}, 
i.e., 
we have $d^1(\tord(D))=0$.
\end{lem}
\begin{proof}
For $x\in X^{(2)}$ we set $Y:=\Spec(\OO_{X,x})$ and consider the boundary map
\[
d_{Y}^1:\Cger_1(Y,\sKMW_1,\OO_X(D)_{|_Y})
\to 
\sKMW_{-1}(k(x),\wedge^2(\mathfrak m_x/\mathfrak m_x^2)^\vee\otimes_{k(x)}\OO_X(D)_x).
\]
Since the boundary homomorphism $d^1$ is defined componentwise it suffices to show the restriction $\tord(D)_Y$ of $\tord(D)$ to $Y$ vanishes under $d_{Y}^1$.
We may choose $i$ such that $x\in U_i$ and write
\[
\tord(D)_Y=\sum_{y\in Y^{(1)}\cap | D|} \tord_y(D).
\] 
Then the map $Y\to U_i$ induces a commutative diagram of complexes:
\[
\begin{tikzcd}
[column sep=1.6em]
\Cger^0(U_i,\sKMW_1,\OO_X(D)_{|_{U_i}}) \ar[r,"{d_{U_i}^0}"] \ar[d,equal] & \Cger^1(U_i,\sKMW_1,\OO_X(D)_{|_{U_i}}) \ar[r,"{d_{U_i}^1}"] \ar[d] & \Cger^2(U_i,\sKMW_1,\OO_X(D)_{|_{U_i}}) \ar[d] \\
\Cger^0(Y,\sKMW_1,\OO_X(D)_{|_Y}) \ar[r,"{d_{Y}^0}"] & \Cger^1(Y,\sKMW_1,\OO_X(D)_{|_Y}) \ar[r,"{d_{Y}^1}"] & \Cger^2(Y,\sKMW_1,\OO_X(D)_{|_Y})
\end{tikzcd}
\]
Here $\tord(D)_Y$ is the image of $\tord(D)_{U_i}$ under the middle vertical map.
Thus it suffices to show that $d_{U_i}^1$ vanishes on $\tord(D)_{U_i}$. 
The latter follows from the equality
\[
\tord(D)_{U_i}=d_{U_i}^0([f_i]\otimes (f_i)^{-1}). \qedhere
\]
\end{proof}

\begin{dfn}
We write $\tdiv(D)$ for the class of $\tord(D)$ in $\H^1_{| D|}(X,\sKMW_1,\OO_X(D))$. 
\index[notation]{divtd@$\tdiv(D)$}%
\end{dfn}

\begin{rem}
The image of $\tdiv(D)$ under the extension of support homomorphism
\[
\H^1_{| D|}(X,\sKMW_1,\OO_X(D))\to \H^1(X,\sKMW_1,\OO_X(D))
\]
equals the Euler class of $\OO_X(D)^\vee$ \cite{Asok16_c}*{\S 3}.
\end{rem}

Having defined the cocycle associated to a Cartier divisor, we can now intersect with cocycles. Let us first recall what we mean by a proper intersection.

\begin{dfn}
Let $T\subset X$ be a closed subset of codimension $d$. If $D$ is a Cartier divisor on $X$, 
we say that $D$ and $T$ intersect properly if the intersection of $T$ and $D$ is proper, 
i.e., 
if any component of $| D| \cap T$ is of codimension $\geq d+1$. 
Equivalently, 
the intersection between $D$ and $T$ is proper if the generic points of $T$ are not in $| D|$.
\end{dfn}

\begin{dfn}
\label{dfn:intersection}
Let $\Lb$ be a (graded) line bundle on $X$ and $\alpha\in H_T^i(X,\sKMW_j,\Lb)$. 
We write $D\cdot \alpha$ for the class of $\tdiv(D)\cdot \alpha$ in $H_{T\cap | D|}^{i+1}(X,\sKMW_{j+1},\OO_X(D)\otimes \Lb)$ (and $\alpha\cdot D$ for the class of $\alpha\cdot \tdiv(D)$). Note in particular that $D\cdot 1=\tdiv(D)\in \H^1_{| D|}(X,\sKMW_1,\OO_X(D))$.
\end{dfn}

\begin{rem}
\label{rem:commutativity}
Since $\OO_X(D)$ has degree $1$ (as a graded line bundle), it follows from \cite{Fasel07_c}*{Lemmas~4.19 and 4.20} or \cite{Fasel19_c}*{\S 3.4} that 
\[
D\cdot \alpha=\alpha\cdot D.
\]
\end{rem}

Recall that the canonical bundle of $X\in\smk$ is defined by $\omega_{X/k}:= \wedge^{d_{X}}\Omega^{1}_{X/k}$.
Here, 
$\wedge^{d_{X}}$ is the exterior power to the dimension $d_{X}$ of $X$ and $\Omega^{1}_{X/k}$ is the cotangent bundle on $X$.

\begin{lem}
\label{lem:degree}
Let $D$ be the principal Cartier divisor on $\Gm=\Spec (k[t^{\pm 1}])$ defined by 
\[
g:=(t^{n+1}-1)/(t^{n+1}-t)\in k(t).
\] 
Denote by $\OO_{\Gm}\to \OO_{\Gm}(D)$ and $\OO_{\Gm}\to \omega_{\Gm/k}$ the evident trivializations, 
and let 
\[
\chi:\H^1_{| D|}(\Gm,\sKMW_1,\OO_{\Gm}(D))\to \H^1_{| D|}(\Gm,\sKMW_1,\omega_{\Gm/k})
\]
denote the induced isomorphism. 
For $p:\Gm\to \Spec (k)$ and the induced push-forward map
\[
p_*:\H^1_{| D|}(\Gm,\sKMW_1,\omega_{\Gm/k})\to \sKMW_0(k),
\]
we have $p_*\chi(D\cdot 1)=\langle -1\rangle$.
\end{lem}
\begin{proof}
Consider the sequence
\[
\begin{tikzcd}
\sKMW_1(k(t),\omega_{k(t)/k}) \ar[r,"d^0"] & \Cger^1(\PP^1,\sKMW_1,\omega_{\PP^1/k}) \ar[r] & \sKMW_0(k),
\end{tikzcd}
\]
where the rightmost map is a sum of the (cohomological) transfer maps defined for instance in \cite{Morel12_c}*{Definition 4.26}, 
see also \cite{Fasel08_c}*{\S 10.4}.
Since the push-forward map
\[
p_*:\H^1(\PP^1,\sKMW_1,\omega_{\PP^1/k})\to \sKMW_0(k)
\]
is well defined \cite{Fasel08_c}*{Cor.~10.4.5}, 
it follows that the composition in the sequence is trivial, 
see also \cite{Fasel08_c}*{Chapitre~8}. 
Let $\overline D$ be the principal Cartier divisor on $\PP^1$ defined by $g$. 
The evident trivialization $\OO_{\PP^1}\to \OO_{\PP^1}(\overline D)$ together with the canonical isomorphism
\[
\H^1(\PP^1,\sKMW_1,\omega_{\PP^1/k})\simeq \H^1(\PP^1,\sKMW_1)
\]
yield an isomorphism 
\[
\overline \chi
\colon
\H^1_{| \overline D|}(\PP^1,\sKMW_1,\OO_{\PP^1}(\overline D))
\xrightarrow{\simeq}
\H^1_{| \overline D|}(\PP^1,\sKMW_1,\omega_{\PP^1/k}),
\] 
and we have $p_*\overline\chi(\overline D\cdot 1)=0$. 
The restriction of the intersection product $\overline D\cdot 1$ to $\Gm$ equals $D\cdot 1$, 
so that 
\[
\overline D\cdot 1=D\cdot 1+\tord_\infty(\overline D)+\tord_0(\overline D).
\] 
It is straightforward to check that $\tord_\infty(\overline D)=0$. 
Next we compute $\tord_0(\overline D)$. 
Using the equality
\[
g=t^{-1}(t^{n+1}-1)/(t^n-1),
\] 
we obtain
\begin{eqnarray*}
\partial_0([g]\otimes g^{-1}) & = & \partial_0([(t^{n+1}-1)/(t^n-1)]\otimes g^{-1}+\langle (t^{n+1}-1)/(t^n-1)\rangle [t^{-1}]\otimes g^{-1}) \\
 & = & \partial_0(\langle (t^{n+1}-1)/(t^n-1)\rangle [t^{-1}]\otimes g^{-1}) \\
 & = & \partial_0(\langle (t^{n+1}-1)/(t^n-1)\rangle\cdot \epsilon[t]\otimes g^{-1}) \\
 & = & \epsilon\otimes (t\otimes g^{-1}).
\end{eqnarray*}
Thus $\overline \chi(\overline D\cdot 1)=\chi(D\cdot 1)+\epsilon$, 
where $\epsilon$ is seen as an element of $\sKMW_0(k)\subset \Cger_1(\PP^1,\sKMW_1,\omega_{\PP^1/k})$. Finally, we have 
\[
0=p_*(\overline \chi(\overline D\cdot 1))=p_*(\chi(D\cdot 1)+\epsilon)=p_*(\chi(D\cdot 1))+\epsilon,
\]
and it follows that $p_*(\chi(D\cdot 1))=-\epsilon=\langle -1\rangle$.
\end{proof}


\section{Cancellation for Milnor-Witt correspondences}
\label{cancellationW}

If $X,Y\in\smk$ we follow the convention in \cite{Voe10_c}*{\S 4} by letting $XY$ be short for the fiber product $X\times_{k} Y$.
We denote the dimension of $Y\in\smk$ by $d_Y$.
Suppose $\alpha\in \cork(\Gm X,\Gm Y)$ is a finite Milnor-Witt correspondence. 
It has a well-defined support by \chfinitecw, Definition~\ref{dfn:support} which we denote by $T:=\supp(\alpha)$, 
so that 
\[
\alpha\in \H^{d_Y+1}_T(\Gm X \Gm Y ,\sKMW_{d_Y+1},\omega_{\Gm Y}),
\]
where $\omega_{\Gm Y}$ is the pull-back of the canonical sheaf of $\Gm Y$ along the relevant projection, sitting in degree $d_Y+1$ as a graded line bundle.
Let $t_1$ and $t_2$ denote the (invertible) global sections on $\Gm X \Gm Y$ obtained by pulling back the coordinate on $\Gm$ under the projections on the first and third factor of $\Gm X\Gm Y$,
respectively.
Recall from \cite{Voe10_c}*{\S 4} that there exists a natural number $N(\alpha)\in\NN$ such that for all $n\geq N(\alpha)$ the rational function 
\[
g_n:=(t_1^{n+1}-1)/(t_1^{n+1}-t_2)
\] 
defines a principal Cartier divisor $D(g_n)$ on $\Gm X \Gm Y$ having the property that $D(g_n)$ and $T$ intersect properly, 
and further that $| D(g_n)| \cap T$ 
is finite over $X$. 
If $n\geq N(\alpha)$, 
we say that $D(g_n)$ is \emph{defined relative to $\alpha$}. Unless otherwise specified, we always assume in the sequel that the Cartier divisors are defined relative to the considered cycles, 
e.g., 
for $D:=D(g_n)$ and $\alpha$.
 
We can consider
\[
D\cdot \alpha \in \H^{d_Y+2}_{T\cap | D|} (\Gm X \Gm Y,\sKMW_{d_Y+2},\OO(D)\otimes\omega_{\Gm Y}),
\]
where $\OO(D)$ is of degree $1$,
and the isomorphisms of graded line bundles (all of degree $1$)
\[
\OO(D)\simeq \OO_{\Gm X \Gm Y}\simeq \omega_{\Gm/k},
\]
where the latter is the usual trivialization of $ \omega_{\Gm/k}$ (choosing $dt_1$ as generator). Using the canonical isomorphism
\begin{equation}\label{eq:orientation}
\omega_{\Gm\Gm Y}\simeq p^*\omega_{\Gm/k}\otimes \omega_{\Gm Y},
\end{equation}  
we obtain
\[
D\cdot \alpha \in \H^{d_Y+2}_{T\cap | D|} (\Gm X \Gm Y,\sKMW_{d_Y+2},\omega_{\Gm\Gm Y}).
\] 
By permuting the first two factors in the product $\Gm X\Gm Y$ we finally obtain
\[
D\cdot \alpha\in \H^{d_Y+2}_{T\cap | D|} (X\Gm\Gm Y,\sKMW_{d_Y+2},\omega_{\Gm\Gm Y}),
\] 
which can be viewed as an element of $\cork(X,\Gm\Gm Y)$. 

\begin{dfn}
Let $\rho_n(\alpha)\in \cork(X,Y)$ be the composite of $D\cdot \alpha\in\cork(X,\Gm\Gm Y)$ with the projection map $\Gm\Gm Y\to Y$. 
\end{dfn}

\begin{lem}
\label{lem:minus1}
For $\alpha\in \cork(X,Y)$ and $1\times\alpha\in \cork(\Gm X,\Gm Y)$ we have
\[
\rho_n(1\times\alpha)=\langle -1\rangle\alpha.
\]
\end{lem}
\begin{proof}
Let $\Delta:\Gm\to \Gm\Gm$ be the diagonal embedding, 
and consider the commutative diagram of Cartesian squares in $\smk$, 
where all the maps apart from $\Delta$ and $\Delta\times 1$ are projections:
\[
\begin{tikzcd}
\Gm XY \ar[r,"{p_2^\prime}"] \ar[d,"{\Delta\times 1}"'] & \Gm \ar[d,"\Delta"] \\
\Gm\Gm XY \ar[d,"p_1"'] \ar[r,"p_2"] & \Gm\Gm\ar[d] \\
XY \ar[r,"p"] & \Spec (k) 
\end{tikzcd}
\]
We have $1\times \alpha=p_2^*\Delta_*(\langle 1\rangle)\cdot p_1^*\alpha$ by definition, 
while the base change formula \chfinitecw, Proposition~\ref{prop:basechange} implies $1\times\alpha=(\Delta\times 1)_*(\langle 1\rangle)\cdot p_1^*\alpha$.
On the other hand, 
let $D^\prime$ be the principal Cartier divisor on $\Gm\Gm$ defined by 
\[
g_n(t_1,t_2):=(t_1^{n+1}-1)/(t_1^{n+1}-t_2),
\] 
and let $D^{\prime\prime}$ be the principal Cartier divisor on $\Gm$ defined by 
\[
g_n(t):=(t^{n+1}-1)/(t^{n+1}-t).
\] 
Since $\Delta^*(D^\prime)=D^{\prime\prime}$ and $p_2^*D^\prime=D$ we obtain $D\cdot (1\times\alpha)=p_2^*D^\prime\cdot (\Delta\times 1)_*(\langle 1\rangle)\cdot p_1^*\alpha$. 
Now, using $(1\times\Delta)^*(p_2^*D^\prime)=(p_2^\prime)^*D^{\prime\prime}$ and the projection formula, we obtain
\[
D\cdot (1\times\alpha)=(\Delta\times 1)_*((p_2^\prime)^*D^{\prime\prime})\cdot p_1^*\alpha.
\]
By composing this Milnor-Witt correspondence with $p_1$, 
which is tantamount to applying $(p_1)_*$ by \chfinitecw, Example~\ref{ex:flat_example}, 
we obtain
\[
\rho_n(1\times\alpha)
=
(p_1)_*((1\times\Delta)_*((p_2^\prime)^*D^{\prime\prime})\cdot p_1^*\alpha)
=
(p^*q_*(D^{\prime\prime}\cdot \langle 1\rangle))\cdot \alpha,
\]
where $q:\Gm\to \Spec(k)$. 
The result follows now from Lemma \ref{lem:degree}.
\end{proof}

\begin{rem}
One has to be a bit careful when applying $(p_1)_*$. Indeed, $D\cdot (1\times\alpha)$ is an element of a Chow-Witt group with twist in the graded line bundle $(2+d_Y,\omega_{\Gm\Gm Y})$. Now, one has to use the canonical isomorphisms (involving the sign $(-1)^{d_Xd_Y}$)
\begin{eqnarray*}
(2+d_Y,\omega_{\Gm\Gm Y}) &= & (d_X,\omega_X)\otimes (-d_X,\omega_X^\vee)\otimes (2+d_Y,\omega_{\Gm\Gm Y}) \\
& = & (2+d_X+d_Y,\omega_{X\Gm\Gm Y})\otimes (-d_X,\omega_X^\vee),
\end{eqnarray*}
and
\begin{eqnarray*}
(d_Y,\omega_{Y}) &= & (d_X,\omega_X)\otimes (-d_X,\omega_X^\vee)\otimes (d_Y,\omega_{Y}) \\
& = & (d_X+d_Y,\omega_{XY})\otimes (-d_X,\omega_X^\vee),
\end{eqnarray*}
involving the same sign.
\end{rem}

\begin{lem}
For the Milnor-Witt correspondence 
\[
e_X:
\Gm X
\xrightarrow{q}
X
\xrightarrow{\{1\}\times \id}
\Gm X,
\]
where $q$ is the projection map, 
we have $\rho_n(e_X)=0$ for any $n\in\NN$.
\end{lem}
\begin{proof}
As in \cite{Voe10_c}*{Lemma~4.3~(ii)}, 
note that the cycle representing the above composite is the image of the unit $\langle 1\rangle\in \sKMW_0(\Gm X)$ under the push-forward homomorphism
\[
f_*
\colon
\sKMW_0(\Gm X)
\to 
\H^{d_X+1}_{f(\Gm X)}(\Gm X\Gm X,\sKMW_{d_X+1}, \omega_{\Gm X})
\]
for the map $f:\Gm X\to \Gm X\Gm X$ given as the diagonal on $X$ and by $t\mapsto (t,1)$ on $\Gm$. 
The result follows now from the projection formula since $f^*D(g_n)=1$ for all $n\in \NN$. 
\end{proof}

\begin{lem}
\label{lem:composite}
Suppose $D:=D(g_n)$ is defined relative to $\alpha\in \cork(\Gm X,\Gm Y)$. 
Then for any Milnor-Witt correspondence $\beta:X^\prime\to X$, 
$D$ is defined relative to $\alpha\circ (1\times \beta)$, 
and we have
\[
\rho_n(\alpha\circ (1\times \beta))=\rho_n(\alpha)\circ \beta.
\]
\end{lem}

\begin{proof}
For the fact that $D$ is defined relative to $\alpha\circ (1\times \beta)$ we refer to \cite{Voe10_c}*{Lemma~4.4}. 
For the second assertion we note that $\rho_n(\alpha)\circ \beta$ is the composite
\[
X^\prime\stackrel{\beta}\to X\xrightarrow{D\cdot \alpha} \Gm\Gm Y\stackrel{q}\to Y,
\]
while $\rho_n(\alpha\circ (1\times \beta))$ is the composite
\[
X^\prime\xrightarrow{D\cdot(\alpha\circ (1\times \beta))} \Gm\Gm Y\stackrel{q}\to Y.
\]
Thus it suffices to prove that 
\[
D\cdot(\alpha\circ (1\times \beta))=(D\cdot \alpha)\circ \beta.
\] 
Consider the following diagram:
\[
\begin{tikzcd}
X^\prime\Gm\Gm Y \ar[rrrd,bend left=1ex,"s_Y"] \ar[rddd,bend right=1ex,"s_X"'] & & & \\
 & X^\prime X\Gm\Gm Y \ar[r,"q_{XY}"'] \ar[d,"p_{X^\prime X}"] \ar[lu,"\pi"'] & X\Gm\Gm Y \ar[r,"q_Y"'] \ar[d,"p_Y"] & \Gm\Gm Y \\
 & X^\prime X \ar[r,"q_X"'] \ar[d,"p_{X^\prime}"] & X & \\
 & X^\prime & & 
\end{tikzcd}
\]
By definition and functoriality of the pull-back we have
\[
(D\cdot \alpha)\circ \beta=\pi_*(p_{X^\prime X}^*\beta\cdot q_{XY}^*(D\cdot \alpha))=\pi_*(p_{X^\prime X}^*\beta\cdot D^\prime\cdot q_{XY}^*\alpha),
\]
where $D^\prime$ is the pull-back of $D$ along the projection $X^\prime X\Gm\Gm Y\to \Gm\Gm$. 
On the other hand, a direct computation involving the base change formula (\chfinitecw, Proposition~\ref{prop:basechange}) shows that 
\[
\alpha\circ (1\times \beta)=\pi_*(p_{X^\prime X}^*\beta\cdot q_{XY}^*\alpha).
\]
Let $D^{\prime\prime}$ be the pull-back of $D$ along the projection $X^\prime\Gm\Gm Y\to \Gm\Gm$. Using respectively the projection formula and the (skew-)commutativity of Chow-Witt groups, we find
\begin{eqnarray*}
D^{\prime\prime}\cdot \pi_*(p_{X^\prime X}^*\beta\cdot q_{XY}^*\alpha) & = & \pi_*(D^\prime\cdot p_{X^\prime X}^*\beta\cdot q_{XY}^*\alpha)\\
 & = & \pi_*(p_{X^\prime X}^*\beta\cdot D^\prime \cdot q_{XY}^*\alpha).
\end{eqnarray*}
\end{proof}

\begin{lem}
\label{lem:rightcomposition}
Suppose $D:=D(g_n)$ is defined relative to $\alpha\in \cork(\Gm X,\Gm Y)$. 
Then for any map of schemes $f:X^\prime\to Y^\prime$, $D(g_n)$ is defined relative to $\alpha\times f$, 
and we have 
\[
\rho_n(\alpha\times f)=\rho_n(\alpha)\times f.
\]
\end{lem}
\begin{proof}
Omitting the required permutations, 
the product $\alpha\times f$ is defined by the intersection product of the pull-backs on $\Gm X\Gm Y X^\prime Y^\prime$ along the corresponding projections of $\alpha$ and the graph $\Gamma_f$. 
To compute $\rho_n(\alpha\times f)$ we multiply (on the left) by $\tdiv D$ (again omitting the permutations). 
On the other hand, 
the correspondence $\rho_n(\alpha)\times f$ is obtained by multiplying $\tdiv D$, $\alpha$, and $\Gamma_f$ along the relevant projections. 
Thus the assertion follows from associativity of the intersection product.
\end{proof}

Recall that $\MWprep\{1\}$ is the cokernel of the morphism $\MWprep(\Spec(k))\to \MWprep(\Gm)$ induced by the unit for the multiplicative group scheme $\Gm$.

\begin{lem}
\label{lem:hopf}
The twist map 
\[
\sigma:\MWprep(\Gm)\otimes\MWprep(\Gm)\to \MWprep(\Gm)\otimes\MWprep(\Gm)
\]
induces a map 
\[
\sigma:\MWprep\{1\}\otimes \MWprep\{1\}\to \MWprep\{1\}\otimes \MWprep\{1\},
\] 
which is $\Aone$-homotopic to multiplication by $\epsilon$.
\end{lem}

\begin{proof}
The claim that $\sigma$ induces a map $\MWprep\{1\}\otimes \MWprep\{1\}\to \MWprep\{1\}\otimes \MWprep\{1\}$ follows immediately from the definition.
By the proof of \cite{Voe10_c}*{Lemma~4.8}, 
$g\otimes f\colon X\to\Gm{}\Gm{}$ is $\Aone$-homotopic to $f\otimes g^{-1}$ for all $X\in\smk$ and $f,g\in \OO(X)^\times$.
We are thus reduced to proving that $\MWprep\{1\}\to\MWprep\{1\}$ induced by the map $\Gm\to \Gm$ sending $z$ to $z^{-1}$  is $\Aone$-homotopic to multiplication by $\epsilon$. 
Its graph in $\Gm\Gm=\Spec(k[t_1^{\pm1},t_2^{\pm1}])$ is given by the prime ideal $(t_2-t_1^{-1})$, 
while the graph of the identity is given by $(t_2-t_1)$. 
We note that their product equals
\[
(t_2-t_1^{-1})(t_2-t_1)=t_2^2-(t_1+t_1^{-1})t_2+1. 
\]
On the other hand, we can consider $(t_2-1)^2$ and the polynomial 
\[
F:=t_2^2-t_2(u(t_1+t_1^{-1})+2(1-u))+1\in k[t_1^{\pm 1},t_2^{\pm 1},u].
\] 
It is straightforward to check that the vanishing locus of $F$ defines a closed subset $\Vanish(F)$ in $\Gm \Aone \Gm$ which is finite and surjective over $\Gm \Aone$. 
The same properties hold for $\Vanish(G)$, 
where $G:=t_2^2-2t_2(1-2u)+1$. 

We define $F_1:=(t_1-t_1^{-1})F$, $F_2:=(t_1+1)G$, and consider the element 
\[
\alpha:=[F_1]\otimes dt_2 -\langle 2\rangle[F_2]\otimes dt_2-\langle 2\rangle[t_1-1]\otimes dt_2 \in \sKMW_1(k(u,t_1,t_2),\omega_{\Gm})
\] 
together with its image under the residue homomorphism
\[
d^0:\sKMW_1(k(u,t_1,t_2),\omega_{\Gm})\to \bigoplus_{x\in (\Gm\Aone\Gm)^{(1)}} \hspace{-4ex} \sKMW_0(k(x),\omega_{\Gm})
\] 
of \cite{Morel12_c}*{Remark~3.21}.
The polynomial $F_1$ ramifies at $F$ and at $(t_1-t_1^{-1})=t_1^{-1}(t_1-1)(t_1+1)$. 
The residue of $[F_1]\otimes dt_2$ at $(t_1-1)$ is 
\[
\langle 2\rangle \otimes (t_1-1)dt_2, 
\]
while its residue at $t_1+1$ is  
\[
\langle 2(t_2^2-2t_2(1-2u)+1)\rangle \otimes (t_1+1)dt_2.
\] 
On the other hand, 
the residue of $\langle 2\rangle[F_2]\otimes dt_2$ at $(t_1+1)$ is 
\[
\langle 2(t_2^2-2t_2(1-2u)+1)\rangle \otimes (t_1+1)dt_2.
\] 
It follows that $\alpha$ is unramified at both $(t_1+1)$ and $(t_1-1)$. 
Hence the residue of $\alpha$ defines a Milnor-Witt correspondence $\Gm\Aone \to \Gm$ of the form
\[
d(\alpha)\in \H^1_{\Vanish(F)\cup \Vanish(G)}(\Gm\Aone\Gm,\sKMW_1,\omega_{\Gm})\subset \cork(\Gm\Aone,\Gm).
\]
We now compute the restriction of $d(\alpha)$ at $u=0$ and $u=1$, 
i.e., 
we compute the image of $d(\alpha)$ along the morphisms $\cork(\Gm\Aone,\Gm)\to \cork(\Gm,\Gm)$ induced by the inclusions $\Spec(k)\to \Aone$ at $u=0$ and $u=1$.
Its restriction $d(\alpha)(0)$ to $u=0$ is supported on $\Vanish(t_2-1)$ and is thus trivial on $\MWprep\{1\}(\Gm)$. 
Its restriction $d(\alpha)(1)$ to $u=1$ takes the form
\[
\langle 1\rangle\otimes (t_2-t_1)dt_2+\langle -1\rangle \otimes (t_2-t_1^{-1})dt_2+ \langle 1,-1\rangle\otimes (t_2+1)dt_2.
\]
It follows that $d(\alpha)(1)$ is $\Aone$-homotopic to $0$. 
To conclude we are reduced to showing that $\langle 1,-1\rangle\otimes (t_2+1)dt_2$ is $\Aone$-homotopic to $0$. 
This is obtained using the polynomial $[G]$ by observing that the restriction to $u=0$ of the residue of $G$ is supported on $\Vanish(t_2-1)$, 
while the restriction to $u=1$ is precisely $\langle 1,-1\rangle\otimes (t_2+1)dt_2$.
\end{proof}


\begin{thm}
\label{thm:weakcancellation}
Let $\mathcal F$ be a Milnor-Witt presheaf and let $p:\MWprep(X)\to \mathcal F$ be an epimorphism of presheaves. 
Suppose 
\[
\phi: \MWprep\{1\}\otimes \mathcal F\to \MWprep\{1\}\otimes \MWprep(Y)
\]
is a morphism of presheaves with Milnor-Witt transfers. 
Then there exists a unique up to $\Aone$-homotopy morphism $\psi:\mathcal F\to \MWprep(Y)$ such that $\phi\cong\id\otimes \psi$.
\end{thm}
\begin{proof}
Precomposing $\phi$ with the projection $\MWprep(\Gm)\otimes \mathcal F\to \MWprep\{1\}\otimes \mathcal F$ and postcomposing with the monomorphism $\MWprep\{1\}\otimes\MWprep(Y)\to \MWprep(\Gm)\otimes \MWprep(Y)$, 
we get a morphism $\tilde \phi:\MWprep(\Gm)\otimes \mathcal F\to \MWprep(\Gm)\otimes \MWprep(Y)$. 
Composing $\tilde\phi$ with $\id\otimes p$ yields a morphism of presheaves 
\[
\alpha: \MWprep(\Gm)\otimes \MWprep(X)\to \MWprep(\Gm)\otimes\MWprep(Y), 
\] 
i.e., 
a Milnor-Witt correspondence $\alpha\in \cork(\Gm X,\Gm Y)$.
Choose $n\geq N(\alpha)$ such that $D(g_n)$ be defined relative to $\alpha$. 
Using Lemma \ref{lem:composite}, 
we get a morphism of presheaves $\rho_n(\tilde \phi):\mathcal F\to \MWprep(Y)$. We set $\psi:=\langle -1\rangle\rho_n(\tilde \phi)$. 
If $\phi=\id\otimes \psi^\prime$ for some $\psi^\prime:\mathcal F\to \MWprep(Y)$ we have $\alpha=\id\otimes\psi^\prime p$. 
Lemma \ref{lem:minus1} shows that $\rho_n(\alpha)=\langle -1\rangle\psi^\prime p$. 
Thus $\psi=\psi^\prime$, implying the uniqueness part of the theorem as in \cite{Voe10_c}*{Theorem~4.6}.

To conclude, 
we need to observe there is an $\Aone$-homotopy $\id\otimes \langle -1\rangle\rho_n(\tilde \phi)\simeq \phi$. 
Let 
\[
\tilde \phi^*:\mathcal F\otimes \MWprep(\Gm)\to \MWprep(Y)\otimes \MWprep(\Gm)
\]
be the morphism obtained by permutation of the factors, 
and let 
\[
\phi^*:\mathcal F\otimes \MWprep\{1\}\to \MWprep(Y)\otimes \MWprep\{1\}
\]
be the naturally induced morphism. 
Using Lemma \ref{lem:hopf}, 
we see that $\id_{\MWprep\{1\}}\otimes \phi^*$ and $\phi\otimes\id_{\MWprep\{1\}}$ are $\Aone$-homotopic. 
Using Lemmas \ref{lem:minus1}, \ref{lem:composite}, and \ref{lem:rightcomposition} it follows that $\rho_n(\tilde\phi)\otimes \id_{\MWprep\{1\}}=\langle -1\rangle\phi^*$. 
That is, 
we have $\phi=\id_{\MWprep\{1\}}\otimes \langle -1\rangle\rho_n(\tilde\phi)$.
\end{proof}

Analogous to  \cite{Voe10_c}*{Corollary~4.9} we deduce from Theorem \ref{thm:weakcancellation} the following result for the Suslin complex of Milnor-Witt sheaves defined via the internal Hom object 
$\uHom$ and the standard cosimplicial scheme.

\begin{coro}\label{cor:cancel}
For any $Y\in\smk$, the morphism
\[
\MWprep (Y)\to \uHom(\MWprep\{1\},\MWprep(Y)\otimes \MWprep\{1\})
\] 
induces for any smooth scheme $X$ a quasi-isomorphism of complexes
\[
\Cstar{\MWprep(Y)}(X)\xrightarrow{\sim} \Cstar{\uHom(\MWprep\{1\},\MWprep(Y)\otimes \MWprep\{1\})}(X).
\]
\end{coro}


\section{Zariski vs.\ Nisnevich hypercohomology}\label{ZvsN}

Throughout this short section we assume that $k$ is an infinite perfect field of characteristic different from $2$.
Let us first recall the following result from \chdmt, Theorem~\ref{thm:Wsh&A1}.

\begin{thm}
\label{thm:strictly}
Let $\mathcal F$ be an $\Aone$-invariant Milnor-Witt presheaf. 
Then the associated Milnor-Witt sheaf $\tilde a(\mathcal F)$ is strictly $\Aone$-invariant. 
Moreover, 
the Zariski sheaf associated with $\mathcal F$ coincides with $\tilde a(\mathcal F)$, 
and for all $i\in \NN$ and $X\in\smk$ there is a natural isomorphism
\[
\H^i_{\zar}(X,\tilde a(\mathcal F))
\xrightarrow{\simeq}
\H^i_{\nis}(X,\tilde a(\mathcal F)).
\]
\end{thm}

We now derive some consequences of Theorem \ref{thm:strictly} following \cite{Mazza06_c}*{\S 13}.

\begin{prop}\label{prop:preparatory}
Let $f\colon \mathscr C \to \mathscr D$ be a morphism of complexes of Milnor-Witt presheaves. 
Suppose their cohomology presheaves are homotopy invariant and that $\mathscr C (\Spec (F))\to \mathscr D (\Spec (F))$ is a quasi-isomorphism for every finitely generated field extension $F/k$. 
Then the induced morphism on the associated complexes of Zariski sheaves is a quasi-isomorphism. 
\end{prop}
\begin{proof}
Consider the mapping cone $\mathscr C_{f}$ of $f$. 
The five lemma implies that the cohomology presheaves $\H^i(\mathscr C_{f})$ are homotopy invariant and have Milnor-Witt transfers. 
The associated Nisnevich Milnor-Witt sheaves are strictly $\Aone$-invariant by Theorem \ref{thm:strictly}.
We have $\tilde a(\H^i(\mathscr C_{f}))(\Spec (F))=0$ by assumption. 
By \cite{Morel12_c}*{Theorem~2.11} it follows that $\tilde a(\H^i(\mathscr C_{f}))=0$. 
This finishes the proof by applying Theorem \ref{thm:strictly}.
\end{proof}

By following the proof of \cite{Mazza06_c}*{Theorem~13.12} with Theorem \ref{thm:strictly} and Proposition \ref{prop:preparatory} in lieu of the references in loc. cit., 
we deduce the following result.

\begin{thm}
Suppose $\mathcal F$ is a Milnor-Witt presheaf such that the Nisnevich sheaf $\tilde a(\mathcal F)=0$. 
Then $\tilde a(\Cstar{\mathcal F})$ is quasi-isomorphic to $0$. 
The same result holds for the complex of Zariski sheaves associated to $\Cstar{\mathcal F}$. 
\end{thm}

\begin{coro}\label{cor:localzero}
Let $f\colon \mathscr C\to \mathscr D$ be a morphism of bounded above cochain complexes of Milnor-Witt presheaves. 
Suppose $f$ induces a quasi-isomorphism $f:\mathscr C(X)\to \mathscr D(X)$ for all Hensel local schemes $X$. 
Then 
\[
\mathrm{Tot}(\Cstar{\mathscr C})(Y)\to \mathrm{Tot}(\Cstar{\mathscr D})(Y) 
\]
is a quasi-isomorphism for every regular local ring $Y$. 
\end{coro}

\begin{proof}
The proof of \cite{Mazza06_c}*{Corollary~13.14} applies mutatis mutandis. 
\end{proof}

\begin{rem}
As pointed out by the referee, there is no need to assume that the cochain complexes $ \mathscr C$ and $ \mathscr D$ are bounded above. This follows from the case of bounded complexes by writing complexes as filtered colimits of their bounded above truncations.
\end{rem}

\begin{coro}\label{cor:local}
For every $X\in\smk$ and regular local ring Z, the morphism of Milnor-Witt presheaves $\MWprep (X)\to \MWrepZ(X)$ induces a quasi-isomorphism of complexes 
\[
\Cstar{\MWprep (X)}(Z)\xrightarrow{\simeq} \Cstar{\MWrepZ(X)}(Z).
\]
There are naturally induced isomorphisms  
\[
\H^i_{\zar}(Y,\Cstar{\MWprep(X)})\simeq \H^i_{\zar}(Y,\Cstar{\MWrepZ(X)})\simeq \H^i_{\nis}(Y,\Cstar{\MWrepZ(X)})
\]
for every $Y\in \smk$.
\end{coro}
\begin{proof}
The isomorphism 
\[
\H^i_{\zar}(Y,\Cstar{\MWprep(X)})
\simeq 
\H^i_{\zar}(Y,\Cstar{\MWrepZ(X)})
\] 
follows immediately from Corollary \ref{cor:localzero}. 
The second isomorphism follows from \chdmt, Corollary~\ref{cor:cohomcomplex}.
\end{proof}

In the same vein, we have the following result.

\begin{coro}
\label{cor:local2}
For every $X\in\smk$ and regular local ring $Z$, the morphism of Milnor-Witt presheaves 
\[
\MWprep(X)\otimes \MWprep\{1\}\to \MWrepZ(X)\otimes \tZcbx 1
\] 
induces a quasi-isomorphism of complexes 
\[
\Cstar{\MWprep(X)\otimes \MWprep\{1\}}(Z)\xrightarrow{\simeq} \Cstar{\MWrepZ(X)\otimes \tZcbx 1}(Z).
\]
There are naturally induced isomorphisms
\[
\H^i_{\zar}(Y,\Cstar{\MWprep(X)\otimes \MWprep\{1\}})\simeq \H^i_{\zar}(Y,\Cstar{\MWrepZ(X)\otimes \tZcbx 1})\simeq \H^i_{\nis}(Y,\Cstar{\MWrepZ(X)\otimes \tZcbx 1})
\]
for every $Y\in\smk$.
\end{coro}
\begin{proof}
In the Nisnevich topology, 
$\MWprep(X)\otimes \MWprep\{1\}\to \MWrepZ(X)\otimes \tZcbx 1$ is locally an isomorphism and the proof of the previous corollary applies verbatim.
\end{proof}


\section{The embedding theorem for Milnor-Witt motives}\label{sec:etfMWmotives}

\begin{thm}
\label{thm:main}
Let $k$ be an infinite perfect field of characteristic not $2$. 
Then for all complexes $\mathscr C$ and $\mathscr D$ of Milnor-Witt sheaves, 
the morphism
\[
\Hom_{\DMtekZ}(\mathscr C,\mathscr D)\to \Hom_{\DMtekZ}(\mathscr C(1),\mathscr D(1))
\]
obtained by tensoring with the Tate object $\tZpx 1$ is an isomorphism.
\index{embedding theorem}%
\end{thm}
\begin{proof}
Obviously,
one can replace the Tate object $\tZpx 1$ by the $\Gm$-twist $\tZcbx 1$.
Using the internal Hom functor, one is reduced to proving there is a canonical isomorphism 
\[
\mathscr D \rightarrow \derR \uHom(\tZcbx 1,\mathscr D\{1\})
\]
for every Milnor-Witt motivic complex $\mathscr D$.
Now $\MWrepZ(X)\{1\}$, being a direct factor of $\MWrepZ(X\times \Gm)$, is a compact object in $\DMtekZ$ by \chdmt, Proposition~\ref{prop:gm_objects} and $\mathscr D$ is a homotopy colimit of complexes obtained by suspensions of sheaves $\MWrepZ(Y)$ for some smooth scheme $Y$. One is then reduced to considering the case where $\mathscr D=\MWrepZ(Y)$.
 
For every $n\in \ZZ$ and $X\in\smk$, we have 
\[
\Hom_{\DMtekZ}(\MWrepZ(X),\MWrepZ(Y)[n])=\H^n_{\nis}(X,\Cstar{\MWrepZ(Y)})
\]
by \chdmt, Corollaries~\ref{cor:compare_Hom&cohomology} and \ref{cor:LA1}, while 
\[
\Hom_{\DMtekZ}(\MWrepZ(X),\derR \uHom(\tZcbx 1,\MWrepZ(Y)\{1\})[n])=\Hom_{\DMtekZ}(\MWrepZ(X)\{1\},\MWrepZ(Y)\{1\}[n])
\]
is the cokernel of the morphism
\[
\H^n_{\nis}(X,\Cstar{\MWrepZ(Y)\{1\}})\to \H^n_{\nis}(X \Gm,\Cstar{\MWrepZ(Y)\{1\}})
\]
induced by the projection $X \Gm\to X$. 
In view of Corollaries \ref{cor:local} and \ref{cor:local2}, we are reduced to proving that the tensor product by $\tZcbx 1$ induces an isomorphism in Zariski hypercohomology. 
This follows from Corollary \ref{cor:cancel}.
\end{proof}

The category of Milnor-Witt motives $\DMtkZ$ is obtained from $\DMtekZ$ by $\otimes$-inversion of the Tate object in the context of the stable model category structure on $\tZpx 1$-spectra (see \chdmt, \S\ref{sec:stablederivedcat}).
By construction there is a "$\tZpx 1$-suspension" functor relating these two categories (see \chdmt, Proposition~\ref{prop:Homlimiso}).
From Theorem \ref{thm:main} we immediately deduce the analogue for Milnor-Witt motives of Voevodsky's embedding theorem.
\begin{coro}
\label{cor:embedding}
Let $k$ be an infinite perfect field of characteristic not $2$.
There is a fully faithful suspension functor 
\[
\DMtekZ
\to
\DMtkZ.
\]
\end{coro}


\section{Examples of $\Aone$-homotopies}
\label{sec:examples}

The purpose of this section is to perform a number of computations which will be useful both in \chcomparison and \chdmt. 
Let us first briefly recall the definition of the first motivic Hopf map in the context of Milnor-Witt correspondences. 
The computation in \chfinitecw, Lemma~\ref{lem:explicitcontraction} shows that 
\[
\cork(\Gm,\Spec (k))=\sKMW_0(\Gm)=\sKMW_0(k)\oplus \sKMW_{-1}(k)\cdot [t].
\] 
Here,
the naturally defined class $[t]\in \sKMW_{1}(k(t))$ is an element of the subgroup $\sKMW_{1}(\Gm)$ since it has trivial residues at all closed points of $\Gm$.
Thus we have the element 
\[
\eta[t]\in\cork(\Gm,\Spec (k))=\sKMW_{0}(\Gm).
\]
Under pull-back with the $k$-rational point $1\colon \Spec(k)\to\Gm$ the element $[t]\in\sKMW_1(\Gm)$ goes to $[1]=0\in \sKMW_1(k)$.
Thus $\eta[t]$ pulls back trivially to $\sKMW_{0}(k)$, 
and defines a morphism of presheaves with MW-transfers $s(\eta)\in \Hom(\MWprep\{1\},\MWprep(\Spec (k)))$. 
Now, $\MWprep(\Spec (k))=\sKMW_0$ is a sheaf with MW-transfers, 
and it follows from \chdmt, Proposition~\ref{prop:exist_associated-W-t-sheaf} that this morphism induces a morphism of sheaves with MW-transfers $\tZcbx 1\to \tZ$, 
which we also denote by $s(\eta)$.

For any (essentially) smooth scheme $X\in\smk$ the bigraded Milnor-Witt motivic cohomology group $\HMW^{n,i}(X,\ZZ)$ is defined using the complex $\tZcbx i$ \chdmt, \S\ref{num:twist-1}.
There is a graded ring structure on the Milnor-Witt motivic cohomology groups satisfying the commutativity rule in \chdmt, Theorem~\ref{thm:commutative}. 
Moreover, 
by \chdmt, \S\ref{sec:naiveMilnorWittpresheaf}, 
every unit $a\in F^{\times}$ (where $F$ is a finitely generated field extension of $k$) gives rise to an element of $\cork(\Spec (F),\Gm)$ and also a Milnor-Witt motivic cohomology class $s(a)\in \HMW^{1,1}(F,\ZZ)$. Next, it follows from \chdmt, \S\ref{num:comparisonmap} that $s(\eta)$ yields a well-defined class in $\HMW^{-1,-1}(k,\ZZ)$ and thus a class in $\HMW^{-1,-1}(F,\ZZ)$ by pull-back. The definitions of $s(\eta)$ and $s(a)$ apply more generally to (essentially) smooth $k$-schemes \chdmt, \S\ref{num:comparisonmap}.

In analogy with motivic cohomology and Milnor $K$-theory, 
the integrally graded diagonal part of Milnor-Witt motivic cohomology can be identified with Milnor-Witt $K$-theory as sheaves of graded rings \chdmt, Theorem~\ref{thm:KMWmotivic}.
Using cancellation for Milnor-Witt correspondences we show results which are absolutely crucial for establishing the mentioned identification.

\begin{lem}
\label{lem:example1}
For every unit $a\in \OO(X)^\times$ we have 
\[
1+s(a)s(\eta)=\langle a\rangle\in \HMW^{0,0}(X,\ZZ)=\sKMW_0(X).
\]
\end{lem}
\begin{proof}
For the function field $k(X)$ of $X$, 
the naturally induced map
\[
\HMW^{0,0}(X,\ZZ)
=
\sKMW_0(X)
\to 
\sKMW_0(k(X))
=
\HMW^{0,0}(k(X),\ZZ)
\] 
is injective. 
We are thus reduced to proving the result for $X=\Spec(F)$, 
where $F$ is a finitely generated field extension of the base field $k$. 
Since this claim is obvious when $a=1$, we may assume that $a\neq 1$.

Consider the following diagram
\[
\begin{tikzcd}
 & \Spec(F)\Gm\Gm \ar[r,"p"] \ar[d,"q"'] & \Gm \\
\Spec(F) \ar[r,"\Gamma_a"] & \Spec(F) \Gm & 
\end{tikzcd}
\]
where $q$ is the projection on the first two factors, $p$ is the projection on the third factor, and $\Gamma_a$ is the graph of $a$.
By construction $s(a)$ is the image of the cycle 
\[
(\Gamma_a)_{*}(\langle 1\rangle)\in \H^1_{(t-a)}(F[t^{\pm 1}],\sKMW_1,\omega_{F[t^{\pm 1}]/F})\subset \cork(F,\Gm).
\] 
On the other hand, 
$s(\eta)$ corresponds to $\eta[t]\in \sKMW_0(\Gm)$ and the product is represented by the exterior product $\alpha$ of correspondences between $s(a)$ and $s(\eta)$, 
i.e., 
$\alpha:=q^*s(a)\cdot p^*s(\eta)$. 
We now want to apply the cancellation theorem to this element. 
However, we need a correspondence from $\Spec(F)\Gm \to \Gm$ (and in particular, we need a twist by $p_3^*\omega_{\Gm}$). 
One way to achieve this is to consider the Cartesian diagram
\[
\begin{tikzcd}
\Spec(F) \Gm \ar[r,"\Gamma"] \ar[d,"p_1"'] & \Spec(F)\Gm\Gm \ar[r,"p_2"] \ar[d,"p_3"] & \Gm \\
\Spec(F)\ar[r,"\Gamma_a"] & \Spec(F) \Gm & 
\end{tikzcd}
\]
where $p_1$ is the projection on the first factor, $p_3$ is the projection on the first and third factor, 
$p_2$ is the projection on the second factor, 
and $\Gamma_a$ is the graph of $a$. 
We now consider $p_3^*s(a)\cdot p_2^*s(\eta)$ which indeed can be seen as a finite Milnor-Witt correspondence. 
The relevant Cartier divisor for the procedure described in \S\ref{cancellationW} is defined by the rational function
\[
g_n:=(t_1^{n+1}-1)/(t_1^{n+1}-t_2).
\]
In this case we may choose $N(\alpha)=0$.  Since $\alpha:=p_3^*s(a)\cdot p_2^*s(\eta)$ and the intersection product is associative, setting $D:=D(g_0)$ we find that 
\[
D\cdot \alpha=(D\cdot p_3^*s(a))\cdot p_2^*s(\eta).
\]
Moreover,
using the base change and projection formulas, 
we get 
\[
D\cdot p_3^*s(a)= D\cdot \Gamma_*(\langle 1\rangle)=\Gamma_*(\Gamma^*(D\cdot \langle 1\rangle)).
\]
Now $D^\prime:=\Gamma^*(D)$ is the Cartier divisor on $\Spec (F) \Gm$ defined by the rational function 
\[
g^\prime:=(t-1)/(t-a).
\] 
Its divisor takes the form 
\[
\tdiv({D^\prime})=\langle 1-a\rangle \otimes (t-1)^\vee\otimes dt+\epsilon \langle a-1\rangle \otimes (t-a)^\vee\otimes dt,
\]
and its push-forward equals
\[
\Gamma_*(\tdiv({D^\prime}))=\langle 1-a\rangle\otimes (t_1-1)^\vee\wedge (t_2-a)^\vee\otimes dt_1\wedge dt_2+\epsilon \langle a-1\rangle  \otimes (t_1-a)^\vee\wedge (t_2-a)^\vee\otimes dt_1\wedge dt_2.
\]
The first summand is supported on the base point (of the first copy of $\Gm$) and we can forget it in our computations below. 
Multiplying by $p_2^*s(\eta)$ amounts to multiplying by $\eta[a]$ on the second summand, and then pushing forward to $\Spec(F)$.
We find that $\rho_0(\alpha)=\epsilon\langle a-1\rangle \eta[a]=-\langle 1-a\rangle \eta[a]$. 
Moreover,
\[
\langle 1-a\rangle \eta[a]=(1+\eta[1-a])\eta[a]=\eta[a], 
\]
and it follows that $\rho_0(\alpha)=-\eta[a]$. 
Hence $\langle -1\rangle\rho_0(\alpha)=\eta[a]$ and the result follows from Theorem \ref{thm:weakcancellation}.
\end{proof}

Let $\mu:\Gm \Gm\to \Gm$ denote the multiplication map, and let $p_i:\Gm \Gm\to \Gm$ for $i=1,2$ denote the projection map. 
We consider the corresponding graphs $\Gamma_\mu$, $\Gamma_i$, and their associated Milnor-Witt correspondences $\tilde\gamma_\mu$, $\tilde\gamma_i$. 
By construction we have $(\Gamma_\mu)_*(\langle 1\rangle)=\tilde\gamma_\mu$ and $(\Gamma_i)_*(\langle 1\rangle)=\tilde\gamma_i$. 
One checks that 
\[
\tilde\gamma_\mu-\tilde\gamma_1-\tilde\gamma_2\in \cork(\Gm\Gm,\Gm)
\] 
induces a morphism of Milnor-Witt presheaves $\alpha:\MWprep\{1\}\otimes \MWprep\{1\}\to \MWprep\{1\}$. 

\begin{lem}
\label{lem:example2}
The morphism of Milnor-Witt presheaves $\alpha$ is $\Aone$-homotopic to the suspension of 
\[
s(\eta):\MWprep\{1\}\to \MWprep(\Spec (k)).
\]
\end{lem}
\begin{proof}
Let $t_1,t_2,t_3$ denote the respective coordinates of $\Gm\Gm \Gm$ with corresponding supports given by $\supp(\tilde\gamma_\mu)=\Vanish(t_3-t_1t_2)$, 
$\supp(\tilde\gamma_1)=\Vanish(t_3-t_1)$, 
and $\supp(\tilde\gamma_2)=\Vanish(t_3-t_2)$. 
The Cartier divisors we want to employ are defined by the equation $g_n:=(t_2^{n+1}-1)/(t_2^{n+1}-t_3)$ for some $n\geq 0$.
We note that $D(g_n)$ intersects properly with both $\supp(\tilde\gamma_\mu)$ and $\supp(\tilde\gamma_1)$ if $n\geq 0$, 
while $D(g_n)$ and $\supp(\tilde\gamma_2)$ intersects properly if $n\geq 1$.
For this reason we set $D:=D(g_1)$. 

Next we compute the intersection product $D\cdot (\tilde\gamma_\mu-\tilde\gamma_1-\tilde\gamma_2)$. 
By the projection formula, 
we get 
\[
D\cdot \tilde\gamma_\mu=(\Gamma_\mu)_*(\langle 1\rangle)\cdot D=(\Gamma_\mu)_*(\Gamma_\mu^*D\cdot \langle 1\rangle).
\]
Note that $D_\mu:=\Gamma_\mu^*D$ is the principal Cartier divisor associated to the rational function
\[
g_\mu:=(t_2^2-1)/(t_2^2-t_1t_2)=t_2^{-1}(t_2^2-1)/(t_2-t_1).
\]
Its associated divisor is given explicitly by
\[
\tdiv(D_\mu)=(\langle 2\rangle\langle 1-t_1\rangle\otimes (t_2-1)^\vee+\langle -2\rangle \langle 1+t_1\rangle\otimes (t_2+1)^\vee-\langle t_1\rangle\langle 1-t_1^2\rangle\otimes (t_2-t_1)^\vee)\otimes dt_2.
\]
Similarly for $\tilde\gamma_i$ we get Cartier divisors $D_i:=\Gamma_i^*D$ for which 
\[
\tdiv(D_1)=(\langle 2\rangle \langle 1-t_1\rangle\otimes (t_2-1)^\vee+\langle -2\rangle \langle 1-t_1\rangle\otimes (t_2+1)^\vee-\langle 1-t_1\rangle\otimes (t_2^2-t_1)^\vee)\otimes dt_2,
\]
and 
\[
\tdiv(D_2)=\langle 1\rangle \otimes (t_2+1)^\vee\otimes dt_2.
\]
Pushing forward along $\Gamma_\mu$, $\Gamma_1$, $\Gamma_2$, 
and cancelling terms, 
we find an expression for $D\cdot (\tilde\gamma_\mu-\tilde\gamma_1-\tilde\gamma_2)$ of the form $(\alpha+\beta) dt_2\wedge dt_3$, 
with
\[
\beta=-\langle t_1(1-t_1^2)\rangle\otimes (t_2-t_1)^\vee\wedge (t_3-t_2^2)^\vee+\langle 1-t_1\rangle\otimes (t_2^2-t_1)^\vee\wedge (t_3-t_2^2)^\vee-\langle 1\rangle \otimes (t_2+1)^\vee\wedge (t_3+1)^\vee, 
\]
and
\[
\alpha=\langle -2\rangle \langle 1+t_1\rangle\otimes (t_2+1)^\vee\wedge (t_3+t_1)^\vee-\langle -2\rangle \langle 1-t_1\rangle\otimes (t_2+1)^\vee\wedge (t_3-t_1)^\vee.
\]
In the next step we take the push-forward to 
\[
\sKMW_0(\Gm)=\sKMW_0(k[t_1^{\pm 1}])\subset \sKMW_0(k(t_1)).
\] 
We set $x_1=\langle t_2-t_1,t_3-t_2^2\rangle$, $x_2=\langle t_2^2-t_1,t_3-t_2^2\rangle$ and $x_3=\langle t_2+1,t_3+1\rangle$.
Since $k(x_1)=k(t_1)$ and $k(x_3)=k(t_1)$, the push-forwards of all terms except the middle one in $\beta$ are evident. 
Note that $k(x_2)/k(t_1)$ is a degree $2$ field extension of the form $k(x_2)=k(t_2)=k(t_1)[t_2]/(t_2^2-t_1)$.
Thus we need to push-forward the form $\langle 1-t_2^2\rangle\langle 2t_2\rangle$ using the trace map, 
see \cite{Fasel19_c}*{Example~1.23}. 
By the projection formula, the result is the product of $\langle 1-t_2^2\rangle=\langle 1-t_1\rangle$ with the push-forward of $\langle 2t_2\rangle$. 
An easy computation shows the push-forward of $\langle 2t_2\rangle$ is hyperbolic, 
and we find that the push-forward of $(\alpha+\beta) dt_2\wedge dt_3$ along the projection $\Gm\Gm\Gm\to \Gm$ to the first factor equals 
\[
\nu:=\langle -1\rangle -\langle t_1(1-t_1^2)\rangle+\langle -2\rangle \langle 1+t_1\rangle-\langle -2\rangle \langle 1-t_1\rangle.
\]
Finally, 
we need to isolate the ``pointed'' component of $\nu\in \cork(\Gm,\Spec(k))$, 
i.e., 
to compute its image along the projection 
\[
\cork(\Gm,\Spec(k))\to \MWprep\{1\}(\Spec(k)).
\] 
This is obtained via the connecting homomorphism
\[
\sKMW_0(\Gm)\to \H_{\{0\}}^1(\Aone,\sKMW_0)=\sKMW_{-1}(k),
\]
under which $\nu\mapsto -\eta[t_1]=\langle -1\rangle \eta[t_1]$. 
Using Theorem \ref{thm:weakcancellation}, we can finally conclude that $\alpha$ is $\Aone$-homotopic to $s(\eta)$. 
\end{proof}

\section{Milnor-Witt motivic cohomology}
\label{MWmcs}

In this section we lay the foundations for Milnor-Witt motivic cohomology from the perspective of motivic functors and structured motivic ring spectra following the work on motivic cohomology in \cite{DRO03_c}.
This is of interest in the context of the very effective slice filtration \cite{Sp12_c}*{\S 5} due to recent work on hermitian $K$-theory in \cite{B17_c}.

A motivic space with Milnor-Witt transfers is an additive contravariant functor
\[
\mathcal{F}
\colon
\cork
\to
\sAb
\]
from the category of Milnor-Witt correspondences to simplicial abelian groups.
Let $\tMk$ 
denote the functor category comprised of motivic spaces with Milnor-Witt transfers. 
\index[notation]{mtk@$\tMk$}%
\index{motivic!space!with Milnor-Witt transfer}%
By composition with the opposite of the graph functor $\smk\to\cork$ (described explicitly in \chfinitecw, \S\ref{subsec:embedding}) we obtain the forgetful functor 
\[
u\colon\tMk\to \m_k
\] 
to pointed motivic spaces over $k$,
\index[notation]{mk@$\m_k$}%
\index{motivic!space}%
i.e.\ the category of pointed simplicial presheaves on the Nisnevich site of $\smk$.
By adding Milnor-Witt transfers to motivic spaces we obtain a left adjoint functor of $u$, 
denoted by
\[
\tZ_{\text{tr}}
\colon
\m_k
\to 
\tMk.
\]
More precisely, 
$\tZtr$ is the left Kan extension determined by setting
\[
\tZtr(h_{X}\wedge \Delta^{n}_{+})
:=
\cork(-,X)\otimes\ZZ[\Delta^{n}].
\]
Here,
$h_{X}$ denotes the motivic space represented by $X\in\smk$.
Recall that a motivic functor is an $\m_k$-enriched functor from finitely presented motivic spaces $\fMk$ to $\m_{k}$ \cite{DRO03_c}*{Definition~3.1}.
\index[notation]{fmk@$\fMk$}%
\index{motivic!space!finitely presented}%
We write $\mathbf{MF}_{k}$ for the closed symmetric monoidal category of motivic functors with unit the full embedding $\fMk\subset\Mk$.
A motivic functor $\mathcal{X}$  is ``continuous'' in the sense that it induces for all $A,B\in\fMk$ a map of internal hom objects
\[
\m(A,B)
\to
\m(\mathcal{X}(A),\mathcal{X}(B)),
\]
which is compatible with the enriched composition and identities.

\begin{dfn}
As a motivic functor, 
Milnor-Witt motivic cohomology $\MFMWZ$ is the composite
\[
\fMk
\xrightarrow{\subset}
\m_{k}
\xrightarrow{\tZtr}
\tMk
\xrightarrow{u}
\m_{k}.
\]
\end{dfn}
\index[notation]{mfzt@$\MFMWZ$}%
\index{motivic!functor!Milnor-Witt cohomology}%
We note that $\MFMWZ$ is indeed a motivic functor since for all $A,B\in\fMk$ there exist natural maps 
\[
\m(A,B)\wedge u\tZtr(A)
\to
u\tZtr\m(A,B)\wedge u\tZtr(A)
\to
u\tZtr(\m(A,B)\wedge A)
\to
u\tZtr(B).
\]

\begin{lem}
\label{lem:MZcommutativemonoid}
Milnor-Witt motivic cohomology $\MFMWZ$ is a commutative monoid in $\MFk$.
\end{lem}
\begin{proof}
We note that $\tMk$ is a closed symmetric monoidal category using \chdmt, \S\ref{num:sht_monoidal}.
Clearly the graph functor $\smk\to\cork$ is strict symmetric monoidal. 
Forgetting the additive structure furnished by Milnor-Witt correspondences is a lax symmetric monoidal functor.
By adjointness it follows that $u$ is lax symmetric monoidal, 
and $\tZtr$ is strict symmetric monoidal.
This yields the desired multiplicative structure on $\MFMWZ$ via the natural maps
\[
u\tZtr(A)\wedge u\tZtr(B)
\to
u(\tZtr(A)\otimes\tZtr(B))
\to
u(\tZtr(A\wedge B))
\]
defined for all $A,B\in\fMk$.
\end{proof}
\begin{rem}
Lemma \ref{lem:MZcommutativemonoid} and \cite{DRO03_c}*{Theorem~4.2} show that the category of modules over $\MFMWZ$ is a cofibrantly generated 
monoidal model category satisfying the monoid axiom.
\end{rem}

If $\mathcal{X}$ is a motivic functor the $\m$-enrichment yields an induced map 
\[
A\wedge \mathcal{X}(B)\to\mathcal{X}(A\wedge B)
\] 
for all finitely presented $A,B\in\fMk$.
In particular, 
for the Thom space $T:=\Aone/\Aone-\{0\}$ of the trivial line bundle on $\Aone$,
there are maps of motivic spaces $T\wedge \mathcal{X}(T^{n})\to\mathcal{X}(T^{n+1})$ for $n\geq 0$.
By \cite{DRO03_c}*{\S 3.7} this yields a lax symmetric monoidal evaluation functor from $\MFk$ to motivic symmetric spectra $\mathrm{Sp}^{\Sigma}(\Mk,T)$ introduced in \cite{Jar00_c}.
Due to the Quillen equivalence between $\MFk$ and $\mathrm{Sp}^{\Sigma}(\Mk,T)$ \cite{DRO03_c}*{Theorem~3.32} we also write $\MFMWZ$ for the corresponding 
motivic symmetric spectrum of Milnor-Witt motivic cohomology.
Equivalently, 
following the approach in \cite{DRO03_c}*{\S 4.2},
we may view $\MFMWZ$ as a motivic symmetric spectrum with respect to the suspension coordinate $(\PP^1,\infty)$.
Recall that a motivic symmetric spectrum $\mathcal{E}$ is fibrant if and only if it is an $\Omega_{T}$-spectrum, 
i.e., 
$\mathcal{E}_n$ is motivic fibrant and $\mathcal{E}_n\to \Omega_{T}\mathcal{E}_{n+1}$ is a motivic weak equivalence for all $n\geq 0$.
Voevodsky's cancellation theorem for finite correspondences implies the motivic cohomology spectrum is fibrant \cite{Voe98_c}*{Theorem~6.2}.
In the context of Bredon motivic cohomology and the cancellation theorem for equivariant finite correspondences \cite{Hel15_c}*{Theorem~9.7}, 
this is shown in \cite{Hel16_c}*{Theorem~3.4, Appendix~A.4}.
By the same type of arguments we obtain:.
\begin{thm}
\label{MWfibrant}
If $k$ is an infinite perfect field then the Milnor-Witt motivic cohomology spectrum $\MFMWZ$ is an $\Omega_{T}$-spectrum in $\mathrm{Sp}^{\Sigma}(\m_{k},T)$.
\end{thm}

\begin{rem}
\label{rem:MWrepresentability}
Along the lines of \cite{Hel16_c}*{Theorem~3.4} one can show that $\MFMWZ$ represents Milnor-Witt motivic cohomology groups as defined in $\DMtekZ$
\chdmt, Definition~\ref{def:generalizedMW}.
\end{rem}

Due to work of Morel \cite{Morel04-2_c} the integral graded homotopy module of the unramified Milnor-Witt $K$-theory sheaf corresponds to a motivic spectrum 
$\MheartZ$ 
\index[notation]{mhZ@$\MheartZ$}%
\index{spectrum!motivic unramified Milnor-Witt}%
in the heart $\SH^{\heartsuit}(k)$ of the homotopy $t$-structure on the stable motivic homotopy category $\SH(k)$.
Here the equivalence of categories between homotopy modules, 
i.e.\ sequences $\{\mathcal{F}_{n}\}_{n\in\ZZ}$ of strictly $\Aone$-homotopy invariant sheaves of abelian groups together with contraction isomorphisms $\mathcal{F}_{n}\cong (\mathcal{F}_{n+1})_{(-1)}$ 
for all $n\geq 0$,
and the heart of the stable motivic homotopy category is induced by the stable $\Aone$-homotopy sheaf $\piaone_{0,\ast}$ with inverse $\Mheart$.
Recall that for every $\mathcal{E}\in\SH(k)$ the Nisnevich sheaf $\piaone_{s,t}(\mathcal{E})$ on $\smk$ is obtained from the presheaf
\[
X
\mapsto
\SH(k)(\Sigma^{s,t}X_{+},\mathcal{E}), 
\text{ where } s,t\in\ZZ.
\]
By using basic properties of effective homotopy modules, 
Bachmann \cite{B17_c}*{Lemma~12} shows there is an isomorphism of Nisnevich sheaves
\[
\piaone_{\ast,\ast}\MheartZ
\cong
\sKMW_{\ast}.
\]

\begin{lem}
There is a canonically induced isomorphism of Nisnevich sheaves
\[
\sKMW_{\ast}
\xrightarrow{\cong}
\piaone_{\ast,\ast}\MFMWZ.
\]
\end{lem}
\begin{proof}
This follows from \chdmt, Theorem~\ref{thm:KMWmotivic} together with the representability of Milnor-Witt motivic cohomology stated in Remark \ref{rem:MWrepresentability}.
\end{proof}


\chapter[Comparison theorem]{A comparison theorem for Milnor-Witt motivic cohomology\author{Baptiste Calmès and Jean Fasel}}
\label{ch:comparison}

\section*{Abstract}
Let $k$ be an infinite perfect field of characteristic not $2$. We prove that 
\[
\HMW^{n,n}(\Spec(L),\ZZ)\simeq \KMW_n(L)
\]
for any finitely generated field extension $L/k$ and any $n\in\ZZ$.

\section*{Introduction}

Our main purpose in this chapter is to compute the MW-motivic cohomology group of a field $L$ in bidegree $(n,n)$, namely the group $\HMW^{n,n}(L,\ZZ)$. In \chdmt, Theorem~\ref{thm:gradedringhomo}, we defined a graded ring homomorphism
\[
\Phi:\sKMW_*\to \bigoplus_{n\in\ZZ} \shHMW^{n,n}
\]
where the left-hand side is the unramified Milnor-Witt $K$-theory sheaf constructed in \cite{Morel12_d}*{\S 3} and the right-hand side is the Nisnevich sheaf associated to the presheaf $U\mapsto \HMW^{n,n}(U,\ZZ)$. The homomorphism $\Phi$ is obtained via a morphism of sheaves $\Gm^{\wedge n}\to \shHMW^{n,n}$ and the right-hand side has the property to be strictly $\Aone$-invariant by \chdmt, Proposition~\ref{prop:exist_associated-W-t-sheaf} and Theorem~\ref{thm:Wsh&A1}. It follows that $\Phi$ is then the universal morphism described in \cite{Morel12_d}*{Theorem~3.37}. In this article, we prove that $\Phi$ is an isomorphism. This can be checked on finitely generated field extensions of the base field $k$ \cite{Morel12_d}*{Theorem~1.12} and thus our main theorem takes the following form.

\begin{thm*}
Let $L/k$ be a finitely generated field extension with $k$ perfect of characteristic different from $2$. Then, the homomorphism of graded rings
\[
\Phi_L:\bigoplus_{n\in\ZZ} \KMW_n(L)\to \bigoplus_{n\in\ZZ} \HMW^{n,n}(L,\ZZ).
\]
is an isomorphism. Consequently, the homomorphism of shaves of graded rings
\[
\Phi:\sKMW_*\to \bigoplus_{n\in\ZZ} \shHMW^{n,n}
\]
is an isomorphism.
\end{thm*}

The isomorphism in the theorem generalizes the result on (ordinary) motivic cohomology in the sense that the following diagram commutes
\[
\begin{tikzcd}
\bigoplus_{n\in\ZZ} \KMW_n(L)\ar[r,"\Phi_L"] \ar[d] & \bigoplus_{n\in\ZZ} \HMW^{n,n}(L,\ZZ) \ar[d] \\
\bigoplus_{n\in\NN} \KM_n(L) \ar[r] & \bigoplus_{n\in\NN} \HM^{n,n}(L,\ZZ)
\end{tikzcd}
\]
where the vertical homomorphisms are the ``forgetful'' homomorphisms and the bottom map is the isomorphism produced by Nesterenko-Suslin-Totaro (\cite{Nesterenko90} and \cite{Totaro92}). Unsurprisingly, our proof is very similar to theirs but there are some essential differences. For instance, the complex in weight one, denoted by $\tZpx 1$, admits an epimorphism to $\sKMW_1[-1]$ paralleling the epimorphism $\ZZ(1)\to \sKM_1[-1]$. However, we are not able to prove directly that the kernel of the epimorphism $\tZpx 1\to \sKMW_1[-1]$ is acyclic (this will be proved elsewhere). 
We are thus forced to compute by hand its cohomology at the right spot in Proposition \ref{prop:inductionstep}. This result being obtained, we then prove that $\Phi$ respects transfers for finitely generated field extensions. This is obtained in Theorem \ref{thm:respect} using arguments essentially identical to \cite{Mazza06_d}*{Lemma~5.11} or \cite{Neshitov18_d}*{Lemma~9.5}.

The paper is organized as follows. In Section \ref{sec:MWmotivic}, we review the basics of MW-motivic cohomology needed in the paper, adding useful results. For instance, we prove a projection formula in Theorem \ref{thm:projection} which is interesting on its own. In Section \ref{sec:main}, we proceed with the proof of our main theorem, starting with the construction of a left inverse of $\Phi$. We then pass to the proof that $\Phi$ is an isomorphism in degree $1$, which is maybe the most technical result of this work. As already mentioned above, we then conclude with the proof that $\Phi$ respects transfers, obtaining as a corollary our main result.

\subsection*{Conventions}
The schemes are separated of finite type over some infinite perfect field $k$ with $\charac k\neq 2$. If $X$ is a smooth connected scheme over $k$, we denote by $\Omega_{X/k}$ the sheaf of differentials of $X$ over $\Spec(k)$ and write $\ome{X}{k}:=\det\Omega_{X/k}$ for its canonical sheaf. In general we define $\ome{X}{k}$ on each connected component separately. We use the same notation if $X$ is the localization of a smooth scheme at any point. If $k$ is clear from the context, we omit it from the notation. If $f\colon X\to Y$ is a morphism of (localizations of) smooth schemes, we set $\omega_{f}=\ome{X}{k}\otimes f^*\ome{Y}{k}^\vee$. If $X$ is a scheme and $n\in\NN$, we denote by $X^{(n)}$ the set of codimension $n$ points in $X$.

\section{Milnor-Witt motivic cohomology}\label{sec:MWmotivic}

The general framework of this article is the category of finite MW-correspondences as defined in \chfinitecw, \S\ref{sec:fmwcorr}. We briefly recall the construction of this category for the reader's convenience. If $X$ and $Y$ are smooth connected schemes over $k$, we say that a closed subset $T\subset X\times Y$ is admissible if its irreducible components (endowed with their reduced structure) are finite and surjective over $X$. The set $\Adm (X,Y)$ of admissible subsets of $X\times Y$ can be ordered by inclusions, and we can consider it as a category. For any $T\in \Adm(X,Y)$, we can consider the Chow-Witt group
\[
\chst {d_Y}T{X\times Y}{\omega_Y}
\]
where $d_Y$ is the dimension of $Y$ and $\omega_Y=p_Y^*\omega_{Y/k}$ with $p_Y:X\times Y\to Y$ the projection. As in \chfinitecw, \S \ref{sec:chowwitt}, we consider $\omega_Y$ as a graded line bundle (in the sense of \cite{Del87_d}*{\S 4}) set in degree $d_Y:=\dimn Y$. If $T\subset T^\prime$ are two admissible subsets, we consider the extension of support homomorphism
\[
\chst {d_Y}T{X\times Y}{\omega_Y}\to \chst {d_Y}{T^\prime}{X\times Y}{\omega_Y}
\]
and set 
\[
\cork(X,Y)=\varinjlim_{T\in \Adm(X,Y)}\chst {d_Y}T{X\times Y}{\omega_Y}.
\]
The composition of finite MW-correspondences is defined using the product of cycles in Chow-Witt groups with supports (\chfinitecw, \S\ref{sec:compositionMWcorr}) and we obtain the category $\cork$ whose objects are smooth schemes and morphisms are $\cork(X,Y)$. The exterior product endows $\cork$ with the structure of a symmetric monoidal category.

Having this category at hand, we may define the category of MW-presheaves $\pshMW(k)$
\index[notation]{pshtk@$\pshMW(k)$}%
as the category of additive functors $\cork\to \Ab$. For any smooth scheme $X$, we can define the presheaf $\MWprep(X)$
\index[notation]{ctx@$\MWprep(X)$}%
by $Y\mapsto \cork(Y,X)$ for any $Y$ and thus obtain the Yoneda embedding functor $\MWprep:\cork\to \pshMW(k)$. The category $\pshMW(k)$ is a symmetric monoidal category, with tensor product $\otimes$ uniquely defined by the property that the Yoneda embedding is monoidal, i.e.\ we have $\MWprep(X)\otimes \MWprep(Y)=\MWprep(X\times Y)$. One can also define an internal Hom functor $\uHom$ which is characterized by the property that $\uHom (\MWprep(X),F)=F(X\times-)$ for any $F\in \pshMW(k)$. 

Recall next that we have a functor $\tilde\gamma:\smk\to \cork$ which is the identity on objects and associates to a morphism of schemes the finite MW-correspondence described in \chfinitecw, \S\ref{subsec:embedding} (which is basically the graph). This yields a functor 
\[
\tilde\gamma_*:\pshMW(k)\to \psh(k)
\] 
where the target is the category of presheaves of abelian groups on $\smk$. As usual, we say that a presheaf with MW-transfer $F$ is a sheaf in a topology $\tau$, and we write that $F$ is a $\tau$-sheaf with MW-transfers, if $\tilde\gamma_*(F)$ is a sheaf in this topology. Usually, we consider either the Zariski or the Nisnevich topology on $\smk$. Interestingly, the representable presheaves $\MWprep(X)$ are Zariski sheaves with MW-transfers (\chfinitecw, Proposition~\ref{prop:zarsheaf}) but not Nisnevich sheaves with transfers (\chfinitecw, Example~\ref{exm:notnis}). However, one can show that the sheaf associated to $F\in \pshMW(k)$ can be endowed with the unique structure of a sheaf with MW-transfers (\chdmt, Proposition~\ref{prop:exist_associated-W-t-sheaf}). Note that it is easy to check that if $F$ is a $\tau$-sheaf with MW-transfers, then $\uHom (\MWprep(X),F)$ is also a $\tau$-sheaf with MW-transfers.

\subsection{Motivic cohomology}

Let $\tZcbx 1$ be the Zariski sheaf with MW-transfers which is the cokernel of the morphism 
\[
\MWprep(k)\to \MWprep(\Gmk)
\]
induced by the unit in $\Gmk$. For any $q\in \ZZ$, we consider next the Zariski sheaf with MW-transfer $\tZcbx q$ defined by 
\[
\tZcbx q=\begin{cases} \tZcbx 1^{\otimes q} & \text{if $q\geq 0$.} \\
\uHom\big(\tZcbx 1^{\otimes q}, \MWprep(k)\big) & \text{if $q<0$.} \end{cases}
\]
\index[notation]{ztq@$\tZcbx q$}%

Let now $\Delta^{\bullet}$ be the cosimplicial object whose terms in degree $n$ are
\[
\Delta^n=\Spec(k[t_0,\ldots,t_n]/\big(\sum t_i-1)\big)
\]
and with the standard face and degeneracy maps. For any presheaf $F\in \pshMW(k)$, we obtain a simplicial presheaf $\uHom(\MWprep(\Delta^{\bullet}),F)$ whose associated complex of presheaves with MW-transfers is denoted by $\Cstar{F}$. 
\index[notation]{cf@$\Cstar F$}%
\index{Suslin(-Voevodsky) singular complex}%
If $F$ is further a $\tau$-sheaf with MW-transfers, then $\Cstar{F}$ is a complex of sheaves with MW-transfers. In particular, $\tZpx q:=\Cstar{\tZcbx q}[-q]$ 
\index[notation]{ztq@$\tZpx q$}%
is such a complex and we have the following definition.

\begin{dfn}
For any $p,q\in \ZZ$ and any smooth scheme $X$, we define the MW-motivic cohomology groups of $X$ as
\[
\HMW^{p,q}(X,\ZZ)=\mathbb{H}^p_{\zar}(X,\tZpx q).
\]
\index[notation]{hmwpqxz@$\HMW^{p,q}(X,\ZZ)$}%
\end{dfn}

\begin{rem}
In \chdmt, \S\ref{num:Tate} and Definition~\ref{def:generalizedMW}, the motivic cohomology groups are defined using the complexes associated to the simplicial Nisnevich sheaves with MW-transfers constructed from the Nisnevich sheaves with transfers associated to the presheaves $\tZpx q$. The two definitions coincide by \chcancellation, Corollary~\ref{cor:local}.
\end{rem}

The complexes $\Cstar{\tZpx q}$ are in fact complexes of Zariski sheaves of $\KMW_0(k)$-modules (\chfinitecw, \S\ref{sec:modulestructure}), and it follows that the MW-motivic cohomology groups are indeed $\KMW_0(k)$-modules. These modules are by construction contravariantly functorial in $X$. Moreover, for any $p,q\in \ZZ$, we have a homomorphism of $\KMW_0(k)$-modules
\[
\HMW^{p,q}(X,\ZZ)\to \HM^{p,q}(X,\ZZ)
\]
where the latter denotes the ordinary motivic cohomology group of $X$ for $q\geq 0$, and $\HM^{p,q}(X,\ZZ)=0$ for $q<0$, the $\KMW_0(k)$-module structure on the right-hand side being obtained via the rank homomorphism $\KMW_0(k)\to \ZZ$ (\chfinitecw, \S\ref{sec:motivic}).

Even though MW-motivic cohomology is defined a priori only for smooth schemes, it is possible to extend the definition to limits of smooth schemes, following the usual procedure (described for instance in \chfinitecw, \S\ref{sec:limits}). In particular, we can consider MW-motivic cohomology groups $\HMW^{p,q}(L,\ZZ)$ for any finitely generated field extension $L/k$. We will use this routinely in the sequel without further comments.

\subsection{The ring structure}\label{sec:ringstructure}

The definition of MW-motivic cohomology given in \chdmt, Definition~\ref{def:generalizedMW} immediately yields a (bigraded) ring structure on MW-motivic cohomology
\[
\HMW^{p,q}(X,\ZZ)\otimes \HMW^{p^\prime,q^\prime}(X,\ZZ)\to \HMW^{p+p^\prime,q+q^\prime}(X,\ZZ)
\]
\index{Milnor-Witt!motivic!cohomology cup-product}%
\index{cup-product|see{Milnor-Witt motivic cohomology cup-product}}%
fulfilling the following properties.
\begin{enumerate}
\item The product is (bi-)graded commutative in the sense that 
\[
\HMW^{p,q}(X,\ZZ)\otimes \HMW^{p^\prime,q^\prime}(X,\ZZ)\to \HMW^{p+p^\prime,q+q^\prime}(X,\ZZ)
\] 
is $(-1)^{pp^\prime}\langle (-1)^{qq^\prime}\rangle$-commutative. In particular, $\HMW^{0,0}(X,\ZZ)$ is central and the $\KMW_0(k)$-module structure is obtained via the ring homomorphism 
\[
\KMW_0(k)=\HMW^{0,0}(k,\ZZ)\to \HMW^{0,0}(X,\ZZ).
\]
\item The homomorphism $\HMW^{*,*}(X,\ZZ)\to \HM^{*,*}(X,\ZZ)$ is a graded ring homomorphism.
\end{enumerate}

\subsection{A projection formula}\label{sec:projection}

In this section, we prove a projection formula for finite surjective morphisms having trivial relative canonical bundles. Let then $f:X\to Y$ be a finite surjective morphism between smooth connected schemes, and let $\chi:\OO_X\to \omega_f=\omega_{X/k}\otimes f^*\omega_{Y/k}^\vee$ be a fixed isomorphism. Recall from \chfinitecw, Example~\ref{ex:push-forwards} that we then have a finite MW-correspondence $\alpha:=\alpha(f,\chi):Y\to X$ defined as the composite
\[
\sKMW_0(X)\simeq \sKMW_0(X,\omega_f)\stackrel {f_*}{\to} \chst {d_X}{\Gamma^t_f(X)}{Y\times X}{\omega_{Y\times X/k}\otimes \omega_{Y/k}^\vee}\simeq \chst {d_X}{\Gamma^t_f(X)}{Y\times X}{\omega_X}
\]
where the first isomorphism is induced by $\chi$, the second homomorphism is the push-forward along the transpose of the graph $\Gamma^t_f:X\to Y\times X$ and the third isomorphism is deduced from the isomorphisms of graded line bundles
\[
(d_Y+d_X,\omega_{Y\times X/k})\otimes (-d_Y,\omega_{Y/k}^\vee)\simeq (d_Y,\omega_Y)\otimes (d_X,\omega_X)\otimes (-d_Y,\omega_{Y/k}^\vee)
\]
and
\[
(d_X,\omega_X)\otimes (-d_Y,\omega_{Y/k}^\vee)\simeq (d_X,\omega_X)\otimes (d_Y,\omega_Y)\otimes  (-d_Y,\omega_{Y/k}^\vee)\simeq (d_X,\omega_X),
\]
the first isomorphism being given by the commutativity constraint of graded line bundles (\chfinitecw, \S\ref{sec:twisting}).

We observe that $\alpha$ induces a ``push-forward'' homomorphism $F(X)\to F(Y)$ for any $F\in \pshMW(k)$ through the composite  
\[
F(X)=\Hom_{\pshMW(k)}(\MWprep(X),F)\stackrel{(-)\circ \alpha}{\to} \Hom_{\pshMW(k)}(\MWprep(Y),F)=F(Y).
\]
In particular, we obtain homomorphisms
\[
f_*:\HMW^{p,q}(X,\ZZ)\to \HMW^{p,q}(Y,\ZZ)
\]
for any $p,q\in\ZZ$, which depend on the choice of $\chi$.

On the other hand, $f$ induces a finite MW-correspondence $X\to Y$ that we still denote by $f$ and therefore a pull-back homomorphism
\[
f^*:\HMW^{p,q}(Y,\ZZ)\to \HMW^{p,q}(X,\ZZ).
\]
We will need the following lemma to prove the projection formula.

\begin{lem}\label{lem:twoways}
Let $f:X\to Y$ be a finite surjective morphism between smooth connected schemes, and let $\chi:\OO_X\to \omega_f$ be an isomorphism. Let $\alpha=\alpha(f,\chi)\in \cork(Y,X)$ and let $\Delta_X$ (resp. $\Delta_Y$) be the diagonal embedding $X\to X\times X$ (resp. $Y\to Y\times Y$). Then, the following diagram of finite MW-correspondences commutes
\[
\begin{tikzcd}
Y \ar[rr,"\Delta_Y"] \ar[d,equal] & & Y\times Y \ar[r,"{(1\times \alpha)}"] & Y\times X \ar[d,equal] \\
Y \ar[r,"\alpha"'] & X \ar[r,"\Delta_X"'] & X\times X \ar[r,"f\times 1"'] & Y\times X,
\end{tikzcd}
\]
i.e.\ $(1\times \alpha)\Delta_Y=(f\times 1)\Delta_X\alpha$.
\end{lem}

\begin{proof}
It suffices to compute both compositions, and we start with the top one. The composite of these two finite MW-correspondences is given by the commutative diagram
\[
\begin{tikzcd}
X \ar[r] \ar[d,"{\Gamma_{(\Delta_Y\circ f)}^t}"'] \ar[r,"{\Gamma_{(\Delta_Y\circ f)}^t}"] & Y \times Y \times X \ar[r] \ar[d,"{1\times \Gamma_{(1\times f)}^t}"] & Y \times X \ar[d,"{\Gamma_{(1\times f)}^t}"] & \\
Y \times Y \times X \ar[r,"{\Gamma_{\Delta_Y}\times 1\times 1}"] \ar[d] & Y \times Y \times Y \times Y \times X \ar[r] \ar[d] & Y \times Y \times Y \times X \ar[r] \ar[d] & Y \times X \\
Y \ar[r,"{\Gamma_{\Delta_Y}}"'] & Y \times Y \times Y \ar[r] \ar[d] & Y \times Y & \\
& Y & & 
\end{tikzcd}
\]
where the squares are Cartesian and the non-labelled arrows are projections (vertically to the first factors and horizontally to the last factors). The composite is given by the push-forward along the projection $p:Y\times Y\times Y\times Y\times X\to Y\times Y\times X$ defined by $(y_1,y_2,y_3,y_4,x)\mapsto (y_1,y_4,x)$ of the product of the respective pull-backs to $Y\times Y\times Y\times Y\times X$ of $(\Gamma_{\Delta_Y})_*(\langle 1\rangle)$ and $(\Gamma_{(1\times f)}^t)_*(\langle 1\rangle)$. Using the base change formula (\chfinitecw, Proposition~\ref{prop:basechange}), we see that it amounts to push-forward the product
\[
(\Gamma_{\Delta_Y}\times 1\times 1)_*(\langle 1\rangle)\cdot (1\times \Gamma_{(1\times f)}^t)_*(\langle 1\rangle).
\]
Using the projection formula for Chow-Witt groups with supports (\chfinitecw, Corollary~\ref{cor:pformula}), the latter equals
\[
(\Gamma_{\Delta_Y}\times 1\times 1)_*((\Gamma_{\Delta_Y}\times 1\times 1)^*((1\times \Gamma_{(1\times f)}^t)_*(\langle 1\rangle)))
\]
and the base-change formula once again shows that we have to push-forward along $p$ the cycle
\[
(\Gamma_{\Delta_Y}\times 1\times 1)_*(\Gamma_{(\Delta_Y\circ f)}^t)_*(\langle 1\rangle)
\]
Finally, the equality  $p\circ (\Gamma_{\Delta_Y}\times 1\times 1)=\id$ shows that the composite $(1\times \alpha)\circ \Delta_Y$ is given by the correspondence $(\Gamma_{(\Delta_Y\circ f)}^t)_*(\langle 1\rangle)$. 

For the second composite, we consider the following commutative diagram 
\[
\begin{tikzcd}
X \ar[r,"\Gamma_f^t"] \ar[d,"{\Gamma_{((f\times 1)\Delta_X)}}"'] & Y \times X \ar[r] \ar[d,"{1\times \Gamma_{((f\times 1)\Delta_X)}}"] & X \ar[d,"{\Gamma_{((f\times 1)\Delta_X)}}"] & \\
X \times Y \times X \ar[r,"{\Gamma_f^t\times 1\times 1}"] \ar[d] & Y \times X \times Y \times X \ar[r] \ar[d] & X \times Y \times X \ar[r] \ar[d] & Y \times X \\
X \ar[r,"{\Gamma_f^t}"'] & Y \times X \ar[d] \ar[r] & X & \\
& Y & & 
\end{tikzcd}
\]
where, as before, the squares are Cartesian and the non-labelled arrows are projections (vertically to the first factors and horizontally to the last factors). Arguing as above, we find that the composite is the push-forward along the projection $q:Y\times X\times Y\times X\to Y\times Y\times X$ omitting the second factor of the product
\[
(\Gamma_f^t\times 1\times 1)_*(\langle 1\rangle)\cdot (1\times \Gamma_{((f\times 1)\Delta_X)})_*(\langle 1\rangle).
\]
The projection and the base-change formulas show that the latter is equal to 
\[
(\Gamma_f^t\times 1\times 1)_*(\Gamma_{((f\times 1)\Delta_X)})_*(\langle 1\rangle)
\]
whose push-forward along $q$ is $(\Gamma_{(\Delta_Y\circ f)}^t)_*(\langle 1\rangle)$ as 
\[
q(\Gamma_f^t\times 1\times 1)(\Gamma_{((f\times 1)\Delta_X)})=\Gamma_{(\Delta_Y\circ f)}^t. \qedhere
\]
\end{proof}

\begin{thm}[Projection formula]\label{thm:projection}
Let $f:X\to Y$ be a finite surjective morphism between smooth connected schemes, and let $\chi:\OO_X\to \omega_f$ be an isomorphism. For any $x\in \HMW^{p,q}(X,\ZZ)$ and $y\in \HMW^{p^\prime,q^\prime}(Y,\ZZ)$, we have 
\[
y\cdot f_*(x)=f_*(f^*y\cdot x)
\]
in $\HMW^{p+p^\prime,q+q^\prime}(Y,\ZZ)$.
\index{Milnor-Witt!motivic!cohomology, projection formula}%
\end{thm}

\begin{proof}
Let $\DMtek$ be the category of MW-motives (\chdmt, \S\ref{sec:AA1derived}).
By \chdmt, Corollary~\ref{cor:cancellation}, we have 
\[
\HMW^{p,q}(X,\ZZ)=\Hom_{\DMtek}(\tMot(X),\tZcbx q[p-q])
\] 
for any $p,q\in\ZZ$. 
The product structure on MW-motivic cohomology is obtained composing the exterior product
\[
\begin{tikzcd}
\Hom_{\DMtek}(\tMot(X),\tZcbx q[p-q])\otimes \Hom_{\DMtek}(\tMot(X),\tZcbx{q^\prime}[p^\prime-q^\prime])\ar[d] \\
\Hom_{\DMtek}(\tMot(X)\otimes \tMot(X),\tZcbx q\otimes \tZcbx{q^\prime}[p+p^\prime-q-q^\prime])
\end{tikzcd}
\] 
with the pull-back 
\[
\begin{tikzcd}
\Hom_{\DMtek}(\tMot(X)\otimes \tMot(X),\tZcbx q\otimes \tZcbx{q^\prime}[p+p^\prime-q-q^\prime])\ar[d] \\
\Hom_{\DMtek}(\tMot(X),\tZcbx q\otimes \tZcbx{q^\prime}[p+p^\prime-q-q^\prime])
\end{tikzcd}
\] 
along the map $\tMot(X)\to \tMot(X\times X)$ induced by the diagonal morphism $\Delta_X\colon X\to X\times X$.

If $x\in \Hom_{\DMtek}(\tMot(X),\tZcbx q[p-q])$, $y\in \Hom_{\DMtek}(\tMot(Y),\tZcbx{q^\prime}[p^\prime-q^\prime])$, the product $y\cdot f_*(x)$ is then of the form 
\[
y\cdot f_*(x)=(y\otimes x)\circ (1\otimes \alpha)\circ \Delta_Y,
\] 
while $f_*(f^*y\cdot x)=(y\otimes x)\circ (f\otimes 1)\circ \Delta_X\circ \alpha$. 
The result then follows from Lemma \ref{lem:twoways}.
\end{proof}

\begin{rem}
It would suffice to have a fixed isomorphism $\Lb\otimes \Lb\simeq \omega_f$ (for some line bundle $\Lb$ on $X$) to get an orientation in the sense of \chfinitecw, \S\ref{sec:canonicalorientations} and thus a finite MW-correspondence $\alpha$ as above. We let the reader make the necessary modifications in the arguments of both Lemma \ref{lem:twoways} and Theorem \ref{thm:projection}.
\end{rem}

\begin{rem}
It follows from \chdmt, Theorem~\ref{thm:commutative} that the same formula holds for the left module structure, i.e.
\[
f_*(x)\cdot y =f_*(x\cdot f^*y).
\]
\end{rem}

\begin{exm}\label{ex:inseparable}
As usual, it follows from the projection formula that the composite $f_*f^*$ is multiplication by $f_*(\langle 1\rangle)$. Let us now compute $f^*f_*$ in some situations that will be used later. Let us start with the general situation, i.e.\ $f:X\to Y$ is a finite surjective morphism and $\chi:\OO_X\to \omega_f$ an isomorphism. The composite $f^*f_*$ is given by precomposition with the correspondence $f\circ\alpha(f,\chi)$ which we can compute using the diagram
\[
\begin{tikzcd} 
X \times_Y X \ar[r,"{(1\times 1)}"] \ar[d,"{(1\times 1)}"'] & X \times X \ar[r] \ar[d,"{1\times \Gamma_f^t}"'] & X \ar[d,"{\Gamma_f^t}"] & \\
X \times X \ar[r,"{\Gamma_f\times 1}"'] \ar[d] & X \times Y \times X \ar[r] \ar[d] & Y \times X \ar[r] \ar[d] & X \\
X \ar[r,"{\Gamma_f}"'] & X \times Y \ar[d] \ar[r] & Y &  \\
 & X & & 
\end{tikzcd}
\] 
where the non-labelled vertical arrows are projections on the first factor and the non-labelled horizontal arrows are projections on the second factor. As usual, the base change formula shows that the composite is equal to the projection on $X\times X$ of 
\[
(\Gamma_f\times 1)_*(\langle 1\rangle)\cdot (1\times \Gamma_f^t)_*(\langle 1\rangle).
\] 
In general the top left square is not transverse, and we cannot use the base-change formula to compute the above product. 

Suppose now that $f:X\to Y$ is finite and \'etale. In that case, we have a canonical isomorphism $f^*\omega_Y\simeq \omega_X$ yielding a canonical choice for the isomorphism
\[
\chi:\OO_X\to \omega_f.
\] 
Moreover, $X\times_Y X$ decomposes as $X\times_Y X=X_1\sqcup X_2\sqcup \ldots \sqcup X_n$ where each term $X_i$ is finite and \'etale over $X$ with a ``structural'' morphism $p_i:X_i\to X$. In that case, the above top right square is transverse and we see that 
\[
(\Gamma_f\times 1)_*(\langle 1\rangle)\cdot (1\times \Gamma_f^t)_*(\langle 1\rangle)=(\Gamma_f\times 1)_*(\Delta_*\sum (p_i)_*(\langle 1\rangle)).
\] 
where $\Delta:X\to X\times X$ is the diagonal map. Thus the composite $f\circ\alpha(f,\chi)$ is equal to $\Delta_*\sum (p_i)_*(\langle 1\rangle)$. It follows immediately that we have a commutative diagram
\[
\begin{tikzcd}
\HMW^{p,q}(X,\ZZ) \ar[r,"{\sum p_i^*}"] \ar[d,"f_*"'] & \bigoplus_i \HMW^{p,q}(X_i,\ZZ) \ar[d,"{\sum (p_i)_*}"] \\
\HMW^{p,q}(Y,\ZZ) \ar[r,"f^*"'] & \HMW^{p,q}(X,\ZZ)
\end{tikzcd}
\]
for any $p,q\in \ZZ$. 

Suppose next that $\charac k=p$, that $X\subset Y\times \Aone$ is the set of zeroes of $t^p-a$ for some global section $a\in \OO_Y(Y)$ (we still suppose that $X$ is smooth over $k$). In that case, we see that the reduced scheme of $X\times_Y X$ is just $X$ (but the former has nilpotent elements) and it follows that $f\circ\alpha(f,\chi)$ is a correspondence supported on the diagonal $\Delta(X)\subset X\times X$. It follows that there is an element $\sigma\in \KMW_0(X)$ such that the following diagram commutes
\[
\begin{tikzcd}
\sKMW_0(X) \ar[r] \ar[d,"{\cdot \sigma}"'] & \cork(X,X) \ar[d,"{f\circ\alpha(f,\chi)}"] \\
\sKMW_0(X) \ar[r] & \cork(X,X)
\end{tikzcd}
\]
where the horizontal arrows are induced by the push-forward map $\Delta_*:\sKMW_0(X)\to \chst {d_X}{\Delta(X)}{X\times X}{\omega_X}$. Now, $\sigma$ can be computed using the composite $\KMW_0(k(X))\to \KMW_0(k(Y))\to \KMW_0(k(X))$, where the first map is the push-forward (defined using $\chi$) and the second map is the pull-back. It follows essentially from \cite{Fasel08_d}*{Lemme~6.4.6} that $\sigma=p_{\epsilon}$.
\end{exm}

\subsection{The homomorphism}\label{sec:homomorphism}

Let $L/k$ be a finitely generated field extension. It follows from the definition of MW-motivic cohomology that $\HMW^{p,q}(L,\ZZ)=0$ provided $p>q$. The next step is then to identify $\HMW^{p,p}(L,\ZZ)$. To this aim,  we constructed in \chdmt, Theorem~\ref{thm:gradedringhomo} a graded ring homomorphism
\[
\KMW_*(L)\to \bigoplus_{n\in\ZZ}\HMW^{n,n}(L,\ZZ) 
\]
which we now recall. For $a\in L^\times$, we can consider the corresponding morphism $a:\Spec(L)\to \Gmk$ which defines a finite MW-correspondence $\Gamma_a$ in $\cork(L,\Gmk)$. Now, we have a surjective homomorphism $\cork(L,\Gmk)\to \HMW^{1,1}(L,\ZZ)$ and we let $s([a])$ be the image of $\Gamma_a$ under this map. Next, consider the element 
\[
\eta[t]\in \sKMW_0(\Gmx L)=\cork(\Gmx L,k)=\cork(\Gmk\times L,k)=\uHom\big(\MWprep(\Gmk),\MWprep(k)\big)(L).
\] 
We define $s(\eta)$ to be its image under the projections
\[
\uHom\big(\MWprep(\Gmk),\MWprep(k)\big)(L)\to \uHom\big(\tZcbx 1,\MWprep(k)\big)(L)\to \HMW^{-1,-1}(L,\ZZ).
\]
The following theorem is proved in \chdmt, Theorem~\ref{thm:gradedringhomo} (using computations of \chcancellation, \S\ref{sec:examples}).

\begin{thm}
The associations $[a]\mapsto s([a])$ and $\eta\mapsto s(\eta)$ induce a homomorphism of graded rings
\[
\Phi_L:\KMW_*(L)\to \bigoplus_{n\in\ZZ}\HMW^{n,n}(L,\ZZ).
\]
\end{thm}

By construction, the above homomorphism fits in a commutative diagram of graded rings
\[
\begin{tikzcd}
\KMW_*(L) \ar[r,"{\Phi_L}"] \ar[d] & \bigoplus_{n\in\ZZ}\HMW^{n,n}(L,\ZZ) \ar[d] \\
\KM_*(L) \ar[r] & \bigoplus_{n\in\ZZ}\HM^{n,n}(L,\ZZ)
\end{tikzcd}
\]
where the vertical projections are respectively the natural map from the Milnor-Witt $K$-theory to the Milnor $K$-theory and the ring homomorphism of Section \ref{sec:ringstructure} respectively, and the bottom horizontal homomorphism is the map constructed by Totaro-Nesterenko-Suslin (e.g. \cite{Mazza06_d}*{Lecture 5}).


\section{Main theorem}\label{sec:main}

\subsection{A left inverse}\label{sec:left}

In this section, we construct for $q\geq 0$ a left inverse to the homomorphism $\Phi_L$ of Section \ref{sec:homomorphism}. By definition, 
\[
\MWprep(\Gm^q)(L)=\bigoplus_{x\in (\Gmx L^q)^{(q)}}\chst qx{\Gmx L^q}{\omega_{\Gmx L^q}}.
\] 
If $x_1,\ldots,x_q$ are parameters for $\Gm^q$, we have $[x_1,\ldots,x_q]\in \H^0(\Gm^q,\KMW_q)$ that we can pull-back to $\H^0(\Gmx L^q,\KMW_q)$. For any $x\in(\Gmx L^q)^{(q)}$, we may consider the product
\[
\chst qx{\Gmx L^q}{\omega_{\Gmx L^q}}\otimes \H^0(\Gm^q,\KMW_q)\to \H^q_x(\Gmx L^q,\KMW_{2q},\omega_{\Gm^q})
\]  
to obtain out of $\alpha\in \chst qx{\Gmx L^q}{\omega_{\Gm^q}}$ and element $\alpha\cdot [x_1,\ldots,x_q]\in \H^q_x(\Gmx L^q,\KMW_{2q},\omega_{\Gm^q})$.  If $p\colon \Gmx L^q\to \Spec(L)$ is the projection, we denote the image of $\alpha\cdot [x_1,\ldots,x_q]\in \H^q_x(\Gmx L^q,\KMW_{2q},\omega_{\Gm^q})$ along the push-forward map (with supports) 
\[
p_*\colon \H^q_x(\Gmx L^q,\KMW_{2q},\omega_{\Gm^q})\to \H^0(L,\KMW_q)=\KMW_q(L)
\]
by $f_x(\alpha)$. In this fashion, we obtain a homomorphism
\[
f:\MWprep(\Gm^q)(L)\to \KMW_q(L)
\]
which is easily seen to factor through $(\tZcbx q)(L)$ since $[1]=0\in \KMW_1(L)$.

We now check that this homomorphism vanishes on the image of $(\tZcbx q)(\Aone_L)$ in $(\tZcbx q)(L)$ under the boundary homomorphism. This will follow from the next lemma.

\begin{lem}
Let $Z\in \Adm(\Aone_L,\Gm^q)$, $p\colon \Gmx L^q\to \Spec(L)$ and $p_{\Aone_L}\colon \Aone_L\times \Gm^q\to \Aone_L$ be the projections and let $Z_i:=p_{\Aone_L}^{-1}(i)\cap Z$ (endowed with its reduced structure) for $i=0,1$. Let $j_i:\Spec(L)\to \Aone_L$ be the inclusions in $i=0,1$ and let $g_i:\Gmx L^q\to \Aone_L\times \Gm^q$ be the induced maps.  Then the homomorphisms
\[
p_*(g_i)^*:\H^q_Z(\Aone_L\times \Gm^q,\KMW_{2q},\omega_{\Gm^q})\to \H^q_{Z_i}(\Gmx L^q,\KMW_{2q},\omega_{\Gm^q})\to \KMW_q(L)
\]
are equal.
\end{lem}

\begin{proof}
For $i=0,1$, consider the Cartesian square
\[
\begin{tikzcd}
\Gmx L^q \ar[r,"g_i"] \ar[d,"p"'] & \Aone_L\times \Gmx L^q \ar[d,"p_{\Aone_L}"] \\ 
\Spec(L) \ar[r,"{j_i}"'] & \Aone_L
\end{tikzcd} 
\]
The vertical maps being in particular flat, we may use the base change formula (\chfinitecw, Proposition \ref{prop:basechange} and Remark \ref{rem:noncartesian}) to obtain $(j_i)^*(p_{\Aone_L})_*=p_*(g_i)^*$.
 The claim follows from the fact that $(j_0)^*=(j_1)^*$ by homotopy invariance \cite{Feld20_a}*{Theorem 9.4} (note that we may indeed apply homotopy invariance as $p_{\Aone_L}(Z)= \Aone_L$).
\end{proof}

\begin{prop}\label{prop:left}
The homomorphism $f:\MWprep(\Gm^q)(L)\to \KMW_q(L)$ induces a homomorphism
\[
\theta_L:\HMW^{q,q}(L,\ZZ)\to \KMW_q(L)
\]
for any $q\geq 1$. 
\end{prop}

\begin{proof}
Observe that the group $\HMW^{q,q}(L,\ZZ)$ is the cokernel of the homomorphism
\[
\partial_0-\partial_1:\tZcbx q(\Aone_L)\to \tZcbx q(L)
\]
and that \chfinitecw, Example~\ref{ex:flat_example} shows that $\partial_i:\MWprep(\Gm^q)(\Aone_L)\to \MWprep(\Gm^q)(L)$ is induced by $g_i^*$. On the other hand $f$ is defined as $f(-)=p_*(-\cdot [x_1,\ldots,x_q])$ Let then $\beta\in \tZcbx q(\Aone_L)$ be represented by $\beta\in \chst qZ{\Aone_L\times \Gm^q}{\omega_{\Gm^q}}$. Pulling back $[x_1,\ldots,x_q]$ along the projection $\Aone_L\times \Gmx L^q\to\Gmx L^q$ we obtain an element $[x_1,\ldots,x_q]\in \H^0(\Aone_L\times \Gmx,\KMW_q)$ which has the property to pull-back to $[x_1,\ldots,x_q]\in \H^0(\Gmx,\KMW_q)$ along $g_i^*$ for $i=0,1$. Thus
\[
p_*((g_i)^*(\beta)\cdot [x_1,\ldots,x_q])=p_*(g_i)^*(\beta\cdot [x_1,\ldots,x_q])
\]
and the above lemma yields
\[
f((g_0)^*(\beta))=p_*((g_0)^*(\beta)\cdot [x_1,\ldots,x_q])=p_*((g_1)^*(\beta)\cdot [x_1,\ldots,x_q])=f((g_1)^*(\beta)).
\]
\end{proof}

\begin{coro}\label{cor:splitinjective}
The homomorphism 
\[
\Phi_L:\bigoplus_{n\in\ZZ} \KMW_n(L) \to \bigoplus_{n\in\ZZ} \HMW^{n,n}(L,\ZZ).
\]
is split injective.
\end{coro}

\begin{proof}
It suffices to check that $\theta_L\Phi_L=\id$, which is straightforward.
\end{proof}

The following result will play a role in the proof of the main theorem.

\begin{prop}\label{prop:transe}
Let $n\in\ZZ$ and let $F/L$ be a finite field extension. Then, the following diagram commutes
\[
\begin{tikzcd}
\HMW^{n,n}(F,\ZZ) \ar[d,"{\Tr_{F/L}}"'] \ar[r,"\theta_F"] & \KMW_n(F) \ar[d,"\Tr_{F/L}"] \\
\HMW^{n,n}(L,\ZZ) \ar[r,"\theta_L"'] & \KMW_n(L). 
\end{tikzcd}
\]
\end{prop}

\begin{proof}
Let $X$ be a smooth connected scheme and let $\beta\in \cork(X,\Gm^n)$ be a finite MW-correspondence with support $T$ (see \chfinitecw, Definition~\ref{dfn:support} for the notion of support). Each connected component $T_i$ of $T$ has a fraction field $k(T_i)$ which is a finite extension of $k(X)$ and we find that $\beta$ can be seen as an element of 
\[
\bigoplus_{i} \KMW_0(k(T_i),\omega_{k(T_i)/k}\otimes \omega_{k(X)/k}^\vee)
\]
Now, the morphism $T_i\subset X\times \Gm^{\times n}\to \Gm^{\times n}$ gives invertible global sections $a_1,\ldots,a_n$ and we define a map
\[
\theta_X:\cork(X,\Gm^{\times n})\to \KMW_n(k(X))
\]
by $\beta\mapsto \sum \Tr_{k(T_i)/k(X)}(\beta_i[a_1,\ldots,a_n])$, where $\beta_i$ is the component of $\beta$ in the group $\KMW_0(k(T_i),\omega_{k(T_i)/k(X)})$. This map is easily seen to be a homomorphism, and its limit at $k(X)$ is the morphism defined at the beginning of Section \ref{sec:left}.

Let now $X$ and $Y$ be smooth connected schemes over $k$, $f:X\to Y$ be a finite morphism and $\chi:\OO_X\to \omega_f$ be an isomorphism inducing a finite MW-correspondence $\alpha(f,\chi):Y\to X$ as in Section \ref{sec:projection}. We claim that the diagram
\[
\begin{tikzcd}
\cork(X,\Gm^{\times n}) \ar[r,"{\theta_X}"] \ar[d,"{\circ\alpha(f,\chi)}"'] & \KMW_n(k(X)) \ar[d,"{\Tr_{k(X)/k(Y)}}"] \\
\cork(Y,\Gm^{\times n}) \ar[r,"{\theta_Y}"'] & \KMW_n(k(Y)),
\end{tikzcd}
\]
where the right arrow is obtained using $\chi$, commutes. If $\beta$ is as above, we have 
\[
\Tr_{k(X)/k(Y)}(\theta_X(\beta))=\sum \Tr_{k(X)/k(Y)}( \Tr_{k(T_i)/k(X)}(\beta_i[a_1,\ldots,a_n]))
\]
and the latter is equal to $\sum  \Tr_{k(T_i)/k(Y)}(\beta_i[a_1,\ldots,a_n])$ by functoriality of the transfers. On the other hand, the isomorphism $\chi:\OO_X\to \omega_f$ can be seen as an element in $\sKMW_0(X,\omega_f)$, yielding an element of $\KMW_0(k(X),\omega_f)$ that we still denote by $\chi$. The image of $\beta\circ\alpha(f,\chi)$ can be seen as the element $\beta\cdot \chi$ of   
\[
\bigoplus_{i} \KMW_0(k(T_i),\omega_{k(T_i)/k(Y)})
\] 
where we have used the isomorphism 
\[
\omega_{k(T_i)/k(X)}\otimes \omega_f= \omega_{k(T_i)/k(X)}\otimes \omega_{k(X)/k(Y)}\simeq \omega_{k(T_i)/k(Y)}.
\]
It is now clear that $\theta_Y(\beta\circ\alpha(f,\chi))=\sum  \Tr_{k(T_i)/k(Y)}(\beta_i[a_1,\ldots,a_n])$ and the result follows. 
\end{proof}


\subsection{Proof of the main theorem}

In this section we prove our main theorem, namely that the homomorphism 
\[
\Phi_L:\bigoplus_{n\in\ZZ} \KMW_n(L) \to \bigoplus_{n\in\ZZ} \HMW^{n,n}(L,\ZZ)
\]
is an isomorphism. We first observe that $\Phi_L$ is an isomorphism in degrees $\leq 0$. In degree $0$, we indeed know from \chfinitecw, Example~\ref{exm:negativecohomology} that both sides are $\KMW_0(L)$. Next, \chcancellation, Lemma~\ref{lem:example1} yields
\[
\Phi_L(\langle a\rangle)=\Phi_L(1+\eta [a])=1+s(\eta)s(a)=\langle a\rangle.
\]
It follows $\Phi_L$ is a homomorphism of graded $\KMW_0(L)$-algebras and the result in degrees $<0$ follows then from the fact that $\HMW^{p,p}(L,\ZZ)=W(L)=\KMW_{-p}(L)$ by \chfinitecw, Example~\ref{exm:negativecohomology} and \chdmt, Proposition~\ref{prop:explicit}.

We now prove the result in positive degrees, starting with $n=1$. Recall that we know from Corollary \ref{cor:splitinjective} that $\Phi_L$ is split injective, and that it therefore suffices to prove that $\Phi_L$ is surjective to conclude.

For any $d,n\geq 1$ and any field extension $L/k$ let $M_n^{(d)}(L)\subset \cork(L,\Gm^{\times n})$ be the subgroup of correspondences whose support is a finite union of field extensions $E/L$ of degree $\leq d$ (see \chfinitecw, Definition~\ref{dfn:support} for the notion of support of a correspondence). Let $\HMW^{n,n}(L,\ZZ)^{(d)}\subset \HMW^{n,n}(L,\ZZ)$ be the image of $M_n^{(d)}(L)$ under the surjective homomorphism
\[
\cork(L,\Gm^{\times n})\to \HMW^{n,n}(L,\ZZ).
\]
Observe that 
\[
\HMW^{n,n}(L,\ZZ)^{(d)}\subset \HMW^{n,n}(L,\ZZ)^{(d+1)} \quad \text{and} \quad \HMW^{n,n}(L,\ZZ)=\cup_{d\in \NN}\HMW^{n,n}(L,\ZZ)^{(d)}.
\]
\begin{lem}\label{lem:Hnn1}
The subgroup $\HMW^{n,n}(L,\ZZ)^{(1)}\subset \HMW^{n,n}(L,\ZZ)$ is the image of the homomorphism
\[
\Phi_L:\KMW_n(L)\to \HMW^{n,n}(L,\ZZ).
\]
\end{lem}

\begin{proof}
By definition, observe that the homomorphism $\KMW_n(L)\to \HMW^{n,n}(L,\ZZ)$ factors through $\HMW^{n,n}(L,\ZZ)^{(1)}$. Let then $\alpha\in \HMW^{n,n}(L,\ZZ)^{(1)}$. We may suppose that $\alpha$ is the image under the homomorphism $\cork(L,\Gm^{\times n})\to \HMW^{n,n}(L,\ZZ)$ of a correspondence $a$ supported on a field extension $E/L$ of degree $1$, i.e.\ $E=L$. It follows that $a$ is determined by a form $\phi\in \KMW_0(L)$ and a $n$-uple $a_1,\ldots,a_n$ of elements of $L$. This is precisely the image of $\Phi_L(\phi\cdot [a_1,\ldots,a_n])$ under the homomorphism $\KMW_n(L)\to \HMW^{n,n}(L,\ZZ)$.
\end{proof}

\begin{prop}\label{prop:inductionstep}
For any $d\geq 2$, we have $\HMW^{1,1}(L,\ZZ)^{(d)}\subset \HMW^{1,1}(L,\ZZ)^{(d-1)}$.
\end{prop}

\begin{proof}
By definition, $\HMW^{1,1}(L,\ZZ)^{(d)}$ is generated by correspondences whose supports are field extensions $E/L$ of degree at most $d$. Such correspondences are determined by an element $a\in E^\times$ given by the composite $\Spec(E)\to \Gmx L\to \Gm$ together with a form $\phi\in \KMW_0(E,\omega_{E/L})$ given by the isomorphism 
\[
\KMW_0(E,\omega_{E/L})\to \chst 1{\Spec(E)}{\Gmx L}{\omega_{\Gmx L/L}}.
\]
We denote this correspondence by the pair $(a,\phi)$. Recall from \chfinitecw, Lemma~\ref{lem:orientation} that there is a canonical orientation $\xi$ of $\omega_{E/L}$ and thus a canonical element $\chi\in \cork\big(\Spec(L),\Spec(E)\big)$ yielding the transfer map 
\[
\Tr_{E/L}:\cork(\Spec(E),\Gm)\to \cork(\Spec(L),\Gm)
\]
which is just the composition with $\chi$ (\chfinitecw, Example~\ref{ex:push-forwards}). Now $\phi=\psi\cdot \xi$ for $\psi\in \KMW_0(E)$, and it is straightforward to check that the Chow-Witt correspondence $(a,\psi)$ in $\cork(\Spec(E),\Gm)$ determined by $a\in E^\times$ and $\psi\in \KMW_0(E)$ satisfies $\Tr_{E/L}(a,\psi)=(a,\phi)$. Now $(a,\psi)\in \HMW^{1,1}(E,\ZZ)^{(1)}$ and therefore belongs to the image of the homomorphism $\KMW_1(E)\to \HMW^{1,1}(E,\ZZ)$. There exists thus $a_1,\ldots,a_n,b_1,\ldots,b_m\in E^\times$ (possibly equal) such that $(a,\psi)=\sum s(a_i)-\sum s(b_j)$. To prove the lemma, it suffices then to show that $\Tr_{E/L}(s(b))\in \HMW^{1,1}(L,\ZZ)^{(d-1)}$ for any $b\in E^\times$. 

Let thus $b\in E^\times$. By definition, $s(b)\in \HMW^{1,1}(E,\ZZ)$ is the class of the correspondence $\tilde\gamma(b)$ associated to the morphism of schemes $\Spec(E)\to \Gm$ corresponding to $b$. If $F(b)\subset E$ is a proper subfield, we see that $\Tr_{E/L}(s(b))\in \HMW^{1,1}(F,\ZZ)^{(d-1)}$, and we may thus suppose that the minimal polynomial $p$ of $b$ over $F$ is of degree $d$. By definition, $\Tr_{E/L}(s(b))$ is then represented by the correspondence associated to the pair $(b,\langle 1\rangle\cdot \xi)$. Consider the total residue homomorphism (twisted by the vector space $\omega_{F[t]/k}\otimes \omega_{F/k}^\vee$)
\begin{equation}\label{eq:residue}
\partial:\KMW_1(F(t),\omega_{F(t)/k}\otimes \omega_{F/k}^\vee)\to \bigoplus_{x\in \Gmx F^{(1)}} \KMW_0(F(x),(\mathfrak m_x/\mathfrak m_x^2)^\vee \otimes _{F[t]} \omega_{F[t]/k}\otimes \omega_{F/k}^\vee)
\end{equation}
where $\mathfrak m_x$ is the maximal ideal corresponding to $x$. Before working further with this homomorphism, we first identify $(\mathfrak m_x/\mathfrak m_x^2)^\vee \otimes _{F[t]} \omega_{F[t]/k}\otimes \omega_{F/k}^\vee$. Consider the canonical exact sequence of $F(x)$-vector spaces
\[
\mathfrak m_x/\mathfrak m_x^2\to \Omega_{F[t]/k}\otimes_{F[t]} F(x) \to \Omega_{F(x)/k}\to 0.
\]
A comparison of the dimensions shows that the sequence is also left exact (use the fact that $\Spec(F(x))$ is the localization of a smooth scheme of dimension $d$ over the perfect field $k$), and we thus get a canonical isomorphism
\[
\omega_{F[t]/k}\otimes_{F[t]} F(x)\simeq \mathfrak m_x/\mathfrak m_x^2 \otimes_{F(x)}\omega_{F(x)/k}.
\]
It follows that
\[
(\mathfrak m_x/\mathfrak m_x^2)^\vee \otimes _{F[t]} \omega_{F[t]/k}\otimes \omega_{F/k}^\vee\simeq \omega_{F(x)/k}\otimes_F \omega_{F/k}^\vee.
\] 
We can thus rewrite the residue homomorphism \eqref{eq:residue} as a homomorphism
\[
\partial:\KMW_1\big(F(t),\omega_{F(t)/k}\otimes \omega_{F/k}^\vee\big)\to \bigoplus_{x\in (\Aone_F\setminus 0)^{(1)}} \hspace{-3ex}\KMW_0\big(F(x),\omega_{F(x)/k}\otimes_F \omega_{F/k}^\vee\big)
\]
Moreover, an easy dimension count shows that the canonical exact sequence
\[
\Omega_{F/k}\otimes F[t]\to \Omega_{F[t]/k}\to \Omega_{F[t]/F}\to 0
\]
is also exact on the left, yielding a canonical isomorphism $\omega_{F(t)/k}\simeq \omega_{F/k}\otimes \omega_{F(t)/F}$ and thus a canonical isomorphism $\omega_{F/k}^\vee\otimes \omega_{F(t)/k}\simeq \omega_{F(t)/F}$.
If $n$ is the transcendance degree of $F$ over $k$, we see that the canonical isomorphism 
\[
\omega_{F/k}^\vee\otimes \omega_{F(t)/k}\simeq\omega_{F(t)/k}\otimes \omega_{F/k}^\vee
\] 
is equal to $(-1)^{n(n+1)}$-times the switch isomorphism, i.e.\ is equal to the switch isomorphism. Altogether, the residue homomorphism reads as 
\[
\partial:\KMW_1\big(F(t),\omega_{F(t)/F}\big)\to \bigoplus_{x\in (\Aone_F\setminus 0)^{(1)}} \hspace{-3ex}\KMW_0\big(F(x),\omega_{F(x)/k}\otimes_F \omega_{F/k}^\vee\big).
\]

Let now $p(t)\in F[t]$ be the minimal polynomial of $b$ over $F$.
Following \cite{Morel12_d}*{Definition~4.26} (or \chfinitecw, \S\ref{sec:cohomological}), write $p(t)=p_0(t^{l^m})$ with $p_0$ separable and set $\omega=p_0^\prime(t^{l^m})\in F[t]$ if $\charac k=l$. If $\charac k=0$, set $\omega=p^\prime(t)$. It is easy to see that the element $\langle \omega\rangle [p]\cdot dt$ of $\KMW_1(F(t),\omega_{F(t)/F})$ ramifies in $b\in \Gmx F^{(1)}$ and on (possibly) other points corresponding to field extensions of degree $\leq d-1$. Moreover, the residue at $b$ is exactly $\langle 1\rangle \cdot \xi$, where $\xi$ is the canonical orientation of $\omega_{F(b)/F}$.

Write the minimal polynomial $p(t)\in F[t]$ of $b$ as $p=\sum_{i=0}^d \lambda_i t^i$ with $\lambda_d=1$ and $\lambda_0\in F^\times$, and decompose $\omega=c\prod_{j=1}^n q_j^{m_j}$, where $c\in F^\times$ and $q_j\in F[t]$ are irreducible monic polynomials. Let $f=(t-1)^{d-1}(t-(-1)^{d}\lambda_0) \in F[t]$. Observe that $f$ is monic and satisfies $f(0)=p(0)$. Let $F(u,t)=(1-u)p+uf$. Since $f$ and $p$ are monic and have the same constant terms, it follows that $F(u,t)=t^d+\ldots+\lambda_0$ and therefore $F$ defines an element of $\Adm(\Aone_F,\Gm)$. For the same reason, every $q_j$ (seen as a polynomial in $F[u,t]$ constant in $u$) defines an element in $\Adm(\Aone_F,\Gm)$. The image of $\langle\omega\rangle [F]\cdot dt \in \KMW_1(F(u,t),\omega_{F(u,t)/F(u)})$ under the residue homomorphism
\[
\partial:\KMW_1\big(F(u,t),\omega_{F(u,t)/F(u)}\big)\to \hspace{-3ex} \bigoplus_{x\in (\Aone_F\times_k \Gm)^{(1)}} \hspace{-3ex}\KMW_0\big(F(x),(\mathfrak m_x/\mathfrak m_x^2)^\vee\otimes _{F[t]}\omega_{F[u,t]/F[u]}\big)
\]
is supported on the vanishing locus of $F$ and the $g_j$, and it follows that it defines a finite Chow-Witt correspondence $\alpha$ in $\cork(\Aone_F,\Gm)$. The evaluation $\alpha(0)$ of $\alpha$ at $u=0$ consists from $\langle 1\rangle \cdot \xi$ and correspondences supported on the vanishing locus of the $q_j$, while $\alpha(1)$ is supported on the vanishing locus of $f$ and the $q_j$. As $\alpha(0)$ and $\alpha(1)$ yield the same class in $\HMW^{1,1}(L,\ZZ)$, we may write $\langle 1\rangle \cdot \xi$ as a sum (and difference) of correspondences supported on the vanishing loci of the $q_j$ and $f$. As the degree of the polynomials $q_j$ are $\leq d-1$ and $f$ vanishes at $1$ and $\pm \lambda_0$, $\langle 1\rangle \cdot \xi$ is then an element of $\HMW^{1,1}(L,\ZZ)^{(d-1)}$.
\end{proof}

\begin{coro}\label{cor:degree1}
The homomorphism 
\[
\Phi_L: \KMW_1(L)\to \HMW^{1,1}(L,\ZZ)
\]
is an isomorphism for any finitely generated field extension $L/k$.
\index{comparison theorem}%
\end{coro}

\begin{proof}
We know that the homomorphism is (split) injective. By Lemma \ref{lem:Hnn1}, $\Phi_L$ factors as
\[
\KMW_1(L)\twoheadrightarrow \HMW^{1,1}(L,\ZZ)^{(1)}\subset  \HMW^{1,1}(L,\ZZ)
\]
The above proposition shows that $\HMW^{1,1}(L,\ZZ)=\HMW^{1,1}(L,\ZZ)^{(1)}$ and it follows that $\Phi_L$ is surjective.
\end{proof}

We can now prove that $\Phi$ respects transfers following \cite{Mazza06_d}*{Lemma~5.11} and \cite{Neshitov18_d}*{Lemma~9.5}.

\begin{thm}\label{thm:respect}
Let $n\in\NN$ and let $F/L$ be a finite field extension. Then the following diagram commutes
\[
\begin{tikzcd}
\KMW_n(F) \ar[r,"{\Phi_F}"] \ar[d,"{\Tr_{F/L}}"'] & \HMW^{n,n}(F,\ZZ) \ar[d,"{\Tr_{F/L}}"] \\
\KMW_n(L) \ar[r,"{\Phi_L}"] & \HMW^{n,n}(L,\ZZ). 
\end{tikzcd}
\]
\end{thm}

\begin{proof}
First, we know from Proposition \ref{prop:transe} that the diagram
\[
\begin{tikzcd}
\HMW^{n,n}(F,\ZZ) \ar[d,"{\Tr_{F/L}}"'] \ar[r,"{\theta_F}"] & \KMW_n(F) \ar[d,"{\Tr_{F/L}}"] \\
\HMW^{n,n}(L,\ZZ) \ar[r,"{\theta_L}"] & \KMW_n(L).
\end{tikzcd}
\]
commutes. If $\Phi_F$ and $\Phi_L$ are isomorphisms, it follows from Corollary \ref{cor:splitinjective} that $\theta_F$ and $\theta_L$ are their inverses and thus that the diagram
\[
\begin{tikzcd}
\KMW_n(F) \ar[r,"{\Phi_F}"] \ar[d,"{\Tr_{F/L}}"'] & \HMW^{n,n}(F,\ZZ) \ar[d,"{\Tr_{F/L}}"] \\
\KMW_n(L) \ar[r,"{\Phi_L}"] & \HMW^{n,n}(L,\ZZ). 
\end{tikzcd}
\]
also commutes. We may then suppose, using Corollary \ref{cor:degree1} that $n\geq 2$. Additionally, we may suppose that $[F:L]=p$ for some prime number $p$. Following \cite{Mazza06_d}*{Lemma~5.11}, we first assume that $L$ has no field extensions of degree prime to $p$. In that case, it follows from \cite{Morel12_d}*{Lemma~3.25} that $\KMW_n(F)$ is generated by elements of the form $\eta^{m}[a_1,a_2,\ldots,a_{n+m}]$ with $a_1\in F^\times$ and $a_i\in L^{\times}$ for $i\geq 2$. We conclude from the projection formula \ref{thm:projection}, its analogue in Milnor-Witt $K$-theory and the fact that $\Phi$ is a ring homomorphism that the result holds in that case.

Let's now go back to the general case, i.e.\ $[F:L]=p$ without further assumptions. Let $L^\prime$ be the maximal prime-to-$p$ field extension of $L$.  Let $\alpha\in \HMW^{n,n}(L,\ZZ)$ be such that its pull-back to $\HMW^{n,n}(L^\prime,\ZZ)$ vanishes. It follows then that there exists a finite field extension $E/L$ of degree $m$ prime to $p$ such that the pull-back of $\alpha$ to $\HMW^{n,n}(E,\ZZ)$ is trivial. Let $f:\Spec(E)\to \Spec(L)$ be the corresponding morphism. For any unit $b\in E^\times$, we have $\langle b\rangle\cdot f^*(\alpha)=0$ and it follows from the projection formula once again that $f_*(\langle b\rangle\cdot f^*(\alpha))=f_*(\langle b\rangle)\cdot \alpha=0$. We claim that there is a unit $b\in E^\times$ such that $f_*(\langle b\rangle)=m_{\epsilon}$. Indeed, we can consider the factorization $L\subset F^{sep}\subset E$ where $F^{sep}$ is the separable closure of $L$ in $E$ and the extension $F^{sep}\subset E$ is purely inseparable. If the claim holds for each extension, then it holds for $L\subset E$. We may thus suppose that the extension is either separable or purely inseparable. In the first case, the claim follows from \cite{Se68_d}*{Lemme 2} while the second case follows from \cite{Fasel08_d}*{Théorème~6.4.13}. Thus, for any $\alpha\in \HMW^{n,n}(L,\ZZ)$ vanishing in $\HMW^{n,n}(L^\prime,\ZZ)$ there exists $m$ prime to $p$ such that $m_{\epsilon}\alpha=0$.

Let now $\alpha\in \KMW_n(F)$ and  $t(\alpha)=(\Tr_{F/L}\circ \Phi_F-\Phi_L\circ \Tr_{F/L})(\alpha)\in \HMW^{n,n}(L,\ZZ)$. Pulling back to $L^\prime$ and using the previous case, we find that $m_{\epsilon}t(\alpha)=0$. On the other hand, the above arguments show that if the pull-back of $t(\alpha)$ to $F$ is trivial, then $p_{\epsilon}t(\alpha)=0$ and thus $t(\alpha)=0$ as $(p,m)=1$. Thus, we are reduced to show that $f^*(t(\alpha))=0$ where $f:\Spec(F)\to \Spec(L)$ is the morphism corresponding to $L\subset F$. 

Suppose first that $F/L$ is purely inseparable. In that case, we know from Example \ref{ex:inseparable} that $f^*f_*:\HMW^{n,n}(F,\ZZ)\to \HMW^{n,n}(F,\ZZ)$ is multiplication by $p_{\epsilon}$. The same property holds for Milnor-Witt $K$-theory. This is easily checked using the definition of the transfer, or alternatively using Proposition \ref{prop:transe}, the fact that $\theta_F$ is surjective and Example \ref{ex:inseparable}. Altogether, we see that $f^*t(\alpha)=p_{\epsilon}\Phi_F(\alpha)-\Phi_F(p_\epsilon\cdot \alpha)$ and therefore $f^*t(\alpha)=0$ since $\Phi_F$ is $\KMW_0(F)$-linear.

Suppose next that $F/L$ is separable. In that case, we have $F\otimes_L F=\prod_i F_i$ for field extensions $F_i/F$ of degree $\leq p-1$. We claim that the diagrams
\[
\begin{tikzcd}
\KMW_n(F) \ar[d,"{\Tr_{F/L}}"'] \ar[r] & \oplus_i \KMW_n(F_i) \ar[d,"{\sum \Tr_{F_i/F}}"] \\
\KMW_n(L) \ar[r] & \KMW_n(F)
\end{tikzcd}
\]
and 
\[
\begin{tikzcd} \HMW^{n,n}(F,\ZZ) \ar[d,"{\Tr_{F/L}}"'] \ar[r] & \oplus_i \HMW^{n,n}(F_i,\ZZ) \ar[d,"{\sum \Tr_{F_i/F}}"] \\
\HMW^{n,n}(L,\ZZ) \ar[r] & \HMW^{n,n}(F,\ZZ)
\end{tikzcd}
\]
commute. The second one follows from Example \ref{ex:inseparable} and the first one from \cite{Neshitov18_d}*{Lemma~9.4} (or alternatively from Proposition \ref{prop:transe}, the fact that $\theta_F$ is surjective and Example \ref{ex:inseparable}). By induction, each of the diagrams
\[
\begin{tikzcd}
\KMW_n(F_i) \ar[r,"{\Phi_{F_i}}"] \ar[d,"{\Tr_{F_i/F}}"'] & \HMW^{n,n}(F_i,\ZZ) \ar[d,"{\Tr_{F_i/F}}"] \\
\KMW_n(F) \ar[r,"{\Phi_F}"] & \HMW^{n,n}(F,\ZZ). 
\end{tikzcd}
\]
commute, and it follows that $f^*(t(\alpha))=0$.
\end{proof}

\begin{thm}\label{thm:comparison}
The homomorphism
\[
\Phi_L:\KMW_n(L)\to \HMW^{n,n}(L,\ZZ)
\]
is an isomorphism for any $n\in\ZZ$ and any finitely generated field extension $L/k$. 
\end{thm}

\begin{proof}
As in degree $1$, it suffices to prove that $\Phi_L$ is surjective. Let then $\alpha\in \cork(L,\Gm^n)$ be a finite Chow-Witt correspondence supported on $\Spec(F)\subset (\Aone_L)^n$. Such a correspondence is determined by an $n$-uple $(a_1,\ldots,a_n)\in (F^\times)^n$ together with a bilinear form $\phi\in \GW(F,\omega_{F/L})$. Arguing as in Proposition \ref{prop:inductionstep}, we see that such a finite MW-correspondence is of the form $\Tr_{F/L}(\Phi_F(\beta))$ for some $\beta\in \KMW_n(F)$. The result now follows from Theorem \ref{thm:respect}.
\end{proof}

\renewcommand{\theequation}{\thethm.\alph{equation}}
\chapter[Milnor-Witt motivic ring spectrum]{The Milnor-Witt motivic ring spectrum and its associated theories\author{Jean Fasel and Frédéric Déglise}}
\label{ch:spectra}

\section*{Abstract}

We build a motivic ring spectrum representing Milnor-Witt motivic
cohomology, as well as its \'etale local version and show how to
deduce out of it three other theories: Borel-Moore homology,
cohomology with compact support and homology.
These theories, as well as the usual cohomology, are defined
for singular schemes and satisfy the properties of their (Voevodsky's) motivic analog
(and more), up to considering more general twists.
In fact, the whole formalism of these four theories can be functorially
attached to any motivic ring spectrum, giving finally maps between 
the Milnor-Witt motivic ones to the classical motivic ones.

\section*{Introduction}

One of the nice features, and maybe part of the success
of Voevodsky's theory of sheaves with transfers 
and motivic complexes is to provide a rich cohomological theory.
Indeed, apart from defining motivic cohomology, Voevodsky also
obtains other theories: Borel-Moore motivic homology, 
(plain) homology and motivic cohomology with compact support.
These four theories are even defined for singular $k$-schemes
and, when $k$ has characteristic $0$ (or with good resolution of
singularities assumptions), they satisfy nice properties, such
as duality. Finally, again in characteristic $0$,
they can be identified with known theories:
Borel-Moore motivic cohomology agrees with
Bloch's higher Chow groups and homology (without twists) with
Suslin homology. This is described
in \cite{FSV00_e}*{Chap.~5, 2.2}.\footnote{The results of Voevodsky
have been generalized in positive characteristic $p$, after 
taking $\ZZ[1/p]$-coefficients, using the results of Kelly \cite{Kel17_e}.
See \cite{CD15_e}*{Sec.~8}.}
 
It is natural to expect that the Milnor-Witt version of motivic complexes
developed in \chdmt allows one to get a similar formalism. That is what
we prove in this chapter. However, we have chosen a different path
from Voevodsky's, based on the tools at our disposal nowadays.
In particular, we have not developed a theory of relative 
MW-cycles, nevertheless an interesting problem which we leave for
future work. Instead, we rely on two classical theories, that of
ring spectra from algebraic topology and that of Grothendieck' six
functors formalism established by Ayoub in $\Aone$-homotopy theory along the
lines of Voevodsky.
 
In fact, it is well known since the axiomatization of Bloch
and Ogus (see \cite{BO74_e}) that a corollary of the six functors formalism
is the existence of twisted homology and cohomology related by duality.
In the present work, we go further than Bloch and Ogus in several directions. 
First, we establish
a richer structure, consisting of Gysin morphisms (wrong-way functoriality),
and establish more properties, such as cohomological descent (see below for
a precise description). Second, following Voevodsky's motivic picture,
we show that the pair of theories, cohomology and Borel-Moore homology,
also comes with a  pair of compactly supported theories, cohomology
with compact support and (plain) homology. And finally we consider
more general twists for our theories, corresponding to the tensor product
with the Thom space of a vector bundle.

Indeed, our main case of interest, MW-motivic cohomology,
is an example of \emph{non orientable} cohomology theory. In practice,
this means there are no Chern classes, and therefore no Thom classes
so that one cannot identify the twists by a Thom space with a twist
by an integer.\footnote{Recall that given an oriented cohomology
theory $\EE^{**}$ and a vector bundle $V/X$ of rank $n$,
the Thom class of $V/X$ provides an isomorphism
$\EE^{**}(\Th(E)) \simeq \EE^{*-2n,*-n}(X)$. See \cite{Pa03_e} for
the cohomological point of view or \cite{Deg12_e} for the ring spectrum
point of view.} In this context, it is especially relevant to take care
of a twist of the cohomology or homology of a scheme $X$
by an arbitrary vector bundle over $X$, and in fact by a virtual
vector bundle over $X$.\footnote{See Paragraph \ref{num:basic_Thom}
for this notion. 
If one is interested in cohomology/homology up to isomorphism, then
one can take instead of a virtual vector bundle a class in the K-theory
ring $K_0(X)$ of $X$ (see Remark \ref{rem:twists&K0}).}
Our first result is to build a ring spectrum $\sHMW R$
(Definition \ref{df:MW_ring_sp})
over a perfect field $k$, with coefficients in an arbitrary ring $R$, which represents
MW-motivic cohomology (see Paragraph \ref{num:MW-sp&concrete_coh}
for precise statements).
In this chapter, the word \emph{ring spectrum} (for instance the object $\sHMW R$)
means a commutative monoid object of the $\Aone$-derived category
$\DA k$ over a field $k$ (see Section \ref{sec:6functors} for
reminders). One can easily get from this one a commutative monoid
of the stable motivic homotopy category $\SH(k)$ (see Remark \ref{rem:on_spectra}).
Out of the ring spectrum $\sHMW R$, we get as expected the following
theories, associated with an arbitrary $k$-variety $X$ (i.e.\ a separated $k$-scheme of finite type) 
and a pair $(n,v)$ where $n \in \ZZ$, $v$ is a virtual vector
bundle over $X$:
\begin{itemize}
\item MW-motivic cohomology, $\sHMW^n(X,v,R)$,
\index[notation]{hmwnxvr@$\sHMW^n(X,v,R)$}%
\index{Milnor-Witt!motivic!cohomology ring spectrum}%
\index{spectrum!Milnor-Witt motivic cohomology|see{Milnor-Witt motivic cohomology ring spectrum}}%
\item MW-motivic Borel-Moore homology, $\sHMWBM_n(X,v,R)$,
\index[notation]{hmwbmxvr@$\sHMWBM_n(X,v,R)$}%
\index{Milnor-Witt!motivic!Borel-Moore homology ring spectrum}%
\index{spectrum!Milnor-Witt motivic Borel-Moore homology|see{Milnor-Witt motivic Borel-Moore homology ring spectrum}}%
\item MW-motivic cohomology with compact support, $\sHMWc^n(X,v,R)$,
\index[notation]{hmwcnxvr@$\sHMWc^n(X,v,R)$}%
\index{Milnor-Witt!motivic!cohomology with compact support ring spectrum}%
\index{spectrum!Milnor-Witt motivic cohomology with compact support|see{Milnor-Witt motivic cohomology with compact support ring spectrum}}%
\item MW-motivic homology, $\sHlMW_n(X,v,R)$.
\index[notation]{hmwnxvr@$\sHlMW_n(X,v,R)$}%
\index{Milnor-Witt!motivic!homology ring spectrum}%
\index{spectrum!Milnor-Witt motivic homology|see{Milnor-Witt motivic homology ring spectrum}}%
\end{itemize}
As expected, one gets the following computations of
MW-motivic cohomology.
\begin{prop*}
Assume $k$ is a perfect field, $X$ a smooth $k$-scheme and
$(n,m)$ a pair of integers.

Then there exists a canonical identification:
\[
\sHMW^n(X,m,R)
 \simeq
 \Hom_{\DMtkR}\big(\tMot(X),\un(m)[n+2m]\big)=\sHMW^{n+2m,m}(X,R)
\]
using the notations of \chdmt.

If in addition, we assume $k$ is infinite and of characteristic not $2$, then one has the following
computations:
\[
\sHMW^n(X,m,\ZZ)=\begin{cases}
\ch mX & \text{if } n=0, \\
\H^n_\zar(X,\sKMW_0) & \text{if } m=0, \\
\H^{n+2m}_\zar(X,\tilde\ZZ(m)) & \text{if } m>0, \\
\H^{n+m}_\zar(X,\sW) & \text{if } m<0.
\end{cases}
\]
\end{prop*}
In fact, the first identification follows from the definition
and basic adjunctions (see Paragraph \ref{num:trivial_comput_coh})
while the second one was proved in \chdmt
(as explained in Paragraph~\ref{num:MW-sp&concrete_coh}).
 
In fact, though the ring spectrum $\sHMW R$ is our main case of interest,
because of the previous computation,
the four theories defined above,
as well as their properties that we will give below,
are defined for an arbitrary ring spectrum $\EE$ --- and indeed, this
generality is useful as will be explained afterwards.
The four theories associated with $\EE$ enjoy
the following functoriality properties (Section \ref{sec:functoriality})
under the following assumptions:
\begin{itemize}
\item $f:Y \rightarrow X$ is an arbitrary morphism of $k$-varieties
(in fact $k$-schemes for cohomology);
\item $p:Y \rightarrow X$ is a morphism of $k$-varieties which is either
smooth or such that
$X$ and $Y$ are smooth over $k$.
In both cases, $p$ is a local complete intersection morphism
and one can define its virtual tangent bundle denoted by $\tau_p$.
\end{itemize}

\hspace{-2.2cm}
{\renewcommand{\arraystretch}{1.5}
\begin{tabular}{|c|p{2cm}|c|p{2cm}|c|}
\hline
&\multicolumn{2}{|c|}{natural variance}
 & \multicolumn{2}{c|}{Gysin morphisms} \\
\hline
Theory & additional assumption on $f$ &  & additional assumption on $p$ &  \\
\hline
cohomology & none & $\EE^n(X,v) \xrightarrow{f^*} \EE^n(Y,f^{-1}v)$
 & proper & $\EE^n(Y,p^{-1}v+\tau_p) \xrightarrow{p_*} \EE^n(X,v)$ \\
\hline
BM-homology & proper & $\EEBM_n(Y,f^{-1}v) \xrightarrow{f_*} \EEBM_n(X,v)$
 & none & $\EEBM_n(X,v) \xrightarrow{p^*} \EEBM_n(Y,p^{-1}v-\tau_p)$ \\
\hline
c-cohomology & proper & $\EE^n_c(X,v) \xrightarrow{f^*} \EE^n_c(Y,f^{-1}v)$
 & none & $\EE^n(Y,p^{-1}v+\tau_p) \xrightarrow{p_*} \EE^n(X,v)$ \\
\hline
homology & none & $\EE_n(Y,f^{-1}v) \xrightarrow{f_*} \EE_n(X,v)$
 & proper & $\EE_n(X,v) \xrightarrow{p^*} \EE_n(Y,p^{-1}v-\tau_p)$ \\
\hline
\end{tabular}}
\index{Gysin morphism}%
\medskip

Given a closed immersion $i:Z \rightarrow X$ between arbitrary $k$-varieties
with complement open immersion $j:U \rightarrow X$,
and a virtual vector bundle $v$ over $X$,
there exist the so-called \emph{localization long exact sequences}
\index{localization sequence}%
(Paragraph~\ref{num:localization_BM&c}):
\begin{align*}
\EEBM_n(Z,i^{-1}v)
& \xrightarrow{i_*} \EEBM_n(X,v)
 \xrightarrow{j^*} \EEBM_n(U,j^{-1}v)
 \rightarrow \EEBM_{n-1}(Z,i^{-1}v), \\
\EE_c^n(U,j^{-1}v)
& \xrightarrow{j_*} \EE_c^n(X,v)
 \xrightarrow{i^*} \EE_c^n(Z,i^{-1}v)
 \rightarrow \EE_c^{n+1}(U,j^{-1}v).
\end{align*}

There exist the following products for a $k$-variety $X$ (or simply a 
$k$-scheme in the case of cup-products), and couples $(n,v)$, $(m,w)$
of integers and virtual vector bundles on $X$:
 
\medskip
{\renewcommand{\arraystretch}{2}
\begin{tabular}{|c|c|c|}
\hline
Name & pairing & symbol \\
\hline
cup-product
 & $\EE^n(X,v) \otimes \EE^m(X,w) \rightarrow \EE^{n+m}(X,v+w)$
 & $x \cupp y$ \\
\hline
cap-product 
 & $\EEBM_n(X,v) \otimes \EE^m(X,w) \rightarrow \EEBM_{n-m}(X,v-w)$
 & $x \capp y$ \\
\hline
cap-product with support
 & $\EEBM_n(X,v) \otimes \EE^m_c(X,w) \rightarrow \EE_{n-m}(X,v-w)$
 & $x \capp y$ \\
\hline
\end{tabular}}
\index{Milnor-Witt!motivic!cohomology spectrum cup-product}%
\index{cap-product|see{Milnor-Witt motivic cohomology spectrum cup-product}}%
\index{Milnor-Witt!motivic!cohomology spectrum cap-product}%
\index{cap-product|see{Milnor-Witt motivic cohomology spectrum cap-product}}%
\index{Milnor-Witt!motivic!cohomology spectrum cap-product with support}%
\index{cap-product with support|see{Milnor-Witt motivic cohomology spectrum cap-product}}%
\index[notation]{xcy@$x \capp y$}%
\index[notation]{xcy@$x \cupp y$}%
\medskip

Given a smooth $k$-scheme $X$ with tangent bundle $T_X$ there exists
(Definition \ref{df:fdl_class})
a \emph{fundamental class} $\eta_X \in \EEBM_0(X,T_X)$
such that the following morphisms
\begin{align*}
\EE^n(X,v) &\rightarrow \EEBM_{-n}(X,T_X-v),
 x \mapsto \eta_X \capp y \\
\EE^n_c(X,v) &\rightarrow \EE_{-n}(X,T_X-v),
 x \mapsto \eta_X \capp y
\end{align*}
are isomorphisms (Theorem \ref{thm:duality}). We call them the \emph{duality isomorphisms}
 associated with $X$ and with coefficients in $\EE$.
\index{duality isomorphism}%

Finally, the four theories satisfy the following descent properties.
Consider a cartesian square 
\[
\begin{tikzcd}
[row sep=10pt,column sep=10pt]
W \ar[r,"k"] \ar[d,"g"'] \ar[rd,phantom,"\Delta" description] & V \ar[d,"f"] \\
Y \ar[r,"i"'] & X
\end{tikzcd}
\]
of $k$-varieties (or simply $k$-schemes in the case of cohomology
(we refer the reader to Paragraph~\ref{num:basic_6functors}
for the definition of Nisnevich and $\cdh$ distinguished
applied to $\Delta$).
Put $h=i \circ g=f \circ k$.
We let $v$ be a virtual vector bundle over $X$.
Then one obtains the following \emph{descent long exact sequences}:

\medskip
\hspace{-1.8cm}
{\renewcommand{\arraystretch}{2}
\begin{tabular}{|c|c|}
\hline
Assumption on $\Delta$ & descent long exact sequence \\
\hline
Nisnevich or cdh
 & $\EE^n(X,v)
 \xrightarrow{i^*+f^*} \EE^n(Y,i^{-1}v) \oplus \EE^n(V,f^{-1}v)
 \xrightarrow{k^*-g^*} \EE^n(W,h^{-1}v)
 \rightarrow \EE^{n+1}(X,v)$ \\
 & $\EE_n(W,h^{-1}v)
 \xrightarrow{k_*-g_*} \EE_n(Y,i^{-1}v) \oplus \EE_n(V,f^{-1}v)
 \xrightarrow{i_*+f_*} \EE_n(X,v)
 \rightarrow \EE_{n-1}(X,v)$ \\
\hline
Nisnevich
 & $\EEBM_n(X,v)
 \xrightarrow{i^*+f^*} \EEBM_n(Y,i^{-1}v) \oplus \EEBM_n(V,f^{-1}v)
 \xrightarrow{k^*-g^*} \EEBM_n(W,h^{-1}v)
 \rightarrow \EEBM_{n-1}(X,v)$ \\
 & $\EE_c^n(W,h^{-1}v)
 \xrightarrow{k_*-g_*} \EE_c^n(Y,i^{-1}v) \oplus \EE_c^n(V,f^{-1}v)
 \xrightarrow{i_*+f_*} \EE_c^n(X,v)
 \rightarrow \EE_c^{n+1}(X,v)$ \\
\hline
$\cdh$
 & $\EEBM_n(W,h^{-1}v)
 \xrightarrow{k_*-g_*} \EEBM_n(Y,i^{-1}v) \oplus \EEBM_n(V,f^{-1}v)
 \xrightarrow{i_*+f_*} \EEBM_n(X,v)
 \rightarrow \EEBM_{n-1}(X,v)$ \\
 & $\EE^n_c(X,v)
 \xrightarrow{i^*+f^*} \EE^n_c(Y,i^{-1}v) \oplus \EE^n_c(V,f^{-1}v)
 \xrightarrow{k^*-g^*} \EE^n_c(W,h^{-1}v)
 \rightarrow \EE^{n+1}_c(X,v)$ \\
\hline
\end{tabular}}
\index{descent sequence}%
\medskip

We illustrate the previous constructions with some examples. In Section \ref{sec:motivicringspectrum}, we briefly survey the properties of the four theories induced by the motivic Eilenberg-Mac Lane spectrum,
 which represents (ordinary) motivic cohomology. They can be identified with the versions
 defined by Voevodsky in \cite{FSV00_e}*{Chap.~5} when working over a base field $k$ whose characteristic exponent is invertible in $R$.
 In the subsequent section, we study the four theories associated to the MW-motivic ring spectrum, restricting to schemes over a fixed field $k$.
 We mainly focus on the Borel-Moore homological one since it has particularly nice features. We recall how to obtain a (co-)niveau spectral sequence out of this spectrum using the formalism of Bloch and Ogus, and we use this spectral sequence to identify some groups with Chow-Witt groups (Corollary \ref{cor:ChowWittBorelMoore}). Interestingly, this comparison theorem also yields a definition for Chow-Witt groups of singular varieties over a perfect field $k$.

To conclude, let us note that the four theories obtained from the MW-ring spectrum map naturally to their ordinary versions,
and that we also construct the \'etale analog of 
the MW-motivic and motivic ring spectra, 
which are linked with their classical (Nisnevich) counterparts by
canonical morphisms. See Section \ref{sec:MW-regulators} for more details.

\subsection*{Plan of the paper}

As said previously,
this paper is an application of our previous work and of general
motivic $\Aone$-homotopy. So we have tried to give complete reminders
for a non-specialist reader.
 
In Section \ref{sec:mothomthringspectra}, we first recall the formalism of the $\Aone$-derived
category, as introduced by Morel, and the associated six functors
formalism as constructed by Ayoub following Voevodsky.
Then we give a brief account of the theory of ring spectra,
specialized in the framework of the $\Aone$-derived category.

In Section \ref{sec:theories}, we construct the four theories associated with
an arbitrary ring spectrum and establish the properties listed 
above. 
 In Section \ref{sec:motivicringspectrum},
 we briefly survey the properties of these four theories for the ring spectrum representing ordinary motivic cohomology.
 
Finally, in Section \ref{sec:MWmotivicringspectrum}, we apply these results to the particular case
of MW-motivic cohomology as well as its \'etale version.
We consider in more details the case of cohomology and Borel-Moore
homology, and conclude this paper with the canonical maps
relating these ring spectra.

\subsection*{Conventions} 

If $S$ is a base scheme, we will say \emph{$S$-varieties} for separated
$S$-schemes of finite type.

We will simply call a symmetric monoidal category $\mathcal C$ monoidal. We generically denote by $\un$
the unit object of a monoidal category. When this category
depends on a scheme $S$, we also use the generic notation $\un_S$.

In Section \ref{sec:MWmotivicringspectrum}, we will fix a perfect base field $k$ and a coefficient ring $R$.
 
\section{Motivic homotopy theory and ring spectra}
\label{sec:mothomthringspectra}

\subsection{Reminder on Grothendieck's six functors formalism}

\label{sec:6functors}

\begin{num}
Let us fix a noetherian finite dimensional base scheme $S$.\footnote{Since the writing of \cite{CD12_e},
 it has been shown how to relax the finiteness assumption on base schemes to
 quasi-compact quasi-coherent ones. We refer the interested reader to \cite{HoySix_e}.}
We briefly recall the construction of Morel's $\PP^1$-stable
and $\Aone$-derived category over $S$ using \cite{CD12_e} as a reference text.
The construction has also been recalled in \chdmt in the particular case
where $S$ is the spectrum of a (perfect) field.

Let $\sh(S)$ be the category of Nisnevich sheaves of abelian groups over
the category of smooth $S$-schemes $\smg_S$.
This is a Grothendieck abelian category\footnote{Recall that an abelian category is Grothendieck abelian
if it admits small coproducts, a family of generators and filtered colimits are exact.
As usual (see e.g. \chdmt, Proposition~\ref{prop:exist_associated-W-t-sheaf}),
 one gets that the category $\sh(S)$ is Grothendieck abelian
from the case of presheaves and the adjunction $(a,\fO)$ where $a$ is
the associated sheaf functor, $\fO$ the obvious forgetful functor.}.
Given a smooth $S$-scheme $X$, one denotes by $\ZZ_S(X)$ the Nisnevich sheaf associated
with the presheaf $Y \mapsto \ZZ[\Hom_S(Y,X)]$. The essentially small family $\ZZ_S(X)$
generates the abelian category $\sh(S)$. The category $\sh(S)$ admits a closed 
monoidal structure, whose tensor product is defined by the formula:
\begin{equation}\label{eq:tensor}
F \otimes G=\varinjlim_{X/F,Y/G} \ZZ_S(X \times_S Y)
\end{equation}
where the limit runs over the category whose objects are couples of morphisms
$\ZZ_S(X) \rightarrow F$ and $\ZZ_S(Y) \rightarrow G$ and morphisms are given
by couples $(x,y)$ fitting into commutative diagrams of the form:
\[
\begin{tikzcd}
[row sep=10pt,column sep=10pt]
\ZZ_S(X) \ar[rd] \ar[rr,"x"] & & \ZZ_S(X'), \ar[ld] & \ZZ_S(Y) \ar[rd] \ar[rr,"y"] & & \ZZ_S(Y'). \ar[ld] \\
& F & & & G &
\end{tikzcd}
\]
Note in particular
that $\ZZ_S(X) \otimes \ZZ_S(Y)=\ZZ_S(X \times_S Y)$ and that this definition coincides with the one given in \chdmt, \S\ref{num:sht_monoidal} when $S=k$.

According to \cite{CD12_e}*{\textsection 5.1.b}, 
the category $\Comp(\sh(S))$ of complexes with coefficients in $\sh(S)$ admits
a monoidal model structure whose weak equivalences are quasi-isomorphisms and:
\begin{itemize}
\item the class of \emph{cofibrations} is given by the smallest class
 of morphisms of complexes closed under suspensions, pushouts,
 transfinite compositions and retracts generated by the inclusions
\begin{equation}
\ZZ_S(X) \rightarrow \mathrm{C}\big(\ZZ_S(X) \xrightarrow{\id} \ZZ_S(X)\big)[-1]
\end{equation}
where $C(-)$ denotes the cone of complexes, for a smooth $S$-scheme $X$.
\item \emph{fibrations} are given by the epimorphisms of complexes whose
 kernel $K$ satisfies the classical \emph{Brown-Gersten property}:
 for any cartesian square of smooth $S$-schemes
\[
\begin{tikzcd}
[row sep=10pt,column sep=10pt]
W \ar[r,"k"] \ar[d,"q"'] & V \ar[d,"p"] \\
U \ar[r,"j"'] & X
\end{tikzcd}
\]
such that $j$ is an open immersion, $p$ is \'etale and induces an isomorphism of
schemes $p^{-1}(Z) \rightarrow Z$ where $Z$ is the complement of $U$ in $X$ endowed with its reduced subscheme structure, the resulting square of complexes of abelian
groups
\[
\begin{tikzcd}
[row sep=10pt,column sep=10pt]
K(X) \ar[r] \ar[d] & K(V) \ar[d] \\
K(U) \ar[r] & K(W)
\end{tikzcd}
\]
is homotopy cartesian.
\end{itemize}
Indeed, from Examples 2.3 and 6.3 of \cite{CD09_e}, one gets a descent structure
$(\mathcal G,\mathcal H)$
(Def. 2.2 of \emph{op.\ cit.}) on $\sh(S)$
where $\mathcal G$ is the essential small family of generators constituted by the sheaves $\ZZ_S(X)$
for $X/S$ smooth, and $\mathcal H$ is constituted by the complexes of the form
\[
0 \rightarrow \ZZ_S(W) \xrightarrow{q_*-k_*} \ZZ_S(U) \oplus \ZZ_S(V)
 \xrightarrow{j_*+p_*} \ZZ_S(X) \rightarrow 0.
\]
This descent structure is flat (\textsection 3.1 of loc.\ cit.)
so that 2.5, 3.5, 5.5 of loc.\ cit.\ give
the assertion about the monoidal model structure described above.
Note moreover that this model structure is proper, combinatorial and satisfies
the monoid axiom.
\end{num}

\begin{rem} \label{rem:compact_gen}
The descent structure defined in the preceding paragraph is also bounded
(\textsection 6.1 of loc.\ cit.). This implies in particular that
the objects $\ZZ_S(X)$, as complexes concentrated in degree $0$,
are compact in
the derived category $\Der(\sh(S))$ (see Th. 6.2 of \emph{op.\ cit.}).
Moreover, one can describe explicitly the subcategory of $\Der(\sh(S))$ 
generated by these objects
(see loc.\ cit.).
\end{rem}

\begin{num}
Recall now that we get the $\Aone$-derived category
by first $\Aone$-localizing the model category $\Comp(\sh(S))$,
which amounts to invert in its homotopy category $\Der(\sh(S))$
morphisms of the form
\[
\ZZ_S(\Aone_X) \rightarrow \ZZ_S(X)
\]
for any smooth $S$-scheme $X$. One gets the so called $\Aone$-local
Nisnevich descent model structure (cf.\ \cite{CD12_e}*{5.2.17}),
which is again proper monoidal. One denotes by $\DAe S$
its homotopy category.
Then one stabilizes the latter model category with respect to
Tate twists, or equivalently with respect to the object:
\[
\un_S\{1\}=\coker\big(\ZZ_S(k) \xrightarrow{s_{1*}} \ZZ_S(\Gm)\big)
\]
where $s_1$ is the unit section of the group scheme $\Gm$.
This is based on the use of symmetric spectra (cf.\ \cite{CD12_e}*{5.3.31}),
called in our particular case \emph{Tate spectra}.
The resulting homotopy category, denoted by $\DA S$
\index[notation]{da1s@$\DA S$}%
\index[notation]{da1effs@$\DAe S$}%
is triangulated monoidal
and is characterized by the existence of an adjunction of triangulated
categories
\[
\Sigma^\infty:\DAe S \leftrightarrows \DA S:\Omega^\infty
\]
\index[notation]{sigmainf@$\Sigma^\infty$, $\Omega^\infty$}%
such that $\Sigma^\infty$ is monoidal and the object
$\Sigma^\infty(\ZZ_S\{1\})$ is $\otimes$-invertible 
in $\DA S$. As usual, one denotes by $K\{i\}$ the tensor product
of any Tate spectrum $K$ with the $i$-th tensor power of
$\Sigma^\infty(\ZZ_S\{1\})$.
Besides, we also use the more traditional Tate twist:
\[
\un_S(1)=\un_S\{1\}[-1].
\]
\end{num}

\begin{rem}\label{rem:generators}
Extending Remark \ref{rem:compact_gen},
let us recall that the Tate spectra of the form
$\Sigma^\infty \ZZ_S(X)\{i\}$, $X/S$ is smooth and $i \in \ZZ$,
are compact and form a family of generators
of the triangulated category $\DA S$ in the sense that
every object of $\DA S$ is a homotopy colimit of spectra
of the preceding form (see \cite{CD12_e}*{5.3.40}).
\end{rem}

\begin{num}\label{num:basic_Thom}
\textit{Thom spaces of virtual bundles}.--
\index{Thom space}%
\index{virtual vector bundle}%
It is important in our context to introduce more general twists
(in the sense of \cite{CD12_e}*{Definition~1.1.39}).
Given a base scheme $X$, and a vector bundle $V/X$ one defines the Thom
space associated with $V$, as a Nisnevich sheaf over $\smg_X$,
by the following formula:
\[
\Th(V)=\coker\big(\ZZ_X(V^\times) \rightarrow \ZZ_X(V)\big),
\]
where $V^\times$ denotes the complement in $V$ of the zero section.
Seen as an object of $\DA X$, which we still denote by $\Th(V)$,
it becomes $\otimes$-invertible
--- as it is locally of the form $\Th(\AAA^n_X) \simeq \un_X(n)[2n]$.
Therefore, we get a functor
\[
\Th:\mathrm{Vec}(X) \rightarrow \Pic(\DA X)
\]
\index[notation]{thv@$\Th(V)$}%
from the category of vector bundles over $X$ (with isomorphisms as morphisms) to that of
$\otimes$-invertible Tate spectra. It is straightforward to
check that this functor sends direct sums to tensor products up to a canonical isomorphism.
Moreover, according to \cite{Ri10_e}*{4.1.1},
given any exact sequence of vector bundles:
\begin{equation*}\tag{$\sigma$}
0 \rightarrow V' \rightarrow V \rightarrow V'' \rightarrow 0,
\end{equation*}
one gets a canonical isomorphism
\[
\Th(V) \xrightarrow{\epsilon_\sigma} \Th(V') \otimes \Th(V'')
\]
allowing to canonically extend the preceding functor
to a functor from the category $\cK(X)$ 
\index[notation]{kx@$\cK(X)$}%
of virtual vector bundles over $X$ introduced in \cite{Del87_e}*{\S 4}
to $\otimes$-invertible objects of $\DA X$
\[
\Th:\cK(X) \rightarrow \Pic(\DA X)
\]
sending direct sums to tensors products.
\end{num}

\begin{rem}
The isomorphism classes of objects of $\cK(X)$
give the K-theory ring $K_0(X)$ of $X$. In other words,
neglecting morphisms, the construction recalled above associates
to any element $v$ of $K_0(X)$ a canonical isomorphism class of 
Tate spectra $\Th(v)$ which satisfies the relation
$\Th(v+w)=\Th(v) \otimes \Th(w)$.
\end{rem}

\begin{num}\label{num:basic_funct}
Let us finally recall the basic functoriality satisfied by sheaves
and derived categories introduced previously.

Let $f:T \rightarrow S$ be a morphism of schemes.
We get a morphism of sites
$f^{-1}:\smg_S \rightarrow \smg_T$ defined by  $X/S \mapsto (X \times_S T/T)$
and therefore an adjunction of categories of abelian Nisnevich sheaves:
\begin{equation}\label{eq:pullback}
f^*:\sh(S) \leftrightarrows \sh(T):f_*
\end{equation}
\index[notation]{ff@$f^*$, $f_*$}%
such that $f_*(G)=G \circ f^{-1}$ and
\[
f^*(F)=\varinjlim_{X/F} \ZZ_T(X \times_S T)
\]
where the colimit runs over the category of morphisms $\ZZ_S(X) \rightarrow F$.
Recall that $f^*$ is not exact in general.

If in addition $p=f$ is smooth, one gets another morphism of sites:
\[
p_\sharp:\smg_T \rightarrow \smg_S,
 (Y \rightarrow T) \mapsto (Y \rightarrow T  \xrightarrow p S).
\]
One can check that $p^*(F)=F \circ p_\sharp$ and we get an adjunction
 of additive categories:
\[
p_\sharp:\sh(T) \leftrightarrows \sh(S):p^*
\]
\index[notation]{pp@$p_\sharp$, $p^*$}%
such that:
\begin{equation}\label{eq:pushforward}
p_\sharp(G)=\varinjlim_{Y/G} \ZZ_S(Y \rightarrow T \xrightarrow p S),
\end{equation}
this time the colimit runs over the category of morphisms
 $\ZZ_T(Y) \rightarrow G$.

Using formulas \eqref{eq:tensor}, \eqref{eq:pullback}
and \eqref{eq:pushforward}, 
one can check the following basic properties:
\begin{enumerate}
\item \textit{Smooth base change formula}.-- For any cartesian square of schemes
\index{base change formula!smooth}%
\[
\begin{tikzcd}
[row sep=10pt,column sep=10pt]
Y \ar[r,"q"] \ar[d,"g"'] & X \ar[d,"f"] \\
T \ar[r,"p"] & S
\end{tikzcd}
\]
such that $p$ is smooth, the canonical map
\[
q_\sharp g^* \rightarrow f^*p_\sharp
\]
is an isomorphism.
\item \textit{Smooth projection formula}.-- 
\index{projection formula!smooth}%
For any smooth morphism $p:T \rightarrow S$
and any Nisnevich sheaves $G$ over $T$ and $F$ over $S$,
the canonical morphism:
\[
p_\sharp(G \otimes p^*(F)) \rightarrow p_\sharp(G) \otimes F
\]
is an isomorphism.
\end{enumerate}
We refer the reader to \cite{CD12_e}*{1.1.6, 1.1.24} for the definition
of the above canonical maps. Note that the properties stated above
shows that $\sh$ is a $\smg$-premotivic abelian category
in the sense of \cite{CD12_e}*{1.4.2} (see also \cite{CD12_e}*{Ex.~5.1.4}).

According to the theory developed in \cite{CD12_e}*{\textsection 5},
the adjunctions $(f^*,f_*)$ and $(p_\sharp,p^*)$ for $p$ smooth
can be derived and induce triangulated functors
\begin{align*}
\derL f^*:\DA S & \leftrightarrows \DA T:\derR f_*, \\
\derL p_\sharp:\DA S & \leftrightarrows \DA T:p^*.
\end{align*}
By abuse of notation,
we will simply denote the derived functors by $f^*, f_*, p_\sharp$.
Then the analogues of smooth base change
and smooth projection formulas stated above hold
(see \cite{CD12_e}*{Ex.~5.3.31}). In other words, we get a premotivic
triangulated category (cf.\ \cite{CD12_e}*{1.4.2}) which by construction
satisfies the homotopy and stability relation (\cite{CD12_e}*{2.1.3, 2.4.4}).
\end{num}
\begin{dfn}\label{def:Thom}
Consider the notations of \ref{num:basic_Thom} and \ref{num:basic_funct}. Let $S$ be a base scheme, $p:X \rightarrow S$ a smooth morphism,
 and $v$ be a virtual vector bundle over $X$. One defines the Thom
 space of $v$ above $S$ as the object of $\DA S$:
\index[notation]{thsv@$\Th_S(v)$}%
\[
\Th_S(v)=p_\sharp(\Th(v)).
\]
\end{dfn}

Of course, unless $X=S$, $\Th_S(v)$ is in general not $\otimes$-invertible,
and we do not have the relation
$\Th_S(v\oplus w)=\Th_S(v) \otimes \Th_S(w)$. Nevertheless, it happens to be the case in the following situation.
 
Suppose that $v$ is a virtual vector bundle over $X$ and that $w$ is a virtual vector bundle over $S$. Then, $p^{-1}w$ is a virtual vector bundle over $X$ and we have $p^*\Th_S(w)=\Th(p^{-1}w)$.
 Thus, one has the following identification
\[
\Th(v\oplus p^{-1}w)=\Th(v) \otimes p^*\Th_S(w),
\]
and it follows from the smooth projection formula in Paragraph~\ref{num:basic_funct} that 
\[
\Th_S(v\oplus p^{-1}w)=\Th_S(v)\otimes \Th_S(w).
\]
In particular, $\Th_S(v\oplus \AAA_X^m)=\Th_S(v)\otimes \Th(\AAA_S^m)$.

To conclude, we give an explicit expression for the Thom space of a vector bundle $V$ over $X$. The complex  $\mathcal C:=\ZZ_X(V^\times)\to \ZZ_X(V)$ is cofibrant and the map $\mathcal C\to \Th_X(V)$ is a quasi-isomorphism. As $p_{\sharp}(\ZZ_X(Y))=\ZZ_S(Y)$ for any scheme $Y$ over $X$, we see that $\Th_S(V)$ is explicitly given by the complex of sheaves $\ZZ_S(V^\times)\to \ZZ_S(V)$ or equivalently by the cokernel of the map.

\begin{num}\label{num:localization}
Consider again the notations of Paragraph~\ref{num:basic_funct}.
One can check the so-called localization property
for the fibred category $\DA -$ (cf.\ \cite{CD12_e}*{2.4.26}):
for any closed immersion $i:Z \rightarrow S$ with complement
open immersion $j:U \rightarrow S$, and any Tate spectrum $K$ over $S$,
there exists a unique distinguished triangle in $\DA S$:
\[
j_\sharp j^*(K) \xrightarrow{j_\star}
 K  \xrightarrow{i^\star} i_*i^*(K) \rightarrow j_\sharp j^*(K)[1]
\]
where $j_\star$ (resp.\ $i^\star$) is the counit (resp.\ unit)
of the adjunction $(j_\sharp,j^*)$ (resp.\ $(i^*,i_*)$).

As we have also seen in Remark \ref{rem:generators}
that $\DA S$ is compactly generated,
we can apply to it the cross-functor theorem of Ayoub and Voevodsky
(cf.\ \cite{CD12_e}*{2.4.50}) which we state here for future reference.
\end{num}

\begin{thm}
\label{thm:6functors}%
Consider the above notations. Then, for any separated morphism of finite type $f:Y \rightarrow X$ of schemes,
 there exists a pair of adjoint functors, the \emph{exceptional functors},
\[
f_!:\DA Y \rightleftarrows \DA X:f^!
\]
\index[notation]{ff@$f_"!$, $f^"!$}%
such that:
\begin{enumerate}
\item
\label{item:covariant-two-functor}%
There exists a structure of a covariant (resp.\ contravariant) 
 $2$-functor on $f \mapsto f_!$ (resp.\ $f \mapsto f^!$).
\item
\label{item:trafo-iso-f-proper}%
There exists a natural transformation $\alpha_f:f_! \rightarrow f_*$
 which is an isomorphism when $f$ is proper.
 Moreover, $\alpha$ is a morphism of $2$-functors.
\item
\label{item:smooth-twisted-inverse}%
For any smooth separated morphism of finite type $f:X \rightarrow S$
 of schemes with tangent bundle $T_f$,
 there are canonical natural isomorphisms
\begin{align*}
\piso_f:f_\sharp & \longrightarrow f_!\big(Th_X(T_f) \otimes_X (\_)\big) \\
\piso'_f:f^* & \longrightarrow Th_X(-T_f) \otimes_X f^!
\end{align*}
which are dual to each other
 -- the Thom premotive $Th_X(T_f)$ is $\otimes$-invertible
 with inverse $Th_X(-T_f)$.
\item
\label{item:twisted-base-change}%
For any cartesian square:
\[
\begin{tikzcd}
[row sep=16pt,column sep=16pt]
Y' \ar[r,"P"] \ar[d,"{g'}"'] \ar[rd,phantom,"\Delta" description] & X' \ar[d,"g"] \\
Y \ar[r,"f"'] & X,
\end{tikzcd}
\]
such that $f$ is separated of finite type,
there exist natural isomorphisms
\begin{align*}
g^*f_! \xrightarrow\sim f'_!{g'}^*\, , \\
g'_*{f'}^! \xrightarrow\sim  f^!g_*\, .
\end{align*}
\item
\label{item:projection-formula-hom}%
For any separated morphism of finite type $f:Y \rightarrow X$ and any Tate spectra $K$ and $L$,
 there exist natural isomorphisms
\begin{align*}
Ex(f_!^*,\otimes):
(f_!K) \otimes_X L &\xrightarrow{\ \sim\ } f_!(K \otimes_Y f^*L)\, ,\ \\
  \uHom_X(f_!(L),K) & \xrightarrow{\ \sim\ } f_* \uHom_Y(L,f^!(K))\, ,\ \\
  f^! \uHom_X(L,M)& \xrightarrow{\ \sim\ } \uHom_Y(f^*(L),f^!(M))\, .
\end{align*}
\end{enumerate}
\end{thm}

\begin{rem}
In items \ref{item:covariant-two-functor} and \ref{item:trafo-iso-f-proper} above, we just mean that there are canonical isomorphisms $f_!g_!\to (fg)_!$ and $g^!f^!\to (fg)^!$, and that we have a commutative diagram
\[
\begin{tikzcd}
f_!g_! \ar[r] \ar[d] & f_*g_* \ar[d] \\
(fg)_! \ar[r] & (fg)_*
\end{tikzcd}
\]
where the horizontal maps are given by the transformation $\alpha$.
\end{rem}

\begin{num}
\label{num:basic_6functors}%
This theorem has many important applications. Let us state
 a few consequences for future use.
\begin{itemize}
\item \textit{Localization triangles}.-- 
\index{localization triangle}%
Consider again the assumptions of Paragraph~\ref{num:localization}.
 Then one gets canonical distinguished triangles:
\begin{align}
\label{eq:localization1}
j_! j^!(K) \xrightarrow{ad_j}
 K  \xrightarrow{ad'_i} i_*i^*(K) \rightarrow j_! j^!(K)[1] \\
\label{eq:localization2}
i_! i^!(K) \xrightarrow{ad_i}
 K  \xrightarrow{ad'_j} j_*j^*(K) \rightarrow i_! i^!(K)[1]
\end{align}
where $ad_j$,  $ad_i$,  $ad'_j$,  $ad'_i$ are the unit/counit
 morphism of the obvious adjunctions. The exactness of first triangle is a direct consequence of Paragraph~\ref{num:localization},
 together with the properties $j_{\sharp}=j_!$ and $j^*=j^!$ derived from (3) above.

\item \textit{Descent properties}.--
\index{descent triangle}%
Consider a cartesian square
 of schemes:
\[
\begin{tikzcd}
[row sep=10pt,column sep=10pt]
W \ar[r,"k"] \ar[d,"g"] \ar[rd,phantom,"\Delta" description] & V \ar[d,"f"] \\
Y \ar[r,"i"'] & X
\end{tikzcd}
\]
One says $\Delta$ is Nisnevich (resp.\ cdh) distinguished
if $i$ is an open (resp.\ closed) immersion,
$f$ is an \'etale (resp.\ proper) morphism and the induced
map $(V-W) \rightarrow (X-Y)$ of underlying reduced subschemes
is an isomorphism.

If $\Delta$ is Nisnevich or cdh distinguished,
then for any object $K \in \DA X$,
there exist canonical distinguished triangles:
\begin{align}
\label{eq:descent1}
&K \xrightarrow{i^\star+f^\star} i_*i^*(K) \oplus f_*f^*(K)
 \xrightarrow{k^\star-g^\star} h_*h^*(K) \rightarrow K[1] \\
\label{eq:descent2}
&h_!h^!(K) \xrightarrow{k_\star-g_\star} i_!i^!(K) \oplus f_!f^!(K)
 \xrightarrow{i_\star+f_\star} K \rightarrow h_!h^!(K)[1]
\end{align}
where $f^\star$ (resp.\ $f_\star$) is the unit (resp.\ counit)
 of the adjunction $(f^*,f_*)$ (resp.\ $(f_!,f^!)$) and $h=ig=fk$.
\item \textit{Pairing}.-- Let us apply point \ref{item:projection-formula-hom} replacing $K$
 by $f^!(K)$;
 one gets an isomorphism which appears in the following
 composite map:
\[
f_!\big(f^!(K) \otimes f^*(L) \big) \xrightarrow{\ \sim\ }
 f_!f^!(K) \otimes L \xrightarrow{f_\star \otimes Id_L} K \otimes L,
\]
where $f_\star$ is the counit map for the adjunction $(f_!,f^!)$.
 Thus, by adjunction, one gets a canonical morphism:
\[
f^!(K) \otimes f^*(L) \rightarrow f^!(K \otimes L).
\]
We will see in Paragraph~\ref{num:products} that this pairing induces
the classical cap-product.
\end{itemize}
\end{num}

\begin{rem}\label{rem:chg_coeff}
Let $R$ be a ring of coefficients.
One can obviously extend the above considerations 
by replacing sheaves of abelian groups by sheaves of $R$-modules
(as in \cite{CD12_e}).
We get a triangulated $R$-linear category $\DA{S,R}$ depending
on an arbitrary scheme $S$, and also obtain the six functors
formalism described above. In brief,
there is no difference between working with $\ZZ$-linear coefficients
or $R$-linear coefficients.

Besides, one gets an adjunction of additive categories:
\[
\rho_R^*:\sh(S) \leftrightarrows \sh(S,R):\rho^R_*
\]
\index[notation]{rr@$\rho_R^*$, $\rho^R_*$}%
where $\rho^R_*$ is the functor that forgets the $R$-linear structure.
The functor $\rho_R^*$ is obtained by taking the associated sheaf
to the presheaf obtained after applying the extension of scalars functor
for $R/\ZZ$. Note that the functor $\rho_R^*$ is monoidal.
According to \cite{CD12_e}*{5.3.36}, these adjoint functors can be derived
and further induce adjunctions of triangulated categories:
\begin{equation}\label{eq:chg_coeff}
\derL \rho_R^*:\DA S \leftrightarrows \DA{S,R}:\derR \rho^R_*
\end{equation}
such that $\derL \rho^*_R$ is monoidal. 
\end{rem}

\subsection{Ring spectra}

Let us start with a very classical definition.
\begin{dfn}
Let $S$ be a base scheme.
A ring spectrum $\EE$ over $S$ is a commutative monoid 
of the monoidal category $\DA S$. A morphism of ring spectra is a morphism of commutative monoids.
\index{ring spectrum}%
\end{dfn}
In other words, $\EE$ is a Tate spectrum equipped with
 a multiplication (resp.\ unit) map
\[
\mu:\EE \otimes \EE \rightarrow \EE,
 \text{ resp.\ } \eta:\un_S \rightarrow \EE
\]
such that the following diagrams are commutative
\begin{equation}\label{eq:axiom_ring}
\begin{tikzcd}
[row sep=10pt]
\EE \ar[r,phantom,shift left=5ex,"{\text{Unity:}}" description] \ar[r,"{1 \otimes \eta}"] \ar[rdd,equal] & \EE \otimes \EE\ar[dd,"\mu"]
 & \EE \otimes \EE \otimes \EE \ar[r,phantom,shift left=5ex,"{\text{Associativity:}}" description] \ar[r,"{1 \otimes \mu}"] \ar[dd,"{\mu \otimes 1}"']
  & \EE \otimes \EE \ar[dd,"\mu"]
 & \EE \otimes \EE \ar[r,phantom,shift left=5ex,"{\text{Commutativity:}}" description] \ar[rd,"\mu"] \ar[dd,"\sim","\gamma"'] & \ \\
& & & & & \EE \\
& \EE
 & \EE \otimes \EE\ar[r,"\mu"] & \EE
 & \EE \otimes \EE\ar[ru,"\mu"']
\end{tikzcd}
\end{equation}
where $\gamma$ is the isomorphism exchanging factors (coming
from the underlying structure of the \emph{symmetric} monoidal category
$\DA S$).

\begin{exm}
\label{ex:trivial_ring}%
\begin{enumerate}
\item
\label{item:constant-Tate-spectrum}%
The constant Tate spectrum $\un_S$ over $S$ is an obvious example of
 ring spectrum over $S$. Besides, for any ring spectrum $(\EE,\mu,\eta)$
 over $S$, the unit map $\eta:\un_S \rightarrow \EE$ is a morphism
 of ring spectra.
\item Let $k$ be a fixed base (a perfect field in our main example).
 Let $\EE$ be a ring spectrum over $k$.
 Then for any $k$-scheme $S$, with structural morphism $f$, we get a canonical
 ring spectrum structure on $f^*(\EE)$ as the functor $f^*$ is monoidal.

In this situation, we will usually denote by $\EE_S$ this ring
spectrum. The family of ring spectra $(\EE_S)$ thus defined
is a cartesian section of the fibred category $\DA -$.
It forms what we call an absolute ring spectrum over the category of
$k$-schemes in \cite{Deg17_e}.
\end{enumerate}
\end{exm}

\begin{num}\label{num:weak_monoidal}
We finally present a classical recipe in motivic homotopy theory to
 produce ring spectra. Let us fix a base scheme $S$.

Suppose given a triangulated monoidal category $\T$ and
 an adjunction of triangulated categories
\[
\phi^*:\DA S \leftrightarrows \T:\phi_*
\]
such that $\phi^*$ is monoidal.

Then, for a couple of objects $K$ and $L$ of $\T$,
we get a canonical map:
\[
\nu_{K,L}:
\phi_*(K) \otimes \phi_*(L)
 \xrightarrow{\phi_\star} \phi_*\phi^*\big(\phi_*(K) \otimes \phi_*(L)\big)
 \xrightarrow{\sim} \phi_*\big(\phi^*\phi_*(K) \otimes \phi^*\phi_*(L)\big)
 \xrightarrow{\phi^\star \otimes \phi^\star}
  \phi_*(K \otimes L)
\]
where $\phi_\star$ and $\phi^\star$ are respectively the unit and
counit of the adjunction $(\phi^*,\phi_*)$.
Besides, one easily checks that this composite morphism is compatible
with the associativity and symmetry isomorphisms of $\DA S$ and $\T$.

We also get a canonical natural transformation
\[
\nu:
\un_S \xrightarrow{\phi^\star}
 \phi_*\phi^*(\un_S) \simeq \phi_*(\un_\T)
\]
which one can check to be compatible with the unit isomorphism of the
monoidal structures underlying $\DA S$ and $\T$.
In other words, the functor $\phi_*$ is weak monoidal.\footnote{We use the terminology
 weak monoidal/monoidal functor. Beware that the terminology monoidal/strong monoidal also
 exists (see e.g. \cite[XI.2.]{MacLane_e}).}

Then, given any commutative monoid $(M,\mu^M,\eta^M)$ in $\T$,
one gets after applying the functor $\phi_*$ a ring spectrum with
multiplication and unit maps
\begin{align*}
\mu:&\phi_*(M) \otimes \phi_*(M) \xrightarrow{\nu_{M,M}}
 \phi_*(M \otimes M) \xrightarrow{\phi_*(\mu^M)} \phi_*(M), \\
\eta:&\un_S \xrightarrow{\nu} \phi_*(\un_\T)
 \xrightarrow{\phi_*(\eta^M)} \phi_*(M).
\end{align*}
The verification of the axioms of a ring spectrum comes from the
fact that $\phi_*$ is weak monoidal.
\end{num}

\begin{exm}\label{ex:main_ringsp}
The main example we have in mind is the case where $M$
 is the unit object $\un_\T$ of the monoidal category $\T$:
\[
\EE=\phi_*(\un_\T).
\]
\end{exm}

\begin{rem}
On can also define a strict ring spectrum over a scheme $S$
as a commutative monoid object of the underlying model category
of $\DA S$ --- in other words a Tate spectrum equipped with a ring structure
such that the diagrams \eqref{eq:axiom_ring} commute in the category
of Tate spectra, rather than in its localization with respect
to weak equivalences.

In the subsequent cases where ring spectra will appear in this paper,
through an adjunction $(\phi^*,\phi_*)$ as above,
this adjunction will be derived from a Quillen adjunction of
monoidal model categories. We can repeat the arguments above by
replacing the categories with their underlying model categories.
Therefore, the ring spectra of the form $\phi_*(\un_\T)$ will in
fact be strict ring spectra. While we will not use this fact here,
it is an information that could be useful to the reader.
\end{rem}

\begin{exm}
Let $R$ be a ring of coefficients.
Consider the notations of Remark \ref{rem:chg_coeff}.
Then we can apply the preceding considerations to the adjunction
\eqref{eq:chg_coeff} so that we get a ring spectrum
\[
\sHAone R_X:=\derR \rho^R_*(\un_X).
\]
\end{exm}

\begin{num}\label{num:morphisms}
Consider again the notations of Paragraph~\ref{num:weak_monoidal}.
As mentioned in Example~\ref{ex:trivial_ring}\ref{item:constant-Tate-spectrum},
the ring spectrum $\phi_*(\un_\T)$ automatically comes with
a morphism of ring spectra:
\begin{equation}\label{eq:regulator}
\un_S \rightarrow \phi_*(\un_\T).
\end{equation}
Moreover, suppose there exists another triangulated monoidal
 categories $\T'$ with an adjunction of triangulated categories:
\[
\psi^*:\T \leftrightarrows \T':\psi_*
\]
such that $\psi^*$ is monoidal.

Then one gets a canonical morphism of ring spectra:
\[
\phi_*(\un_\T)
 \rightarrow \phi_*\psi_* \psi^*(\un_\T) \simeq \phi_*\psi_*(\un_{\T'})
\]
which is compatible with the canonical morphism of
the form \eqref{eq:regulator}.
\end{num}

\section{The four theories associated with a ring spectrum}
\label{sec:theories}

\begin{num}\label{num:notations_4theories}
In this section, we will fix a base scheme $k$
together with a ring spectrum $(\EE,\mu,\eta)$ over $k$.
Given any $k$-scheme $X$ with structural morphism $f$,
we denote by $\EE_X=f^*(\EE)$ the pullback ring spectrum over $X$
(Example \ref{ex:trivial_ring}).
\index[notation]{ex@$\EE_X$}%
We will still denote by $\mu$ (resp.\ $\eta$) the multiplication
(resp.\ unit) map of the ring spectrum $\EE_X$.

When $X/k$ is separated of finite type,
it will also be useful to introduce the following notation for the following
 spectrum in $\DA X$:
\[
\EE'_X=f^!(\EE).
\]
\index[notation]{exp@$\EE'_X$}%
Note that this is not a ring spectrum in general, but that there is a pairing 
\begin{equation}\label{eq:pre_capp}
\mu':\EE'_X \otimes \EE_X=f^!(\EE) \otimes f^*(\EE)
 \rightarrow f^!(\EE \otimes \EE) \xrightarrow \mu f^!(\EE)=\EE'_X
\end{equation}
using the last point
of \S \ref{num:basic_6functors}.

In the following subsections, we will show how to associate
cohomological/homological theories with $\EE$ and deduce
the rich formalism derived from the six functors formalism
(Theorem~\ref{thm:6functors}).

Note finally that the constructions will be functorial in the
ring spectra $\EE$. To illustrate this fact, we will also fix
a morphism of ring spectra
\[
\phi:\EE \rightarrow \FF.
\]
\end{num}

\subsection{Definitions and basic properties}

\begin{dfn}\label{df:4theories}
Let $p:X \rightarrow \Spec(k)$ be a morphism of schemes,
 $n \in \ZZ$ be an integer and $v \in \cK(X)$ be a virtual vector bundle over $X$.

We define the cohomology of $X$ in degree $(n,v)$
and coefficients in $\EE$ as the following abelian group:
\[
\EE^{n}(X,v):=
\Hom_{\DA k}\Big(\un_k,p_*\big(p^*(\EE) \otimes \Th(v)\big)[n]\Big).
\]
If $X$ is a $k$-variety,
we also define respectively the cohomology with compact
support, Borel-Moore homology and homology
of $X$ in degree $(n,v)$ and coefficients in $\EE$ as:
\index[notation]{enxv@$\EE^n(X,v)$}%
\index[notation]{encxv@$\EE^n_c(X,v)$}%
\index[notation]{ebmnxv@$\EEBM_n(X,v)$}%
\index[notation]{enxv@$\EE_n(X,v)$}%
\index{ring spectrum!cohomology}%
\index{ring spectrum!cohomology with compact support}%
\index{ring spectrum!Borel-Moore homology}%
\index{ring spectrum!homology}%
\begin{align*}
\EE^n_c(X,v)&:=
\Hom_{\DA k}\Big(\un_k,p_!\big(p^*(\EE) \otimes \Th(v)\big)[n]\Big), \\
\EEBM_n(X,v)&:=
\Hom_{\DA k}\Big(\un_k,p_*\big(p^!(\EE) \otimes \Th(-v)\big)[-n]\Big), \\
\EE_n(X,v)&:=
\Hom_{\DA k}\Big(\un_k,p_!\big(p^!(\EE) \otimes \Th(-v)\big)[-n]\Big). 
\end{align*}
We will sometime use the abbreviations
\emph{c-cohomology} and \emph{BM-homology} for cohomology with compact support
and Borel-Moore homology respectively.

Finally, $v$ can be replaced by an integer $m \in \ZZ$,
 in the case where $v$ is the class $m.[\mathcal O_X]$.
\end{dfn}
We will describe below the properties satisfied by these four theories,
which can be seen as a generalization of the classical 
Bloch-Ogus formalism (see \cite{BO74_e}).
\index{Bloch-Ogus}%

\begin{rem}\label{rem:twists&K0}
It is clear from the construction of Paragraph~\ref{num:basic_Thom}
and from the above definition that cohomology and cohomology with compact
support (resp.\ Borel-Moore homology and homology) depend covariantly
(resp.\ contravariantly) upon the virtual vector bundle $v$
--- i.e.\ with respect to morphisms of the category
$\cK(X)$.
In particular, if one considers these theories up to isomorphism,
one can take for $v$ a class in the $K$-theory ring $\K_0(X)$
of vector bundles over $X$.
\end{rem}

\begin{exm}\label{ex:coh_smooth}
Let us assume that $X$ is a smooth $k$-scheme, with structural
morphism $p$. Consider a couple of integers $(n,m) \in \ZZ^2$.

Then, one gets the following computations:
\begin{align*}
\EE^n(X,m)
 &=\Hom_{\DA k}\big(\un_k,p_*\big(p^*(\EE) \otimes \Th(\AAA^m_X)\big)[n]\big) \\
 &=\Hom_{\DA k}\big(\un_k,p_*p^*\big(\EE \otimes \Th(\AAA^m_k)\big)[n]\big) \\
 &\stackrel{(1)}=\Hom_{\DA k}\big(p_\sharp p^*(\un_k),\EE \otimes \Th(\AAA^m_k)[n]\big) \\
 &\stackrel{(2)}=\Hom_{\DA k}\big(\ZZ_k(X),\EE(m)[n+2m]\big) \\
 &:=\EE^{n+2m,m}(X),
\end{align*}
where the last notation follows a more classical usage,
 from motivic cohomology theory for example.
The identification (1) comes from the (derived) adjunctions
described in Paragraph~\ref{num:basic_funct}
and (2) comes from the definition of $p^*$ (resp.\ $p_\sharp$)
--- see again Paragraph~\ref{num:basic_funct}.
%
\end{exm}

\begin{rem}\label{rem:relative_MWspectrum}
Using the conventions stated in the beginning of this section,
one can rewrite the previous definitions as follows:
\begin{align*}
\EE^{n}(X,v)
 & =\Hom_{\DA X}\big(\un_X,\EE_X \otimes \Th(v)[n]\big), \\
\EE^{n}_{c}(X,v)
 & =\Hom_{\DA k}\big(\un_k,p_!(\EE_X \otimes \Th(v))[n])\big), \\
\EEBM_n(X,v)
 & =\Hom_{\DA X}\big(\un_X,\EE'_X \otimes \Th(-v)[-n]\big), \\
\EE_{n}(X,v)
 & =\Hom_{\DA k}\big(\un_k,p_!(\EE'_X \otimes \Th(-v))[-n]\big).
\end{align*}
In particular, if one interprets $\EE_X'$ as a \emph{dual} of $\EE_X$,
our definition of Borel-Moore homology is analogous
to that of Borel and Moore relative to singular homology
(see \cite{BM60_e}).
\end{rem}

\begin{num}
Assume $p:X \rightarrow \Spec(k)$ is separated of finite type.
From the natural transformation $\alpha_p:p_! \rightarrow p_*$
of Theorem~\ref{thm:6functors}\ref{item:trafo-iso-f-proper},
one gets canonical natural transformations:
\begin{align*}
\EE^{n}_{c}(X,v) & \rightarrow \EE^{n}(X,v), \\
\EE_{n}(X,v) & \rightarrow \EEBM_n(X,v)
\end{align*}
which are isomorphisms whenever $X/k$ is proper.
\end{num}

\begin{rem}
\label{rem:pre_homotopy}%
Consider an arbitrary $k$-scheme $X$.

One must be careful about the homotopy invariance property.
\index{homotopy invariance}%
Indeed, if $v$ is a virtual bundle over $\Aone_X$ which comes from $X$,
that is $v=\pi^{-1}(v_0)$ where $\pi:\Aone_X \rightarrow X$ is the
canonical projection, then one gets:
\[
\EE^n(\Aone_X,v) \simeq \EE^n(X,v_0)
\]
from the homotopy property of $\DA X$ --- more precisely,
the isomorphism $Id \xrightarrow \sim \pi_*\pi^*$.

This will always happen if $X$ is regular.
But in general, $v$ could not be of the form $\pi^{-1}(v_0)$
and there is no formula as above.

Similarly, if $v=p^{-1}(v_0)$, one gets:
\[
\EE_n(\Aone_X,v) \simeq \EE_n(X,v_0).
\]
Note finally there is no such formula for BM-homology (resp. c-cohomology):
 more precisely, the affine line behaves like the sphere $\PP^1$ (resp. its "inverse")
 for this theory.
\end{rem}

\begin{num}\label{num:morphisms_ringsp}
It is clear that a morphism of ring spectra $\phi:\EE \rightarrow \FF$
induces morphisms of abelian groups, all denoted by $\phi_*$:
\begin{align*}
\EE^{n}(X,v) &\rightarrow \FF^{n}(X,v) \\
\EE^{n}_c(X,v) &\rightarrow \FF^{n}_c(X,v) \\
\EEBM_n(X,v) &\rightarrow \FFBM_n(X,v) \\
\EE_{n}(X,v) &\rightarrow \FF_{n}(X,v).
\end{align*}
\end{num}

\subsection{Functoriality properties} \label{sec:functoriality}

\begin{num} \label{num:basic_functoriality}
\textit{Basic functoriality}.
Let $f:Y \rightarrow X$ be a morphism of $k$-schemes
and consider $(n,v) \in \ZZ \times \cK(X)$.

Letting $p$ (resp.\ $q$) be the structural projection of $X/k$ (resp.\ $Y/k$),
we deduce the following maps in $\DA X$, where in the second one
we have assumed that $p$ and $q$ are separated of finite type
\begin{align*}
\Th(v) \otimes p^*(\EE)
 \xrightarrow{ad} f_*f^*(\Th(v) \otimes p^*\EE)
 \stackrel{(1)} \simeq f_*\big(\Th(f^{-1}v) \otimes q^*\EE\big) \\
f_!\big(\Th(f^{-1}v) \otimes q^!\EE\big)
 \stackrel{(2)}\simeq f_!f^!(\Th(v) \otimes p^!\EE)
 \xrightarrow{ad'} \Th(v) \otimes p^!(\EE)
\end{align*}
where $ad$ (resp.\ $ad'$) is the unit (resp.\ counit) map
of the adjunction $(f^*,f_*)$ (resp.\ $(f_!,f^!)$)
and the isomorphism (1) (resp.\ (2)) follows from the fact $f^*$ is monoidal
(resp.\ $\Th(v)$ is $\otimes$-invertible).

Composing respectively with $p_*$ and $p_!$, we get canonical morphisms:
\begin{align*}
p_*\big(\Th(v) \otimes p^*\EE\big)
 &\xrightarrow{\ \pi(f)\ } q_*\big(\Th(f^{-1}v) \otimes q^*\EE\big) \\
q_!\big(\Th(f^{-1}v) \otimes q^!\EE\big)
 &\xrightarrow{\ \pi'(f)\ }  p_!\big(\Th(v) \otimes p^!\EE\big),
\end{align*}
which induces the following pullback and pushforward maps:
\begin{align*}
\EE^{n}(X,v) &\xrightarrow{\ f^*\ } \EE^{n}(Y,f^{-1}v) \\
\EE_{n}(Y,f^{-1}v) &\xrightarrow{\ f_*\ }  \EE_{n}(X,v).
\end{align*}
\index[notation]{ff@$f^*$, $f_*$}%
It is straightforward to check that these maps are compatible with composition,
turning cohomology (resp.\ homology) into a \emph{contravariant}
(resp.\ \emph{covariant})
functor with source the category of $k$-schemes 
(resp.\ $k$-varieties).

Assume now that $f$ is proper.
Then from Theorem~\ref{thm:6functors}\ref{item:trafo-iso-f-proper},
one gets a canonical isomorphism $\alpha_f:f_! \simeq f_*$
and the maps $\pi(f)$, $\pi'(f)$ respectively induce canonical morphisms:
\begin{align*}
\EE^{n}_{c}(X,v) &\xrightarrow{\ f^*\ } \EE^{n}_c(Y,f^{-1}v) \\
\EEBM_n(Y,f^{-1}v) &\xrightarrow{\ f_*\ }  \EEBM_n(X,v).
\end{align*}
Again, notably because $\alpha_f$ is compatible with composition,
these maps are compatible with composition so
that cohomology with compact support
(resp.\ Borel-Moore homology) is a \emph{contravariant}
(resp.\ \emph{covariant})
functor with respect to \emph{proper morphisms} of $k$-varieties.
\end{num}

\begin{rem}
\index{homotopy invariance}%
With that functoriality at our disposal, we can understand
the homotopy property described in Remark \ref{rem:pre_homotopy}
as follows. Given any scheme $X$ and any virtual bundle $v_0$ over $X$,
the canonical projection $\pi:\Aone_X \rightarrow X$ induces
isomorphisms:
\begin{align*}
\pi^*:&\EE^{n}(X,v_0) \rightarrow \EE^{n}(\Aone_X,\pi^{-1}(v_0)) ,\\
\pi_*:&\EE_{n}(\Aone_X,\pi^{-1}(v_0)) \rightarrow \EE_{n}(X,v_0).
\end{align*}
\end{rem}

\begin{num}\label{num:localization_BM&c}
\textit{Localization long exact sequences}.
\index{localization sequence}%
One of the main
properties of Borel-Moore homology, as well as cohomology with
compact support is the existence of the so-called localization long exact
sequences.
In our case, it follows directly from the localization triangle
stated in Paragraph~\ref{num:basic_6functors}.

Indeed, for a closed immersion $i:Z \rightarrow X$ of $k$-varieties
with complement open immersion $j:U \rightarrow X$,
and a virtual vector bundle $v$ over $X$,
one gets localization sequences:
\begin{align*}
\EEBM_n(Z,i^{-1}v)
& \xrightarrow{i_*} \EEBM_n(X,v)
 \xrightarrow{j^*} \EEBM_n(U,j^{-1}v)
 \rightarrow \EEBM_{n-1}(Z,i^{-1}v), \\
\EE_c^n(U,j^{-1}v)
& \xrightarrow{j_*} \EE_c^n(X,v)
 \xrightarrow{i^*} \EE_c^n(Z,i^{-1}v)
 \rightarrow \EE_c^{n+1}(U,j^{-1}v).
\end{align*}
More explicitly, the first (resp.\ second) exact sequence is obtained by
using the distinguished triangle \eqref{eq:localization2} with $K=\EE'_X$
(resp.\ \eqref{eq:localization1} with $K=\EE_X$)
and applying the cohomological functor $\Hom_{\DA X}(\un_X,-)$.
Note we also use the identifications $i_!=i_*$
(resp.\ $j^!=j^*$) which follows from Theorem~\ref{thm:6functors}
part~\ref{item:trafo-iso-f-proper} (resp.\ \ref{item:smooth-twisted-inverse}).
\end{num}

\begin{num}\label{num:Gysin}
\textit{Gysin morphisms}.
\index{Gysin morphism}%
Let us fix a morphism $f:Y \rightarrow X$ of $k$-schemes which is separated
of finite type and consider the notations
of Remark \ref{rem:relative_MWspectrum}.

Assume $f$ is smooth with tangent bundle $\tau_f$.
Then, according to Theorem \ref{thm:6functors}\ref{item:smooth-twisted-inverse} and the $\otimes$-invertibility of Thom spectra,
we get a canonical isomorphism:
\[
\piso'_f:f^!(\EE_X) \simeq f^*(\EE_X) \otimes \Th(\tau_f)
 =\EE_Y \otimes \Th(\tau_f).
\]
Suppose now that $X$ and $Y$ are smooth $k$-varieties,
with respective structural morphisms $p$ and $q$.
Then $f$ is a local complete intersection morphism, and has for
relative virtual tangent bundle the virtual bundle in $\cK(Y)$:
\[
\tau_f=[T_q]-[f^{-1}(T_p)].
\]
Then one can compute $q^!\EE$ in two ways:
\begin{align*}
q^!\EE & \stackrel{(1)}\simeq q^*(\EE) \otimes \Th(T_q)
 =\EE_Y \otimes \Th(T_q) \\
& =f^!p^!(\EE) \stackrel{(2)}\simeq f^! \big( p^*(\EE) \otimes \Th(T_p) \big)
 \stackrel{(3)} \simeq f^!(\EE_X) \otimes \Th(f^{-1}T_p)
\end{align*}
where (1) and (2) are given by the relative purity isomorphisms
of Theorem \ref{thm:6functors}\ref{item:smooth-twisted-inverse},
respectively for $p$ and $q$,
and (3) follows from the fact $\Th(T_p)$ is $\otimes$-invertible.
Putting the two formulas together, one gets as in the previous case
an isomorphism:
\begin{equation}\label{eq:fdl_class1}
\tilde \eta_f:f^!(\EE_X) \simeq \EE_Y \otimes \Th(\tau_f).
\end{equation}
Similarly, using the same procedure but exchanging the role of $f^*$
and $f^!$, one gets a canonical isomorphism:
\begin{equation}\label{eq:fdl_class2}
\tilde \eta'_f:f^*(\EE'_X) \simeq \EE'_Y \otimes \Th(-\tau_f),
\end{equation}
assuming either $f$ is smooth or $f$ is a morphism of smooth $k$-varieties.

Therefore one gets using adjunctions the following trace maps:
\begin{align*}
tr_f&:f_!\big(\EE_Y \otimes \Th(\tau_f)\big) \longrightarrow \EE_X, \\
tr'_f&:\EE'_X \longrightarrow f_*\big(\EE'_Y \otimes \Th(-\tau_f)\big).
\end{align*}
We can tensor these maps with the Thom space of an arbitrary
virtual vector bundle $v$ over $X$, and compose the map with $p_!$
for the first one and $p_*$ for the second one to get
the following maps:
\begin{align*}
q_!\big(q^*\EE \otimes \Th(f^{-1}v+\tau_f)\big)
 &\longrightarrow p_!(p^*\EE \otimes \Th(v)), \\
p_*(p^!\EE_X \otimes \Th(v))
 &\longrightarrow q_*\big(q^!\EE \otimes \Th(f^{-1}v-\tau_f)\big).
\end{align*}
If we assume moreover that $f$ is proper,
then we get using the same procedure and using the identification $f_*=f_!$
the following maps:
\begin{align*}
q_*\big(q^*\EE \otimes \Th(f^{-1}v+\tau_f)\big)
 &\longrightarrow p_*(p^*\EE \otimes \Th(v)), \\
p_!(p^!\EE_X \otimes \Th(v))
 &\longrightarrow q_!\big(q^!\EE \otimes \Th(f^{-1}v-\tau_f)\big).
\end{align*}
Let us state the result in term of the four theories in the following
proposition.
\end{num}
\begin{prop}
\label{prop:Gysin}%
Let 
$f:Y \rightarrow X$ be a morphism of $k$-varieties satisfying one of the
following assumptions:
\begin{enumerate}
\item 
\label{item:assumition-smooth}%
$f$ is smooth;
\item
\label{item:assumtion-X-Y-smooth}%
$X$ and $Y$ are smooth $k$-varieties.
\end{enumerate}
Then the maps defined above induce the following
 \emph{Gysin morphisms}:
\begin{align*}
f_*:&\EE^{n}_{c}(Y,f^{-1}v+\tau_f) \longrightarrow \EE^{n}_{c}(X,v), \\
f^*:&\EEBM_n(X,v) \longrightarrow \EEBM_n(Y,f^{-1}v-\tau_f).
\end{align*}
Assume moreover that $f$ is proper.
Then using again the previous constructions,
one gets the following maps:
\begin{align*}
f_*:&\EE^{n}(Y,f^{-1}v+\tau_f) \longrightarrow \EE^{n}(X,v), \\
f^*:&\EE_{n}(X,v) \longrightarrow \EE_{n}(Y,f^{-1}v-\tau_f).
\end{align*}
These Gysin morphisms are compatible with composition.

Under assumption \ref{item:assumition-smooth}, for any cartesian square,
\[
\begin{tikzcd}
[row sep=10pt,column sep=10pt]
Y' \ar[r,"g"] \ar[d,"q"'] & X' \ar[d,"p"] \\
Y \ar[r,"f"] & X
\end{tikzcd}
\]
one has the classical base change formulas:
\index{base change formula}%
\begin{itemize}
\item $p^*f_*=g_*q^*$ in case of cohomologies,
\item $f^*p_*=q_*g^*$ in case of homologies.
\end{itemize}
\end{prop}
Indeed, the construction of maps are directly obtained
from the maps defined in Paragraph~\ref{num:Gysin}.
The compatibility with composition is a straightforward check once
we use the compatibility of the relative purity isomorphism with
composition (due to Ayoub, see \cite{CD12_e}*{2.4.52} for the precise statement).
The base change formulas in the smooth case are similar
and are ultimately reduced to the compatibility of the relative purity
isomorphism with base change
(see for example the proof of \cite{Deg17_e}*{Lemma~2.3.13}).

\begin{rem}
Gysin morphisms can be defined under the weaker assumption
that $f$ is a global complete intersection.
Similarly, the base change formula can be extended to cover
also the case of assumption \ref{item:assumtion-X-Y-smooth} of the previous proposition, as well as the general case.
We refer the reader to \cite{DJK17_e} for this generality, as well as for
more details on the proofs.
\end{rem}

\begin{rem}\label{rem:morphisms_ringsp1}
According to these constructions, 
it is clear that the maps in Paragraph~\ref{num:morphisms_ringsp} induced by a morphism of ring spectra $\phi$
are natural in $X$ with respect to the basic functoriality
and Gysin morphisms of each of the four theories.
\end{rem}

\subsection{Products and duality}

\begin{num}\label{num:products}
One can define a \emph{cup-product} on cohomology,
\[
\EE^n(X,v) \otimes \EE^m(X,w) \rightarrow \EE^{n+m}(X,v+w),
 (x,y) \mapsto x \cupp y
\]
\index{spectrum!cup-product}%
\index[notation]{xcy@$x \cupp y$}%
where, using the presentation of Remark \ref{rem:relative_MWspectrum}
one defines $x \cupp y$ as the map:
\begin{align*}
\un_X & \xrightarrow{x \otimes y}
 \EE_X \otimes \Th(v) \otimes \EE_X \otimes \Th(w)[n+m]
 \simeq  \EE_X \otimes \EE_X \otimes \Th(v+w)[n+m] \\
 & \xrightarrow{\ \mu \otimes Id\ } \EE_X \otimes \Th(v+w)[n+m].
\end{align*}
Here we use the fact that $p_*$ is weak monoidal --- as the right adjoint
of a monoidal functor; see Paragraph~\ref{num:weak_monoidal}.
One can easily check that the pullback morphism on cohomology
 (see Paragraph~\ref{num:basic_functoriality}) is compatible with cup product
(see for example \cite{Deg12_e}*{1.2.10,~(E5)}).

Besides one gets a \emph{cap-product}:
\[
\EEBM_n(X,v) \otimes \EE^m(X,w) \rightarrow \EEBM_{n-m}(X,v-w),
 (x,y) \mapsto x \capp y
\]
\index{spectrum!cup-product}%
\index[notation]{xcy@$x \capp y$}%
defined, using again the presentation of Remark \ref{rem:relative_MWspectrum},
as follows:
\begin{align*}
x \capp y:\un_X[n-m] & \xrightarrow{x \otimes y}
 \EE'_X \otimes \Th(-v) \otimes \EE_X \otimes \Th(w)
 \simeq  \EE'_X \otimes \EE_X \otimes \Th(-v+w) \\
 & \xrightarrow{\ \mu'\ } \EE'_X \otimes \Th(-v+w).
\end{align*}
where $\mu'$ is defined in \eqref{eq:pre_capp}.

There is finally a \emph{cap-product with support}:
\[
\EEBM_n(X,v) \otimes \EE^m_c(X,w) \rightarrow \EE_{n-m}(X,v-w),
 (x,y) \mapsto x \capp y
\]
\index{spectrum!cup-product with support}%
defined, using Remark \ref{rem:relative_MWspectrum}, as follows:
\begin{align*}
x \capp y:\un_k[n-m]
 & \xrightarrow{x \otimes y}
  p_*\big(\EE'_X \otimes \Th(-v)\big) \otimes p_!(\EE_X \otimes \Th(w)) \\
 & \xrightarrow{(1)} 
  p_!\big(p^*p_*(\EE'_X \otimes \Th(-v)\big) \otimes \EE_X \otimes \Th(w) \big) \\
 & \xrightarrow{ad'} 
  p_!\big(\EE'_X \otimes \Th(-v) \otimes \EE_X \otimes \Th(w) \big)
  \simeq  p_!\big((\EE'_X \otimes \EE_X \otimes \Th(-v+w)\big) \\
 & \xrightarrow{\ \mu'\ } p_!\big((\EE'_X \otimes \Th(-v+w))\big),
\end{align*}
where (1) is obtained using the projection formula of Theorem~\ref{thm:6functors}\ref{item:projection-formula-hom} (first isomorphism),
 $ad'$ is the counit map of the adjunction $(p^*,p_*)$ and $\mu'$ is defined in \eqref{eq:pre_capp}.
\end{num}

\begin{rem}\label{rem:PF}
These products satisfy projection formulas with respect to Gysin
morphisms, whose formulation we leave to the reader (see also \cite{DJK17_e}).
\end{rem}

\begin{rem}\label{rem:morphisms_ringsp2}
Clearly, the map
\[
\phi_*:\EE^n(X,v) \rightarrow \FF^n(X,v)
\]
defined in \ref{num:morphisms_ringsp} is compatible with cup-products.
Similarly, the other natural transformations associated with 
the morphism of ring spectra $\phi$ are compatible with cap-products.
\end{rem}

\begin{num}\label{num:fdl}
\textit{Fundamental class}.
\index{fundamental class}%
Let us fix a smooth morphism $f:X \rightarrow \Spec(k)$,
with tangent bundle $\tau_X$.

Applying theorem \ref{thm:6functors}\ref{item:smooth-twisted-inverse} to $f$, we get an isomorphism:
\[
\piso'_f:\EE_X=f^*(\EE) \longrightarrow \Th_X(-T_X) \otimes f^!(\EE)=\Th_X(-T_X) \otimes \EE^\prime_X
\]
Then, in view of Remark \ref{rem:relative_MWspectrum},
 the composite map:
\[
\eta_X:\un_X\xrightarrow{\eta} \EE_X\xrightarrow{\piso'_f}\Th_X(-T_X) \otimes \EE^\prime_X\simeq \EE^\prime_X\otimes \Th_X(-T_X)
\]
corresponds to a class in the Borel-Moore homology
group $\EEBM_0(X,T_X)$.
\end{num}
\begin{dfn}\label{df:fdl_class}
Under the assumptions above, we call the class
$\eta_X \in \EEBM_0(X,T_X)$ the \emph{fundamental class} of
the smooth $k$-scheme $X$ with coefficients in $\EE$.
\end{dfn}

The following Poincaré duality theorem is now a mere consequence of
the definitions and of part \ref{item:smooth-twisted-inverse} of Theorem~\ref{thm:6functors}.
\begin{thm}\label{thm:duality}
\index{Poincaré duality}%
Let $X/k$ be a smooth $k$-variety
and $\eta_X$ its fundamental class with coefficients in $\EE$
as defined above.

Then the following morphisms
\begin{align*}
\EE^n(X,v) &\rightarrow \EEBM_{-n}(X,T_X-v),
 x \mapsto \eta_X \capp y \\
\EE^n_c(X,v) &\rightarrow \EE_{-n}(X,T_X-v),
 x \mapsto \eta_X \capp y
\end{align*}
are isomorphisms, simply called the \emph{duality isomorphisms}.
\end{thm}
\begin{proof}
Let us consider the first map. Using Remark \ref{rem:relative_MWspectrum},
 we can rewrite it as follows:
\[
\Hom_{\DA X}\big(\un_X,\EE_X \otimes \Th(v))[n]\big)
 \rightarrow \Hom_{\DA X}\big(\un_X,\EE'_X \otimes \Th(v-T_X))[n]\big).
\]
Then it follows from the definition of the fundamental class
(Paragraph \ref{num:fdl})
and that of cap-products (Paragraph \ref{num:products}),
that this map is induced by the morphism
\[
\piso'_f:\EE_X=f^*(\EE) \longrightarrow \Th_X(-T_X) \otimes f^!(\EE)=\Th_X(-T_X) \otimes \EE^\prime_X
\]
which is an isomorphism according to Theorem \ref{thm:6functors}\ref{item:smooth-twisted-inverse}.

The proof in the second case is the same.
\end{proof}

\begin{rem}
According to this theorem, and the basic functoriality of the four
theories (Paragraph \ref{num:basic_functoriality}),
one gets the following exceptional functoriality for morphisms of smooth
$k$-varieties:
\begin{itemize}
\item cohomology becomes covariant with respect to proper morphisms;
\item BM-homology becomes contravariant with respect to all morphisms;
\item c-cohomology becomes covariant with respect to all morphisms;
\item homology becomes contravariant with respect to proper morphisms.
\end{itemize}
An application of the projection formulas alluded to in Remark
\ref{rem:PF} gives that this extra functorialities coincide
with Gysin morphisms constructed in Proposition \ref{prop:Gysin}.
\end{rem}

\begin{exm}\label{ex:localization}
One deduces from the above duality isomorphisms and the localization
long exact sequences of Paragraph \ref{num:localization_BM&c}
that, for a closed immersion $i:Z \rightarrow X$ of smooth $k$-varieties
with complement open immersion $j$, one has long exact sequences:
\begin{align*}
\EE^n(Z,T_Z-i^{-1}v)
& \xrightarrow{i_*} \EE^n(X,T_X-v)
\xrightarrow{j^*} \EE^n(U,j^{-1}(T_X-v))
\rightarrow \EE^{n+1}(Z,T_Z-i^{-1}v), \\
\EE_n(U,j^{-1}(T_X-v))
& \xrightarrow{j_*} \EE_n(X,T_X-v)
\xrightarrow{i^*} \EE_n(Z,T_Z-i^{-1}v)
\rightarrow \EE_{n+1}(U,j^{-1}(T_X-v)).
\end{align*}
Besides, replacing $(T_X-v)$ by $v$ and using the exact
sequence over vector bundles over $Z$:
\[
0 \rightarrow T_Z \rightarrow i^{-1}T_X
 \rightarrow N_ZX \rightarrow 0
\]
where $N_ZX$ is the normal bundle of $Z$ in $X$,
the above exact sequences can be written more simply as:
\begin{align*}
\EE^n(Z,i^{-1}v-N_ZX)
& \xrightarrow{i_*} \EE^n(X,v)
 \xrightarrow{j^*} \EE^n(U,j^{-1}v)
 \rightarrow \EE^{n+1}(Z,i^{-1}v-N_ZX), \\
\EE_n(U,j^{-1}v)
& \xrightarrow{j_*} \EE_n(X,v)
 \xrightarrow{i^*} \EE_n(Z,i^{-1}v-N_ZX)
 \rightarrow \EE_{n+1}(U,j^{-1}v).
\end{align*}
For further reference, we note that the first long exact sequence was obtained using duality and the localization exact sequence
 for Borel-Moore homology induced by the localization triangle \eqref{eq:localization2}.
 However, we could alternatively use the localization triangle \eqref{eq:localization1} for $\un_X$,
 Theorem~\ref{thm:6functors} \ref{item:trafo-iso-f-proper} and \ref{item:smooth-twisted-inverse}, and then apply $\Hom_{\DA X}(-,p^*\EE\otimes_X \Th(v))$ to get a long exact sequence
\[
\EE^n(Z,i^{-1}v-N_ZX)
 \xrightarrow{i_*} \EE^n(X,v)
 \xrightarrow{j^*} \EE^n(U,j^{-1}v)
 \rightarrow \EE^{n+1}(Z,i^{-1}v-N_ZX)
\]
which is fact the same as above. This follows easily using the various adjunctions and the duality between Borel-Moore homology and cohomology
 (Theorem~\ref{thm:duality}).
\end{exm}

\subsection{Descent properties}

The following result is a direct application of Paragraph
 \ref{num:basic_6functors}.
\begin{prop}
Consider a cartesian square of $k$-schemes:
\[
\begin{tikzcd}
[row sep=10pt,column sep=10pt]
W \ar[r,"k"] \ar[d,"g"'] \ar[rd,phantom,"\Delta" description] & V \ar[d,"f"] \\
Y \ar[r,"i"'] & X
\end{tikzcd}
\]
which is Nisnevich or cdh distinguished
(see  Paragraph \ref{num:basic_6functors}).
Put $h=i \circ g$ and let $v$ be a virtual bundle over $X$.

Then one has canonical long exact sequences of the form:
\begin{align*}
\EE^n(X,v)
& \xrightarrow{i^*+f^*} \EE^n(Y,i^{-1}v) \oplus \EE^n(V,f^{-1}v)
 \xrightarrow{k^*-g^*} \EE^n(W,h^{-1}v)
 \rightarrow \EE^{n+1}(X,v) \\
\EE_n(W,h^{-1}v)
& \xrightarrow{k_*-g_*} \EE_n(Y,i^{-1}v) \oplus \EE_n(V,f^{-1}v)
 \xrightarrow{i_*+f_*} \EE_n(X,v)
 \rightarrow \EE_{n-1}(W,v).
\end{align*}
\index{descent sequence}%

If the square $\Delta$ is Nisnevich distinguished
(in which case all of its morphisms are \'etale),
one has canonical long exact sequences of the form:
\begin{align*}
\EEBM_n(X,v)
& \xrightarrow{i^*+f^*} \EEBM_n(Y,i^{-1}v) \oplus \EEBM_n(V,f^{-1}v)
 \xrightarrow{k^*-g^*} \EEBM_n(W,h^{-1}v)
 \rightarrow \EEBM_{n-1}(X,v) \\
\EE_c^n(W,h^{-1}v)
& \xrightarrow{k_*-g_*} \EE_c^n(Y,i^{-1}v) \oplus \EE_c^n(V,f^{-1}v)
 \xrightarrow{i_*+f_*} \EE_c^n(X,v)
 \rightarrow \EE_c^{n+1}(W,v)
\end{align*}
where we have used the Gysin morphisms with respect to \'etale maps
for BM-homology and c-cohomology (cf. Proposition~\ref{prop:Gysin}).

If the square $\Delta$ is cdh distinguished (in which case all of its
morphisms are proper),
one has canonical long exact sequences of the form:
\begin{align*}
\EEBM_n(W,h^{-1}v)
& \xrightarrow{k_*-g_*} \EEBM_n(Y,i^{-1}v) \oplus \EEBM_n(V,f^{-1}v)
 \xrightarrow{i_*+f_*} \EEBM_n(X,v)
 \rightarrow \EEBM_{n-1}(W,v), \\
\EE^n_c(X,v)
& \xrightarrow{i^*+f^*} \EE^n_c(Y,i^{-1}v) \oplus \EE^n_c(V,f^{-1}v)
 \xrightarrow{k^*-g^*} \EE^n_c(W,h^{-1}v)
 \rightarrow \EE^{n+1}_c(X,v)
\end{align*}
where we have used the proper covariance (resp.\ contravariance)
of BM-homology (resp.\ c-cohomology) constructed
in Paragraph \ref{num:basic_functoriality}.
\end{prop}
\begin{proof}
The proof is a simple application of the descent properties
obtained in Paragraph \ref{num:basic_6functors}.

For example, one gets the case of cohomology by using
the distinguished triangles \eqref{eq:descent1} with $K=\EE_X$:
\[
\EE_X \xrightarrow{i^\star+f^\star} i_*(\EE_Y) \oplus f_*(\EE_V)
 \xrightarrow{k^\star-g^\star} h_*(\EE_W) \rightarrow \EE_X[1]
\]
and applying the cohomological functor $\Hom_{\DA X}(\un_X,-)$.
The description of the maps in the long exact sequence obtained
follows directly from the description of the contravariant
functoriality of cohomology (see \ref{num:basic_functoriality}).

The other exact sequences are obtained similarly.
\end{proof}

We can state the existence of the long exact sequences in the
preceding proposition by saying that the four theories satisfies
Nisnevich and cdh cohomological descent.\footnote{One can express cohomological descent for the cdh or Nisnevich topology in the style
of \cite{SGA4_e}*{Vbis}, or \cite{CD12_e}*{\textsection 3},
using the fact the four theories admits an extension to simplicial
schemes and stating that cdh or Nisnevich hypercovers induces
isomorphisms. The simplification of our formulation comes as
these topologies are defined by cd-structure in the sense
of \cite{Veo10_e}.} This also shows that, 
under the existence of resolution of singularities,
they are essentially determined by their restriction to smooth $k$-varieties.

Let us make a precise statement.
\begin{prop}\label{prop:unique_extension}
Let us assume $k$ is of characteristic $0$
or more generally 
that any reduced $k$-variety admits a non-singular blow-up
and that $k$ is perfect.

Let $\tilde{\mathcal V}_k$ be the category whose objects
 are pairs $(X,v)$ such that $X$ is a $k$-variety and
 $v$ a virtual vector bundle over $X$,
 and morphisms $(Y,w) \rightarrow (X,v)$ are given
 by pairs $(f,\epsilon)$ where $f:Y \rightarrow X$ is a morphism of $k$-schemes
and $\epsilon:w \rightarrow f^{-1}v$ an isomorphism of virtual vector bundles
 over $X$.

Suppose one has a functor $H^*:\tilde{\mathcal V}_k^{op} \rightarrow \Ab^{\ZZ}$
 to graded abelian groups and a natural transformation:
\[
\phi_X:\EE^n(X,v) \rightarrow H^n(X,v)
\]
such that for any cdh distinguished square as in the preceding proposition,
one has a long exact sequence:
\[
\H^n(X,v)
 \xrightarrow{i^*+f^*} \H^n(Y,i^{-1}v) \oplus \H^n(V,f^{-1}v)
 \xrightarrow{k^*-g^*} \H^n(W,h^{-1}v)
 \rightarrow \H^{n+1}(X,v)
\]
which is compatible via $\phi$ with the one for $\EE^*$.

Then, if $\phi_X$ is an isomorphism when $X/k$ is smooth,
it is an isomorphism for any $k$-variety $X$.
\end{prop}
\begin{proof} Note that the long exact sequence for $H^*(X,v)$
 implies that any nil-immersion $\nu:V \rightarrow X$ induces
 an isomorphism $\nu^*:H^*(X,v) \rightarrow H^*(V,\nu^{-1}v)$
 --- take $Y=W=\varnothing$ to get a cdh-distinguished square
 as in Paragraph~\ref{num:basic_6functors}.
 Thus we can focus on reduced $k$-varieties.
 The proof is then an easy induction on the dimension of $X$.
When $X$ has dimension $0$, it is necessarily smooth
over $k$ as the latter field is assumed to be perfect.
The noetherian induction argument follows from the existence
of a blow-up $f:V \rightarrow X$ such that $V$ is smooth.
Let $Y$ be the (reduced) locus where $f$ is not an isomorphism,
$W=V \times_X Y$ with its reduced structure.
 Then the dimension of $Y$ and $W$ is strictly less than the dimension of $X$ and $V$.
By assumptions, $\phi_V$ is an isomorphism.
By the inductive assumption, $\phi_Y$ and $\phi_W$ are isomorphisms.
So the existence of the cdh descent long exact sequences,
and the fact $\phi$ is compatible with these,
allow us to conclude.
\end{proof}

\begin{rem}
Similar uniqueness statements,
with the same proof, hold for the other three theories.
We leave the formulation to the reader.
\end{rem}

\section{The motivic ring spectrum}\label{sec:motivicringspectrum}

\subsection{Motivic spectra}

The purpose of this section is to recall how classical motivic cohomology theories fit into the formalism introduced in the previous sections. Recall from \eqrefdmt{eq:chg_top&tr_DM}
that we have adjunctions of triangulated categories:
\begin{equation}
\begin{tikzcd}
[column sep=30pt,row sep=24pt]
\DAkR \ar[r,shift left=2pt,"{\derL \tilde \gamma^*}"] \ar[d,shift left=2pt,"a"]
 & \DMtkR \ar[r,shift left=2pt,"{\derL \pi^*}"] \ar[d,shift left=2pt,"{\tilde a}"]
     \ar[l,shift left=2pt,"{\tilde \gamma_{*}}"]
 & \DMkR \ar[d,shift left=2pt,"{a^{\mathrm{tr}}}"]
     \ar[l,shift left=2pt,"{\pi_{*}}"] \\
\DAxkR{\et} \ar[r,shift left=2pt,"{\derL \tilde \gamma^*_\et}"]
    \ar[u,shift left=2pt,"{\derR \fO}"]
 & \DMtxkR{\et} \ar[r,shift left=2pt,"{\derL \pi^*_\et}"]
	  \ar[u,shift left=2pt,"{\derR \fO}"]
		\ar[l,shift left=2pt,"{\tilde \gamma_{\et*}}"]
 & \DMxkR{\et}.
    \ar[u,shift left=2pt,"{\derR \fO}"]
		\ar[l,shift left=2pt,"{\pi_{\et*}}"]
\end{tikzcd}
\end{equation}
such that each left adjoint is monoidal. Setting $\gamma=\tilde\gamma\pi$ as in \chdmt, \ref{num:symmonstruc}, we get a diagram of adjunctions
\begin{equation}
\begin{tikzcd}
[column sep=30pt,row sep=24pt]
\DAkR \ar[r,shift left=2pt,"{\derL \gamma^*}"] \ar[d,shift left=2pt,"a"]
  & \DMkR \ar[d,shift left=2pt,"{a^{\mathrm{tr}}}"]
     \ar[l,shift left=2pt,"{\gamma_{*}}"] \\
\DAxkR{\et} \ar[r,shift left=2pt,"{\derL  \gamma^*_\et}"]
    \ar[u,shift left=2pt,"{\derR \fO}"]
 & \DMxkR{\et}.
    \ar[u,shift left=2pt,"{\derR \fO}"]
		\ar[l,shift left=2pt,"{\gamma_{\et*}}"]
\end{tikzcd}
\end{equation}

\begin{dfn}\label{df:M_ring_sp}
Applying the general procedure of Paragraph
\ref{num:weak_monoidal} to deduce ring spectra from the above adjunctions, we obtain respectively the $R$-linear motivic Eilenberg-Mac Lane spectrum
 and \'etale motivic Eilenberg-Mac Lane spectrum as follows:
\begin{align*}
\sHM R&:=\gamma_*(\un), \\
\sHMet R&:=\derR \fO \gamma_{\et*}(\un).
\end{align*} 
\index{motivic!Eilenberg-MacLane spectrum|see{spectrum motivic Eilenberg-MacLane}}
\index{spectrum!motivic!Eilenberg-MacLane}
\index[notation]{hmr@$\sHM R$}
\index[notation]{hmetr@$\sHMet R$}
\end{dfn}
\vspace{-0.5cm}
By adjunction, we get:
\begin{align*}
\sHM^n(X,m,R)
&=\Hom_{\DMkR}\big(\Mot(X),\un(m)[n+2m]\big), \\
\sHMet^n(X,m,R)
&=\Hom_{\DMxkR{\et}}\big(\Mot_\et(X),\un(m)[n+2m]\big)
\end{align*}
\index[notation]{hmnxmr@$\sHM^n(X,m,R)$}%
\index[notation]{hmetnxmr@$\sHMet^n(X,m,R)$}%
where $\Mot(X)$ (resp.\ $\Mot_{\et}(X)$) denotes the motive
 (resp.\ \'etale motive)
\index[notation]{mx@$\Mot(X)$}%
\index[notation]{metx@$\Mot_{\et}(X)$}%
associated with the smooth $k$-scheme $X$.
 
\begin{rem}\label{rem:on_spectra}
One derives from the Dold-Kan
\index{Dold-Kan}%
equivalence an adjunction of triangulated categories:
\[
N:\SH(k) \leftrightarrows \DA k:K
\]
see \cite{CD12_e}*{5.3.35}. The functor $N$ is monoidal (loc.\ cit.)
so that using the arguments of Paragraph \ref{num:weak_monoidal},
the functor $K$ is weak monoidal.
Therefore, if one applies $K$ to any ring spectrum of the above 
definition, one obtains a commutative monoid in the stable homotopy
category $\SH(k)$. The ring spectrum obtained via this procedure from the Eilenberg-MacLane motivic ring spectrum defined here coincides with the ring spectrum defined by Voevodsky (\cite{Voe98_e}*{6.1}; see
also \cite{CD12_e}*{11.2.17}).
\end{rem}

\subsection{Associated cohomology}

\begin{num}
Applying Definition \ref{df:4theories},
we can associate to the preceding ring spectra four
cohomological/homological theories, namely motivic cohomology, Borel-Moore motivic homology, motivic cohomology with compact support, and (Suslin) homology.

There are many computations for these four theories, and we briefly recall a few of them, starting with motivic cohomology.
For a smooth $k$-variety $X$ and a virtual bundle $v$ over $X$
of rank $m$, oner has:
\[
\sHM^n(X,v,\ZZ)=\begin{cases}
\CH^m(X) & \text{if } n=0, \\
\H^{n+2m}_\zar(X,\ZZ(m)) & \text{if } m \geq 0, \\
0 & \text{if } m<0,
\end{cases}
\]
where $\ZZ(m)$ is Suslin-Voevodsky motivic complex.\footnote{For example,
$\ZZ(m)=\Cstar{\VrepZ(\Gm^{\wedge,n})\big}[-n]$,
where $\Cstar{}$ is the Suslin (singular) complex functor.
See \cite{FSV00_e}*{Chap.~5}.} Note that the computation uses the fact that
motivic cohomology is an oriented cohomology theory
(see for example \cite{Deg12_e}*{2.1.4~(2)}).
\end{num}

\begin{num}\label{num:motivic_singular} \label{num:mot_coh}
\textit{Motivic cohomology (singular case)}.
For any $k$-scheme $f:X \rightarrow \Spec(k)$,
following the conventions of Paragraph \ref{num:notations_4theories},
we consider the following ring spectrum over $X$:
\[
\sHMx X:=f^*(\sHM).
\]
Apart from the fact that we are working in $\DA X$ instead of $\SH(X)$,
this ring spectrum agrees with the ring spectrum defined in \cite{CD15_e}*{3.8}.

Assume further that:
\begin{equation}\label{eq:assumption_DM}
\text{The characteristic exponent of $k$ is invertible in $R$.}
\end{equation}
Then according to \cite{CD15_e}, one can extend the construction
of the category $\DMkR$ to an arbitrary $k$-scheme $X$ and obtain
a triangulated $R$-linear category
\[
\DMcdh(X,R)
\]
which satisfies the six functors formalism for various $X$,
as described in Section
\ref{sec:6functors} (see \cite{CD15_e}*{5.11}).
Indeed, $\DMcdh(-,R)$ form what we called a motivic triangulated
category in \cite{CD12_e}*{2.4.45} --- here, the assumption
\eqref{eq:assumption_DM} is essential.
Note in particular that we can define, as in Paragraph \ref{num:basic_Thom},
Thom motives of virtual bundles (see \cite{CD12_e}*{2.4.15}).
Given a virtual bundle $v$ over $X$, we denote by $\Mot\Th(v)$
\index{Thom motive}%
\index[notation]{mthv@$\Mot\Th(v)$}%
this Thom motive, an object of $\DMcdh(X,R)$.

Moreover, according to \cite{CD12_e}*{11.2.16},
one has a natural adjunction of triangulated categories:
\[
\gamma^*:\DA{X,R} \leftrightarrows \DMcdh(X,R):\gamma_*
\]
extending the adjunction
$(\derL \gamma^*,\gamma_*)$.
In fact, the family of adjunctions $(\gamma^*,\gamma_*)$
for various schemes $X$ form what
we called a premotivic adjunction in \cite{CD12_e}*{1.4.6}.

As a consequence, $\gamma^*$ is monoidal and commutes with functors
of the form $f^*$ and $p_!$ while $\gamma_*$ commutes
with functors of the form $f_*$ and $p^!$ (see \cite{CD12_e}*{2.4.53});
here $f$ is any morphism of $k$-schemes while $p$ is separated of finite
type.

Note that by construction, $\gamma^*(\Th(v))=\Mot\Th(v)$.
As the objects $\Th(v)$ and $\Mot\Th(v)$ are $\otimes$-invertible,
one deduces that%
\footnote{The argument goes as follows. Consider the functors:
\[
\fTh(v):K \mapsto \Th(v) \otimes K,
\qquad 
\fTh^M(v):K \mapsto \Mot\Th(v) \otimes K.
\]
Then we obtain an isomorphism of functors:
\[
\gamma^* \circ \fTh(v) \simeq \fTh^M(v) \circ \gamma^*.
\]
Now, as the Thom objects are $\otimes$-invertible,
the functors $\fTh(v)$ and $\fTh^M(v)$ are equivalences
of categories. Their quasi-inverses are respectively:
$\fTh(-v)$ and $\fTh^M(-v)$.
Then from the preceding isomorphism of functors, one deduces an isomorphism
of the right adjoint functors:
\[
\fTh(-v) \circ \gamma_*
 \simeq \gamma_* \circ \fTh^M(-v).
\]}
\[
\gamma_*(\Mot\Th(v)) \simeq \Th(v).
\]

Applying \cite{CD15_e}*{Prop.~4.3 and Th.~5.1}, we also obtain
that $\gamma_*$ commutes with $f^*$. As a consequence,\footnote{Note this
identification is by no means obvious. In fact,
it answers a conjecture of Voevodsky (cf.\ \cite{Voe02_e}*{Conj.~17})
in the particular case of the base change map $f:X \rightarrow \Spec(k)$.
See also \cite{CD15_e}*{3.3, 3.6}.}
\[
\sHMx X=f^*(\gamma_*(\un_k))\simeq \gamma_*(f^*(\un_k)) \simeq \gamma_*(\un_X).
\]
Finally, under assumption \eqref{eq:assumption_DM}, we can do
the following computation for an arbitrary $k$-scheme $X$
and a virtual vector bundle $v$ over $X$ of rank $m$:
\begin{align*}
\sHM^n(X,v,R)
&\simeq\Hom_{\DA X}\big(\un_X,\sHMx X \otimes \Th(v))[n]\big) \\
&\simeq\Hom_{\DA X}\big(\un_X,\gamma_*(\un_X) \otimes \Th(v))[n]\big) \\
&\simeq\Hom_{\DA X}\big(\un_X,\gamma_*(\Mot\Th(v))[n])\big) \\
&\simeq\Hom_{\DMcdh(X,R)}\big(\un_X,\Mot\Th(v)[n])\big) \\
&\simeq\Hom_{\DMcdh(X,R)}\big(\un_X,\un_X(m)[n+2m])\big) 
\end{align*}
which uses the identifications recalled above,
and for the last one, the fact that motivic cohomology is an oriented
cohomology theory (equivalently, $\DMcdh$ is an oriented motivic
triangulated category, see \cite{CD12_e}*{2.4.38, 2.4.40, 11.3.2}).

We can give more concrete formulas as follows.
Assume $X$ is a $k$-variety.
Recall Voevodsky has defined in \cite{FSV00_e}*{Chap.~5, \textsection 4.1}
a motivic complex $\Cstar X$
\index[notation]{cx@$\Cstar X$}%
\index{Suslin(-Voevodsky) singular complex}%
in $\DMe{k,\ZZ}$ by considering the Suslin complex of the sheaf
with transfers $\VrepZ(X)$ represented by $X$.%
\footnote{$\VrepZ(X)$ is denoted by $L(X)$ in loc.\ cit.}
With $R$-coefficients, let us put:
\[
\Cstar{X}_R:=\Cstar X \otimes^\derL_\ZZ R.
\]
Then according to \cite{CD15_e}*{8.4, 8.6}, one gets:
\[
\sHM^n(X,v,R)=\begin{cases}
\Hom_{\DMekR}\big(\Cstar X_R,R(m)[n+2m]\big)
 & \text{if $m \geq 0$}, \\
0 & \text{if $m<0$}
\end{cases}
\]
where $R(m)$ is the $R$-linear Tate motivic complex:
$R(m)=\Cstar{\Gm^{\wedge m}}_R[-m]$.
Note also that one can compute the right-hand side
as the following cdh-cohomology group (see \cite{CD12_e}*{(8.3.1)}:
\[
\Hom_{\DMekR}\big(\Cstar X,R(m)[n+2m]\big)
=\H^{n+2m}_{\cdh}\big(X,R(m)\big),
\]
where $R(m)$ is seen as a complex of cdh-sheaves on the site
of $k$-schemes of finite type.
\end{num}

\subsection{Associated Borel-Moore homology}

Let us consider the situation of Paragraph \ref{num:motivic_singular}.
We assume further that $R$ is a localization of $\ZZ$ satisfying condition \eqref{eq:assumption_DM}.

Then we can compute Borel-Moore motivic homology, for
a $k$-variety $f:X \rightarrow \Spec(k)$
and a virtual bundle $v/X$ of rank $m$ as follows:
\begin{align*}
\sHMBM_n(X,v,R)
&\simeq\Hom_{\DA{X,R}}\big(\un_X,f^!(\sHM R) \otimes \Th(-)[-n]\big) \\
&\simeq\Hom_{\DA{X,R}}\big(\un_X,f^!(\gamma_*(\un_k)) \otimes \Th(-v)[-n]\big) \\
&\stackrel{(1)}\simeq\Hom_{\DA{X,R}}\Big(\un_X,\gamma_*\big(f^!(\un_k) \otimes \Mot\Th(-v)[-n]\big)\Big) \\
&\simeq\Hom_{\DMcdh(X,R)}\big(\un_X,f^!(\un_k) \otimes \Mot\Th(-v)[-n]\big) \\
&\stackrel{(2)}\simeq\Hom_{\DMcdh(X,R)}\big(\un_X,f^!(\un_k)(-m)[-2m-n]\big) \\
&\stackrel{(3)}\simeq\CH_m(X,n) \otimes_\ZZ R
\end{align*}
\index[notation]{hmbmnxvr@$\sHMBM_n(X,v,R)$}%
\index{spectrum!motivic Borel-Moore homology}%
\index{Borel-Moore!motivic homology spectrum|see{spectrum, motivic Borel-Moore homology}}%
where (1) follows from the properties mentioned in
Paragraph \ref{num:motivic_singular}, (2) as $\DMcdh$ is oriented
and (3) using \cite{CD15_e}*{Cor.~8.12}.

\section{The Milnor-Witt motivic ring spectrum}
\label{sec:MWmotivicringspectrum}%

\subsection{The ring spectra}

In this section, we apply the machinery developed in Section \ref{sec:theories} to MW-motives. We also discuss generalizations of the results of the previous section in this new framework.
 We now assume that $k$ is a perfect field and that $R$ is a ring of coefficients.

We again start with the following adjunctions of triangulated categories (\chdmt, \eqref{eq:chg_top&tr_DM})
\begin{equation}\label{eq:diagram_DMs}
\begin{tikzcd}
[column sep=30pt,row sep=24pt]
\DAkR \ar[r,shift left=2pt,"{\derL \tilde \gamma^*}"] \ar[d,shift left=2pt,"a"]
 & \DMtkR \ar[d,shift left=2pt,"{\tilde a}"]
     \ar[l,shift left=2pt,"{\tilde \gamma_{*}}"] \\
\DAxkR{\et} \ar[r,shift left=2pt,"{\derL \tilde \gamma^*_\et}"]
    \ar[u,shift left=2pt,"{\derR \fO}"]
 & \DMtxkR{\et}
	  \ar[u,shift left=2pt,"{\derR \fO}"]
		\ar[l,shift left=2pt,"{\tilde \gamma_{\et*}}"]
\end{tikzcd}
\end{equation}
and we define ring spectra out of these
 adjunctions as before.

\begin{dfn}\label{df:MW_ring_sp}
We define respectively the $R$-linear MW-spectrum and the \'etale MW-spectrum
 as follows:
\index[notation]{hmwr@@$\sHMW R$, $\sHMWet R$}%
\index{Milnor-Witt!motivic!cohomology ring spectrum}%
\begin{align*}
\sHMW R&:=\tilde \gamma_*(\un), \\
\sHMWet R&:= \derR \fO \tilde \gamma_{\et*}(\un).
\end{align*}
\end{dfn}

\begin{rem}
\begin{enumerate}
\item As in Remark \ref{rem:on_spectra}, we observe that the above ring spectra induce ring spectra in $\SH(k)$.
\item One can easily get rid of the assumption that $k$ is perfect by taking pullbacks from the prime field.
 Actually, the pullback procedure provides a well-defined $R$-linear MW-motivic ring spectrum over any $\QQ$-scheme
 of $\mathbb F_p$-scheme: that is any scheme of \emph{equal characteristics}.
\end{enumerate} 
\end{rem}

\begin{num}\label{num:trivial_comput_coh}
Each of these ring spectra represents a corresponding cohomology theory
 on smooth $k$-schemes. This follows by adjunction using Example
 \ref{ex:coh_smooth}. Explicitly, for a smooth $k$-scheme $X$
 and integers $(n,m) \in \ZZ^2$, one gets:
\begin{align*}
\sHMW^n(X,m,R)
&\simeq\Hom_{\DAkR}\big(\Sigma^\infty \ZZ_k(X),\tilde \gamma_*(\un)(m)[n+2m]\big) \\
&\simeq\Hom_{\DMtkR}\big(\tMot(X),\un(m)[n+2m]\big), \\
&\simeq\HMW^{n+2m,m}(X,R)
\end{align*}
where the first identification follows from Example \ref{ex:coh_smooth},
the second by adjunction --- here, $\tMot(X)$ denotes the MW-motive
associated with the smooth $k$-scheme $X$, 
following the notation of \chdmt.
The last group was introduced in \chdmt, Definition~\ref{def:generalizedMW}. Consequently, we have the following computation for any smooth $k$-scheme $X$, any couple of integers $(n,m)\in \ZZ^2$:
\[
\sHMW^n(X,m,\ZZ)=\begin{cases}
\ch mX & \text{if } n=0, \\
\H^n_\zar(X,\sKMW_0) & \text{if } m=0, \\
\H^{n+2m}_\zar(X,\tilde\ZZ(m)) & \text{if } m>0, \\
\H^{n+m}_\zar(X,\sW) & \text{if } m<0.
\end{cases}
\]
See \chdmt, Cor.~\ref{cor:ChowWitt} and Prop.~\ref{prop:explicit}.
\end{num}

Similarly we get:
\begin{align*}
\sHMWet^n(X,m,R)
&=\Hom_{\DMtxkR{\et}}\big(\tMot_\et(X),\un(m)[n+2m]\big)
\end{align*}
where $\tMot_\et(X)$ denotes the étale MW-motive
associated with the smooth $k$-scheme $X$.

If $V$ is a vector bundle of rank $r$ over $X$, then we can extend the previous computations to the case of $\sHMW^n(X,m-V,R)$ as follows. First, The Thom space is represented by the explicit complex of sheaves
\[
\ZZ_X(V^\times)\to \ZZ_X(V)
\]
(which is in fact a cofibrant resolution of the actual Thom space introduced in \ref{num:basic_Thom}). It follows that $\Th_k(V)=p_{\sharp}(\Th(V))$ is given by the complex 
 \[
\ZZ(V^\times)\to \ZZ(V)
\]
while $\tilde\gamma^*(p_{\sharp}(\Th(V)))$ is of the form
\[
\tilde\ZZ(V^\times)\to \tilde\ZZ(V).
\]
By abuse of notation, we still denote by $\Th(V)$ the complex $\tilde\gamma^*(p_{\sharp}(\Th(V)))$ and observe that it coincides with the Thom space considered in \chdmt, Remark~\ref{rem:preThom}. It follows from the discussion after Definition \ref{def:Thom} that $\Th(\AAA^m_X)=\tilde\ZZ(X)(m)[2m]$.

\begin{lem}
We have a canonical isomorphism
\[
\sHMW^n(X,m-V,R)\simeq\sHMW^{n+2m,m}(\Th_k(V),R).
\]
\end{lem}

\begin{proof}
We have
\begin{eqnarray*}
\sHMW^n(X,m-V,R) & = & \Hom_{\DAkR}\big(\un_k,p_*(p^*\tilde \gamma_*(\un)\otimes \Th(-V)(m)[n+2m])\big) \\
 & \simeq & \Hom_{\DA{X,R}}\big(\un_X,p^*\tilde \gamma_*(\un)\otimes \Th(-V) (m)[n+2m]\big) \\
 & \simeq & \Hom_{\DA{X,R}}\big(\Th(V),p^*\tilde \gamma_*(\un) (m)[n+2m]\big) \\
 & \simeq & \Hom_{\DAkR}\big(p_{\sharp}\Th(V),\tilde \gamma_*(\un) (m)[n+2m]\big) \\
 & \simeq & \Hom_{\DMtkR}\big(\tilde \gamma^*p_{\sharp}\Th(V),\tilde\ZZ (m)[n+2m]\big) \\
 & = & \Hom_{\DMtkR}\big(\Th(V),\tilde\ZZ (m)[n+2m]\big).
\end{eqnarray*}
The result now follows from \cite{Asok16_e}*{\S 3.5.2} or \cite{Y17_e}*{Proposition~3.1}.
\end{proof}

Consequently, we obtain for any smooth $k$-scheme $X$ and any couple of integers $(n,m)\in\ZZ^2$ the following computation:
\[
\sHMW^n(X,m-V,\ZZ)=\begin{cases}
\cht{m-r}{X}{\det(V)^\vee} & \text{if } n=0, \\
\H^{n-r}_\zar(X,\sKMW_0,\det(V)^\vee) & \text{if } m=0, \\
\H^{n+2m}_{\zar,X}(V,\tilde\ZZ(m)) & \text{if } m>0, \\
\H^{n+m-r}_\zar(X,\sW,\det(V)^\vee) & \text{if } m<0.
\end{cases}
\]

It is possible to compute $\sHMW^n(X,V,\ZZ)$ for any vector bundle $V$ of rank $r$ over $X$. In fact, it follows \cite[Th. 6.1]{Y17_e}
 that $\Th_k(-V)=\Th_k(V^\vee)(-2r)[-4r]$ in $\DMtkZ$. Consequently, we obtain 
\[
\sHMW^n(X,V,\ZZ) \simeq \sHMW^n(X,2r-V^\vee,\ZZ)\simeq \sHMW^{n+4r,2r}(\Th_k(V^\vee),\ZZ).
\]
We will prove a special case in the next section using Borel-Moore homology. We start by observing that if $V$ is a trivial vector bundle over $X$, then a choice of a trivialization yields an isomorphism $\Th(V)\simeq \Th(\AAA^r_X)$ and consequently isomorphisms
\[
\sHMW^n(X,V,\ZZ)\simeq \sHMW^n(X,r,\ZZ)\simeq \sHMW^{n+2r,r}(X,\ZZ)
\]
A different choice of a trivialization will give another isomorphism, and we now study this difference starting with a lemma.

Let $\alpha\in \OO(X)^\times$ be an invertible global section of $X$. Let $X\times \Aone\setminus 0\to X\times \Aone\setminus 0$ be the morphism induced by $\alpha$ (i.e.\ $(x,t)\mapsto (x,\alpha(x)\cdot t)$), which we can see as an automorphism of the presheaf $\MWprep(X\times \Aone\setminus 0)$ (\chfinitecw, Definition~\ref{dfn:reprpresheaf}), still denoted by $\alpha$. This automorphism in turn induces an automorphism of the associated Nisnevich sheaf $\tilde{\ZZ}(X\times \Aone\setminus 0)$ and its pointed version. On the other hand, we have a homomorphism
\[
\sKMW_0(X)\to \Hom_{\DMtkZ}(\tMot(X),\tMot(X))
\]
induced by \chfinitecw, Example~\ref{ex:action}.

\begin{lem}\label{lem:alpha}
Under the assumptions of the cancellation theorem \chcancellation, Theorem~\ref{thm:weakcancellation}, the map $\alpha$ corresponds to the endomorphism of $\tilde{\ZZ}(X)$ given by the class of $\langle\alpha\rangle$ in $\sKMW_0(X)$.
\end{lem}

\begin{proof}
Consider the graph of $\alpha$
\[
\Gamma:X\times (\Aone\setminus 0)\to X\times (\Aone\setminus 0)\times X\times (\Aone\setminus 0)
\]
and the Cartier divisor $D$ on $\Gm\times \Gm$  given by the rational function $(t_1^{n+1}-1)/(t_1^{n+1}-t_2)$. The first step in the cancellation theorem is to compute the intersection product $D\cdot \Gamma_*(1)$ of the Cartier divisor associated to the rational function above (see \chcancellation, \S\ref{sec:CartierDivisors}) and $\Gamma_*(1)$ (endowed with the "correct" orientation made explicit in \chfinitecw, \S\ref{subsec:embedding}). Using the projection formula of \chfinitecw, Corollary~\ref{cor:pformula} and Remark~\ref{rem:leftmodule}, we obtain 
\[
D\cdot \Gamma_*(1)=\Gamma_*(\Gamma^*(D)).
\] 
Next, observe that we can suppose that $\alpha\neq 1$ (in which case the result is obvious) and that under this assumption $D$ and $\Gamma_*(1)$ intersect properly, i.e.\ we can compute $\Gamma^*(D)$ by just pulling back the rational function above to $X\times (\Aone\setminus 0)$. 
Thus, $\Gamma^*(D)$ is the Cartier divisor on $X\times (\Aone\setminus 0)$ given by the rational function $(t^{n+1}-1)/(t^{n+1}-\alpha t)$. 

The second step in the cancellation theorem is to push-forward $\Gamma_*(\Gamma^*(D))$ along the projection on $X\times X$. Now, we have a commutative diagram
\[
\begin{tikzcd}
X \times (\Aone\setminus 0) \ar[r,"{\Gamma}"] \ar[d] & X \times (\Aone\setminus 0) \times X \times (\Aone\setminus 0) \ar[d] \\
X \ar[r,"{\Delta}"'] & X\times X
\end{tikzcd}
\]
where the vertical maps are just projections and $\Delta$ is the diagonal embedding. By functoriality of the push-forward homomorphism, the push-forward of $\Gamma_*(\Gamma^*(D))$ along the projection on $X\times X$ is the push-forward along the diagonal embedding of the push-forward along the projection on $X$. Using the argument of \chcancellation, Lemma~\ref{lem:degree}, we see that the latter is just $\langle -\alpha\rangle$. The last step in the theorem is multiplication by $\langle -1\rangle$, giving the result.
\end{proof}

Let again $V\to X$ be a trivial vector bundle of rank $r$ over a scheme $X$. Let $\det(V)^\times$ be the set of trivializations of $\det(V)$ (i.e.\ nowhere vanishing sections) and consider for any $i\in\ZZ$ the groups 
\[
\sHMW^{n}(X,r+i,\ZZ)\otimes_{\ZZ[\OO(X)^\times]}\ZZ[\det(V)^\times] 
\]
where $\ZZ[\OO(X)^\times]$ is the group algebra of $\OO(X)^\times$, $\ZZ[\det(V)^\times]$ is the free abelian group on the set $\det(V)^\times$ and the actions of $\OO(X)^\times$ on $\sHMW^{n}(X,r+i,\ZZ)$ and $\ZZ[\det(V)^\times]$ are given respectively by the composite $\OO(X)^\times \to \sKMW_0(X)=\sHMW^0(X,0,\ZZ)$ (together with the multiplicative structure of MW-motivic cohomology) and the obvious multiplication. We can now prove the following proposition.

\begin{prop}\label{prop:twisted}
Assume $k$ is a perfect infinite field of characteristic not $2$.\footnote{This is to be able to apply the cancellation theorem \chcancellation, Theorem~\ref{thm:weakcancellation}.}
Let $X$ be an essentially smooth $k$-scheme such that $\mathrm{SL}_n(X)$ is generated by transvections
 (for example $X$ can be the semi-localization of a smooth $k$-scheme at a finite number of points). Let $V$ be a trivial vector bundle of rank $r$ over $X$ and let $i\in\ZZ$. Then, we have a canonical isomorphism
\[
\sHMW^n(X,V+i,\ZZ)\simeq \sHMW^{n}(X,r+i,\ZZ)\otimes_{\ZZ[\OO(X)^\times]}\ZZ[\det(V^\vee)^\times] 
\]
for any $n\in\ZZ$.
\end{prop}

The kind of twists that appear in the above proposition comes from \cite{Mor12_e}*{Rem. 3.21}.
 Given a $k$-scheme $X$ as in the above proposition, and a line bundle $\mathcal L$ over $X$,
 we will use the following notation:
\begin{equation}\label{eq:HMW;twists}
\HMW^{n,m}(X,\ZZ\{L\}):=\HMW^{n-2m}(X,m,\ZZ)\otimes_{\ZZ[\OO(X)^\times]}\ZZ[\mathcal L^\times].
\end{equation}
With this notation, the preceding isomorphism can restated as:
\[
\sHMW^n(X,V+i,\ZZ)\simeq \HMW^{n+2(r+i),r+i}(X,\ZZ\{\det(V^\vee)\}).
\]

\begin{proof}
As seen above, the choice of an isomorphism $\varphi:\AAA^r_X\to V$ yields an isomorphism 
\[
\alpha(\varphi):\sHMW^n(X,V+i,\ZZ)\simeq \sHMW^{n}(X,r+i,\ZZ)
\]
and an isomorphism $\det(\varphi^\vee)^{-1}:\det(\AAA^r_X)\to \det(V^\vee)$ providing (together with the choice of the usual orientation of (the dual of) $\AAA^r_X$) a nowhere vanishing section denoted by $\chi(\varphi)$.  We then obtain an isomorphism
\[
\beta(\varphi):\sHMW^n(X,V+i,\ZZ)\to  \sHMW^{n}(X,r+i,\ZZ)\otimes _{\ZZ[\OO(X)^\times]}\ZZ[\det(V)^\times] 
\]
defined by $\beta(\varphi)(x)=\alpha(\varphi)(x)\otimes \chi(\varphi)$. If $\psi:\AAA^r_X\to V$ is another isomorphism, then $\psi^{-1}\varphi\in GL_n(X)$ and our assumption that $SL_n(X)$ is generated by transvections shows that we may suppose that $\psi^{-1}\varphi$ is of the form $\mathrm{diag}(u,Id_{m-1})$ for some $u\in \OO(X)^\times$. We obtain $\chi(\varphi)=u^{-1}\chi(\psi)$ and $\langle u^{-1}\rangle\cdot\alpha(\varphi)= \alpha(\psi)$. The claim follows.
\end{proof}

\begin{coro}\label{cor:fieldcase}
Assume $k$ is a perfect infinite field of characteristic not $2$.
Let $L/k$ be a finitely generated field extension, $V$ be a vector bundle of rank $r$ over $L$ and $i\in\ZZ$. Then, we have a canonical isomorphism
\[
\sHMW^n(L,V+i,\ZZ)\simeq 
\begin{cases}
\sKMW_{r+i}(L)\otimes_{\ZZ[L^\times]}\ZZ[\det(V^\vee)^\times] & \text{if $n=-r-i$.} \\
0 & \text{if $n> -r-i$}.
\end{cases}
\]
\end{coro}
Following Morel \cite{Mor12_e}*{Rem. 3.21},
 given a field $L$, a $1$-dimensional vector space $V$ over $L$, and an integer $n \in \ZZ$,
 we will use the following notation for the right hand-side:
\begin{equation}\label{eq:KMW;twists}
\sKMW_{n}(L;V):=\sKMW_{n}(L)\otimes_{\ZZ[L^\times]}\ZZ[\det(V^\vee)^\times].
\end{equation}
\begin{proof}
This is an obvious consequence of \chcomparison, Theorem~\ref{thm:comparison}.
\end{proof}

\subsection{Associated theories}

\begin{num}\label{num:MW-sp&concrete_coh}
\textit{MW-cohomology}.
Applying Definition \ref{df:4theories},
we can associate to the preceding ring spectra four
cohomological/homological theories: the MW-motivic cohomology, the MW-motivic Borel-Moore homology, the MW-motivic cohomology with compact support and the MW-motivic homology.
\index{Milnor-Witt!motivic!cohomology}%
\index{Milnor-Witt!motivic!Borel-Moore homology}%
\index{Milnor-Witt!motivic!cohomology with compact support}%
\index{Milnor-Witt!motivic!homology}%

This version of MW-motivic cohomology extends the definition of
the MW-cohomology theory of \chfinitecw, Definition~\ref{def:generalizedMW},
with its product,
to the case of possibly singular $k$-schemes.
In characteristic $0$,
this extension is the unique one satisfying cdh descent
(see Proposition \ref{prop:unique_extension}).
Additionally, the previous constructions yield Gysin morphisms on cohomology,
with respect to proper morphisms of smooth $k$-varieties
(or proper smooth morphisms of arbitrary $k$-schemes); see Proposition
 \ref{prop:Gysin}.
\end{num}

\begin{num}\textit{ MW-motivic Borel-Moore homology}\label{num:BM}.
We also get the MW-motivic Borel-Moore homology
of $k$-varieties, covariant (resp.\ contravariant)
with respect to proper (resp.\ \'etale) maps,
satisfying the localization long exact sequence
(see Paragraph \ref{num:localization_BM&c})
and, up to a twist, contravariant for any smooth maps or arbitrary morphisms
of smooth $k$-varieties (Gysin morphisms,
Proposition \ref{prop:Gysin}).

Recall that, using the duality theorem \ref{thm:duality},
we get for any smooth $k$-scheme $X$ with tangent bundle $T_X$
and any couple of integers $(n,m) \in \ZZ^2$
the following computation:
\[
\sHMWBM_n(X,v,R)\simeq \sHMW^{-n}(X,T_X-v,R).
\]
\index[notation]{hmwbmxvr@$\sHMWBM_n(X,v,R)$}%
One of the nice features of Borel-Moore homology is the existence of a (co)niveau spectral sequence which is constructed using the localization sequence discussed in Section \ref{num:localization_BM&c}. We briefly recall this construction (omitting the ring of coefficients $R$ for convenience of notation), with \cite{BD17_e}*{\S 3} as reference.

We start with a $k$-variety $X$ (not necessarily smooth) and we consider sequences $Z_*:=(Z_p)_{p\in\NN}$ of reduced closed (possibly empty) subschemes of $X$ with the property that $Z_{p}\subset Z_{p+1}$ for any $p\in \NN$ and $\dimn{Z_p}\leq p$. We denote by $\FF(X)$ the set of all such sequences, ordered by term-wise inclusions. Note that this set is cofiltered (and obviously non-empty). For any $Z_p\subset Z_{p+1}$ with complement $U_{p+1}:=Z_{p+1}\setminus Z_p$, any $q\in \ZZ$ and any virtual vector bundle $v$ over $X$, we obtain a long exact sequence
\begin{align*}
\sHMWBM_q(Z_p,v_{\vert_{Z_p}}) \stackrel{i_*}\to & \sHMWBM_q(Z_{p+1},v_{\vert_{Z_{p+1}}})\stackrel{j^*}\to \sHMWBM_q(U_{p+1},v_{\vert_{U_{p+1}}}) \\
&\to \sHMWBM_{q-1}(Z_p,v_{\vert_{Z_p}}).
\end{align*}
Setting 
\[
F_p\sHMWBM_q(X,v):=\colim_{Z_*\in \FF(X)}\sHMWBM_q(Z_p,v_{\vert_{Z_p}})
\]
and 
\[
G_p\sHMWBM_q(X,v):=\colim_{Z_*\in \FF(X)}\sHMWBM_q(U_p,v_{\vert_{U_p}}),
\]
where the transition morphisms are as in \cite{BO74_e}*{\S 3},
we obtain a long exact sequence
\begin{align*}
F_p\sHMWBM_q(X,v) \stackrel{i_*}\to & F_{p+1}\sHMWBM_q(X,v)\stackrel{j^*}\to G_{p+1}\sHMWBM_q(X,v) \\
 & \to F_p\sHMWBM_{q-1}(X,v).
\end{align*}
As usual, we can organize these long exact sequences in an exact couple to obtain a convergent spectral sequence
\[
E^1_{p,q}(X,v):=\bigoplus_{x\in X_{(p)}} \sHMWBM_{p+q}(k(x),v_{\vert_{k(x)}})\implies \sHMWBM_{p+q}(X,v)
\]
computing MW-motivic Borel-Moore homology. 

\textbf{From now on, we assume that $k$ is an infinite perfect field of characteristic not $2$.}
 Under this assumption, we can compute some differentials of the above spectral sequence
\[
d_{p}:\bigoplus_{x\in X_{(p)}} \sHMWBM_{p+q}(k(x),v_{\vert_{k(x)}})\to \bigoplus_{x\in X_{(p-1)}} \sHMWBM_{p+q-1}(k(x),v_{\vert_{k(x)}})
\]
in an important case, specializing to $R=\ZZ$. Let $x\in X_{(p)}$ and $y\in \overline{\{x\}}\cap X_{(p-1)}$. Let $U\subset \overline{\{x\}}$ be the smooth locus, and suppose that $y\in U$. Further, let $Z$ be the (closed) singular locus of $\overline{\{y\}}$ and set $V=U\setminus Z$. Then, $x,y\in V$ and $V$ is smooth. By definition, any element $\alpha$ of $\sHMWBM_{p+q}(k(x),v_{\vert_{k(x)}})$ comes from $\sHMWBM_{p+q}(W,v_{\vert_{W}})$ for some open subvariety $W\subset V$ with closed complement $D$. If $y\not\in D$, then the component of $d_p(\alpha)$ in $\sHMWBM_{p+q-1}(k(y),v_{\vert_{k(y)}})$ is trivial. If $y\in D$, we may suppose that (up to shrinking further $V$) that $\overline {\{y\}}=D$, which is then smooth, and that $W=V\setminus D$. We have a diagram
\[
\begin{tikzcd}
\sHMWBM_{p+q}(V,v_{\vert_{V}}) \ar[r] & \sHMWBM_{p+q}(V\setminus D,v_{\vert_{V\setminus D}}) \ar[r] &  \sHMWBM_{p+q-1}(D,v_{\vert_{D}}) \ar[d] \\
 & & \sHMWBM_{p+q-1}(k(y),v_{\vert_{k(y)}})
\end{tikzcd}
\]
where the horizontal row is the localization sequence and the vertical homomorphism is the restriction to the generic point. Then, the component in $\sHMWBM_{p+q-1}(k(y),v_{\vert_{k(y)}})$ of $d_p(\alpha)$ is just the composite of the last two maps. Since $V$ and $D$ are smooth, the localization sequence coincides with the localization sequence for MW-motivic cohomology obtained from the localization triangle \eqref{eq:localization1} (Example \ref{ex:localization}). Explicitly, we may use the diagram
\[
\begin{tikzcd}
[column sep=2ex]
\sHMW^{-p-q}(V,T_V-v_{\vert_{V}}) \ar[r] & \sHMW^{-p-q}(V\setminus T,T_{V\setminus D}-v_{\vert_{V\setminus D}}) \ar[r] &  \sHMW^{1-p-q}(D,T_D-v_{\vert_{D}})\ar[d] \\
 & & \sHMW^{1-p-q}\Big(k(y),T_{k(y)}-v_{\vert_{k(y)}}\Big)
\end{tikzcd}
\]
to compute the differential. This leads to the following lemma.

\begin{lem}\label{lem:smoothdifferentials}
Let $x\in X_{(p)}$ and $y\in X_{(p-1)}$ be such that $y$ is in the smooth locus of $\overline{\{x\}}$. Let further $n\in \NN$. Then, the residue homomorphism
\[
d:\sHMWBM_{p+n}(k(x),-n)\to \sHMWBM_{p+n-1}(k(y),-n) 
\]
in the spectral sequence (with $v=-n$ and $q=n$) coincides with the residue homomorphism
\[
d:\sKMW_{p+n}(k(x);\omega_{k(x)/k}) \to \sKMW_{p+n-1}(k(y);\omega_{k(y)/k})
\]
of \cite{Mor12_e}*{\S5} under the isomorphisms (for $i=p,p-1$)
\[
\sHMWBM_{i+n}(k(x),-n)\simeq \sHMW^{-i-n}(k(x),T_{k(x)}+n)\simeq \sKMW_{i+n}(k(x);\omega_{k(x)/k})
\]
obtained from duality (Theorem~\ref{thm:duality})
 and Corollary \ref{cor:fieldcase}, using the notation of \eqref{eq:KMW;twists}.
\end{lem}

\begin{proof}
In view of the above discussion, we may work with MW-motivic cohomology and suppose that $x$ is the generic point of a smooth connected variety $V$ of dimension $p$, with $y$ a point of codimension 1 in $V$. As before, to compute the boundary of some $\alpha\in \sHMW^{-i}(k(x),T_{k(x)})$, we may suppose that $\alpha$ is defined in $\sHMW^{-p}(V\setminus D,T_{V\setminus D})$ where $\overline{\{y\}}=D$ is smooth and use the localization sequence
\[
\begin{tikzcd}
[column sep=1.3em]
\sHMW^{-p-n}(V,T_V+n) \ar[r] & \sHMW^{-p-n}(V\setminus T,T_{V\setminus D}+n) \ar[r] &  \sHMW^{1-p-n}(D,T_D+n)
\end{tikzcd}
\]
Further, we may assume that $T_V$ is trivial, and choose such a trivialization. In that case, the localization sequence looks as follows
\[
\begin{tikzcd}
[column sep=1.3em]
\sHMW^{-p-n}(V,p+n) \ar[r] & \sHMW^{-p-n}(V\setminus T,p+n) \ar[r] &  \sHMW^{1-p-n}(D,p+n-N_{D}V)
\end{tikzcd}
\]
or equivalently as follows
\[
\begin{tikzcd}
[column sep=1.3em]
\sHMW^{p+n,p+n}(V) \ar[r] & \sHMW^{p+n,p+n}(V\setminus T) \ar[r] &  \sHMW^{p+n+1,p+n}(\Th(N_{D}V))
\end{tikzcd}
\]
where the relevant cohomology groups are Zariski hypercohomology with coefficients in the complex of MW-sheaves $\tilde\ZZ(p+n)$ defined in \chdmt, \ref{num:Tate} (or rather its explicit $\Aone$-local model of \chdmt, Corollary~\ref{cor:LA1}). This follows from the cancellation theorem \chcancellation, Theorem~\ref{thm:main}, \chdmt, Corollary~\ref{cor:compare_Hom&cohomology} and \chdmt, Corollary~\ref{cor:cohomcomplex}.

 If $p+n\leq 0$, then this complex is precisely the single sheaf $\sKMW_{p+n}$ (in degree $p+n$) and the above localization sequence is just the localization long exact sequence in the associated cohomology. By definition, this yields the residue homomorphism of \cite{Mor12_e}*{\S5} and the result is proved in that case. If $p+n>0$, then the first step in the truncation of the $t$-structure of \chdmt, Definition~\ref{df:htp_tstruct_eff} yields a morphism of complexes $\tilde\ZZ(p+n)\to \sKMW_{p+n}$ (the latter in degree $p+n$). This morphism provides a commutative ladder of long exact sequences of localization, and we can conclude as in the previous case since the truncation is an isomorphism in the relevant range.
\end{proof}

\end{num}

\begin{num}
Suppose that $f:X\to Y$ is a smooth morphism of constant relative dimension $d$, so that $\tau_f$ is defined. We know by Proposition \ref{prop:Gysin} that Borel-Moore homology is contravariantly functorial in $f$, i.e.\ we have a Gysin morphism
\[
f^*:\EEBM_n(Y,v) \longrightarrow \EEBM_n(X,f^{-1}v-\tau_f)
\]
for any virtual vector bundle $v$ over $Y$. If $Z_*\in \FF(Y)$ (see Notation in Paragraph~\ref{num:BM}),
 then we obtain an element $(f^{-1}Z)_*\in \FF(X)$ by setting $(f^{-1}Z)_{p+d}=f^{-1}(Z_p)$ and $(f^{-1}Z)_i=\emptyset$ if $i<d$.
 For each $Z_*\in \FF(Y)$, we obtain a commutative ladder of long exact sequences (using $f^*$) and we can take limits to obtain a morphism of spectral sequences
\[
f^*:E^1_{p,q}(Y,v)\to E^1_{p+d,q-d}(X,f^{-1}v-\tau_f).
\]
The morphism on the abutments of the corresponding spectral sequences is precisely the morphism $f^*$ we started with.
\end{num}

\begin{num}\label{num:pf}
Suppose next that $f:X\to Y$ is a proper morphism of $k$-varieties. Borel-Moore homology being covariantly functorial with respect to such morphisms (\ref{num:basic_functoriality}), we get a homomorphism
\[
\sHMWBM_n(X,f^{-1}v)\to \sHMWBM_n(Y,v).
\]
If $Z_*$ is a sequence of reduced closed subschemes of $X$ as above, then $f(Z_*)$ is also such a sequence. Consequently, we obtain a map $\FF(X)\to \FF(Y)$. Moreover, we obtain a commutative ladder of long exact sequences for each $Z_*$ and $f(Z_*)$. Taking limits, it follows that $f$ induces a morphism 
\[
f_*:E^1_{p,q}(X,f^{-1}v)\to E^1_{p,q}(Y,v)
\]
such that the morphism at the abutments of the spectral sequences is the group homomorphism $\sHMWBM_{p+q}(X,f^{-1}v)\to \sHMWBM_{p+q}(Y,v)$ described above. If $x\in X_{(p)}$ and $y=f(x)$, then the construction of the spectral sequence implies in particular the existence of homomorphisms
\[
(f_*)^x_{y}:\sHMWBM_{p+q}(k(x),f^{-1}v_{\vert_{k(x)}})\to \sHMWBM_{p+q}(k(y),v_{\vert_{k(y)}}). 
\]
Suppose that $f(x)\in Y_{(n)}$ for $n<p$, i.e.\ that the field extension $k(y)\subset k(x)$ is infinite. In particular, there exists $z\in X_{(p-1)}$ such that $f(z)=f(x)=y$. We may choose $Z_*$ such that $\overline{\{x\}}=Z_p$ and $z\in Z_{p-1}$ and consider $f(Z_*)$. Since $f(Z_p)\subset f(Z_{p-1})$, it follows immediately that $(f_*)^x_{y}=0$ in that case. Therefore, the morphism $f_*:E^1_{p,q}(X,f^{-1}v)\to E^1_{p,q}(Y,v)$ can be understood by computing the homomorphisms $(f_*)^x_y$ with $k(y)\subset k(x)$ finite. There is a well-understood procedure for this computation, and we describe it now in our context.

Let $F$ be a finitely generated field extension of the base field $k$. As $k$ is perfect, we may suppose that $F=k(U)$ for some smooth $k$-scheme $U$. We also consider $\PP^1_F$, which can be seen as a localization of $\PP^1_U$. Following the construction of the spectral sequence for the Borel-Moore homology of $U$ twisted by a virtual vector bundle $v$, we obtain (taking an appropriate limit) a spectral sequence computing the Borel-Moore homology of $\PP^1_F$ whose $E^1$-page consists of complexes of the form 
\[
\sHMWBM_n(F(t),v_{\vert F(t)})\stackrel{d}\to \bigoplus_{x\in (\PP^1_F)^{(1)}} \sHMWBM_{n-1}(F(x),v_{\vert F(x)})
\]
On the other hand, the corresponding spectral sequence for the Borel-Moore homology of $F$ only consists of terms of the form $\sHMWBM_n(F,v_F)$ and the morphism of spectral sequences associated to $\pi:\PP^1_F\to \Spec(F)$ described above gives commutative diagrams
\[
\begin{tikzcd}
\sHMWBM_n(F(t),v_{\vert F(t)}) \ar[r,"d"] \ar[d] & \displaystyle{\bigoplus_{x\in (\PP^1_F)^{(1)}} \sHMWBM_{n-1}(F(x),v_{\vert F(x)})} \ar[d,"{\pi_*}"] \\
0\ar[r] & \sHMWBM_{n-1}(F,v_F)
\end{tikzcd}
\]
showing that $\pi_*d=0$. As seen above, $d$ is of the form $\sum_{x}d_x$, where 
\[
d_x:\sHMWBM_n(F(t),v_{\vert F(t)})\to \sHMWBM_{n-1}(F(x),v_{\vert F(x)}).
\]
Now, we may write $\pi_*=\sum_x (\pi_*)_x$ and the above equality reads as $\sum_x (\pi_*)_xd_x=0$. For $x=\infty$, we observe that $(\pi_*)_{\infty}$ is just the identity (being induced by the identity on $\Spec(F)$) and consequently $\sum_{x\neq \infty}(\pi_*)_xd_x=-d_{\infty}$.

\begin{lem}
Let $F/k$ be a finitely generated field extension of transcendence degree $p$ and let $x\in \PP^1_F$ be a closed point. Then, the homomorphism 
\[
(\pi_*)_x:\sHMWBM_{p+n}(F(x),-n)\to \sHMWBM_{p+n}(F,-n)
\] 
obtained above coincides with the push-forward homomorphism
\[
\Tr^{F(x)}_F:\sKMW_{p+n}(F(x);\omega_{F(x)/k})\to \sKMW_{p+n}(F;\omega_{F/k})
\]
of \cite{Mor12_e}*{Definition~4.4} under the isomorphisms 
\[
\sHMWBM_{p+n}(F(x),-n)\simeq \sHMW^{-p-n}(F(x),T_{F(x)}+n)\simeq \sKMW_{p+n}(F(x);\omega_{F(x)/k})
\]
obtained from Theorem~\ref{thm:duality} and Corollary~\ref{cor:fieldcase},
 using again the notation of \eqref{eq:KMW;twists}.
\end{lem}

\begin{proof}
We know from Lemma \ref{lem:smoothdifferentials} that the differentials in the diagram
\[
\begin{tikzcd}
\sHMWBM_{p+n+1}(F(t),-n) \ar[r,"d"] \ar[d] & \displaystyle{\bigoplus_{x\in (\PP^1_F)^{(1)}} \sHMWBM_{p+n}(F(x),-n)} \ar[d,"{\pi_*}"] \\
0 \ar[r] & \sHMWBM_{p+n}(F,-n)
\end{tikzcd}
\]
coincide with the differentials in Milnor-Witt $K$-theory. In particular, the sequence
\[
\begin{tikzcd}
\sHMWBM_{p+n+1}(F(t),-n) \ar[r,"d"] & \displaystyle{\bigoplus_{x\in (\Aone_F)^{(1)}} \sHMWBM_{p+n}(F(x),-n)}
\end{tikzcd}
\]
splits and it follows from the equation $\sum_{x\neq \infty}(\pi_*)_xd_x=-d_{\infty}$ that $(\pi_*)_*$ is determined by any splitting and $d_{\infty}$. We conclude applying Lemma \ref{lem:smoothdifferentials} again.
\end{proof}

We are now ready to prove that the push-forwards in Borel-Moore homology coincide with those in Milnor-Witt $K$-theory when the relevant groups are isomorphic.

\begin{thm}
Let $F\subset L$ be a finite extension of finitely generated field extensions of $k$. Let $p$ be the transcendence degree of $F$ (and thus $L$) and let $n\in\ZZ$ be an integer. Then, the push-forward 
\[
\sHMWBM_{p+n}(L,-n)\to \sHMWBM_{p+n}(F,-n)
\]
coincides with the push-forward
\[
\Tr^{L}_F:\sKMW_{p+n}(L;\omega_{L/k})\to \sKMW_{p+n}(F;\omega_{F/k})
\]
defined in \cite{Mor12_e}*{Definition~4.4} under the isomorphisms of the previous lemma.
\end{thm}

\begin{proof}
The field extension $L/F$ being finitely generated, we can see $F$ as a closed point of $\PP^1_L\times \ldots\times \PP^1_L$. We conclude using the respective projections and the previous lemma.
\end{proof}

\end{num}

\begin{num}
We finally collect the results of the previous sections to identify the Borel-Moore homology groups $\sHMWBM_n(X,-n)$ for any $n\in\ZZ$. 
\index[notation]{hmwbmx-n@$\sHMWBM_n(X,-n)$}%
Consider the spectral sequence constructed in \ref{num:BM}
\[
E^1_{p,q}(X,-n):=\bigoplus_{x\in X_{(p)}} \sHMWBM_{p+q}(k(x),-n)\implies \sHMWBM_{p+q}(X,-n).
\]
According to the duality theorem, we may replace the groups of the $E^2$-page by MW-motivic cohomology groups and it follows from Proposition \ref{prop:twisted},
 using notation \eqref{eq:HMW;twists},
 that the spectral sequence takes the form
\[
E^1_{p,q}(X,-n):=\bigoplus_{x\in X_{(p)}}\HMW^{p-q+2n,p+n}\big(k(x),\ZZ\{\omega_{k(x)/k}\}\big) \implies \sHMWBM_{p+q}(X,-n)
\]
Let now $X$ be of dimension $d$. Then $E^1_{p,q}=0$ if $p<0$ or $p>d$. On the other hand, it follows from Corollary \ref{cor:fieldcase} that $E^1_{p,q}=0$ if $q<n$ while the line $q=n$ reads as
\[
\begin{tikzcd}
[column sep=1.3em]
\cdots & \displaystyle{\bigoplus_{x\in X_{(p)}} \sKMW_{p+n}(k(x);\omega_{k(x)/k}}) \ar[l] & \displaystyle{\bigoplus_{x\in X_{(p+1)}} \sKMW_{p+1+n}(k(x);\omega_{k(x)/k}}) \ar[l] & \cdots \ar[l]
\end{tikzcd}
\]
using this time Morel's notation \eqref{eq:KMW;twists}.

\begin{dfn}
Let $X$ be a $k$-variety of dimension $d$. For any $n\in \ZZ$, we denote by $\CBM(X,\sKMW_{d+n})$ the complex $E^1_{p,n}(X,-n)$ (graded by the dimension of the points) in the spectral sequence computing Borel-Moore homology. 
\end{dfn}

The above vanishing results imply the following proposition.

\begin{prop}
Assume $k$ is a perfect infinite field of characteristic not $2$.
Let $n\in \NN$ and $X$ be a $k$-variety of dimension $d$. Then the spectral sequence yields isomorphisms
\[
\sHMWBM_{i}(X,n)\simeq \H_{n+i}\big(\CBM(X,\sKMW_{d-n})\big)
\]
for $i=0,1$.
\end{prop}

\begin{proof}
For $p\leq n$, it follows from \chfinitecw, Proposition~\ref{prop:explicit} that $\sHMW^{p-q-2n,p-n}(k(x))=0$ if $p-q-2n\neq p-n$, i.e.\ if $q\neq -n$. We then have
\[
E^1_{p,q}(X,n)=\begin{cases}  0 & \text{if $p<0$ and $p>d$.} \\ 0 & \text{if $q<-n$.} \\ 0 & \text{if $p\leq n$ and $q\neq -n$.}\end{cases}.
\]
The result follows immediately.
\end{proof}

\end{num}

We now focus on the complex $\CBM(X,\sKMW_{d-n})$ for a while, starting with the differentials. Let then $x\in X_{(p)}$. Denote by $Z$ the normalization of $\overline{\{x\}}$ and observe that the morphism $f:Z\to X$ is finite. By \ref{num:pf}, we get a morphism of spectral sequences 
\[
f_*:E^1_{p,q}(Z,n)\to E^1_{p,q}(X,n)
\]
and in particular a morphism of complexes
\[
f_*:\CBM(Z,\sKMW_{p+n})\to \CBM(X,\sKMW_{d+n})
\]
Now, $Z$ is smooth in codimension $1$ and it follows that the differential
\[
\sKMW_{p+n}(k(x);\omega_{k(x)/k}) \to \displaystyle{\bigoplus_{z\in Z_{(p-1)}}}  \sKMW_{p+n-1}(k(z);\omega_{k(z)/k})
\]
can be computed as in Lemma \ref{lem:smoothdifferentials}; in particular, we can use the residue homomorphism defined in \cite{Mor12_e}*{Theorem~3.15, \S 5}. On the other hand, the morphism of complexes $f_*$ yields a commutative diagram
\[
\begin{tikzcd}
\sKMW_{p+n}(k(x);\omega_{k(x)/k}) \ar[r] \ar[d] & \displaystyle{\bigoplus_{z\in Z_{(p-1)}}} \sKMW_{p+n-1}(k(z);\omega_{k(z)/k}) \ar[d] \\
\displaystyle{ \bigoplus_{x\in X_{(p)}}} \sKMW_{p+n}(k(x);\omega_{k(x)/k}) \ar[r] & \displaystyle{\bigoplus_{x\in X_{(p-1)}}} \sKMW_{p+n-1}(k(x);\omega_{k(x)/k})
\end{tikzcd}
\]
where the left vertical map is the inclusion of the relevant factor, the horizontal maps are the differentials, and the right vertical map is the sum of the push-forwards described in \ref{num:pf}. This immediately implies the following theorem.

\begin{thm} Assume $k$ is a perfect infinite field of characteristic not $2$.
Let $n\in\NN$ and let $X$ be a smooth $k$-variety of dimension $d$. For any $i\in \NN$, we have an identification
\[
\H_i\big(\CBM(X,\sKMW_{d-n})\big) \simeq \H^{d-i}(X,\sKMW_{d-n}\{\omega_{X/k}\})
\]
which is natural with respect to push-forward homomorphisms. 
 For the definition of the \emph{twisted} Nisnevich sheaf $\sKMW_{d-n}\{\omega_{X/k}\}$,
 we refer the reader to \cite{Feld21}, \textsection 2.2.4.
\end{thm}

\begin{dfn}
Let $X$ be a smooth $k$-variety of pure dimension $d$. We define the homological grading on its Chow-Witt groups by 
\[
\chht iX L:=\cht {d-i}X{\omega_{X/k} \otimes L^\vee}.
\]
If $X$ has several connected components of various dimensions, we extend this definition by additivity.
 We put further: $\chhtnotw iX=\chht iX {\mathcal O_X}$.
\index{Chow-Witt!group!homological}%
\index[notation]{chtnxo@$\chht nX{\omega_{X/k}}$}%
\end{dfn}
The convention chosen in this definition corresponds to that of \cite{Fasel20}, Definition 2.5.
 It is motivated by the duality formula of \cite{Fasel20}, beginning of Section 2.2.

\begin{coro}\label{cor:ChowWittBorelMoore}
Assume $k$ is a perfect infinite field of characteristic not $2$.
Let $X$ be a smooth $k$-variety of dimension $d$. For any $n\in\NN$, we have an isomorphism
\[
\sHMWBM_{0}(X,n)\simeq \chhtnotw n X
\]
which is natural with respect to push-forward maps.
\end{coro}

\begin{rem}
In case $X$ is not smooth, we can take the identification 
\[
\H_n(\CBM(X,\sKMW_{d-n}))=\chhtnotw nX
\]
of the previous theorem as a definition of the Chow-Witt group of dimension $n$-points. These "singular" Chow-Witt groups are covariantly functorial with respect to push-forwards and contravariantly functorial with respect to \'etale morphisms. Additionally, there is an obvious forgetful homomorphism $\chhtnotw nX\to \CH_n(X)$, where the right-hand side is the usual Chow group of $n$-dimensional cycles
 (see \cite{Fult98_a}).

Besides, let $L$ be a line bundle over $X$. Considering now $\sHMWBM_i(X,(n+1)-L)$ instead of $\sHMWBM_i(X,n)$ and running the same arguments as above, we find complexes $\CBM(X,\sKMW_{d+n},L)$ whose degree $p$ term is of the form
\[
\displaystyle{\bigoplus_{x\in X_{(p)}}  \sKMW_{p+n}(k(x);\omega_{k(x)/k}\otimes L_x})
\]
and which enjoy the same properties as their "untwisted" versions. In particular, for any $n\geq 0$ one has:
\[
\H_0(\CBM(X,\sKMW_{d-n},L))=\chht nX{\omega_{X/k}\otimes L}.
\]
\end{rem}

\subsection{Generalized regulators}\label{sec:MW-regulators}
\index{regulator}%

\begin{num}\label{num:MW-regulators}
As explained in Paragraph \ref{num:morphisms},
the commutativity of the diagram \eqref{eq:diagram_DMs} automatically
induces morphisms of ring spectra as follows:
\[
\begin{tikzcd}
[row sep=20pt,column sep=20pt]
\sHAone R \ar[r,"\psi"] & \sHMW R \ar[r,"\varphi"] \ar[d] & \sHM R\ar[d] \\
& \sHMWet R \ar[r,"{\varphi_\et}"] & \sHMet R.
\end{tikzcd}
\]
Moreover, as explained in Paragraph \ref{num:morphisms_ringsp} and Remarks
\ref{rem:morphisms_ringsp1}, \ref{rem:morphisms_ringsp2},
these morphisms induce natural transformations of the four associated
theories, compatible with products.

In particular, given a $k$-scheme $X$ and virtual bundle $v$ over $X$
of rank $m$, we get morphisms
\[
\sHAone^n(X,v,R)
 \xrightarrow{\psi_*} \sHMW^n(X,v,R)
 \xrightarrow{\varphi_*} \sHM^{n+2m,m}(X,R)
\]
where the right-hand side is Voevodsky's motivic cohomology
(see \ref{num:mot_coh}).
In brief, these maps are compatible with all the structures on
cohomology described in Section \ref{sec:theories}.

Moreover, for any $k$-variety $X$, we get also natural morphisms:
\[
\sHBMAone_n(X,v,R)
 \xrightarrow{\psi_*} \sHMWBM_n(X,v,R)
 \xrightarrow{\varphi_*} \CH_m(X,n) \otimes_\ZZ R,
\]
compatible with contravariant and covariant functorialities,
and localization long exact sequences
(Paragraph \ref{num:localization_BM&c}).

Of course, when $X$ is a smooth $k$-variety, 
the two maps $\varphi_*$ (resp.\ $\psi_*$)
that appear above can be compared by duality
(Theorem \ref{thm:duality}).
Besides, if $v=[m]$ and $n=0$, one can check the
latter map $\varphi_*$ is simply the canonical map:
\[
\ch mX \otimes_\ZZ R \longrightarrow \CH^m(X) \otimes_\ZZ R
\]
from Chow-Witt groups to Chow groups.
\end{num}

\begin{rem}
Regulators, according to the reformulation of Beilinson
 (see \cite{Bei85_a}, after 5.10), are maps from motivic cohomology.
\end{rem}

Note finally, as a corollary of our definitions and \chdmt, Proposition \ref{prop:compare_DMt_DM},
 the following result:\footnote{Intuitively, quadratic forms disappear for the \'etale topology
 as very unit is a square \'etale locally.}
\begin{thm}\label{thm:compare_HMW_et&HM_et}
We assume that $k$ is an arbitrary perfect field.
Then the map $\varphi_\et:\sHMWet R \rightarrow \sHMet R$ is an isomorphism of ring spectra
 over $k$.
\end{thm}

\renewcommand{\theequation}{\arabic{equation}}

\renewcommand\thethm{\thesection.\arabic{thm}}
\makeatletter
\@addtoreset{thm}{section}
\makeatother

\chapter[Effectivity of motivic spectra]{On the effectivity of spectra representing motivic cohomology theories\author{Tom Bachmann and Jean Fasel}}
\label{ch:effective}

\section*{Abstract}
Let $k$ be an infinite perfect field. We provide a general criterion for a spectrum $E\in \SHk$ to be effective, i.e.\ to be in the localizing subcategory of $\SHk$ generated by the suspension spectra $\Sigma_T^{\infty}X_+$ of smooth schemes $X$.
As a consequence, we show that two recent versions of generalized motivic cohomology theories coincide.

\section*{Introduction}

In \cite{Bac17_f}, the first author undertook the study of the very effective slice spectral sequence of Hermitian $K$-theory, which could be seen as a refinement of the analogue in motivic homotopy theory of the famous Atiyah-Hirzebruch spectral sequence linking singular cohomology with topological $K$-theory. He observed that the generalized slices were 4-periodic and consisting mostly of well understood pieces, such as ordinary motivic cohomology with integral and mod 2 coefficients. However, there is a genuinely new piece given by a spectrum that he called \emph{generalized motivic cohomology}. Thus, Hermitian $K$-theory can be ``understood'' in terms of ordinary motivic cohomology and generalized motivic cohomology in his sense. Even though he was able to deduce abstractly some properties for this motivic cohomology, some questions remained open.

On the other hand, different generalizations of ordinary motivic cohomology recently appeared in the literature, always aimed at understanding better both the stable homotopy category of schemes and its ``abelian'' version. First, Garkusha-Panin-Voevodsky developed the formalism of framed correspondences and its linear version. Among many possible applications, this formalism allows to define an associated motivic cohomology, the first computations of which were performed in \cite{Neshitov18_f}. Second, Calmès-Déglise-Fasel introduced the category of finite MW-correspondences and its associated categories of motives (Chapters \ref{ch:finitecw} and \ref{ch:dmt}) and performed computations allowing to recast most of the well-known story in the ordinary motivic cohomology in this new framework. Third, Druzhinin introduced the category of GW-motives (\cite{Dr17-3_f}) producing yet another version of motivic cohomology.

This flurry of activity leads to the obvious question: how are these theories related? Note the parallels with the situation at the beginnings of singular cohomology. This is the question we address in this paper with a quite general method. To explain it, note first that all these motivic cohomologies are represented by ring spectra in the motivic stable homotopy category (of $\mathbb{P}^1$-spectra) $\SHk$. This category is quite complicated, but the situation becomes much better if the ring spectra are in the localising subcategory $\SHkeff$ generated by the image of the suspension spectrum functor $\Sigma_T^{\infty}:\SHS\to \SHk$. This category is endowed with a $t$-structure (\cite[Definition 5.5]{Spitzweck12_f} \cite{Bac17_f}*{Proposition~4}) whose heart is much easier to understand than the heart of the (usual) $t$-structure of $\SHk$. Moreover, many naturally occurring spectra turn out to be in this heart. Thus, our strategy is to prove that the relevant spectra are in $\SHkeff$, or \emph{effective}, then show that they are represented by objects in the heart, and finally compare them via the natural maps linking them. Unsurprisingly, the first step is the hardest and the main part of the paper is devoted to this point. The criterion for effectivity that we obtain is the following (Theorem \ref{thm:eff-crit}).

\begin{thm*}
Let $E \in \SHk$, where $k$ is a perfect field. Then $E \in \SHkeff$ if and only if for all $n \ge 1$ and all finitely generated fields $F/k$, we have $(E \wedge \Gm^{\wedge n})(\hat \Delta^\bullet_F) \simeq 0$.
\end{thm*}

In the statement, $\hat \Delta^\bullet_F$ denotes the essentially smooth cosimplicial scheme whose component in degree $n$ is the semi-localization at the vertices of the standard algebraic $n$-simplex over $F$. Making sense of $(E \wedge \Gm^{\wedge n})(\hat \Delta^\bullet_F)$ requires some contortions which are explained in Section \ref{sec:general-criterion}. The appearance of $\hat \Delta^\bullet_F$ is explained by the need to compute the zeroth (ordinary) slice of a spectrum, using Levine's homotopy coniveau filtration (\cite{L08_f}).

Having this criterion in the pocket, the last two (much easier) steps of our comparison theorem take place in the proof of our main result (Theorem \ref{thm:comparisonUHt}).

\begin{thm*}
Let $k$ be an infinite perfect field of exponential characteristic $e \ne 2$ and let 
\[
M: \SHk \leftrightarrows  \DMtk: U
\]
be the canonical adjunction. Then the spectrum $U(\un)$ representing MW-motivic cohomology with $\ZZ$-coefficients is canonically isomorphic to the spectrum $\HtZ$ representing abstract generalized motivic cohomology with $\ZZ$-coefficients.
\end{thm*}

The organization of the paper is as follows. We briefly survey the main properties of the category of MW-motives, before proving in Section \ref{sec:rational} that the presheaf represented by $\Gm^{\wedge n}$ is rationally contractible (in the sense of \cite{S03_f}*{\S 2}) for any $n\geq 1$. Unsurprisingly, our proof follows closely Suslin's original method. However, there is one extra complication due to the fact that the presheaf represented by $\Gm^{\wedge n}$ is in general not a sheaf. We thus have to compare the Suslin complex of a presheaf and the one of its associated sheaf in Section \ref{sec:semilocal}. This part can be seen as an extension of the results in \chcancellation, \S\ref{sec:etfMWmotives} to the case of semi-local schemes, i.e.\ localizations of a smooth scheme at a finite number of points. The proof of our criterion for effectivity takes place in the subsequent section. Finally, we prove our comparison result in Section \ref{sec:AppMWmotives}, where all the pieces fall together. 

In the last few paragraphs of the article, we give some examples of applications of our results, one of them being a different way to prove the main result of \cite{S03_f} avoiding polyrelative cohomology. 

\subsection*{Conventions}
Schemes are separated and of finite type over a base field $k$, assumed to be infinite perfect of characteristic different from $2$.
Recall that a field $k$ is said to have exponential characteristic $e=1$ if $\charac k = 0$, and $e = \charac k$ else.


\section{Recollections on MW-correspondences}\label{sec:recollections}

In this section, we briefly survey the few basic features of MW-correspondences (as constructed in \chfinitecw, \S\ref{sec:fmwcorr}) and the corresponding category of motives (\chdmt, \S\ref{sec:MWmotives}) that are needed in the paper. Finite MW-correspondences are an enrichment of finite correspondences in the sense of Voevodsky using symmetric bilinear forms. The category whose objects are smooth schemes and whose morphisms are MW-correspondences is denoted by $\cork$ and we have a sequence of functors
\[
\smk\stackrel{\tilde\gamma}\to \cork\stackrel{\pi}\to \corVk
\]
such that the composite is the classical embedding of the category of smooth schemes into the category of finite correspondences. For a smooth scheme $X$, the corresponding representable presheaf on $\cork$ is denoted by $\MWprep(X)$. This is a Zariski sheaf, but not a Nisnevich sheaf in general (\chfinitecw, Proposition~\ref{prop:zarsheaf}, Example~\ref{exm:notnis}). The associated Nisnevich sheaf also has MW-transfers (i.e.\ is a (pre-)sheaf on $\cork$) by \chdmt, Proposition~\ref{prop:exist_associated-W-t-sheaf} and is denoted by $\MWrepZ(X)$.

Consider next the cosimplicial object $\Delta^\bullet_k$ in $\smk$ defined as usual (see \chfinitecw, \S\ref{sec:motivic}). Taking the complex associated to a simplicial object, we obtain the Suslin complex $\Cstar {\MWrepZ(X)}$ associated to $X$, which is the basic object of study. Applying this to $\Gm^{\wedge n}$, we obtain complexes of Nisnevich sheaves $\tZcbx n$ for any $n\in \NN$ and complexes $\tZpx n:=\tZcbx n[-n]$ whose hypercohomology groups are precisely the MW-motivic cohomology groups in weight $n$. In this paper, we will also consider the cosimplicial object $\hat\Delta^\bullet_k$ obtained from $\Delta^{\bullet}_k$ by semi-localizing at the vertices (see \cite{L08_f}*{5.1}, \cite{S03_f}*{paragraph before Proposition~2.5}). Given a finitely generated field extension $L$ of the base field $k$, the same definition yields cosimplicial objects $\Delta^{\bullet}_L$ and $\hat\Delta^\bullet_L$ that will be central in our results.
If $L/k$ is separable, then note that both $\Delta^{\bullet}_L$ and $\hat\Delta^\bullet_L$ are simplicial essentially smooth schemes.

The category $\cork$ is the basic building block in the construction of the category of effective MW-motives (aka the category of MW-motivic complexes) $\DMtek$ and its $\PP^1$-stable version $\DMtk$ (\chdmt, \S\ref{sec:MWmotives}). The category of effective MW-motives fits into the following diagram of adjoint functors (where $R$ is a ring)
\begin{equation}\label{eq:unstable}
\begin{tikzcd}
[row sep=24pt,column sep=30pt]
\DAekR \ar[r,shift left=2pt,"{\derL \tilde \gamma^*}"]
 & \DMtekR\ar[r,shift left=2pt,"{\derL \pi^*}"]
     \ar[l,shift left=2pt,"{\tilde \gamma_{*}}"]
 & \DMekR
     \ar[l,shift left=2pt,"{\pi_{*}}"]
\end{tikzcd}
\end{equation}
where the left-hand category is the effective $\Aone$-derived category (whose construction is for instance recalled in \chspectra, \S\ref{sec:mothomthringspectra}).

More precisely, each category is the homotopy category of a proper cellular model category and the functors, which are defined at the level of the underlying closed model categories, are part of a Quillen adjunction. Moreover, each model structure is symmetric monoidal, the respective tensor products admit a total left derived functor and the corresponding internal homs admit a total right derived functor. The left adjoints are all monoidal and send representable objects to the corresponding representable object, while the functors from right to left are conservative. The corresponding diagram for stable categories reads as
\begin{equation}\label{eq:stable}
\begin{tikzcd}
[row sep=24pt,column sep=30pt]
\DAkR \ar[r,shift left=2pt,"{\derL \tilde \gamma^*}"]
 & \DMtkR \ar[r,shift left=2pt,"{\derL \pi^*}"]
     \ar[l,shift left=2pt,"{\tilde \gamma_{*}}"]
 & \DMkR
     \ar[l,shift left=2pt,"{\pi_{*}}"] 
\end{tikzcd}
\end{equation}
and enjoys the same properties as in the unstable case.


\section{Rational contractibility}\label{sec:rational}

Recall the following definition from \cite{S03_f}*{\S 2}. For any presheaf $F$ of abelian groups, let $\tilde{C}_1F$ be the presheaf defined by
\[
\tilde{C}_1F(X)=\colim_{X\times\{0,1\}\subset U\subset X\times \Aone} F(U),
\]
where $U$ ranges over open subschemes of $X \times \Aone$ containing $X \times \{0,1\}$.
Observe that the restriction of $\tilde{C}_1F(X)$ to both $X\times \{0\}$ and $X\times\{1\}$ make sense, i.e.\ that we have morphisms of presheaves $i_0^*:\tilde{C}_1F\to F$ and $i_1^*: \tilde{C}_1F\to F$.

\begin{dfn}
A presheaf $F$ is called rationally contractible if there exists a morphism of presheaves $s:F\to \tilde{C}_1F$ such that $i_0^*s=0$ and $i_1^*s=\id_F$.
\end{dfn}

We note the following stability property.

\begin{lem} \label{lemm:rat-contractible-pullback}
Let $K/k$ be a field extension and write $p: Spec(K) \to Spec(k)$ for the associated morphism of schemes. Then $p^* \tilde{C}_1F \simeq \tilde{C}_1 p^*F$. In particular, $p^*F$ is rationally contractible if $F$ is.
\end{lem}
\begin{proof}
Since $k$ is perfect, $p$ is essentially smooth and so for $X \in \sm K$ there exists a cofiltered diagram with affine transition maps $\{X_i\} \in \smk$ with $X = \lim_i X_i$. Then for any sheaf $G$ on $\smk$ we have $(p^*G)(X) = \colim_i G(X_i)$. Now, note that $X\times \Aone= \lim_i (X_i\times_k \Aone)$ and \cite{EGA4-3}*{Corollaire~8.2.11} shows that any open subset in $X \times \Aone$ containing $X \times \{0,1\}$ is pulled back from an open subset of $X_i \times \Aone$ containing $X_i \times \{0, 1\}$ for some $i$. The result follows.
\end{proof}

The main property of rationally contractible presheaves is the following result which we will use later. 

\begin{prop}[Suslin] \label{prop:suslin}
Let $F$ be a rationally contractible presheaf of abelian groups on $\smk$. Then
$(\Cstar F)(\hat\Delta^\bullet_K) \simeq 0$, for any field $K/k$.
\end{prop}
\begin{proof}
Combine \cite{S03_f}*{Lemma~2.4 and Proposition~2.5}, and use Lemma \ref{lemm:rat-contractible-pullback}.
\end{proof}

Examples of rationally contractible presheaves are given in \cite{S03_f}*{Proposition~2.2}, and we give here a new example that will be very useful in the proof of our main result. 

\begin{prop}\label{prop:ratcontractible}
Let $X$ be a smooth connected scheme over $k$ and $x_0\in X$ be a rational $k$-point of $X$. Assume that there exists an open subscheme $W\subset X\times \Aone$ containing $(X\times \{0,1\})\cup (x_0\times \Aone)$ and a morphism of schemes $f:W\to X$ such that $f_{\vert_{X\times 0}}=x_0$, $f_{\vert_{X\times 1}}=\id_X$ and $f_{\vert_{x_0\times \Aone}}=x_0$. Then the presheaf $\MWprep(X)/\MWprep(x_0)$ is rationally contractible.
\end{prop}

\begin{proof}
We follow closely Suslin's proof in \cite{S03_f}*{Proposition~2.2}. Let $Y$ be a smooth connected scheme and let $\alpha\in \cork(Y,X)$. There exists then an admissible subset $Z\subset Y\times X$ (i.e.\ $Z$ endowed with its reduced structure is finite and surjective over $X$) such that 
\[
\alpha\in \chst {d_X}{Z}{Y\times X}{\omega_X}.
\]
where $\omega_X$ is the pull-back along the projection $Y\times X\to X$ of the canonical sheaf of $X$.
On the other hand, the class of $\tilde\gamma(\id_{\Aone})$ is given by the class of the MW-correspondence $\Delta_*(\langle 1\rangle)$ where 
\[
\Delta_*:\ch 0{\Aone}\to \chst 1{\Delta}{\Aone\times \Aone}{\omega_{\Aone}}
\]
is the push-forward along the diagonal $\Delta:\Aone\to \Aone\times \Aone$, and $\Delta=\Delta(\Aone)$. Considering the Cartesian square
\[
\begin{tikzcd} Y\times X\times \Aone\times \Aone \ar[r,"{p_2}"] \ar[d,"{p_1}"'] & \Aone\times \Aone \ar[d] \\
Y\times X \ar[r] & \Spec(k)
\end{tikzcd}
\]
we may form the exterior product $p_1^*\alpha\cdot p_2^*\Delta_*(\langle 1\rangle)$ and its image under the push-forward along $\sigma:Y\times X\times \Aone\times\Aone\to Y\times\Aone\times X\times \Aone$ represents the MW-correspondence $\alpha\times \id_{\Aone}$ defined in \chfinitecw, \S \ref{sec:tensorproducts}. Using this explicit description, we find a cycle
\[
\alpha\times \id_{\Aone}\in \chst {d_X+1}{Z\times \Delta}{Y\times \Aone\times X\times \Aone}{\omega_{X\times \Aone}}
\]
where $Z\times \Delta$ is the product of $Z$ and $\Delta$. Now, we may consider the closed subset $T:=(X\times \Aone)\setminus W\subset X\times \Aone$. It is readily verified that $T^\prime:=(Z\times \Delta)\cap (Y\times \Aone\times T)$ is finite over $Y\times \Aone$. Thus $p_{Y\times \Aone}(T^\prime)\subset Y\times \Aone$ is closed and we can consider its open complement $U$ in $(Y\times \Aone)$. It follows from \cite{S03_f}*{proof of Proposition~2.2} that $Y\times \{0,1\}\subset U$. By construction, we see that $\left(U\times (X\times \Aone)\right)\cap (Z\times \Delta)\subset U\times W$ and is finite over $U$. Restricting $\alpha\times \id_{\Aone}$ to $U\times W$, we find 
\[
i^*(\alpha\times \id_{\Aone})\in \chst {d_X+1}{(Z\times \Delta)\cap (U\times W) }{U\times W}{i^*\omega_{X\times \Aone}}
\]
where $i:U\times W\to Y\times \Aone\times X\times \Aone$ is the inclusion. Now, we see that we have a canonical isomorphism $i^*\omega_{X\times \Aone}\simeq \omega_W$ and it follows that we can see $i^*(\alpha \times \id_{\Aone})$ as a finite MW-correspondence between $U$ and $W$. Composing with $f:W\to X$, we get a finite MW-correspondence $f\circ s(\alpha):U\to X$, i.e.\ an element of $\cork(U,X)=\MWprep(X)(U)$ with $Y\times \{0,1\}\subset U\subset Y\times \Aone$. Using now the canonical morphism $\MWprep(X)(U)\to  \tilde{C}_1(\MWprep(X))(Y)$, we obtain an element denoted by $s(\alpha)$. It is readily checked that this construction is (contravariantly) functorial in $Y$ and thus that we obtain a morphism of presheaves
\[
s:\MWprep(X)\to \tilde{C}_1(\MWprep(X)).
\]
We check as in \cite{S03_f}*{Proposition~2.2} that this morphism induces a morphism
\[
s:\MWprep(X)/\MWprep(x_0)\to \tilde{C}_1\big(\MWprep(X)/\MWprep(x_0)\big).
\]
with the prescribed properties.
\end{proof}

\begin{coro}\label{cor:gmcase}
For any $n\geq 1$, the presheaf $\MWprep(\Gm^{\times n})/\MWprep(1,\ldots,1)$ is rationally contractible.
\end{coro}

\begin{proof}
Let $t_1,\ldots,t_n$ be the coordinates of $\Gm^{\times n}$ and $u$ be the coordinate of $\Aone$. We consider the open subscheme $W\subset \Gm^{\times n}\times \Aone$ defined by $ut_i+(1-u)\neq 0$. It is straightforward to check that $\Gm^{\times n}\times \{0,1\}\subset W$ and that $(1,\ldots,1)\times \Aone\subset W$. We then define 
\[
f:W\to \Gm^{\times n}
\] 
by $f(t_1,\ldots,t_n,u)=u(t_1,\ldots,t_n)+(1-u)(1,\ldots,1)$ and check that it fulfills the hypothesis of Proposition \ref{prop:ratcontractible}.
\end{proof}

We would like to deduce from this result that Proposition \ref{prop:suslin} also holds for the sheaf $\MWrepZ(\Gm^{\times n})/\MWrepZ(1,\ldots,1)$ associated to the presheaf $\MWprep(\Gm^{\times n})/\MWprep(1,\ldots,1)$, or more precisely that it holds for its direct summand $\MWrepZ(\Gm^{\wedge n}):=\tZcbx n$ for $n\geq 1$. This requires some comparison results between the Suslin complex of a presheaf and the Suslin complex of its associated sheaf, which are the objects of the next section.


\section{Semi-local schemes}\label{sec:semilocal}

In this section, a \emph{semi-local scheme} will be a localization of a smooth integral scheme $X$ at finitely many points.

Our aim in this section is to extend \chcancellation, Corollary~\ref{cor:localzero} to the case of semi-local schemes. Let us first recall a result of H. Kolderup (\cite{K17_f}*{Theorem~9.1}).

\begin{thm}\label{thm:excision}
Let $X$ be a smooth $k$-scheme and let $x\in X$ be a closed point. Let $U=\Spec(\OO_{X,x})$ and let $\mathrm{can}:U\to X$ be the canonical inclusion. Let $i:Z\to X$ be a closed subscheme with $x\in Z$ and let $j:X\setminus Z\to X$ be the open complement. Then there exists a finite MW-correspondence $\Phi\in \cork(U,X\setminus Z)$ such that the following diagram 
\[
\begin{tikzcd}
 & X \setminus Z\ar[d,"j"] \\
U \ar[r,"{\mathrm{can}}"'] \ar[ru,"{\Phi}"] & X
\end{tikzcd}
\]
commutes up to homotopy.
\end{thm}

We note that this result uses a proposition of Panin-Stavrova-Vavilov (\cite{PSV09_f}*{Proposition~1}) which is in fact true for the localization of a smooth scheme at finitely many closed points and that the proof of Theorem \ref{thm:excision} goes through in this setting. This allows us to prove the following corollary. We thank M. Hoyois for pointing out the reduction to closed points used in the proof. The same reduction also allows removing the closedness hypothesis of Theorem \ref{thm:excision} and its many-point version.

\begin{coro}
Let $X$ be a smooth scheme and let $x_1,\ldots,x_n\in X$ be finitely many points. Let $U=\Spec(\OO_{X,{x_1,\ldots,x_n}})$ and let $\mathrm{can}:U\to X$ be the inclusion. Let $i:Z\to X$ be a closed subscheme containing $x_1,\ldots,x_n$ and let $j:X\setminus Z\to X$ be the open complement. Then, there exists a finite MW-correspondence $\Phi\in \cork(U,X\setminus Z)$ such that the following diagram 
\[
\begin{tikzcd}
 & X\setminus Z\ar[d,"j"] \\
U\ar[r,"{\mathrm{can}}"'] \ar[ru,"{\Phi}"] & X
\end{tikzcd}
\]
commutes up to homotopy.
\end{coro}

\begin{proof}
Let $v_1,\ldots,v_n$ be (not necessarily distinct) closed specializations of $x_1,\ldots,x_n$ and let $V$ be the semi-localization of $X$ at these points. We have a sequence of inclusions $U\stackrel{\iota}\to V\stackrel{\mathrm{can}}\to X$. As $Z$ is closed, we see that $v_1,\ldots,v_n$ are also in $Z$ and we may apply the previous theorem to get a finite MW-correspondence $\Phi^\prime$ and a homotopy commutative diagram
\[
\begin{tikzcd}
 & X \setminus Z \ar[d,"j"] \\
V \ar[r,"{\mathrm{can}}"'] \ar[ru,"{\Phi^\prime}"] & X.
\end{tikzcd}
\]
Composing with the map $U\stackrel{\iota}\to V$, we get the result with $\Phi=\Phi^\prime\circ \iota$.
\end{proof}

We deduce the next result from the above, following \cite{K17_f}*{Corollary~9.2}.

\begin{coro}\label{cor:restriction}
Let $F$ be a homotopy invariant presheaf with MW-transfers. Let $Y$ be a semi-local scheme. Then the restriction homomorphism $F(Y)\to F(k(Y))$ is injective.
\end{coro}

\begin{proof}
Let $Y$ be the semi-localization of the smooth integral $k$-scheme $X$ at the points $x_1, \dots, x_n$.  By definition, we have $F(Y)=\colim_{x_1,\ldots,x_n\in V}F(V)$, whereas $F(k(Y)) = F(k(X))=\colim_{W\neq \emptyset}F(W)$. Here $V,W$ are open subschemes of $X$. Let then $s\in \colim_{x_1,\ldots,x_n\in V}F(V)$ mapping to $0$ in $F(k(X))$. There exists $V$ containing $x_1\ldots,x_n$ and $t\in F(V)$ such that $s$ is the image of $t$ under the canonical homomorphism, and there exists $W\neq \emptyset$ such that $t_{\vert_{W\cap V}}=0$. Shrinking $W$ if necessary, we may assume that $x_1,\ldots,x_n\not\in W$. We can now use Theorem \ref{thm:excision} with $X=V$, $Y=U$ and $Z=V\setminus (V\cap W)$. Since $F$ is homotopy invariant, we then find a commutative diagram
\[
\begin{tikzcd}
 & F(V\cap W) \ar[ld,"{\Phi^*}"'] \\
F(Y) & \ar[l,"{\mathrm{can}^*}"] F(V)\ar[u,"{j^*}"']
\end{tikzcd}
\]
showing that $s=0$.
\end{proof}

\begin{coro}
Let $F\to G$ be a morphism of homotopy invariant MW-presheaves such that for any finitely generated field extension $L/k$ the induced morphism $F(L)\to G(L)$ is an isomorphism. Then the homomorphism $F(X)\to G(X)$ is an isomorphism for any semi-local scheme $X$.
\end{coro}

\begin{proof}
As the category of MW-presheaves is abelian, we can consider both the kernel $K$ and the cokernel $C$ of $F\to G$. An easy diagram chase shows that $C$ and $K$ are homotopy invariant and our assumption implies that $C(L)=0=K(L)$ for any finitely generated field extension $L/k$. By Corollary \ref{cor:restriction}, it follows that $C(X)=0=K(X)$, proving the claim.
\end{proof}

\begin{coro}\label{cor:equality}
Let $F$ be a homotopy invariant MW-presheaf. Let respectively $F_{\zar}$ be the associated Zariski sheaf and $F_{\nis}$ be the associated Nisnevich sheaf. Then the canonical sequence of morphisms of presheaves
\[
F\to F_{\zar} \to F_{\nis}
\]
induces isomorphisms $F(X)\simeq F_{\zar}(X)\simeq F_{\nis}(X)$ for any semi-local scheme $X$.
\end{coro}

\begin{proof}
First note that $F_{\nis}$ is indeed an MW-sheaf by \chdmt, Proposition~\ref{prop:exist_associated-W-t-sheaf}. Moreover, the associated Zariski sheaf $F_{\zar}$ coincides with $F_{\nis}$ and they are both homotopy invariant by \chdmt, Theorem \ref{thm:Wsh&A1}. To conclude, we observe that the sequence $F\to F_{\zar} \to F_{\nis}$ induces isomorphisms when evaluated at finitely generated field extensions and we can use the previous corollary to obtain the result.
\end{proof}

We now pass to the identification of the higher cohomology presheaves of the sheaf associated to a homotopy invariant MW-presheaf $F$.

\begin{lem}\label{lem:superior}
Let $F$ be a homotopy invariant MW-presheaf. Then $\H^n_{\zar}(X,F_{\zar})=\H^n_{\nis}(X,F_{\nis})=0$ for any semi-local scheme $X$ and any $n>0$.
\end{lem}

\begin{proof}
Using \chdmt, Theorem \ref{thm:Wsh&A1}, it suffices to prove the result for $F_{\nis}$. Now, the presheaf $U\mapsto \H^n_{\nis}(U,F_{\nis})$ is an MW-presheaf (as the category of MW-sheaves has enough injectives by \chdmt, Proposition~\ref{prop:exist_associated-W-t-sheaf} and \cite{Gr57_f}*{Théorème~1.10.1}) which is homotopy invariant by \chdmt, Theorem~\ref{thm:Wsh&A1} again. As any field has Nisnevich cohomological dimension $0$, we find $\H^n_{\nis}(L,F_{\nis})=0$ for any finitely generated field extension $L/k$. We conclude using Corollary \ref{cor:restriction}.
\end{proof}

Recall that $\DMtek$ is the homotopy category of a certain model category. This model category is obtained as a localization of a model structure on the category $\Comp(\shMWxk{\nis})$ of unbounded chain complexes of MW-sheaves. We call a fibrant replacement functor for this localized model structure the \emph{$\Aone$-localization} functor, and denote it $\mathrm{L}_{\Aone}$. If $K$ is a complex of MW-presheaves, then we can take the associated complex of Nisnevich MW-sheaves $a_\nis K$. We write $\mathrm{L}_\nis K$ for a fibrant replacement of $a_\nis K$ in the usual (i.e.\ non-$\Aone$-localized) model structure on $\Comp(\shMWxk{\nis})$ (\chdmt, \S\ref{sec:dercat}).

We will need the following slight strengthening of \chdmt, Corollary~\ref{cor:LA1}.
\begin{lem} \label{lem:motivic-localization}
Let $F$ be an MW-presheaf. Then the $\Aone$-localization (of $F_{Nis}$) is given by $\mathrm{L}_{\nis} \Cstar F$.
\end{lem}
\begin{proof}
Throughout the proof we abbreviate $\Delta^\bullet := \Delta^\bullet_k$.
We claim that $F_\nis$ and $a_\nis \Cstar F$ are $\Aone$-equivalent. To see this, let $\CstarS F$ denote the complex constructed like $\Cstar F$, but with the constant cosimplicial object $*$ in place of $\Delta^\bullet$. In other words $\CstarS F = F \xleftarrow{0} F \xleftarrow{1} F \xleftarrow{0} \dots$. The projection $\Delta^\bullet \to *$ induces $\alpha: \CstarS F \to \Cstar F$. Since $\CstarS F$ is chain homotopy equivalent to $F$, it will suffice to show that $a_\nis \alpha: a_\nis \CstarS F \to a_\nis \Cstar F$ is an $\Aone$-equivalence. For this, it is enough to prove that $a_\nis \alpha$ is a levelwise $\Aone$-equivalence (because $\Aone$-equivalences are closed under filtered colimits), for which in turn it is enough to prove that $\alpha$ is a levelwise $\Aone$-homotopy equivalence. This is clear, since $\alpha_n$ is $F \to F^{\Delta^n}$, and $\Delta^n$ is $\Aone$-contractible. This proves the claim.
It thus remains to show that $a_\nis \Cstar F$ is $\Aone$-local. This follows from \chdmt, Corollary~\ref{cor:A1-local_complexes}.
\end{proof}

\begin{coro} \label{coro:compute-semilocal-suslin}
Let $F$ be a MW-presheaf and let $\Cstar F$ be its associated Suslin complex. For any $n\in\ZZ$, let $\H^n(\Cstar F)$ be the $n$-th cohomology presheaf of $\Cstar F$. Then for any semi-local scheme $X$ over $k$, we have canonical isomorphisms
\[
\H^n(\Cstar F)(X)\to \mathbb{H}^n_{\nis}(X,\mathrm{L}_{\Aone} F_\nis).
\]
\end{coro}

\begin{proof}
By Lemma \ref{lem:motivic-localization}, we have $\mathrm{L}_{\Aone} F_\nis \simeq \mathrm{L}_\nis \Cstar F$.
Observe first that the cohomology presheaves $\H^n(\Cstar F)$ are homotopy invariant and have MW-transfers. Denote by $h^n_{\nis}$ the associated Nisnevich sheaves (which are homotopy invariant MW-sheaves by \chdmt, Theorem~\ref{thm:Wsh&A1}). Considering the hypercohomology spectral sequence, we see that it suffices to prove that $\H^n_{\nis}(\Cstar F)(X)=\H^0(X,h^n_{\nis})$ and that $\H_{\nis}^i(X,h^n_{\nis})=0$ for $i>0$. The first claim follows from Corollary \ref{cor:equality}, while the second one follows from Lemma \ref{lem:superior}.
\end{proof}

\begin{rem}
Using the fact that the Zariski sheaf $h^n_{\zar}$ associated to $\H^n(\Cstar F)$ coincides with $h^n_{\nis}$ (\chdmt, Theorem~\ref{thm:Wsh&A1}), the same arguments as above give a canonical isomorphism
\[
\H^n(\Cstar F)(X)\to \hypH^n_{\zar}(X,\Cstar {F_{\zar}}).
\]
\end{rem}

Finally, we are in position to prove the result we need. In the statement, the complexes are the total complexes associated to the relevant bicomplexes of abelian groups.

\begin{coro} \label{coro:compute-semilocal-suslin-2}
Let $F$ be a MW-presheaf and $K/k$ be a finitely generated field extension. The canonical map
\[ 
\Cstar F(\hat\Delta^\bullet_K) \to (\mathrm{L}_{\Aone}F_{\nis})(\hat\Delta^\bullet_K)
\]
is a weak equivalence of complexes of abelian groups.
\end{coro}

\begin{proof}
We have strongly convergent spectral sequences
\[ \H^p(\Cstar F(\hat\Delta^q_K)) \Rightarrow \H^{p+q}(\Cstar F(\hat\Delta^\bullet_K)) \]
and
\[ 
\H^p((\mathrm{L}_{\Aone} F_\nis)(\hat\Delta^q_K)) \Rightarrow \H^{p+q}((\mathrm{L}_{\Aone} F_\nis)(\hat\Delta^\bullet_K)). 
\]
Since $\mathrm{L}_{\Aone} F_\nis$ is Nisnevich-local, we have $\H^p((\mathrm{L}_{\Aone} F_\nis)(\hat\Delta^q_K)) = \hypH^p_{\nis}(\hat\Delta^q_K, \mathrm{L}_{\Aone} F_\nis)$.
Thus the claim follows from Corollary \ref{coro:compute-semilocal-suslin} and spectral sequences comparison. Here we use that $\hat\Delta^q_K$ is semilocal: if $K = k(U)$ for some smooth irreducible scheme with generic point $\eta$, then $\hat\Delta^q_K$ is the semilocalization of $\Delta^q \times U$ in the points  $(v_i, \eta)$.
\end{proof}

\begin{thm}\label{thm:rational}
For any $n\geq 1$ and $K/k$ finitely generated, we have 
\[
\mathrm{L}_{\Aone}\big(\tZpx n\big)(\hat\Delta^\bullet_K)\simeq 0.
\]
\end{thm}

\begin{proof}
Since $\tZpx n[n]$ is $\Aone$-equivalent to $\MWprep(\Gm^{\wedge n})$, and the latter is a direct factor of $\MWprep(\Gm^{\times n})/\MWprep(1,\dots,1)$, by Corollary \ref{coro:compute-semilocal-suslin-2} it suffices to show that 
\[
\Cstar{\MWprep(\Gm^{\times n})/\MWprep(1,\dots,1)}(\hat\Delta^\bullet_K)\simeq 0.
\] 
This follows from Corollary \ref{cor:gmcase} and Proposition \ref{prop:suslin}.
\end{proof}


\section{A General Criterion}
\label{sec:general-criterion}

In this section we study when the motivic spectrum representing a generalized
cohomology theory of algebraic varieties is effective. We first recall a few facts about the slice filtration of \cite{Voe02_f}.

Let $\SHS$ be the motivic homotopy category of $S^1$-spectra and let $\SHk$ be the stable motivic homotopy category. We have an adjunction
\[
\Sigma^\infty_T: \SHS \leftrightarrows \SHk:\Omega^{\infty}_T
\]
and we write $\SHkeff$ for the localising subcategory (in the sense of \cite{Nee01_f}*{3.2.6}) of $\SHk$
generated by the image of $\Sigma^\infty_{T}$. The inclusion $i_0:\SHkeff\to \SHk$ has a right adjoint $r_0:\SHk \to \SHkeff$ and we obtain a functor $f_0=i_0r_0:\SHk \to \SHk$ called the effective cover functor. More generally, we may consider the localising subcategories $ \SHS(d)$ and $\SHkeff(d)$ of respectively $\SHS$ and $\SHk$ generated by the images of $X\wedge T^d$ for $X$ smooth and $d\in\NN$. We obtain a commutative diagram of functors
\[
\begin{tikzcd}
\SHS(d) \ar[r,"{\Sigma_T^\infty}"] \ar[d,"{i_d}"'] & \SHkeff(d)\ar[d,"{i_d}"] \\
\SHS \ar[r,"{\Sigma_T^\infty}"'] & \SHk
\end{tikzcd}
\]
Both of the inclusions $i_d:\SHS(d)\to \SHS$ and $i_d:\SHkeff(d)\to \SHk$ admit right adjoints $r_d$ and we set $f_d=i_dr_d$ (on both categories). We obtain a sequence of endofunctors 
\[
\ldots\to f_d\to f_{d-1}\to\ldots \to f_1\to f_0
\]
and we define $s_0$, the \emph{zeroth slice functor}, as the cofiber of $f_1\to f_0$. More generally, we let $s_d$ be the cofiber of $f_{d+1}\to f_d$.

The following result is due to M. Levine (\cite{L08_f}*{Theorems~9.0.3 and 7.1.1}).

\begin{thm}[Levine]\label{thm:omegas0}
The following diagram of functors
\[
\begin{tikzcd}
\SHk \ar[r,"{\Omega_T^{\infty}}"] \ar[d,"{s_0}"'] & \SHS \ar[d,"{s_0}"] \\
\SHk \ar[r,"{\Omega_T^\infty}"'] & \SHS 
\end{tikzcd}
\]
is commutative.
\end{thm}

One essential difference between $\SHS$ and $\SHk$ is that in the latter case, the above sequence of functors extends to a sequence of endofunctors
\[
\ldots\to f_d\to f_{d-1}\to\ldots \to f_1\to f_0\to f_{-1}\to \ldots\to f_{-n}\to\ldots
\]
Let us recall the following well-known lemma for the sake of completeness.

\begin{lem} \label{lemm:eff}
Let $E \in \SHk$. Then $\hocolim_{n \to \infty} f_{-n} E \to E$ is an equivalence.
\end{lem}

\begin{proof}
It suffices to show that for any $X \in \smk$ and any $i, j \in \ZZ$ we get 
\[
\Hom_{\SHk}(\Sigma^\infty X_+ [i] \wedge \Gm^{\wedge j}, E) = \Hom_{\SHk}(\Sigma^\infty X_+ [i] \wedge \Gm^{\wedge j}, \hocolim f_{-n} E).
\] 
Since $\Sigma^\infty X_+ [i] \wedge \Gm^{\wedge j}$ is compact, the right hand side is equal to 
\[
\colim_n \Hom_{\SHk}(\Sigma^\infty X_+ [i] \wedge \Gm^{\wedge j}, f_{-n} E).
\] 
For $j > -n$, we have $\Sigma^\infty X_+ [i] \wedge \Gm^{\wedge j} \in \SHkeff(-n)$ and hence 
\[
\Hom_{\SHk}(\Sigma^\infty X_+ [i] \wedge \Gm^{\wedge j}, f_{-n} E) = \Hom_{\SHk}(\Sigma^\infty X_+ [i] \wedge \Gm^{\wedge j}, E).
\] 
The result follows.
\end{proof}

We now make use of the \emph{spectral enrichment} of $\SHk$. To explain it, consider $E \in \SHk$. This yields a presheaf $rE \in PSh(\smk)$ given by 
\[
(rE)(U) = \Hom_{\SHk}(\Sigma^\infty_T U_+, E).
\] 
Write $\Spt(\smk)$ for the homotopy category of \emph{spectral presheaves} (\cite{L06_f}*{\S 1.4}). Then there exists a functor $R: \SHk \to \Spt(\smk)$ such that $rE = \pi_0 RE$. Indeed $R$ is constructed as the following composite
\[ \SHk \xrightarrow{\Omega^\infty_T} \SHS \xrightarrow{R_0} \Spt(\smk), \]
where $R_0$ is the (fully faithful) right adjoint of the localization functor. Now note that if $P \in \Spt(\smk)$ is a spectral presheaf and $F/k$ is a finitely generated field extension, then we can make sense of the expression $P(\hat\Delta^\bullet_F) \in \SH$: it is obtained by choosing a bifibrant model of $P$ as a presheaf of spectra, and then taking the geometric realization of the induced simplicial diagram \cite{L06_f}*{1.5}. If $E \in \SHk$, then we abbreviate $(RE)(\hat\Delta^\bullet)$ to $E(\hat\Delta^\bullet)$. Similarly if $E \in \SHS$, then we abbreviate $(R_0E)(\hat\Delta^\bullet)$ to $E(\hat\Delta^\bullet)$

\begin{lem} \label{lemm:s0-vanishing}
Let $E \in \SHk$, where $k$ is a perfect field. Then $s_0(E) \simeq 0$ if and only if for all finitely generated fields $F/k$ we have $E(\hat\Delta^\bullet_F) \simeq 0$.
\end{lem}

\begin{proof}
By definition, we have an exact triangle 
\[
f_1E\to f_0E\to s_0(E)\to f_1E[1]
\]
and it follows that $s_0(E)\in \SHkeff$. On the other hand, we have an adjunction
\[
\Sigma^\infty_T: \SHS \leftrightarrows \SHkeff: \Omega^\infty_T
\] 
and $\Omega^\infty_T$ is conservative on $\SHkeff$ (its left adjoint has dense image). Thus $s_0(E) \simeq 0$ if and only if $\Omega_T^{\infty}s_0(E)\simeq 0$, and the latter condition is equivalent to $s_0\Omega_T^{\infty}E\simeq 0$ by Theorem \ref{thm:omegas0}. By definition, we have $(\Omega_T^{\infty}E)(\hat \Delta^\bullet_F) = E(\hat \Delta^\bullet_F)$, and
we are thus reduced to proving that for $E\in \SHS$, we have $s_0(E)\simeq 0$ if and only if $E(\hat\Delta^\bullet_F)=0$ for $F/k$ finitely generated.

Let then $E\in \SHS$. We can (and will) choose a fibrant model for $E$, which we denote by the same letter. Now $s_0(E)$ is given by the $E^{(0/1)}$ construction of Levine (\cite{L08_f}*{Theorem 7.1.1}) and then $s_0(E) \simeq 0$ if and only if we have $E^{(0/1)}\simeq 0$. Since strictly homotopy invariant sheaves are unramified (\cite{Morel12_f}*{Example~2.3}), $E^{(0/1)}\simeq 0$ if and only if $E^{(0/1)}(F)\simeq 0$ for any finitely generated field extension $F/k$. Since $E^{(0/1)}(F) \simeq E(\hat \Delta^\bullet_F)$ (this argument is used for example in \cite{L08_f}*{proof of Lemma~5.2.1}), this concludes the proof.
\end{proof}

\begin{thm} \label{thm:eff-crit}
Let $E \in \SHk$, where $k$ is a perfect field. Then $E \in \SHkeff$ if and only if for all $n \ge 1$ and all finitely generated fields $F/k$, we have $(E \wedge \Gm^{\wedge n})(\hat \Delta^\bullet_F) \simeq 0$.
\end{thm}
\begin{proof}
By Lemma \ref{lemm:s0-vanishing}, we know that the condition is equivalent to $s_0(E \wedge \Gm^{\wedge n}) \simeq 0$. This is clearly necessary for $E \in \SHkeff$ and we are left to prove sufficiency.

Note that $s_0(E \wedge \Gm^{\wedge n}) \simeq s_{-n}(E) \wedge \Gm^{\wedge n}$. Thus our condition is equivalent to $s_{-n}(E) \simeq 0$ for all $n \ge 1$, or equivalently $f_0(E) \simeq f_{-n}(E)$ for all $n \ge 0$. Consequently we get $f_0(E) \simeq \hocolim_n f_{-n}(E)$. But this homotopy colimit is equivalent to $E$, by Lemma \ref{lemm:eff}. This concludes the proof.
\end{proof}

\begin{coro} \label{coro:main}
Let $\mathcal{D}$ be a symmetric monoidal category and let 
\[
M:\SHk  \leftrightarrows \mathcal{D}:U
\] 
be a pair of adjoint functors such that $M$ is symmetric monoidal.
Then $U(\un_{\mathcal{D}}) \in \SHkeff$ if and
only if $U(M \Gm^{\wedge n})(\hat \Delta^\bullet_F) \simeq 0$ for all $F/k$
finitely generated and all $n \ge 1$.
\end{coro}

\begin{proof}
Let $E = U(\un_{\mathcal{D}})$. Note that by Lemma \ref{lemm:inv} below, we have $U(M(\Gm^{\wedge n})) \simeq E \wedge \Gm^{\wedge n}$. Thus the result reduces to Proposition \ref{thm:eff-crit}.
\end{proof}

For the convenience of the reader, we include a proof of the following well-known result.

\begin{lem} \label{lemm:inv}
Let $M: \mathcal{C} \leftrightarrows \mathcal{D}: U$ be an adjunction of
symmetric monoidal categories, with $M$ symmetric monoidal. Then for any
rigid (e.g. invertible) object $G \in \mathcal{C}$ and any $E \in \mathcal{D}$, there is a
canonical isomorphism $U(E \wedge MG) \simeq U(E) \wedge G$.
\end{lem}

\begin{proof}
Let $DG$ be the dual object of $G$. As $M$ is symmetric monoidal, we see that $MG$ also admits a dual object, namely $M(DG)$.
For any object $F \in \mathcal{C}$, we get 
$\Hom_{\mathcal C}(F, U(E \wedge MG)) 
= \Hom_{\mathcal D}(MF, E \wedge MG) 
= \Hom_{\mathcal D}(MF\wedge M(DG), E) 
= \Hom_{\mathcal D}(M(F \wedge DG), E) 
= \Hom_{\mathcal C}(F \wedge DG, UE) 
= \Hom_{\mathcal C}(F, UE \wedge G)$. 
Thus we conclude by the Yoneda lemma.
\end{proof}

We can simplify this criterion in a special case.

\begin{coro} \label{corr:main-simplified}
Consider the following diagram of functors
\[
\begin{tikzcd}
[column sep=30pt,row sep=24pt]
\SHS \ar[r,shift left=2pt,"{M_0}"] \ar[d,shift left=2pt,"{\Sigma^\infty_T}"']
 & \mathcal D_0 \ar[d,"L"]
     \ar[l,shift left=2pt,"{U_0}"] \\
\SHk \ar[r,shift left=2pt,"M"]
 & \mathcal D \ar[l,shift left=2pt,"U"]
\end{tikzcd}
\]
where the rows are adjunctions, $M_0, M$ and $L$ are symmetric monoidal and $LM_0 \simeq M\Sigma^\infty_T$.
Suppose furthermore that $L$ is fully faithful and has a right adjoint $R$.

Then $U(\un_{\mathcal D}) \in \SHkeff$ if and only if $U_0(M_0 \Gm^{\wedge n})(\hat \Delta^\bullet_F) \simeq 0$ for $F$ as in Corollary \ref{coro:main}.
\end{coro}

\begin{proof}
First, observe that there is an isomorphism $\Omega^\infty_T U \simeq U_0 R$ since $LM_0 \simeq M\Sigma^\infty_T$. Moreover, $R L \simeq \id$ since $L$ is assumed to be fully faithful. For any $E \in \mathcal{D}_0$, we then get $\Omega^\infty_T ULE \simeq U_0 R L E \simeq U_0 E$. Next,
\[
(ULE)(\hat \Delta^\bullet_F) = (\Omega^\infty_T ULE)(\hat \Delta^\bullet_F) \simeq (U_0 E)(\hat \Delta^\bullet_F)
\] 
where the first equality is by definition.

By Corollary \ref{coro:main}, we have $U(\un_{\mathcal D}) \in \SHkeff$ if and only if $U(M \Gm^{\wedge n})(\hat \Delta^\bullet_F) \simeq 0$ for $F$ as stated. Note that $M \Gm^{\wedge n} \simeq LM_0  \Gm^{\wedge n}$ by assumption. Hence by the first paragraph, we find that $U(M \Gm^{\wedge n})(\hat \Delta^\bullet_F) \simeq (U_0 M_0 \Gm^{\wedge n})(\hat \Delta^\bullet_F)$. This concludes the proof.
\end{proof}


\section{Application to MW-Motives}\label{sec:AppMWmotives}

In this section, we apply the result of the previous section to the category of MW-motives. We have a diagram of functors 
\[
\begin{tikzcd}
[column sep=30pt,row sep=24pt]
\SHS \ar[r,shift left=2pt,"N"] \ar[d,"{\Sigma^\infty_T}"']
 & \DAek \ar[r,shift left=2pt,"{\derL \tilde \gamma^*}"] \ar[d,"{\Sigma^\infty_T}"]
     \ar[l,shift left=2pt,"K"]
 & \DMtek\ar[d,"{\Sigma^\infty_T}"']
     \ar[l,shift left=2pt,"{\gamma_*}"] \\
\SHk \ar[r,shift left=2pt,"N"]
 & \DA k\ar[r,shift left=2pt,"{\derL \tilde \gamma^*}"]
		\ar[l,shift left=2pt,"K"]
 & \DMtk
		\ar[l,shift left=2pt,"{\gamma_*}"]
\end{tikzcd}
\]
where the vertical functors are given by $T$-stabilization, the adjunctions in the right-hand square are those discussed in Section \ref{sec:recollections}, and the adjunctions in the left-hand square are derived from the classical Dold-Kan correspondence (see \cite{CD12_f}*{5.2.25} for the unstable version, and \cite{CD12_f}*{5.3.35} for the $\PP^1$-stable version). Both $N$ and $\derL \tilde \gamma^*$ commute with $T$-stabilization, and the stabilization functor 
\[
\Sigma_T^\infty: \DMtek\to \DMtk
\]
is fully faithful by \chcancellation, Corollary~\ref{cor:embedding}. It follows that the diagram
\[
\begin{tikzcd}
[column sep=30pt,row sep=24pt]
\SHS \ar[r,shift left=2pt,"{\derL \tilde \gamma^*N}"] \ar[d,"{\Sigma^\infty_T}"'] & \DMtek \ar[d,"{\Sigma^\infty_T}"] \ar[l,shift left=2pt,"{K\gamma_*}"] \\
\SHk\ar[r,shift left=2pt,"{\derL \tilde \gamma^*N}"]
 & \DMtk \ar[l,shift left=2pt,"{K\gamma_*}"]
\end{tikzcd}
\]
satisfies the assumptions of Corollary \ref{corr:main-simplified}. We can thus apply Theorem \ref{thm:rational} to obtain the following result, where $M:=\derL \tilde \gamma^*N$ and $U:=K\gamma_*$.

\begin{coro} \label{coro:effective}
In the stabilized adjunction $M: \SHk \leftrightarrows  \DMtkZ: U$,
we have $U(\un) \in \SHkeff$.
\end{coro}
\begin{proof}
Having Theorem \ref{thm:rational} and Corollary \ref{corr:main-simplified} at hand, the only subtle point is to show the following: if $E \in \DMtek$ has a fibrant model still denoted by $E$, then $K_s(E(\hat\Delta^\bullet_F)) \simeq (U_0 E)(\hat\Delta^\bullet_F)$, where $U_0 = K\gamma^*$. Here $K_s: \Der(\Ab) \to \SH$ denotes the classical stable Dold-Kan correspondence. Essentially this requires us to know that $K_s$ preserves homotopy colimits (at least we need filtered homotopy colimits and geometric realizations). This is well-known. In fact since this is a stable functor, it preserves all homotopy colimits if and only if it preserves arbitrary sums, if and only if its left adjoint preserves the compact generator(s), which is clear.
\end{proof}

We are now in position to prove our main result. To this end, recall that the motivic spectrum of abstract generalized motivic cohomology $\HtZ \in \SHk$ was defined in \cite{Bac17_f}*{\S 4} as the effective cover of the homotopy module of Milnor-Witt $K$-theory. Equivalently, $\HtZ$ is the effective cover of the homotopy module $\{\piaone_{n,n}(\mathbb{S})\}_n$, where $\mathbb{S}$ is the sphere spectrum.

\begin{thm} \label{thm:comparisonUHt}
Let $k$ be an infinite perfect field of exponential characteristic $e \ne 2$ and let 
\[
M: \SHk \leftrightarrows  \DMtk: U
\]
be the above adjunction. Then the spectrum $U(\un)$ representing MW-motivic cohomology with $\ZZ$-coefficients is canonically isomorphic to the spectrum $\HtZ$ representing abstract generalized motivic cohomology with $\ZZ$-coefficients. In particular, $U(\un) \in \SHkeff$.
\end{thm}
\begin{proof}
For an effective spectrum $E \in \SHkeff$, let $\tau_{\le 0}^\eff E \in \SHkeff_{\le 0}$ denote the truncation in the effective homotopy $t$-structure \cite{Bac17_f}*{Proposition~4}.

Let $X$ be (the localization of) a smooth scheme over $k$ and let $\HMW^{p,q}(X,\ZZ)$ be the MW-motivic cohomology groups introduced in \chfinitecw, Definition \ref{def:genmotiviccohom} (or equivalently in \chdmt, Definition \ref{def:generalizedMW}). We note that for $X$ local, (1) $\HMW^{n, 0}(X, \ZZ) = 0$ for $n \ne 0$ and (2) $\H^{0, 0}(X, \tZ)= \sKMW_0(X)$ by \chdmt, Proposition \ref{prop:explicit}. The unit map $\un \to U(\un)$ induces $\alpha: \HtZ \simeq \tau_{\le 0}^\eff \un \to \tau_{\le 0}^\eff U(\un) \simeq U(\un)$, where the first equivalence is by definition and the second since $U(\un) \in \SHkeff_{\le 0}$, by Corollary \ref{coro:effective} and (1). Now $\alpha$ is a map of objects in $\SH^{\eff,\heartsuit}$ (again by (1)) and hence an equivalence if and only if it induces an isomorphism on $\piaone_{0,0}$. This follows from (2).
\end{proof}

Next, we would like to show that ordinary motivic cohomology is represented by an explicit (pre-)sheaf in $\DMtk$. We start with the following lemma (see also \cite{G17_f}*{Theorem~5.3} and \cite{EK17_f}*{Theorem~1.1}).

\begin{lem} \label{lemm:modules}
Under the assumptions of the theorem, the category $\DMtekinv$ is equivalent
to the category of highly structured modules over $U(\un_{\DMtekinv}) \simeq H\tZ[1/e]$.
\end{lem}

\begin{proof}
Let $\mathcal{M}$ be this category of modules. By abstract nonsense \cite{MNN17_f}*{Construction~5.23} there is an
induced adjunction 
\[
M' :\mathcal{M} \leftrightarrows \DMtekinv: U'
\] 
which satisfies $U'M' (\un_{\mathcal M}) \simeq \un_{\mathcal M}$. Under our assumptions, the category
$\SHk[1/e]$ is compact-rigidly generated \cite{LYZ13_f}*{Corollary~B.2} and hence so are the categories $\mathcal{M}$ and 
$\DMtekinv$. It follows that $M'$ and $U'$ are inverse equivalences, see e.g.
\cite{Bac16_f}*{Lemma~22}.
\end{proof}

\begin{coro} \label{coro:modules-Z}
Under the same assumptions, the presheaf $\ZZ \in \DMtk$ represents
ordinary motivic cohomology with $\ZZ$-coefficients.
\end{coro}
\begin{proof}
Let $H = f_0 U(\ZZ)$. Then $\piaone_{0,0}(H) = \ZZ$ whereas $\piaone_{n,0}(H) = 0$ for $n \ne 0$. Also $\piaone_{-1,-1}(H) = (\piaone_{0,0}(H))_{-1} = 0$ and consequently $f_1 H = 0$, $s_0 H \simeq H$. The unit map $\un \to U(\un) \to U(\ZZ)$ induces $\un \to H$ and hence $\H\ZZ \simeq s_0(\un) \to s_0(H) \simeq H$. This is an equivalence since it is a map between objects in $\SH^{\eff,\heartsuit}$ inducing an isomorphism on $\piaone_{0,0}(\bullet)$. We have thus found a canonical map $\alpha: \H\ZZ \to f_0 U(\ZZ) \to U(\ZZ)$, which we need to show is an equivalence. We show separately that $\alpha[1/e]$ and $\alpha[1/2]$ are equivalences; since $e \ne 2$ this is enough.

We claim that $U(\ZZ)[1/e] \in \SHkeff$. This will imply that $\alpha[1/e]$ is an equivalence.
For $X \in \smk$ we have $UM(X)[1/e] = \Sigma^\infty X_+ \wedge U(\un)[1/e]$, by the
previous lemma. In particular $UM(X)[1/e] \in \SHkeff$. It follows that for $E
\in \DMtekinv$ we get $U(E) \in \SHkeff$ (indeed $U$ commutes with filtered colimits, being right adjoint to a functor preserving compact generators). This applies in particular to $E = \ZZ[1/e]$.

Recall that if $E \in \SHk$, then $E[1/2]$ canonically splits into two spectra, which we denote by $E^+$ and $E^-$. They are characterised by the fact that the motivic Hopf map $\eta$ is zero on $E^+$ and invertible on $E^-$ \cite{Bac17-2_f}*{Lemma~39}.
Now consider $U(\ZZ)[1/2]$. The action of $\sKMW$ on $\piaone_{0,0}(U\ZZ) = \ZZ$ is by definition via the canonical epimorphism $\sKMW_0 \to \sKM_0= \ZZ$. This implies that $(U\ZZ)^- = 0$, just like $(H\ZZ)^- = 0$. On the other hand $\ZZ^+ \in \DMtek^+ \simeq \DMe{k,\ZZ[1/2]}$ \chdmt, \S\ref{sec:relordmot} is the unit, by construction, whence $U\ZZ^+ = \H\ZZ[1/2]$.
\end{proof}

\begin{exm}[Grayson's Motivic Cohomology]
\index{Grayson's motivic cohomology}%
In \cite{S03_f}, Suslin proves that Grayson's definition of motivic cohomology coincides with Voevodsky's. To do so he proves that Grayson's complexes satisfy the cancellation theorem, and then employs an induction using poly-relative cohomology. We cannot resist pointing out that the second half of this argument is subsumed by our criterion. Indeed, it is easy to see that $\K_0^\oplus$-presheaves admit framed transfers in the sense of \cite{GP14_f}*{\S 2}. Consequently the $\Aone$-localization functor for Grayson motives is given by $\mathrm{L}_{\nis} \Cstar{-}$ (\cite{GP15_f}*{Theorem~1.1}). Arguing exactly as in the proof of Corollary \ref{coro:effective} (using \cite{S03_f}*{Remark~2.3} instead of Proposition \ref{prop:ratcontractible}) we conclude that the spectrum $\H\ZZ^{Gr}$ representing Grayson's motivic cohomology is effective. But $\ZZ^{Gr}(0) \simeq \ZZ$ and so $\H\ZZ \simeq \H\ZZ^{Gr}$, arguing as in the proof of Theorem~\ref{thm:comparisonUHt}.
\end{exm}

\begin{exm}[GW-motives]
\index{Grothendieck-Witt motives}%
In \cite{Dr17-3_f}, a category of GW-motives $\DMGWk$ is defined and the usual properties are established. Arguing very similarly to the proof of Proposition \ref{prop:ratcontractible}, one may show that the reduced GW-presheaf corresponding to $\Gm^{\times n}$ is rationally contractible. Then, arguing as in Theorem \ref{thm:comparisonUHt} and Lemma \ref{lemm:modules}, using the main results of \cite{Dr17-3_f, Dr17-2_f, Dr17_f}, one may show that the spectrum representing $\un \in \DMGWk$ is $\HtZ$ again, and that $\DMGWk$ is equivalent to the category of highly structured modules over $\HtZ$. In particular $\DMGWk \simeq \DMtk$. We leave the details for further work.
\end{exm}

\begin{rem}
The assumption that $k$ is infinite in our results can be dropped by employing the techniques of \cite{EHKSY17_f}*{Appendix~B}.
\end{rem}

\appendix

\backmatter
\printindex
\printindex[notation]

\end{document}